\documentclass[a4paper,leqno]{amsart}
\usepackage{a4wide}
\usepackage{color}
\usepackage{amsfonts}
\usepackage{amsmath, amsfonts, amsthm, amssymb, amscd}
\usepackage[mathcal]{euscript}
\usepackage{mathrsfs}
\usepackage{dsfont}
\usepackage{ulem}
\usepackage{bbm}
\usepackage{graphicx}

\usepackage{mathtools}
\usepackage{slashed}
\usepackage{lineno} 
\usepackage{multicol}

\usepackage{tikz-cd}

\usepackage{fancyhdr}
\newtheorem{theorem}{Theorem}[section]
\newtheorem{proposition}[theorem]{Proposition}
\newtheorem{lemma}[theorem]{Lemma}
\newtheorem{corollary}[theorem]{Corollary}
\newtheorem{definition}[theorem]{Definition}
\newtheorem{remark}[theorem]{Remark}

\numberwithin{equation}{section}

\usepackage{multirow}
\usepackage{tabularx}
\usepackage{lscape}
\newcommand \Rbb {\mathbb{R}}

\newcommand \Ecal {\mathcal E}
\newcommand \Fcal{\mathcal{F}}
\newcommand \Hcal {\mathcal H}

\newcommand \Kcal {\mathcal K}

\newcommand \Pcal {\mathcal P}

\newcommand \Scal{\mathcal{S}}
\newcommand \Tcal{\mathcal T}
\newcommand \TPcal{\mathcal{TP}}

\newcommand \Econ {{\bf E}_{con}}
\newcommand \econ {{\bf e}_{con}}

\newcommand \nablab{\bar{\nabla}}

\newcommand \delb {\bar {\del}}

\newcommand \Qb{\bar Q}

\newcommand \Tb {\overline {T}}
\newcommand \Phib{\overline{\Phi}}
\newcommand \Psib{\overline{\Psi}}

\newcommand {\chih}{\hat{\chi}}
\newcommand {\delh}{\hat{\del}}

\newcommand {\hh}{\hat{h}}

\newcommand \m{\text{m}}

\newcommand \del \partial
\newcommand \delu {\uline{\del}}

\newcommand \Tu {\uline{T}}

\newcommand \minu{\uline{\text{m}}}

\newcommand \Psiu{\uline{\Psi}}
\newcommand \Phiu{\uline{\Phi}}
\newcommand \Qu{\uline{Q}}

\newcommand {\delt} {\tilde{\del}}

\newcommand {\Tt} {\tilde{T}}

\newcommand {\Phit}{\tilde{\Phi}}
\newcommand {\Psit}{\tilde{\Psi}}
\newcommand {\Qt}{\tilde{Q}}

\newcommand {\vt}{\tilde{v}}

\newcommand \RR{\mathbb{R}}

\newcommand {\vep}{\varepsilon}

\newcommand {\rhoH} {\rho^{\mathcal{H}}}

\newcommand {\near} {\text{\sl near}}
\newcommand {\far} {\text{\sl far}}
\newcommand {\la}{\langle}
\newcommand {\ra}{\rangle}

\newcommand {\Ebf}{{\bf E}}
\newcommand {\ebf}{{\bf e}}
\newcommand {\Fbf}{{\bf F}}

\newcommand {\ord}{\text{ord}}
\newcommand {\rank}{\text{rank}}

\newcommand {\ih}{\hat{\imath}}
\newcommand {\jh}{\hat{\jmath}}
\newcommand {\kh}{\hat{k}}
\newcommand {\lh}{\hat{l}}
\newcommand {\xh}{\hat{x}}
\newcommand {\phih}{\hat{\phi}}

\makeatletter
\def\hlinew#1{%
  \noalign{\ifnum0=`}\fi\hrule \@height #1 \futurelet
   \reserved@a\@xhline}
\makeatother

\begin{document}
\title{Nonlinear stability of a type of totally geodesic wave maps in non-isotropic manifolds}

\author[S.~Duan] {Senhao Duan}
\address{D\'epartement ing\'enierie Math\'ematique et informatique, \'Ecole des Ponts ParisTech, Paris France, 77420}
\email{senhao.duan@eleves.enpc.fr}

\author[Y.~ Ma] {Yue MA}
\address{School of Mathematics and Statistics, Xi'an Jiaotong University, Xi'an, Shaanxi 710049, China.}
\email{yuemath@xjtu.edu.cn}

\author[W. ~Zhang] {Weidong Zhang}
\address{School of Mathematics and Statistics, Xi'an Jiaotong University, Xi'an, Shaanxi 710049, China.}
\email{zwd13892650621@stu.xjtu.edu.cn}

\maketitle

%
%
\begin{abstract}
In this article we investigate a type of totally geodesic map which has its image being a geodesic in an non-isotropic Riemannian manifold. We consider its nonlinear stability in the family of wave maps. We first establish the factorization property. Then with a specially constructed chart of geodesic normal coordinates, we formulate this stability problem into a Cauchy problem associate to a wave-Klein-Gordon system with small initial data. Compared with the isotropic cases (Abbrescia-Chen (2021), Duan-Ma (2022) and Dong-Wyatt (2021)), this system contains some new nonlinearities. Furthermore, we remove the restriction on the support of initial data. Our strategy relies on a generalization of the hyperboloidal foliation, which smoothly foliates the entire (future) space time. A conformal energy estimate on these hypersurfaces and a type of normal form transform are also developed in order to treat the newly appeared nonlinear terms.  
\end{abstract}
\section{Introduction}
\subsection{Backgrounds and objectives}
Let $(M,g)$ be a Lorentzian manifold and $(N,h)$ be a Riemannian manifold. A sufficiently regular map $\phi: (M,g)\mapsto (N,h)$ is called a {\it wave map}, if it is formally a critical point of the following Lagrangian action functional
\begin{equation}\label{eq1-09-02-2022}
S[\phi] = \int_{M}(\phi_*, \phi_* )_{T^*M\otimes \phi^*(TN)}\, d\text{V}_g.
\end{equation}
The research on wave maps has been a attractive theme for a long time. One may see \cite{Shatah-1998,Kri07} for a general review. In the present article we discuss a special type of such maps, called the {\it totally geodesic wave maps}. They are wave maps which preserve geodesics. More precisely, a map $\phi: M\mapsto N$ is called {\sl totally geodesic}, if $\phi\circ\gamma$ being a geodesic of $N$ provided that $\gamma: (-\vep,\vep)\mapsto M$ is a geodesic. 

A wave map is not necessarily totally geodesic, while it can be easily checked that a totally geodesic map is always a wave map (see in detail below \eqref{eq-main-geo}). Then a natural question arises: are the totally geodesic maps stable in the family of wave maps? 

This problem turn out to be complicated in general case. However, when $(N,h)$ is a space form (simply connected with constant sectional curvature) and $M = \RR^{1+d}$ the standard Minkowski space with $\phi(M)$ being the image of a geodesic of $N$ (which are called {\it rank-one} maps, see Remark \ref{rk1-11-04-2022-M} for detail.), some results are obtained. To be more precise, we formulate this problem as follows.  Given a totally geodesic map defined in $\RR^{1+d}$ and perturb it on $\{t=0\}$. Does there exist a wave map defined in $\RR^{1+d}$ (or at least locally for $t\leq T$) such that the restriction of this wave map coincides with the perturbed map on $\{t=0\}$? 

As formulated in \cite{Ab-2019} with the {\sl geodesic normal coordinates} introduced in \cite{Gray-2003}, the perturbed wave map, if exists, satisfies a hyperbolic partial differential system composed by one wave equation and $d$ Klein--Gordon equations. The geometric stability problem then reduces to a global existence problem of the associate PDE system with small initial data. With the techniques of vector fields and the application of hyperbolic variables / hyperboloidal foliation, the case $d\geq 3$ was established in \cite{Ab-2019} and the case $d=2$ was established in \cite{Duan-Ma-2020}. In \cite{Dong-Wyatt-2021}, the demand on the initial regularity is greatly relaxed via a much simpler technique. All the above results demand that the initial perturbation should be of compact support.  

In the present article, we first give a complete classification on the rank-one totally geodesic maps when the departure is $\RR^{1+d}$. Then we concentrate on the critical case $d=2$ and establish the nonlinear stability result in a more general context. More precisely, the target manifold $N$ need not to be isotropic, and the initial perturbation need not to be compactly supported.

Form a PDE point of view, the totally geodesic wave maps belong to a particular class of solutions to the wave map equation \eqref{eq-main-geo} or \eqref{eq1-18-11-2021}. The above stability problem then reduces to the stability problem of these particular solutions. For this reason, we mention particularly the results of global well-posedness of the wave maps \cite{Chrsitodoulou-Shadi-1993-1,Chrsitodoulou-Shadi-1993-2,Muller-Struwe-1996,Freire-1997,Tao-2000,Klainerman-Selberg-2001,Tataru-2001,Kri-2004,Tataru-2005,Sterbenz-Tataru-2009,Kir-2009,Sterbenz-Tataru-2010,Wang-Yu-2012,Lawrie-Oh-2016,Gavrus-Jiao-Tataru-2021} etc. in the critical dimension $2+1$. We expect that some techniques or ideas may be shared between these two closely related themes of research.

\subsection{Notation and main result}

In Section \ref{sec1-09-02-2022} we will show that a totally geodesic map $\varphi: M\mapsto N$ is of constant rank, i.e., the dimension of the image of the tangent map $\varphi_*$ is constant (Remark \ref{rk1-11-04-2022-M}). Furthermore, a totally geodesic map $\varphi: \RR^{1+d}\mapsto N$ with rank one can be factorized as following (see in detail in Proposition \ref{prop1-09-02-2022}):
\begin{equation}\label{eq2-09-03-2022-M}
\varphi:\RR^{1+d}\stackrel{\varphi_S}{\longrightarrow}\RR\stackrel{\varphi_I}{\longrightarrow} N,
\end{equation}
where $\varphi_S$ is affine and $\varphi_I$ is totally geodesic. We will show that $\m(d\varphi_S,d\varphi_S)$ is constant in $\RR^{1+d}$ (see the discussion at the end of Subsection \ref{subsec1-14-04-2022-M}). The map $\varphi$ is said to be {\it space-like / null /  time-like}, if $d\varphi_{S}$ is space-like / null / time-like. Remark that in this case, $\varphi$ is of {\sl infinite energy} in sens of \eqref{eq1-09-02-2022}. This is different form \cite{Grigoryan2010,Sideris1989} where $\varphi_S$ an arbitrary {\it finite energy solution} to a wave equation, which leads to the fact that $\varphi$ is not totally geodesic (see \cite{Ab-2019}).

It is clear that $\varphi_I(\RR)$ is a geodesic of $N$. Its geometric properties influence essentially the nonlinear stability of $\varphi$. For our purpose we introduce the following geometric descriptions on a geodesic of $N$. 
\begin{definition}\label{def2-09-03-2022-M}	
Let $\gamma$ be a complete geodesic of a Riemannian manifold $(N,h)$. 
\\
$\bullet$ The manifold $N$ is said to be locally symmetric along $\gamma$, if 
\begin{equation}\label{eq1-17-02-2022-M}
\nabla_{\dot\gamma}\nabla^j R\big|_{\gamma} = 0,\quad \forall j\in\mathbb{N},
\end{equation}
where $\nabla$ is the associate metric connection and $R$ is the Riemann curvature tensor. 
\\
$\bullet$ Let $\kappa(\cdot,\cdot)$ be the symmetric quadratic form acting on $\dot\gamma^{\perp}$ defined by
\begin{equation}\label{eq1-09-03-2022-M}
\kappa(v,w) = \big(R(v,\dot\gamma)\dot\gamma,w\big)_h\big|_{\gamma},	
\end{equation}
then $\gamma$ is said to be:
\\
1. {\sl time-like linear stable}, if $\kappa(\cdot,\cdot)$ is  {positive} defined along $\gamma$; 
\\
2. {\sl space-like linear stable}, if $\kappa(\cdot,\cdot)$ is  {negative} defined along $\gamma$.
\\
$\bullet$ Let $\lambda_{\ih}$ be the  eigenvalues of $\kappa$ with $\lambda_{\ih}\neq 0$. Then $\gamma$ is said to be {\it non-resonant}, if at each point of $\gamma$,
\begin{equation}\label{eq2-09-03-2022_M}
|\lambda_{\ih}|^{1/2} + |\lambda_{\jh}|^{1/2}\neq |\lambda_{\kh}|^{1/2}.
\end{equation}
\end{definition} 
As we will see through the PDE system formulated in \eqref{eq5-20-11-2021}, the case with $d\geq 3$ is relatively easy and we do not need the non-resonant assumption. In the present article we concentrate on the two-dimensional case with $\varphi_I(\RR)$ satisfying the above three assumptions. A typical example in our case is the
$2$-torus embedded in $\RR^3$ with its two equators: one is time-like linear stable and the other is space-like linear stable. In Section \ref{sec1-13-05-2022-M} we will present more examples in detail.

From \eqref{eq2-09-03-2022-M}, $\varphi(\RR^{1+2})$ is the image of the geodesic $\varphi_I(\RR)$. Then we state the main result.
\begin{theorem}[Main result, rough version]
Let $\varphi$ be a rank-one totally geodesic wave map from $\RR^{1+2}$ to a Riemannian manifold $(N,h)$. Suppose that $\varphi(\RR^{1+2})$ is the image of a non-resonant geodesic of $N$, and $N$ is locally symmetric along $\varphi(\RR^{1+2})$. Then $\varphi$ is nonlinear stable provided that $\varphi$ is time-like / space-like and $\varphi_I(\RR)$ time-like / space-like linear stable.
\end{theorem}

\subsection{Main difficulties and strategy of proof}
We first explain the roles of the assumptions made on $\varphi(\RR^{1+2})$. If we follow the formulation in \cite{Ab-2019} and carefully avoid the simplifications based on the fact that $N$ is a space form, we arrive at \eqref{eq1-19-11-2021}. However, this system is too genera to be handle.

The first difficulty comes from the coefficients of this system, which are derivatives of the Christoffel symbols. Without the assumption of being space form, these coefficients (and their derivatives) are not even bounded. That is why we need some type of symmetry assumption along the geodesic. The most typical choice is the local symmetry along $\varphi(\RR^{1+2})$ announced in Definition \ref{def2-09-03-2022-M}. It implies that the coefficients of \eqref{eq1-19-11-2021} are constants (Proposition \ref{prop1-19-02-2022-M}). In fact our mechanism can handle the situations ``not far from'' it. For example we accept the following weaker assumption:
\begin{equation}\label{eq1-10-03-2022-M}
	\big|\nabla_{\dot\gamma}^k\nabla^jR|_{\gamma}\big|\leq C\ell^{-k},
\end{equation}
where $\ell$ is a fixed arc-length parameter. This assumption is clearly verified in the isotropic case by the homogeneity of the space forms.

It is nature to analyze the linearized system before regarding the fully nonlinear one. Due to the non-isotropy of the target manifold, the mass of each Klein-Gordon equation may differ from one to another. We need to guarantee the positivity of all these masses. This leads to the assumption of linear stability.  On the other hand, $\kappa(\vec{e}_{\ih},\vec{e}_{\ih}) = -\frac{R(\vec{e}_{\ih},\vec{e}_1,\vec{e}_{\ih},\vec{e}_1)}{\|\vec{e}_{\ih}\wedge \vec{e}_1\|^2}$ is a component of the sectional curvature. This assumption is essentially a sign condition on the sectional curvature. In the isotropic case analyzed in \cite{Ab-2019}, $N$ is time-like linear stable provided that $ {\kappa>0}$ and space-like linear stable if $ {\kappa<0}$. 

The fact that the Klein-Gordon equations may have different mass may bring us a very important and inconvenient phenomena called the {\it mass resonance}. This was revealed by H. Sunagawa in \cite{Sunagawa-2003}. More precisely, in lower dimensional spacetime, the asymptotic behavior of a system composed by Klein-Gordon equations may depend on the ratio of the masses of the components. For example (see \cite{Tsutsumi-2003-2,Katayama-Ozawa-Sunagawa-2012,Dfx,KS-2011}), a Klein-Gordon system composed by two components with quadratic nonlinear terms in 1+2 dimensional spacetime always enjoys global existence with small localized initial data, provided that their masses {\it do not} satisfy the relation $m_1 = 2m_2$. And when it occurs, additional structural conditions are necessary for global existence. 

In the present case, due to the lack of homogeneity of the target manifold, the unpleasant Klein-Gordon quadratic terms appear on the right-hand side of the Klein-Gordon equations of \eqref{def2-09-03-2022-M}. Together with the fact that we have different masses, we face the risk of mass resonance. In this initial step for non-isotropic target manifold, we concentrate on the {non-resonant} case, i.e., the assumption \eqref{eq2-09-03-2022_M} is made.


Taking the above assumptions, the main PDE system reduces to \eqref{eq5-20-11-2021}. Regarding this system we have two main challenges. The first is to remove the restriction on the support of initial data. A natural idea is to combine the hyperboloidal foliation in the light cone with the flat foliation outside. See \cite{K-W-Y-2018} for an analysis on the  massive Maxwell-Klein-Gordon system in the complement of a fixed light cone, and more recently,  \cite{Stingo-Huneau-2021} on a quasilinear wave system in the entire space time. 
Our approach relies on a generalized hyperboloidal foliation introduced in \cite{M-2018} and \cite{LM-2022}(which is called  ``Euclidean-hyperboloidal foliation'' therein) . This is a smooth combination of hyperboloidal foliation in the interior of the light cone $\{r\leq t-1\}$ and flat foliation outside of the light cone $\{r\geq t\}$. The global solution will be constructed simultaneously in and out of the light cone $\{t=r-1\}$. In the present article we reformulate this method in $\RR^{1+2}$ (Sections \ref{sec1-11-04-2022-M} -- \ref{sec1-22-10-2021}). Then we make some necessary improvements on Soblev decay estimates in Section \ref{sec1-10-04-2022-M} in order to be adapted to this lower dimensional case. 
For other type of techniques aimed at wave-Klein-Gordon system with non-compactly supported initial data, the reader is referred to \cite{Stingo-2015,Ionescu-Pausader-2017,Stingo-2018,Stingo-2019,Dong-2020,Dong-2021,Dong-Ma-2021,Ionescu-Pausader-2022,Dong-Ma-Yuan-2022-1}. 

As showed in \cite{Duan-Ma-2020, Dong-Wyatt-2021}, the application of divergence structure  enjoyed by the pure Klein-Gordon quadratic terms coupled in the wave equation of \eqref{eq5-20-11-2021} demands the application of conformal energy estimate on hyperboloids (see also \cite{Dong-Wyatt-2021-2,Dong-Wyatt-2020} and \cite{Kubota-Yokoyama-2001,Ka} for an alternative treatment for this divergence structure). This is because we need the fact that a cubic wave equation enjoys much better decay compared with the energy-Sobolev decay.  In the present case, with sufficient decay assumptions at spatial infinity on the initial data, one will see that this property still holds through a conformal energy estimate on the Euclidean-hyperboloidal slices. This will be established in Section \ref{sec2-11-04-2022-M}. 

A second difficulty comes from the pure Klein-Gordon quadratic terms coupled in the Klein-Gordon equations. Traditionally they can be handled by a normal form transform method which dates back to the pioneer work \cite{Shatah85} in $\RR^{1+3}$. In \cite{Ozawa-1996,Dfx} this idea was applied in on Klein-Gordon systems in $\RR^{1+2}$. However, the application of the hyperbolic variables in \cite{Dfx} demands the compactness of the support of the initial data. On the other hand, the normal form techniques realized in Fourier space such as \cite{Shatah85,Ozawa-1996} can not be easily integrated in the framework of Euclidean-hyperboloidal foliation. In the present article, following a similar idea in \cite{Yagi-1994,Moriyama-1997,Katayama-1999}, we introduce an alternative formulation purely in physical space and independent of the hyperbolic variables. This will be presented in Section \ref{sec1-03-04-2022-M}. This formulation, combined with the Euclidean-hyperbolidal foliation, can be applied on general Klein-Gordon systems with non-compactly supported initial data. Furthermore, based on this technique, one demands less decay at spatial infinity on the initial data for linear decay rate (say, compared with the famous result \cite{Georgiev-1992}, where one needs a initial weight $\la r\ra^{-1-\delta}$, and we only need $\la r\ra^{-1/2-\delta}$).

\subsection{Organization of the article}
This article is composed by three parts. The first part aims at the geometric preparations including two aspects. The first is the factorization property of rank one totally geodesic maps from $\RR^{1+d}$ to a general Riemannian manifold. This is contained in Section \ref{sec1-09-02-2022} . The second is the formulation of the PDE system based on the geodesic normal coordinates and the assumptions introduced in Definition \ref{def2-09-03-2022-M}. This will be done in Section \ref{sec1-19-02-2022-M} and \ref{sec1-08-02-2022}.

Part II is devoted to the general framework of the Euclidean-hyperboloidal foliation in $\RR^{1+2}$. The structure of this part has been explained in the previous subsection.

In Part III, we firstly recall the auxiliary system in Section \ref{sec3-11-04-2022-M}. Then we apply the techniques in Part II on the system \eqref{eq5-20-11-2021} formulated in Part I. Our strategy is the standard bootstrap argument, which is initialized in Section \ref{sec4-11-04-2022-M}. For the convenience of the reader, some details included in \cite{M-2018} and \cite{LM-2022} about the Eulidean-hyperboloidal foliation are sketched in Appendix.

\subsection{Convention}
In $\RR^{1+2}$, Greek indices $\alpha,\beta,\gamma,\cdots$ take values in $\{0,1,2\}$ for spacetime components while Latin indices $a,b,c$ take values in $\{1,2\}$ for space components. 
In the target manifold $(N,h)$, the Latin indices $i,j,k,l$ take values in $\{1,2,\cdots,n\}$, and $\ih, \jh, \kh, \lh$ take values in $\{2,3,\cdots, n\}$ for ``transversal components''(which will be explained latter). The Einstein convention on summation is applied unless otherwise specified.  When repeating $\alpha,\beta,\gamma,\cdots$, the sum is taken over $\{0,1,2,3\}$ and when repeating $\ih,\jh,\kh,\lh,\cdots$, the sum is taken over $\{2,3,\cdots,n\}$, etc.
\subsection*{Acknowledgement}
The present work belongs to a research project ``Global stability of quasilinear wave-Klein--Gordon system in $2 + 1$ space-time dimension'' (11601414), supported by NSFC.

\part{Geometric aspects of the totally geodesic wave maps.}
\section{Factorization properties of totally geodesic maps}\label{sec1-09-02-2022}
\subsection{Affine maps and totally geodesic maps}\label{subsec1-18-11-2021}
We consider two pseudo-Riemannian manifolds $(M,g)$ and $(N,h)$, and a $C^{\infty}$ map $\phi: M\mapsto N$. We denote by $\phi_*: TM\mapsto TN$ the tangent map of $\phi$ and $\phi^*(TN)$ the pull-back bundle on $M$. Let $\nabla$ and $\nablab$ be the metric connections on $M,N$ associate to $g,h$ respectively. Then $\phi_*$ becomes a linear map from $TM$ to $\phi^*(TN)$. It is then a section of the bundle $L(TM,\phi^*(TN))$. We introduce the connection $D$ on $L(TM,\phi^*(TN))$ as follows:
\begin{equation}\label{eq3-24-02-2022-M}
D_vA(w) := \nablab_{\phi_*(v)}A(w) - A(\nabla_v w),
\end{equation}
where $A\in L(TM,\phi^*(TN))$.
$D\phi_*$ can be regarded as a quadratic form defined in $TM$ (and taking value in $\phi^{*}(TN)$)):
$$
\beta(\phi)(v,w) := D_v\phi_*(w) = \nablab_{\phi_*(v)}\phi_*(w) - \phi_*(\nabla_vw).
$$
Recalling that $\nablab,\nabla$ are torsion-free, one has (thanks to Lemma 1.4 of \cite{Vil70}), $\beta(\phi)(v,w) = \beta(\phi)(w,v)$. 

The map $\phi$ is called {\sl affine}, if it preserves connections, i.e., $D\phi_* = 0$. $\phi$ is called {\sl totally geodesic}, if it maps a geodesic of $M$ to a geodesic of $N$, i.e.,
$$
\phi\circ\gamma: (a,b)\mapsto N \text{ being a geodesic, provided that } \gamma: (a,b)\mapsto M \text{ is a geodesic},
$$
or equivalently (see Lemma 1.4 of \cite{Vil70}),  
\begin{equation}\label{eq4-24-02-2022-M}
\beta(\phi)(v,v) = D_v\phi_*(v) = 0\ \Leftrightarrow\ \nablab_{\phi_*(v)}\phi_*(v) = \phi_*(\nabla_vv).
\end{equation}
In our context, $\beta(\phi)$ is symmetric. Thus $\phi$ is affine iff $\phi$ is totally geodesic. 

\subsection{Factorization property of totally geodesic maps}
We first recall the following notation. Let $(M,g)$ be a $d-$dimensional manifold. Let $m\in M$. A proper subspace $T_mM^1$ of $T_mM$ is called {\sl non-degenerate}, if the restriction of $g$ on $T_mM^1$ is non-degenerate. The holonomy group $\Phi_m$ of $M$ at $m$ is said to be {\sl non-degenerately reducible}, if it preserves a non-degenerate proper subspace of $T_mM$. We say $M$ is complete in the sens that all geodesics can be infinitely extended in terms of the parameter.  An affine map $\phi: (M,g)\mapsto (N,h)$ is called {\sl non-degenerate} if $\exists m\in M$, $\ker(\phi_*)_m\subset T_mM$ is a non-degenerate subspace.

From now on we restrict ourselves to the case where $(M,g)$ is complete pseudo-Riemannian, simply connected, and $(N,h)$ is Riemannian.  We will establish the following result. 
\begin{proposition}\label{thm1-16-11-2021}
Let $(M,g)$ be a complete simply connected pseudo-Riemannian manifold and $(N,h)$ a Riemannian manifold. Let $\phi: M\mapsto N$ be a non-degenerate totally geodesic map. Then $\phi$ factors into a totally geodesic  submersion followed by a totally geodesic immersion.
\end{proposition}
This result is {\it not} really necessary for the factorization result Proposition \ref{prop1-09-02-2022} from $\RR^{1+d}$ to a Riemannian manifold (for the proof one may rely on a Riemannian version). However, we would like to present it here for our further work in more general cases. The proof is relied on the following result established in \cite{Wu-1962}:
\begin{theorem}[The Decomposition Theorem of de Rham -- Wu]\label{thm-wu}
	Let $(M,g)$ be a complete simply connected pseudo-Riemannian manifold. Suppose that the holonomy group $\Phi_m$ of $M$ at $m\in M$ is {\sl non-degenerately reducible}; let $T_mM^1$ and $T_mM^2 = (T_mM^1)^{\perp}$ be the subspaces of $T_mM$ left invariant by $\Phi_m$. Then $M$ is naturally isometric to the direct product of the maximal integral manifolds of distributions $T_1$ and $T_2$, obtained by parallel transport of $T_mM^1$ and $T_mM^2$ over $M$. More precisely, let $M_i$ be the integral manifold of $T_i$ through $m$, $i=1,2$. Then there exists an isometry $\psi: M_1\times M_2\mapsto M$ which maps $(M_1,m)$ identically onto $M_1$ and $(m, M_2)$ identically onto $M_2$.
\end{theorem}
In order to apply the above result, we need the following lemma.
\begin{lemma}\label{lem1-11-04-2022-M}
Suppose that $\phi$ is totally geodesic (thus affine). Then $\ker(\phi_*)$ is holonomy-invariant. Furthermore, if $\ker(\phi_*)$ is non-degenerate at $m\in M$, then it is non-degenerate everywhere.
\end{lemma}
\begin{proof}
This is essentially the Proposition 1.3 of \cite{Vil70}. In fact we prove that, for $v\in \Gamma(M,TM)$ a vector field parallel along a curve $\gamma:(a,b)\mapsto M$, then $(\phi_*(v),\phi_*(v))_h$ is constant along $\gamma$. To see this, we remark that
$$
\frac{d}{dt}(\phi_*(v),\phi_*(v))_h = 2\big(\nablab_{\phi_*(\dot\gamma)}\phi_*(v),\phi_*(v)\big)_h = 2(\phi_*(\nabla_{\dot\gamma}v),\phi_*(v))_h = 0.
$$
Then if at one point $m\in M$, $\phi_*(v(m)) = 0$, then $\phi_*(v) = 0$ along the curve. That is, $\ker(\phi_*)$ is invariant under parallel transport thus holonomy-invariant.

For the second point, we only need to remark that the parallel transport is an isometry from one tangent space to another.
\end{proof}
\begin{remark}\label{rk1-11-04-2022-M}
The fact that $\ker(\phi_*)$ is invariant under parallel transport leads to the fact that the dimension of $\ker(\phi_*)$ is constant, and thus $\phi$ is of constant rank.
\end{remark}

\begin{proof}[Proof of Proposition \ref{thm1-16-11-2021}]
According to the notation of Theorem \ref{thm-wu}, we take 
$$
T_1 = \ker(\phi_*),\quad T_2 = \ker(\phi_*)^{\perp},
$$ 
which coincide with the distributions obtained by parallel transport of $\ker(\phi_*)_m$ and $\ker(\phi_*)_m^{\perp}$ (Lemma \ref{lem1-11-04-2022-M}). Recalling that $\ker(\phi_*)_m$ is non-degenerate. Then there exists an isometry $\psi: M \mapsto M_1\times M_2$ where $M_1$ is a maximal integral manifold of $\ker(\phi_*)$. Recalling that on each slice $M_1\times \{y\} $, $\phi\circ\psi^{-1}$ is constant. Then $\phi$ factors as follows:
$$
\begin{tikzcd}
	M \arrow[d, "\psi"] \arrow[rr, "\phi"] \arrow[rrd, "\varphi_S"] &  & N                          \\
	M_1\times M_2 \arrow[rr, "\pi_2"]                               &  & M_2 \arrow[u, "\varphi_I"]
\end{tikzcd}
$$
where $\varphi_S = \pi_2\circ\psi$ is obviously a submersion and totally geodesic (thus affine). 

To show $\varphi_I$ is an immersion, we consider a point $(x,y)\in M_1\times M_2$ with $\psi^{-1}(x,y) = m\in M$. Let $v\in T_yM_2$ and $\gamma:(-\vep,\vep)\mapsto M_2$ a $C^{\infty}$ curve through $y$ with $\gamma (0) = y$, $\dot\gamma(0) = v$. Then 
$$
\tilde{\gamma}:(-\vep,\vep)\mapsto M,
\qquad
t\mapsto \psi^{-1}(x,\gamma(t))
$$
is a $C^{\infty}$ curve with $\tilde\gamma(0) = m$ and $\dot{\tilde{\gamma}}(0) = \psi_*^{-1}(0,v)$. Remark that
$\bar{\gamma}:=\varphi_I\circ \gamma = \phi\circ\tilde{\gamma}$ is a $C^{\infty}$ curve in $N$ with $\bar{\gamma}(0) = \phi(m)$. Then 
\begin{equation}\label{eq1-17-11-2021}
\dot{\bar{\gamma}}(t) = (\varphi_I)_*(v) = \phi_*\circ\psi^{-1}_*(0,v).
\end{equation}
When $v\in\ker(\varphi_I)_*$, one has
$\phi_*\circ\psi^{-1}_*(0,v) = 0$, thus $\psi^{-1}_*(0,v)\in\ker (\phi_*)$. On the other hand, by construction $\psi^{-1}(0,v)\in \ker(\phi_*)^{\perp}$. Then due to the fact that $\ker(\phi_*)$ is non-degenerate, one has $\ker(\phi_*)^{\perp}\cap \ker(\phi_*) = \{0\}$, thus $v = 0$. Then $\ker(\varphi_I) = \{0\}$, i.e., $\varphi_I$ is immersion.

Finally, we show that $\varphi_I$ is affine. For this purpose we calculate $\beta(\varphi_I)$. Remark that $\eta:=\phi\circ\psi^{-1}$ is affine from $M_1\times M_2$ to $N$. We denote by $\nabla^P$ the metric connection defined in $M_1\times M_2$. Then $\nabla^P = \nabla^{M_1}\oplus \nabla^{M_2}$. On the other hand, by \eqref{eq1-17-11-2021} one has $(\varphi_I)_*(v) = \eta_*(0,v)$.  Then
$$
\aligned
(\varphi_I)_*\big(\nabla^{M_2}_vw\big) &= \eta_*(0,\nabla^{M_2}_vw) = \eta_*(\nabla^P_{(0,v)}(0,w)) = \nablab_{\eta_*(0,v)}\eta_*(0,w) 
\\
&= \nablab_{(\varphi_I)_*(v)}(\varphi_I)_*(w).
\endaligned
$$
This complete the proof.
\end{proof}

\subsection{Rank-one totally geodesic maps defined in Minkowski space}\label{subsec1-14-04-2022-M}
From now on we consider totally geodesic maps from Minkowski $(\RR^{1+d},\m)$ to a Riemannian manifold $(N,h)$. We establish the following result.
\begin{proposition}\label{prop1-09-02-2022}
Let $\varphi: (\RR^{1+d},m)\mapsto N$ a totally geodesic map and $\rank(\varphi_*) = 1$. Then $\varphi$ factors as follows:
\begin{equation}
\varphi:\RR^{1+d}\stackrel{\varphi_S}{\longrightarrow}\RR\stackrel{\varphi_I}{\longrightarrow} N
\end{equation}
with $\varphi_S$ totally geodesic submersion and $\varphi_I$ totally geodesic immersion.
\end{proposition} 
\begin{remark}
	When $\varphi$ totally geodesic (thus affine), the rank of $\varphi_*$ is constant. Thus $\rank(\varphi_*)=1$ can be replaced by $\rank(\varphi_*)=1$ at any point of $\RR^{1+d}$.
\end{remark}
\begin{proof}
We consider the identity map $i:(\RR^{1+d},\m) \mapsto (\RR^{1+d},e)$ with $e$ the standard Euclidean metric and $\m$ the standard Minkowski metric. Obviously $i$ is affine (with respect to the two metric connections). Then $\phi := \varphi\circ i^{-1}$ is totally geodesic from $(\RR^{1+d},e)$ to $(N,h)$. Recalling that $(\RR^{1+d},e)$ is a complete Riemannian manifold and simply connected, every subspace of its tangent space is non-degenerate. By the above Proposition \ref{thm1-16-11-2021}, one obtains the isometry
	$$
	\psi:\RR^{1+d}\mapsto M_1\times M_2
	$$
	and $\varphi$ factors as $\varphi = \varphi_S\circ \varphi_I$. This is showed in the following diagram:
	$$
	\begin{tikzcd}
		&  &  &                            & {(\mathbb{R}^{1+d},\m)} \arrow[lllld, "i"'] \arrow[ldd, "\varphi_S = \tilde{\varphi}_S\circ i "] \arrow[ld, "\varphi"'] \\
		{(\mathbb{R}^{1+d},e)} \arrow[rrr, "\phi = \varphi\circ i^{-1}" description] \arrow[d, "\psi"] \arrow[rrrd, "\tilde{\varphi}_S"] &  &  & {(N,h)}                    &                                                                                                                        \\
		M_1\times M_2 \arrow[rrr, "\pi_2"]                                                                                               &  &  & M_2 \arrow[u, "\varphi_I"] &                                                                                                                       
	\end{tikzcd}
	$$
	Here $M_1, M_2$ are the maximal integral manifolds of $\ker(\phi_*)$ and $\ker(\phi_*)^{\perp}$ respectively (in $(\RR^{1+d},e)$, i.e., with respect to Euclidean metric). $M_2$ is a one-dimensional manifold. Thus $M_2 = \RR$ or $M_2 = \mathbb{S}^1$. It is clear that $M_2$ cannot be $\mathbb{S}^1$ due to the trivial topology of $\RR^{1+d}$. 
\end{proof}

Let us concentrate on $\varphi_S$. Let $x\in M_1$, we consider $\RR_x := \psi^{-1}(\{x\}\times \RR)$ which is an embedding into $(\RR^{1+d},e)$. If we parameterize $\RR$ by arc-length, we obtain a curve $\gamma_x: \RR\mapsto \RR^{1+d}, t\mapsto \psi^{-1}(x,t)$ which is also parameterized by arc-length  (modulo a nonzero constant factor) in $(\RR^{1+d},e)$. It is showed that the tangent space of $\gamma_x$ ($ =\ker(\phi_*)^{\perp}$, with respect to the Euclidean metric) is invariant by parallel transport. Along $\gamma_x$ itself, this leads to $\nabla_{\dot\gamma_x(t)}\dot\gamma_x(t) = 0$. Thus $\gamma_x$ is a geodesic of $\RR^{1+d}$, i.e., a straight line. Furthermore, also due to the fact that $T_x\gamma = \ker(\phi_*)^{\perp}$ is invariant by parallel transport,  these strait lines are parallel. 
Then $\m(d\varphi_S,d\varphi_S) = |\dot\gamma_x|^2_{\m}$ is a constant. Then there are three cases:
$$
\aligned
&\bullet\quad\gamma_x \text{ are time-like in }(\RR^{1+d},m) && \Leftrightarrow \m(d\varphi_S,d\varphi_S) < 0. \text{ We say }\varphi\text{ is {\sl time-like}}.
\\
&\bullet\quad\gamma_x\text{ are null in }(\RR^{1+d},m) &&  \Leftrightarrow \m(d\varphi_S,d\varphi_S) = 0.\text{ In this case }\varphi\text{ is said {\sl null}}.
\\
&\bullet\quad \gamma_x\text{ are space-like in }(\RR^{1+d},m) &&  \Leftrightarrow \m(d\varphi_S,d\varphi_S) > 0.\text{ In this case }\varphi\text{ is said {\sl space-like}}.
\endaligned
$$
On the other hand, remark that at each point $\psi^{-1}(x,y)$, $\gamma_x\perp \psi^{-1}(M_1\times\{y\})$. Therefore $\psi^{-1}(M_1\times \{y\})$ are $d-$dimensional hyperplanes which are parallel.

Remark that $\varphi_I: \RR\mapsto (N,h)$ is immersion and totally geodesic. Thus $\gamma: t\rightarrow \varphi(t)$ is a geodesic in $(N,h)$.

\section{Local symmetry along a geodesic}\label{sec1-19-02-2022-M}
\subsection{Geodesic normal coordinates}\label{subsec1-24-02-2022-M}
We firstly recall the geodesic normal  coordinates introduced in \cite{Gray-2003}. It permits one to parameterize a tubular neighborhood of an arbitrary geodesic, in which the Christoffel symbols vanish along the geodesic. The following outline is based on the Section 2 and 3 of \cite{Ab-2019}. Let $(N,h)$ be a complete Riemannian manifold and $\gamma:\RR\rightarrow N$ be a fixed geodesic. We parameterize it with arc-length. At $\gamma(0)$, let $\vec{e}_1 = \dot\gamma(0)$ and
$$
e^{\perp} := (\vec{e}_2,\cdots, \vec{e}_n),\quad \vec{e}_{\ih}\perp \vec{e}_{\jh},\quad |\vec{e}_{\ih}|=1.
$$
For $x^1\in \RR$, define $\vec{e}_{\ih}$ by parallel transporting along $\gamma$. This forms a normalized orthogonal frame along $\gamma$.  Let $\exp_{\gamma(x^1)}(t\vec{v})$ be the geodesic satisfying
$$
\frac{d}{dt}\exp_{\gamma(x^1)}(t\vec{v})\big|_{t=0} = \vec{v},\quad \exp_{\gamma(x^1)}(t\vec{v})\big|_{t=0} = \gamma(x^1)
$$ 
with $\vec{v}\in \dot\gamma(x^1)^{\perp}$ and for $(x^1,\xh) = (x^1,x^2,\cdots x^n)$ with $|\xh|$ sufficiently small,
\begin{equation}\label{eq1-22-02-2022-M}
\sigma: (x^1,\xh)\rightarrow \exp_{\gamma(x^1)}(x^{\jh}\vec{e}_{\jh})
\end{equation}
gives a parameterization of a tubular neighborhood of $\gamma$. Here remark that the sum over $\jh$ is taken over the range $\jh=2,3,\cdots, n$. This is called the {\sl geodesic normal coordinates}. We denote by $\{\del_i\}_{i=1,2,\dots, n}$ the associate natural frame. The coordinates $x^2,x^3,\cdots, x^n$ are called {\sl transversal coordinates} and $\del_{\ih}$ are called {\sl transversal directions}. Other terms such as {\sl transversal components} of a tensor are understood in an natural way. 
\begin{remark}
Let $r_{\text{foc}(x^1)}$ be the {\it focal radius} of $\exp_{\gamma(x^1)}$, i.e., the maximal radius such that the exponential map is non-critical. Remark that $r_{\text{foc}(x^1)}$ is determined by the Jacobi fields along each geodesic $\exp_{\gamma(x^1)}(x^{\jh}\vec{e}_{\jh})$.
The following Proposition \ref{prop1-19-02-2022-M} shows that the Riemann curvature together with all its derivatives are constants along $\gamma$, thus uniformly bounded. Recalling that the Jacobi field equations are ODEs  with coefficients determined by the sectional curvature. 
This leads to a uniform lower bound on $r_{\text{foc}(x^1)}$ along $\gamma$. Thus the geodesic normal coordinates are well defined in a tubular neighborhood $\RR\times\{|\xh|<\delta\}$.
\end{remark}
Then we recall several basic properties of this system of coordinates, which are established in \cite{Ab-2019}.
\begin{lemma}
In the geodesic normal coordinates,  $g\big|_{\gamma} = \delta_{ij}dx^idx^j$. Furthermore, 
\begin{equation}\label{eq16-07-10-2020}
\del_1^q\Gamma_{jk}^i\big|_{\gamma} = 0,\quad q=0,1,2,\cdots,n,
\end{equation}
\begin{equation}\label{eq2-21-11-2021}
\del_m\Gamma_{1k}^i\big|_{\gamma} = R_{m1k}^i\big|_{\gamma}.
\end{equation}
where $\nabla_i\nabla_j\del_k - \nabla_j\nabla_i\del_k = R_{ijk}^l\del_l$ is the Riemann curvature.
\end{lemma}

The main objective of this section is to establish the following result.
\begin{proposition}\label{prop1-19-02-2022-M}
Let $\gamma$ be a geodesic of $M$. Suppose that $M$ locally symmetric along $\gamma$. Then the following equations holds in geodesic normal coordinates around $\gamma$:
\begin{equation}\label{eq5-15-02-2022-M}
\del_1\del^I\Gamma_{ij}^k\big|_{\gamma} \equiv 0, 
\end{equation}
\begin{equation}\label{eq6-24-02-2022-M}
\del_1\del^IR_{jkl}^i\big|_{\gamma} \equiv 0,
\end{equation}
for all $I = (i_1,i_2,\cdots,i_m),\quad m\in \mathbb{N}$.
\end{proposition}
The rest of this section is devoted to the proof of the above result.

\subsection{Derivatives of Christoffel symbols}
We firstly remark that in an arbitrary natural frame, 
\begin{equation}\label{eq6-22-02-2022-M}
R_{\jh\kh\lh}^i = \del_{\jh}\Gamma_{\kh\lh}^i - \del_{\kh}\Gamma_{\jh\lh}^i + \Gamma_{\jh h}^i\Gamma_{\kh\lh}^h - \Gamma_{\kh h}^i\Gamma_{\jh\lh}^h.
\end{equation}
We then establish the following identity:
\begin{lemma}
In an arbitrary local coordinate chart $\{x^1,x^2,\cdots,x^n\}$ and suppose that $\nabla^I = \nabla_{\ih_1}\nabla_{\ih_2}\cdots\nabla_{\ih_m}$, one has
\begin{equation}\label{eq9-22-02-2022-M}
\aligned
\nabla^I\big(dx^j\otimes dx^k\otimes dx^l\otimes \del_i\big) 
=& \sum_{|J|\leq |I|-1}\!\!\!\!\!\!
\Lambda_{i,Ji'',j'k'l'}^{Ijkl,j''k'',i'}\delh^J\Gamma_{j''k''}^{i''}\,dx^{j'}\otimes dx^{k'}\otimes dx^{l'}\otimes \del_{i'}
\\
&+ P^{I,i'}_{j'k'l'}(\delh^{J'} \Gamma)dx^{j'}\otimes dx^{k'}\otimes dx^{l'}\otimes \del_{i'}
\endaligned
\end{equation}
with $\Lambda_{i,Ji'',j'k'l'}^{Ijkl,j''k'',i'}$ constants and $P^{I,i'}_{j'k'l'}(\delh^{J'} \Gamma)$ constant coefficient polynomials acting on $\delh^{J'}\Gamma_{jk}^i$ with $|J'|\leq |I|-2$. When $|J|\leq 1$ the last term does not exist.
\end{lemma}
\begin{proof}
This is by induction on $|I|$. We remark that for $|I|=1$, 
\begin{equation}\label{eq10-22-02-2022-M}
\aligned
\nabla_{\ih_1}\big(dx^j\otimes dx^k\otimes dx^l\otimes \del_i\big)
=& -\Gamma_{\ih_1j'}^jdx^{j'}\otimes dx^k\otimes dx^l\otimes\del_i
  -\Gamma_{\ih_1k'}^kdx^j\otimes dx^{k'}\otimes dx^l\otimes\del_i 
\\
 &-\Gamma_{\ih_1l'}^ldx^j\otimes dx^k\otimes dx^{l'}\otimes\del_i 
  +\Gamma_{\ih_1i}^{i'}dx^j\otimes dx^k\otimes dx^l\otimes\del_{i'}.
\endaligned
\end{equation}
Then suppose that \eqref{eq9-22-02-2022-M} holds for $|I|\leq m$. We consider
$$
\aligned
\nabla_{\hh}\nabla^I\big(dx^j\otimes dx^k\otimes dx^l\otimes \del_i\big) 
=&\sum_{|J|\leq |I|-1}\!\!\!\!\!\!
\Lambda_{i,Ji'',j'k'l'}^{Ijkl,j''k'',i'}\del_{\hh}\delh^J\Gamma_{j''k''}^{i''}\,dx^{j'}\otimes dx^{k'}\otimes dx^{l'}\otimes \del_{i'}
\\
&+
\sum_{|J|\leq |I|-1}\!\!\!\!\!\!\Lambda_{i,Ji'',j'k'l'}^{Ijkl,j''k'',i'}\delh^J\Gamma_{j''k''}^{i''}\,\nabla_{\hh}\big(dx^{j'}\otimes dx^{k'}\otimes dx^{l'}\otimes \del_{i'}\big)
\\
&+\del_{\hh}\big( P^{I,i'}_{j'k'l'}(\delh^{J'} \Gamma)\big)\,dx^{j'}\otimes dx^{k'}\otimes dx^{l'}\otimes \del_{i'}
\\
&+ P^{I,i'}_{j'k'l'}(\delh^{J'} \Gamma)\nabla_{\hh}\big(dx^{j'}\otimes dx^{k'}\otimes dx^{l'}\otimes \del_{i'}\big).
\endaligned
$$

We substitute \eqref{eq10-22-02-2022-M} into the second the last term, and remark that these terms becomes constant coefficient polynomials acting on $\delh^{J'}\Gamma$ with $|J'|\leq |I|-1$. The third term, thanks to the assumption of induction, is a polynomial acting on $\delh^{J'}\Gamma$ with $|J'|\leq |I|-2+1 = |I|-1$. The first term is already in the form desired. Thus we conclude by induction.
\end{proof}

Based on the above result, we have the following expression on the covariant derivatives of Riemann curvature.
\begin{lemma}
Suppose that $\nabla^I = \nabla_{\ih_1}\nabla_{\ih_2}\cdots\nabla_{\ih_m}$. In a local coordinate chart, the following expression holds:
\begin{equation}\label{eq3-23-02-2022-M}
\nabla^I R_{jkl}^i = \delh^IR_{jkl}^i + P^{Ii}_{jkl}(\delh^{J'}R,\delh^{J''}\Gamma).
\end{equation}
Here $P^{Ii}_{jkl}(\delh^{J'}R,\delh^{J''}\Gamma)$ is a constant coefficient polynomial acting on the derivatives of the components of Riemann curvature and the Christoffel symbols of order $\leq |I|-1$. 
\end{lemma}
\begin{proof}
This is by direct calculation. We remark that
$$
\aligned
\nabla^IR =& \nabla^I(R_{jkl}^idx^j\otimes dx^k\otimes dx^l\otimes \del_i) 
\\
=& \delh^IR_{jkl}^i dx^j\otimes dx^k\otimes dx^l\otimes \del_i
 + \sum_{I_1+I_2=I\atop |I_1|\leq |I|-1}\delh^{I_1}R_{jkl}^i \nabla^{I_2}\big(dx^j\otimes dx^k\otimes dx^l\otimes \del_i\big).
\endaligned
$$
Then by \eqref{eq9-22-02-2022-M}, we obtain the desired result.
\end{proof}

Then we concentrate on the derivatives of the Christoffel symbols. For the convenience of expression, we denote by 
$\{\theta^i\}_{i=1,2,\cdots, n}$ the dual frame of $\{\del_i\}_{i=1,2,\cdots,n}$.
\begin{lemma}
Let $I = (\ih_1,\ih_2,\cdots,\ih_m)$. Then in a local coordinate chart, one has
\begin{equation}\label{eq1-23-02-2022-M}
\delh^I\Gamma_{1\jh}^i  = \del_1\delh^{I'}\Gamma_{\ih_m\jh}^i + \delh^{I'}R_{\ih_m1\jh}^i 
+ \delh^{I'}\big( \Gamma_{\ih_m\jh}^j\Gamma_{1j}^i-\Gamma_{\ih_m j}^i\Gamma_{1\jh}^j\big).
\end{equation}
where $I' = (\ih_1,\ih_2,\cdots,\ih_{m-1})$.
\end{lemma}
\begin{proof}
We denote by $\la \omega, v\ra$ the action of a differential from $\omega$ on a vector field $v$. Then
$$
\aligned
\del_{\ih}\Gamma_{1\jh}^i = \del_{\ih}\la \theta^i,\nabla_1\del_{\jh}\ra
=& \la\nabla_{\ih}\theta^i,\nabla_1\del_{\jh}\ra + \la \theta^i,\nabla_{\ih}\nabla_1\del_{\jh} \ra
= \la\nabla_{\ih}\theta^i,\nabla_1\del_{\jh}\ra + \la \theta^i,\nabla_1\nabla_{\ih}\del_{\jh} \ra
+ R_{\ih1\jh}^i 
\\
=& -\Gamma_{\ih j}^i\Gamma_{1\jh}^j + \del_1\Gamma_{\ih\jh}^i + \Gamma_{\ih\jh}^j\Gamma_{1j}^i + R_{\ih1\jh}^i.
\endaligned
$$
Then we obtain \eqref{eq1-23-02-2022-M} by differentiating the above identity with respect to $\delh^{I'}$.
\end{proof}

Then we establish the following result in the geodesic normal coordinates.
\begin{lemma}\label{lem1-22-02-2022-M}
	In the geodesic normal coordinates, one has
	\begin{equation}\label{eq1-15-02-2022-M}
		\Gamma_{\jh\kh}^ix^{\jh}x^{\kh} = 0,\qquad \forall x^1,\quad |\xh|\in (-\vep,\vep).
	\end{equation}	
\end{lemma}
\begin{proof}
	We remark that for a fixed $x_0 = (x_0^1,\xh_0)$, 
	$$
	\aligned
	&\eta_{x_0} :(-\vep,\vep)\mapsto M,\qquad &&s\mapsto \sigma(x_0^1,s\xh_0),
	\endaligned
	$$
	are geodesics. Then from the geodesic equations satisfied by $\eta_{x_0}$, we obtain \eqref{eq1-15-02-2022-M}.
\end{proof}

Finally, we establish the following relation.
\begin{lemma}\label{lem2-23-02-2022-M}
	In geodesic normal coordinates, 
\begin{equation}\label{eq4-23-02-2022-M}
\delh^I\Gamma_{\jh\kh}^i\big|_{\gamma} = P^{iI}_{\jh\kh}(\delh^{J'}R,\delh^{J''}\Gamma)\big|_{\gamma}
\end{equation}
where $P^{iI}_{\jh\kh}(\delh^{J'}R,\delh^{J''}\Gamma)$ is a constant coefficient polynomial acting on the transversal derivatives of components of the Riemann curvature and Christoffel symbols with order $\leq |I|-1$. When $|I| = 0$, $P = 0$. 
\end{lemma}
\begin{proof}
We differentiate \eqref{eq1-15-02-2022-M} with respect to $\delh^I$ with $|I|\geq 3$, and obtain
$$
0 = x^{\ih}x^{\jh}\delh^I\Gamma_{\ih\jh}^i + 2\sum_{I_p}x^{\ih}\delta_{\ih_p}^{\jh}\delh^{I_p}\Gamma_{\ih\jh}^i
 + 2\sum_{I_{pq},p\neq q}\delta_{\ih_p}^{\ih}\delta_{\ih_q}^{\jh}\delh^{I_{pq}}\Gamma_{\ih\jh}^i
$$
where 
$$
I_p = (\ih_1,\ih_2,\cdots,\ih_{p-1},\ih_{p+1},\cdots\ih_m),
\quad
I_{pq} = (\ih_1,\cdots,\ih_{p-1},\ih_{p+1},\cdots,\ih_{q-1},\ih_{q+1},\cdots,\ih_m).
$$
Thus, evaluating the above relation along $\gamma$, i.e., $\xh = 0$, one obtains
\begin{equation}\label{eq2-23-02-2022-M}
	\sum_{I_{pq},p\neq q}\delh^{I_{pq}}\Gamma_{\ih_p\ih_q}^i\big|_{\gamma} = 0.
\end{equation}
On the other hand, form \eqref{eq6-22-02-2022-M}, we have
\begin{equation}\label{eq3-23-02-2022_M}
\delh^J\del_{\jh}\Gamma_{\kh\lh}^i - \delh^J\del_{\kh}\Gamma_{\jh\lh}^i = \delh^JR_{\jh\kh\lh}^i - \delh^J\big(\Gamma_{\jh h}^i\Gamma_{\kh\lh}^h - \Gamma_{\kh h}^i\Gamma_{\jh\lh}^h\big).
\end{equation}

We denote by $\mathcal{D}_N = \big\{(\ih_1\ih_2,\cdots\ih_N,\ih\jh)|\ih_k,\ih,\jh\in\{1,2,\cdots,n\}\big\}$ the collection of $(N+2)-$tuples. We introduce the following relation $\sim$:
$$
(\ih_1\ih_2\cdots\ih_N,\jh\kh)\sim (\jh_1\jh_2\cdots\jh_N,\jh'\kh'),
$$
if for all $i=1,2,\cdots,n$,
$$
\delh^I\Gamma_{\jh\kh}^i =\delh^J\Gamma_{\jh'\kh'}^i + P(\delh^{J'}R_{i''j''k''}^{l''},\delh^{J''}\Gamma_{j'''k'''}^{i'''}),
$$
with
$$
I = (\ih_1\ih_2\cdots\ih_N),\quad J = (\jh_1\jh_2\cdots\jh_N),
$$
and $P$ is a constant coefficient polynomial acting on $\delh^{J'}R_{i''j''k''}^{l''}$ and $\delh^{J''}\Gamma_{j'''k'''}^{i'''}$ with $|J'|,|J''|\leq N-1$. When $N=0$, this relation is reduced to
$$
(\ih\jh) \sim (\ih'\jh') \quad \Leftrightarrow\quad \Gamma_{\ih\jh}^i = \Gamma_{\ih'\jh'}^i.
$$
It is clear that $\sim$ is a relation of equivalence. From \eqref{eq3-23-02-2022_M}, one has
\begin{subequations}
\begin{equation}\label{eq3a-23-02-2022-M}
(\ih_1\ih_2\cdots\ih,\ih_N\jh)\sim (\ih_1\ih_2\cdots\ih_N,\ih\jh).
\end{equation}
Furthermore, due to the fact that the connection $\nabla$ is torsion free, 
\begin{equation}\label{eq3b-23-02-2022-M}
(\ih_1\ih_2\cdots\ih_N,\ih\jh) \sim (\ih_1\ih_2\cdots\ih_N,\jh\ih).
\end{equation}
At last, recalling that $[\del_{\ih},\del_{\jh}] = 0$, 
\begin{equation}\label{eq3c-23-02-2022-M}
(\ih_1\ih_2\cdots\ih_N,\ih\jh)\sim (\ih_{\sigma(1)}\ih_{\sigma(2)}\cdots\ih_{\sigma(N)},\ih\jh),\quad \forall \sigma \in S_N,
\end{equation}
\end{subequations}
where $S_N$ is the $N-$permutation group. For the simplicity of expression, we denote by
$$
\aligned
&\tau: \mathcal{D}_N\mapsto \mathcal{D}_N, \quad (\ih_1\ih_2...\ih_N,\ih\jh)\mapsto (\ih_1\ih_2...\ih,\ih_N \jh);
\\
&\eta: \mathcal{D}_N\mapsto \mathcal{D}_N, \quad (\ih_1\ih_2...\ih_N,\ih\jh)\mapsto (\ih_1\ih_2...\ih_N,\jh\ih).
\endaligned
$$
Both $\tau$ and $\eta$ can be regarded as permutations in $S_{N+2}$. We denote by $G = \la S_N, \tau,\sigma\ra$ the group generated by $\sigma\in S_N, \tau$ and $\eta$. Then $G$ is a subgroup of $S_{N+2}$. It is clear that $\forall \alpha\in G$, 
$$
(\ih_1\ih_2\cdots\ih_N,\ih\jh)\sim \alpha(\ih_1\ih_2\cdots\ih_N,\ih\jh).
$$

Following a result of elementary theory of permutation groups (see Lemma \ref{lem1-23-02-2022-M} below), one has $G = S_{N+2}$. Then  $\forall I = (\ih_1\ih_2\cdots\ih_N\ih_{N+1}\ih_{N+2})$ and $\forall 1\leq p<q\leq N+2$, all $(I_{pq}, \ih_p\ih_q)$ are equivalent to each other. Then from \eqref{eq2-23-02-2022-M}, we obtain the desired result.
\end{proof}

\begin{lemma}\label{lem1-23-02-2022-M}
	Following the notation in the proof of Lemma \ref{lem2-23-02-2022-M}, $G = \la S_N,\tau,\eta\ra = S_{N+2}$. 
\end{lemma}
\begin{proof}
	For all $N$, recall that $S_{N}=\la \{(i,k)|i\neq k,\ i,k\in \{1,2,3,\cdots, N\}\}\ra$. Since for any $k<N$, it is obvious that $(k,N)=(k,N-1)(N-1,N)(k,N-1)^{-1}$, therefore $S_N$ is also generated by $\{(i,k),(N-1,N)|i\neq k,\ i,k\in \{1,2,3,\cdots, N-1\}\}$. This is equivalent to say $S_N=\la S_{N-1},(N-1,N)\ra$. Hence $S_{N+2}=\la S_{N+1},(N+1,N+2)\ra = \la S_{N},(N,N+1),(N+1,N+2)\ra$. The desired result can be directly deduced from this relation.   
\end{proof}

\subsection{Proof of Proposition \ref{prop1-19-02-2022-M}}
We proceed by induction on the following set of relations (with respect to $|I|$):
\begin{equation}\label{eq1-24-02-2022-M}
\del_1\delh^IR_{jkl}^i\big|_{\gamma} = 0,\quad \del_1\delh^I\Gamma_{jk}^i\big|_{\gamma} = 0.
\end{equation}
This is true when $|I| = 0$. The first is due to the condition of local symmetry applied on $R$:
$$
0 = \nabla_1R_{jkl}^i = \del_1R_{jkl}^idx^j + \Gamma_{1i'}^iR_{jkl}^{i'}
 - \Gamma_{1j}^{j'} R_{j'kl}^i - \Gamma_{1k}^{k'}R_{jk'l}^i - \Gamma_{1l}^{l'}R_{jkl'}^i .
$$
Evaluating along $\gamma$ and recalling that $\Gamma_{jk}^i\big|_{\gamma} = 0$, we obtain the desired relation. 

Then suppose that \eqref{eq1-24-02-2022-M} holds for $|I|\leq N-1$. We consider the case $|I| = N$. Remark that
\begin{equation}\label{eq2-24-02-2022-M}
\nabla_1\nabla^IR = 0 \Rightarrow \nabla_1\big(\nabla^IR_{jkl}^idx^j\otimes dx^k\otimes dx^l\otimes \del_i\big) = 0.
\end{equation}
Thanks to \eqref{eq10-22-02-2022-M} and \eqref{eq16-07-10-2020}, 
$$
\nabla_1\big(dx^j\otimes dx^k\otimes dx^l\otimes \del_i\big)\big|_{\gamma} = 0.
$$
Thus from \eqref{eq2-24-02-2022-M},
$$
\del_1(\nabla^IR_{jkl}^i)\big|_{\gamma} = 0.
$$
Then thanks to \eqref{eq3-23-02-2022-M}, we obtain
$$
0 = \del_1\delh^IR_{jkl}^i\big|_{\gamma}  + \del_1\big(P^{Ii}_{jkl}(\delh^{J'}R,\delh^{J''}\Gamma)\big)\big|_{\gamma} 
$$
where $|J'|,|J''|\leq N-1$. By the assumption of induction,  one obtains
\begin{equation}\label{eq3-24-02-2022_M}
\del_1\delh^IR_{jkl}^i\big|_{\gamma} = 0.
\end{equation}

On the other hand, we differentiate \eqref{eq1-23-02-2022-M} with respect to $\del_1$ and evaluate along $\gamma$ together with the assumption of induction and obtain
\begin{equation}\label{eq4-24-02-2022_M}
\del_1\delh^I\Gamma_{1k}^i\big|_{\gamma} = 0.	
\end{equation}
Finally, we recall \eqref{eq4-23-02-2022-M} (remark that $\del_1$ tangents to $\gamma$)
$$
\del_1\delh^I\Gamma_{\jh\kh}^i\big|_{\gamma} = \del_1P^{iI}_{\jh\kh}(\delh^{J'}R,\delh^{J''}\Gamma)\big|_{\gamma}
$$
where $|J'|,|J''|\leq N-1$. By the assumption of induction, we obtain
$$
\del_1\delh^I\Gamma_{\jh\kh}^i\big|_{\gamma} = 0.
$$
Then by \eqref{eq3-24-02-2022_M}, \eqref{eq4-24-02-2022_M} together with the above relation, we conclude that \eqref{eq1-24-02-2022-M} holds for all high-order transversal derivatives $\delh^I$. 

Once \eqref{eq1-24-02-2022-M} is established, it is clear that for all $m\in\mathbb{N}$ and all all high-order transversal derivatives $\delh^I$,
$$
\del_1\del_1^m\delh^I\Gamma_{jk}^i\big|_{\gamma} = 0.
$$
This concludes \eqref{eq5-15-02-2022-M}. The equation \eqref{eq6-24-02-2022-M} is established in the same manner, we omit the detail.

\subsection{Special components of the Christoffel symbols}
In this subsection we establish the following result:
\begin{proposition}\label{prop2-27-11-2021}
Suppose that $\gamma$ is a geodesic of $(N,h)$, which is locally symmetric along $\gamma$. Then for $q\geq 0$, 
	\begin{equation}\label{eq1-21-11-2021}
		\del_1^q\del_1\del_j\Gamma_{1k}^i\big|_{\gamma} = \del_1^q\del_1\Gamma_{1k}^i\big|_{\gamma} = 0,
	\end{equation}
	\begin{equation}\label{eq3-21-11-2021}
		\del_1^q\del_j\del_k\Gamma_{11}^1\big|_{\gamma} = \del_1^q\del_k\Gamma_{11}^1\big|_{\gamma} = \del_1^q\Gamma_{11}^1\big|_{\gamma} = 0.
	\end{equation}
\end{proposition}
Before the proof of this result, we need the following observation.
\begin{lemma}
	In the geodesic normal coordinates, one has
	\begin{equation}\label{eq7-24-02-2022-M}
		\del_kh_{ij}\big|_{\gamma} \equiv 0,
	\end{equation}
where we recall that $h_{ij}$ are the components of the metric $h$.	Furthermore, if the manifold enjoys the local symmetry along $\gamma$, 
	\begin{equation}\label{eq8-24-02-2022-M}
		\del_1\del^Ih_{jk}\big|_{\gamma} \equiv 0.
	\end{equation}
\end{lemma}
\begin{proof}
	This is because 
	\begin{equation}\label{eq9-24-02-2022-M}
		\del_kh_{ij} = h_{lj}\Gamma_{ki}^l + h_{li}\Gamma_{kj}^l,
	\end{equation}
	which is equivalent to $\del_k(\vec{e}_i,\vec{e}_j)_h = (\nabla_{\vec{e}_k}\vec{e}_i,\vec{e}_j)_h + (\vec{e}_i,\nabla_{\vec{e}_k}\vec{e}_j)$. Then by \eqref{eq16-07-10-2020} we obtain \eqref{eq7-24-02-2022-M}.	Furthermore by \eqref{eq5-15-02-2022-M} and induction on $|I|$, we obtain \eqref{eq8-24-02-2022-M}.
\end{proof}

\begin{proof}[Proof of Proposition \ref{prop2-27-11-2021}]
We remark that
$$
\aligned
\del_j\Gamma_{1k}^i =& \del_j\la \theta^i,\nabla_1\del_k\ra 
= \la \nabla_j\theta^i,\nabla_1\del_k\ra + \la \theta^i,\nabla_j\nabla_1\del_k\ra
\\
=& \la \nabla_j\theta^i,\nabla_1\del_k\ra + \la \theta^i, \nabla_1\nabla_j\del_k \ra + R_{j1k}^i.
\endaligned
$$
Evaluating the above identity along $\varphi_I(\RR)$, we obtain (thanks to \eqref{eq16-07-10-2020} and \eqref{eq2-21-11-2021})
$$
\del_j\Gamma_{1k}^l\big|_{\gamma} = R_{j1k}^i\big|_{\gamma} = R_{j1ki}\big|_{\gamma}. 
$$
Then differentiate the above identity along $\varphi_I$ and thanks to \eqref{eq6-24-02-2022-M}, we obtain 
$$
\del_1\del_j\Gamma_{1k}^i\big|_{\gamma} = 0.
$$
This leads to $\del_1^q\del_1\del_j\Gamma_{1k}^i\big|_{\gamma} = 0$. The second relation in \eqref{eq1-21-11-2021} is a direct result of \eqref{eq16-07-10-2020}.

	For \eqref{eq3-21-11-2021}, we only need to remark that $\Gamma_{11}^ig_{i1} = (\nabla_1\del_1,\del_1)_h = \frac{1}{2}\del_1h_{11}$. Then
	$$
	\del_j(\Gamma_{11}^ig_{i1}) = \frac{1}{2}\del_1\del_jh_{11},\qquad \del_j\del_k(\Gamma_{11}^ih_{i1}) = \frac{1}{2}\del_1\del_j\del_kh_{11} .
	$$
	Here if we evaluate the right-hand side along $\gamma$, we have
	\begin{equation}\label{eq10-24-02-2022-M}
		\del_j(\Gamma_{11}^ih_{i1})\big|_{\gamma} = \del_j\del_k(\Gamma_{11}^ih_{i1})\big|_{\gamma} = 0
	\end{equation}
	by \eqref{eq8-24-02-2022-M}. On the other hand, 
	$$
	\aligned
	\del_j(\Gamma_{11}^ih_{i1}) =& \del_j\Gamma_{11}^ih_{i1} + \Gamma_{11}^i\del_jh_{i1},
	\\
	\del_j\del_k(\Gamma_{11}^ih_{i1}) =& \del_j\del_k\Gamma_{11}^ih_{i1} + \del_j\Gamma_{11}^i\del_kh_{i1} + \del_k\Gamma_{11}^i\del_jh_{i1} + \Gamma_{11}^i\del_j\del_kh_{i1}.
	\endaligned
	$$
	Evaluating the above equation along $\gamma$ and thanks to \eqref{eq7-24-02-2022-M} and \eqref{eq8-24-02-2022-M}, we have
	$$
	\del_j(\Gamma_{11}^ih_{i1})\big|_{\gamma} = \del_j\Gamma_{11}^1\big|_{\gamma},
	\quad
	\del_j\del_k(\Gamma_{11}^ih_{i1})\big|_{\gamma} = \del_j\del_k\Gamma_{11}^1\big|_{\gamma}.
	$$
	Combine this with \eqref{eq10-24-02-2022-M}, we obtain \eqref{eq3-21-11-2021}.
\end{proof}

\section{Formulation in PDE system}\label{sec1-08-02-2022}
\subsection{Formulation with the geodesic normal coordinates}
We restrict to the case where $(M,g) = (\RR^{1+d},m)$. Recalling the energy functional \eqref{eq1-09-02-2022} and the definition of wave maps, we write the associate Euler-Lagrange equation:
\begin{equation}\label{eq-main-geo}
\text{trace}_g D\phi_* = 0,
\end{equation}
where the connection $D$ acting on the tangent map $\phi_*$ is defined in \eqref{eq3-24-02-2022-M}. It is direct by \eqref{eq4-24-02-2022-M} that all totally geodesic maps are wave maps. 

We recall the stability problem formulated in Introduction. In the following Subsections we will give a quantitative description with the help of the geodesic normal coordinates  when the totally geodesic map is of rank-one.
 
Firstly, in arbitrary local coordinates, \eqref{eq-main-geo} is written as
\begin{equation}\label{eq1-18-11-2021}
\Box_m \phi^i + \Gamma_{jk}^i(\phi)\, \m^{\mu\nu}\del_{\mu}\phi^j\del_{\nu}\phi^k = 0,
\end{equation}
where $\phi^k = r^k\circ\phi$ with $r^k$ the $k$-th coordinate function, $\Gamma_{ij}^k(\phi)$ are the Christoffel symbols of $(N,h)$ evaluated along the image of $\phi$. $\Box_m := -\del_t\del_t + \Delta_{\RR^d}$. Within geodesic normal coordinates, a perturbation of $\varphi$ is described as following (see in detail \cite{Ab-2019}).

Recalling the factorization described in Proposition \ref{prop1-09-02-2022}, a rank-one totally geodesic map $\varphi: \RR^{1+2}\mapsto (N,h)$ is factorized as
$$
\RR^{1+d}\stackrel{\varphi_S}{\longrightarrow}\RR\stackrel{\varphi_I}{\longrightarrow} N.
$$
We construct the geodesic normal coordinates in a tubular neighborhood of $\varphi_I(\RR)$. Then
$$
\varphi: (t,x)\mapsto \sigma\big(\varphi_S(t,x),0\big).
$$
We perturb $\varphi$ as following. Consider 
$$
\tilde{\varphi} : \RR^{1+d}\rightarrow N,\quad (t,x)\rightarrow \sigma\big(\varphi_S(t,x) + \phi^1(t,x),\hat{\phi}(t,x)\big), \quad \hat\phi = (\phi^2,\phi^3,\cdots,\phi^n),
$$ 
and demand that $\tilde{\varphi}$ is again a wave map. This leads to, thanks to \eqref{eq1-18-11-2021} and the fact that $\varphi_S$ being linear,
\begin{equation}\label{eq15-07-10-2020}
	\begin{aligned}
		&\Box_m \phi^1 + \Gamma_{jk}^1(\varphi_S+\phi^1,\hat{\phi})\, \m^{\mu\nu}\del_{\mu} \bar{\phi}^j\del_{\nu} \bar{\phi}^k = 0,
		\\
		&\Box_m \phi^{\ih} + \Gamma_{jk}^{\ih}(\varphi_S+\phi^1,\hat{\phi})\, \m^{\mu\nu}\del_{\mu} \bar{\phi}^j\del_{\nu}\bar{\phi}^k = 0,
		\quad  \ih = 2,\cdots, n
	\end{aligned}
\end{equation}
with $\bar{\phi}^1 = \varphi_S(t,x) + \phi^1, \bar{\phi}^{\ih} = \phi^{\ih}, \ih=2,3,\cdots n$. Then  one develops the nonlinear terms $\Gamma_{jk}^1(\varphi_S+\phi^1,\hat{\phi}) \m^{\mu\nu}\del_{\mu} \bar{\phi}^j\del_{\nu} \bar{\phi}^k$ into Taylor series at each point of $\varphi_I(\RR)$, i.e., at $(\phi^1,\hat{\phi}) = 0$. 
Applying \eqref{eq16-07-10-2020} and \eqref{eq2-21-11-2021}, one arrive at the following formulation obtain in \cite{Ab-2019} which is valid for all $(N,h)$ in the geodesic normal coordinates. For the tangent perturbation $\phi^1$,
\begin{subequations}\label{eq1-19-11-2021}
\begin{equation}\label{eq1a-19-11-2021}
\aligned
\Box_m\phi^1 
=&-2 R_{\ih1\jh}^1(\varphi_S,0)\phi^{\ih}\m^{\mu\nu}\del_{\mu}\phi^{\jh}\del_{\nu}\varphi_S
 - \rho\!\!\!\!\sum_{(j,k)\neq(1,1)}\!\!\!\!\del_j\del_k\Gamma_{11}^1(\varphi_S,0)\phi^j\phi^k
\\
&- 2\!\!\!\!\sum_{(j,k)\neq(1,1)}\!\!\!\!
\del_j\del_k\Gamma_{i1}^1(\varphi_S,0)\phi^j\phi^k\m^{\mu\nu}\del_{\mu}\phi^i\del_{\nu}\varphi_S 
-  \del_{\jh}\Gamma_{jk}^1(\varphi_S,0)\phi^{\jh}\,\m^{\mu\nu}\del_{\mu}\phi^j\del_{\nu}\phi^k
\\
&- \rho\!\!\!\!\sum_{(i,j,k)\neq(1,1,1)}\!\!\!\!
\del_i\del_j\del_k\Gamma_{11}^1(\varphi_S,0)\phi^i\phi^j\phi^k 
 + \text{h.o.t.},
\endaligned	
\end{equation}
and for the transversal components $\phi^{\ih}$,
\begin{equation}\label{eq1b-19-11-2021}
\aligned
&\Box_m\phi^{\ih} + \rho R_{\jh11}^{\ih}(\varphi_S,0)\phi^{\jh}
\\
=&- 2  R_{\jh1j}^{\ih}(\varphi_S,0)\phi^{\jh}\m^{\mu\nu}\del_{\mu}\phi^j\del_{\nu}\varphi_S
- \rho\!\!\!\!\!\sum_{(j,k)\neq(1,1)}\!\!\!\!\!\!\!\del_j\del_k\Gamma_{11}^{\ih}(\varphi_S,0)\phi^j\phi^k
\\
&- 2\!\!\!\!\sum_{(j,k)\neq(1,1)}\!\!\!\!\!\!\!\!
\del_j\del_k\Gamma_{i1}^{\ih}(\varphi_S,0)\phi^j\phi^k\m^{\mu\nu}\del_{\mu}\phi^i\del_{\nu}\varphi_S
\\
&-  \del_{\jh}\Gamma_{jk}^{\ih}(\varphi_S,0)\phi^{\jh}\,\m^{\mu\nu}\del_{\mu}\phi^j\del_{\nu}\phi^k 
- \rho\!\!\!\!\sum_{(i,j,k)\neq(1,1,1)}\!\!\!\!\!\!\!\!
\del_i\del_j\del_k\Gamma_{11}^{\ih}(\varphi_S,0)\phi^i\phi^j\phi^k + \text{h.o.t.},
\endaligned	
\end{equation}
\end{subequations}
where $\rho := \m(d\varphi_s,d\varphi_s) = \m^{\mu\nu}\del_{\mu}\varphi_S\del_{\nu}\varphi_S$ and $R_{ijk}^l$ is the components of the Riemann curvature of $(N,h)$. However, in $\RR^{1+2}$, it seems to be difficult to obtain gloabl stability result for \eqref{eq1-19-11-2021}. In the next Subsection this system will be reduced with the condition of local symmetry along $\varphi_I$, more precisely, we will apply Proposition \ref{prop1-19-02-2022-M} and Proposition \ref{prop2-27-11-2021}.
\subsection{Reductions with local symmetry along geodesic}
In this subsection we suppose that $(N,h)$ is symmetric along $\varphi_I(\RR)$. Then we simplify \eqref{eq1-19-11-2021} under this condition. The essential work is to calculate the high-order derivatives of $\Gamma_{jk}^i$ along $\varphi_I(\RR)$. Remark that here we need to develop $\Gamma_{ij}^k$ to third order. This is because in $\RR^{1+2}$, cubic terms may lead to finite time blow-up (see for example \cite{Zhou-Han-2011} in the case of a semi-linear wave equation). We thus need the information on the structure of the cubic terms. Recalling  Proposition \ref{prop1-19-02-2022-M} and Proposition \ref{prop2-27-11-2021}, we obtain
\begin{equation}\label{eq2-14-12-2021}
\aligned
\Gamma_{11}^1(\varphi_S + \phi^1,\phih) =&\, \sum_{\jh,\kh,\lh}C^{\jh\kh\lh}\del_{\jh}\del_{\lh}\del_{\kh}\Gamma_{11}^1(\varphi_S,0)\phi^{\jh}\phi^{\kh}\phi^{\lh}
+ \sum_{|A|\geq 1} C^A
\del^A \del_{\jh}\del_{\kh}\del_{\lh}\Gamma_{11}^1(\varphi_S,0)\phi^A\phi^{\jh}\phi^{\kh}\phi^{\lh},
\\
\Gamma_{11}^{\ih}(\varphi_S + \phi^1,\phih) =&\, \del_{\jh}\Gamma_{11}^{\ih}(\varphi_S,0)\phi^{\jh}
+ \sum_{\jh,\kh}C^{\kh\jh}\del_{\kh}\del_{\jh}\Gamma_{11}^{\ih}(\varphi_S,0)\phi^{\jh}\phi^{\kh}
+ C^{\jh\kh l} \del_l\del_{\kh}\del_{\jh}\Gamma_{11}^{\ih}(\varphi_S,0)\phi^{\jh}\phi^{\kh}\phi^l
\\
&+ \sum_{|A|\geq 1}C^A  \del_l\del_{\kh}\del_{\lh}\Gamma_{11}^{\ih}(\varphi_S,0)\phi^A\phi^{\jh}\phi^{\kh}\phi^l,
\\
\Gamma_{1\jh}^{i}(\varphi_S + \phi^1,\phih) =&\, \del_{\kh}\Gamma_{1\jh}^{i}(\varphi_S,0)\phi^{\kh}
\\
&+ \sum_{\kh,\lh}C^{\lh\kh}\del_{\lh}\del_{\kh}\Gamma_{1\jh}^{i}(\varphi_S,0)\phi^{\lh}\phi^{\kh}
+ \sum_{|A|\geq 1}C^A \del^A\del_i\del_{\kh} \Gamma_{1\jh}^{\ih}(\varphi_S,0)\phi^A\phi^i\phi^{\kh},
\\
\Gamma_{\jh\kh}^i(\varphi_S + \phi^1,\phih) =& \del_{\lh}\Gamma_{\jh\kh}^i(\varphi_S,0)\phi^{\lh}
+\sum_{|A|\geq 1} C^A \del^A\del_{\lh}\Gamma_{\jh\kh}^i(\varphi_S,0)\phi^A\phi^{\lh}.
	\endaligned
\end{equation}
Here $C^A$ are constants (determined by the Taylor's expansion). On the other hand, recall that
$$
\m(d\bar{\phi}^j,d\bar{\phi}^k) = \delta^1_j\delta^1_k\rho + 2\delta^1_j\m(d\phi^j,d\varphi_S) + \m(d\phi^j,d\phi^k).
$$
Then we obtain
\begin{equation}\label{eq1-14-12-2021}
	\aligned
	&\Gamma_{jk}^i(\varphi_S+\phi^1,\phih)\m(d\bar{\phi}^j,d\bar{\phi}^k)
	\\
	=&\rho\Gamma_{11}^i(\varphi_S+\phi^1,\phih) + 2\Gamma_{11}^i(\varphi_S+\phi^1,\phih)\m(d\phi^1,d\varphi_S) 
	+ \Gamma_{11}^i(\varphi_S+\phi^1,\phih)\m(d\phi^1,d\phi^1)
	\\
	&+ 2\Gamma_{1\jh}^i(\varphi_S+\phi^1,\phih)\m(d\phi^{\jh},d\varphi_S) 
	+ 2\Gamma_{1\jh}^i(\varphi_S+\phi^1,\phih)\m(d\phi^{\jh},d\phi^1)
	+ \Gamma_{\jh\kh}^i(\varphi_S+\phi^1,\phih)\m(d\phi^{\jh},d\phi^{\kh}).
	\endaligned
\end{equation}
Substitute \eqref{eq2-14-12-2021} and \eqref{eq1-14-12-2021} into \eqref{eq15-07-10-2020}, we obtain\footnote{here and in the following sections, we take $\Box := \del_t\del_t - \sum_a\del_a\del_a$.}
\begin{subequations}\label{eq5-20-11-2021}
\begin{equation}\label{eq5a-20-11-2021}	
\Box \phi^1= 2\m(d\phi^{\jh},d\varphi_S)\del_{\kh}\Gamma_{1\jh}^1(\varphi_S,0)\phi^{\kh} + T_W[\phi] + S_{W}[\phi],
\end{equation}
with trilinear terms
	$$
	\aligned
	T_W[\phi] :=&\rho\sum_{\jh,\kh,\lh} 
	C^{\jh\kh\lh}\del_{\jh}\del_{\lh}\del_{\kh}\Gamma_{11}^1(\varphi_S,0)\phi^{\jh}\phi^{\kh}\phi^{\lh} 
	+ 2\m(d\phi^1,d\phi^{\jh})\,\del_{\kh}\Gamma_{1\jh}^1(\varphi_S,0)\phi^{\kh}
	\\ 
	&+2\sum_{\kh,\lh}\m(d\phi^{\jh},d\varphi_S)\,\del_{\kh}\del_{\lh}\Gamma_{1\jh}^1(\varphi_S,0)\phi^{\kh}\phi^{\lh}
	+ \m(d\phi^{\jh},d\phi^{\kh})\,\del_{\lh}\Gamma_{\jh\kh}^1(\varphi_S,0)\phi^{\lh}
	\endaligned
	$$
	and higher order terms
	$$
	\aligned
	S_{W}[\phi] =& (2\m(d\phi^1,d\varphi_S) + \m(d\phi^1,d\phi^1))
	\sum_{\jh,\kh,\lh}C^{\jh\kh\lh}\del_{\jh}\del_{\lh}\del_{\kh}\Gamma_{11}^1(\varphi_S,0)\phi^{\jh}\phi^{\kh}\phi^{\lh}
	\\
	& + 2\m(d\phi^1,d\phi^{\jh})\sum_{\kh,\lh}\del_{\kh}\del_{\lh}\Gamma_{1\jh}^1(\varphi_S,0)\phi^{\kh}\phi^{\lh}
	\\
	& + \m(d(\varphi_S+\phi^1),d(\varphi_S+\phi^1))
	\sum_{|A|\geq 1}C^A\del^A \del_{\jh}\del_{\kh}\del_{\lh}\Gamma_{11}^1(\varphi_S,0)\phi^A\phi^{\jh}\phi^{\kh}\phi^{\lh}
	\\
	& +\m(d\phi^{\jh},d(\varphi_S+\phi^1))
	\sum_{|A|\geq 1}C^A\del^A
	\del_{\kh}\del_{\jh}\del_{\lh}\Gamma_{1\jh}^1(\varphi_S,0)\phi^A\phi^{\kh}\phi^{\lh}
	\\
	& +\m(d\phi^{\jh},d\phi^{\kh}) \sum_{|A|\geq 1}C^A\del^A\del_{\lh}\Gamma_{\jh\kh}^1(\varphi_S,0)\phi^A\phi^{\lh}.
	\endaligned
	$$
	\begin{equation}\label{eq5b-20-11-2021}
		\aligned
		\Box \phi^{\ih} - \rho\del_{\jh}\Gamma_{11}^{\ih}\phi^{\jh}
		=& 2\m(d\varphi_S,d\phi^{\jh})\, \del_{\kh}\Gamma_{1\jh}^{\ih}(\varphi_S,0)\phi^{\kh} 
		+ 2\m(d\varphi_S,d\phi^1) \,\del_{\jh}\Gamma_{11}^{\ih}(\varphi_S,0)\phi^{\jh}
		\\
		&+ \rho \sum_{\jh,\kh} C^{\jh\kh}\del_{\jh}\del_{\kh}\Gamma_{11}^{\ih}(\varphi_S,0)\phi^{\jh} \phi^{\kh}
		+ T_{KG}^{\ih}[\phi] + S_{KG}^{\ih}[\phi],
		\endaligned
	\end{equation}
with trilinear terms,
	$$
	\aligned
	T_{KG}^{\ih}[\phi] :=& \rho\sum_{\jh,\kh,\lh}C^{\jh\kh\lh}	\del_{\jh}\del_{\kh}\del_l\Gamma_{11}^{\ih}(\varphi_S,0)\phi^{\jh}\phi^{\kh}\phi^{\lh}
	+ 2\sum_{\jh,\kh}C^{\jh\kh}\m(d\phi^1,d\varphi_S)\,\del_{\jh}\del_{\kh}\Gamma_{11}^{\ih}(\varphi_S,0)\phi^{\jh}\phi^{\kh}
	\\
	&+\m(d\phi^1,d\phi^1)\, \del_{\jh}\Gamma_{11}^{\ih}\phi^{\jh}
	+ 2\m(d\phi^1,d\phi^{\jh}) \,\del_{\kh}\Gamma_{1\jh}^{\ih}(\varphi_S,0)\phi^{\kh}
	\\
	&+ 2\sum_{\kh,\lh}C^{\kh\lh}\m(d\phi^{\jh},d\varphi_S)\,\del_{\kh}\del_{\lh}\Gamma_{1\jh}^{\ih}\phi^{\kh}\phi^{\lh}
	+ \m(d\phi^{\jh},d\phi^{\kh})\,\del_{\lh}\Gamma_{\jh\kh}^{\ih}\phi^{\lh}
	\endaligned
	$$
	and higher order terms
	$$
	\aligned
	S_{KG}^{\ih}[\phi] 
	=&(2\m(d\phi^1,d\varphi_S) + \m(d\phi^1,d\phi^1)) C^{\jh\kh l}\del_{\jh}\del_{\kh}\del_l \Gamma_{11}^{\ih}\phi^{\jh}\phi^{\kh}\phi^l
	\\
	&+ \m(d\phi^1,d\phi^1)\del_{\jh}\del_{\kh}\Gamma_{11}^{\ih}(\varphi_S,0)\phi^{\jh}\phi^{\kh}
	+2\m(d\phi^{\jh},d\phi^1)\del_{\lh}\del_{\kh}\Gamma_{1\jh}^1\phi^{\kh}\phi^{\lh}
	\\
	&+\m(d(\varphi_S+\phi^1),d(\varphi_S+\phi^1)) \sum_{|A|\geq 1}C^A  \del_{\jh}\del_{\kh}\del_{\lh}\Gamma_{11}^{\ih}(\varphi_S,0)\phi^A\phi^{\jh}\phi^{\kh}\phi^{\lh}
	\\
	&+ \m(d(\varphi_S+\phi^1),d\phi^{\jh})\sum_{|A|\geq 1}C^A \del^A\del_{\kh}\del_{\lh}\Gamma_{1\jh}^{\ih}(\varphi_S,0)\phi^A\phi^{\kh}\phi^{\lh}
	\\
	&+ \m(d\phi^{\jh},d\phi^{\kh})\sum_{|A|\geq 1} C^A \del^A\del_{\lh}\Gamma_{\jh\kh}^i(\varphi_S,0)\phi^A\phi^{\lh}.
	\endaligned
	$$
\end{subequations}

\subsection{Linear stability and geodesic principle coordinates}\label{subsec1-09-04-2022-M}
We regard the linearization of \eqref{eq5-20-11-2021}:
\begin{equation}\label{eq1-26-11-2021}
\Box\phi^1  = 0,
\quad
\Box\phi^{\ih} - \rho \del_{\jh}\Gamma_{11}^{\ih}(\varphi_S,0)\phi^{\jh} = 0.
\end{equation}
Recalling \eqref{eq2-21-11-2021} and \eqref{eq6-24-02-2022-M}, $\del_{\jh}\Gamma_{11}^{\ih}(\varphi_S,0) = R_{\jh11}^{\ih}(\varphi_S,0) = R_{\jh11\ih}(\varphi_S,0)$ are constants along $\varphi_I(\RR)$. Let $v,w$ be two orthogonal vector fields defined along $\varphi_I(\RR)$. We define
$$
\kappa(v,w) := (R(v,\vec{e}_1)\vec{e}_1,w)_h
$$
which is a bilinear form. We remark that at one point $p = \varphi_I(0)$ and in the geodesic normal coordinates, $\kappa(v,w) = R_{\ih11\jh}v^{\ih}w^{\jh}$ and by the symmetry of Riemann curvature, $\kappa(v,w) = \kappa(w,v)$. By the theory of real quadratic form, there exists another orthonormal base $\{\vec{f}_{\ih}\}_{\ih=2,3,\cdots,n}$ of $T_p\varphi_I(\RR)^{\perp}$, such that within this base, 
$$
\kappa_{\ih\jh}(0) := R\big((\vec{f}_{\ih},\vec{e}_1)\vec{e}_1,\vec{f}_{\jh}\big)_h\big|_{\varphi_I(0)}
$$
forms a diagonal matrix, i.e., $\kappa_{\ih\jh} = 0$ if $\ih\neq\jh$. Then with the base $\{\del_1,\vec{f}_{\ih}\}$ of $T_p\varphi_I(\RR)$ we construct another geodesic normal coordinates chart around $\varphi_I(\RR)$. For the simplicity of expression, we suppose that the geodesic normal coordinates introduced in Subsection \ref{subsec1-24-02-2022-M} satisfy the following property:
\begin{equation}\label{eq1-25-02-2022-M}
\kappa_{\ih\jh}(0) = \big(R(\vec{e}_{\ih},\vec{e}_1)\vec{e}_1,\vec{e}_{\jh}\big)_h\big|_{\varphi_I(0)} = 0,\quad \text{provided that } \ih\neq\jh.
\end{equation}
Recalling that $(N,h)$ is locally symmetric along $\varphi_I$, we apply Proposition \ref{prop1-19-02-2022-M} and see that
\begin{equation}
\kappa_{\ih\jh}(t) =  \big(R(\vec{e}_{\ih},\vec{e}_1)\vec{e}_1,\vec{e}_{\jh}\big)_h\big|_{\varphi_I(t)}	= 0,\quad \forall \ih\neq\jh,\quad t\in \RR.
\end{equation}
Furthermore, the eigenvalues $\kappa_{\ih\ih}$ are also constants along $\varphi_I(\RR)$. These special geodesic normal coordinates are called {\sl geodesic principle coordinates}. With these ones, \eqref{eq1-26-11-2021} reduces to a linear system of one wave equation and $(n-1)$ Klein--Gordon equations. 

It is clear that, when $\varphi$ is space-like / time-like and space-like linear stable /  time-like linear stable,
\begin{equation}\label{eq2-25-02-2022-M}
-\rho\kappa_{\ih\ih}\geq 0,\quad \forall\  \ih=2,3,\cdots, n,
\end{equation}
the linear system \eqref{eq1-26-11-2021} is stable. 
From now on we always suppose that $\gamma$ enjoys one of the above linear stability property.

\subsection{Formulation in geodesic principle coordinates}\label{subsec1-15-04-2022-M}
Under the condition of local symmetry and linear stability, we now give the final formulation of \eqref{eq1-19-11-2021} with the geodesic principle coordinates. 
\begin{subequations}\label{eq2-27-02-2022-M}
\begin{equation}\label{eq1-26-02-2022-M}	
\Box \phi^1 = A^{\alpha}_{\jh} \del_{\alpha}(|\phi^{\jh}|^2) + T_W[\phi] + S_W[\phi],
\end{equation}
where $A_{\jh}^{\alpha} = -\kappa_{\jh\jh}\m^{\alpha\beta}\del_{\beta}\varphi_S$.
\begin{equation}\label{eq1-27-02-2022_M}
\big(\Box+ c_{\ih}^2 \big) \phi^{\ih}  
= K^{\ih\alpha}\phi^{\ih}\del_{\alpha}\phi^1 + E_{\jh\kh}^{\ih} \phi^{\jh}\phi^{\kh}
+  F_{\jh\kh}^{\ih\alpha}\phi^{\jh}\del_{\alpha}\phi^{\kh}
+ T_{KG}^{\ih}[\phi] + S_{KG}^{\ih}[\phi] , \quad 2\leq \ih\leq n,
\end{equation}
where $c_{\ih}^2 = -\rho\kappa_{\ih\ih}$, and
$$
\aligned
K^{\ih\alpha} =& 2R_{\ih11}^{\ih}(\varphi_S,0)\m^{\alpha\beta}\del_{\beta}\varphi_S,\quad \text{remark that } R_{\jh11}^{\ih} = 0 \text{ if } \ih\neq \jh,
\\
E^{\ih}_{\jh\kh} =& \rho \sum_{\jh,\kh}C^{\jh\kh}\del_{\jh}\del_{\kh}\Gamma_{11}^{\ih}(\varphi_S,0),\quad
F^{\ih\alpha}_{\jh\kh} = 2\m^{\alpha\beta}\del_{\beta}\varphi_S\del_{\kh}\Gamma_{1\jh}^{\ih}(\varphi_S,0)
= 2R_{\kh1\jh}^{\ih}(\varphi_S,0)\m^{\alpha\beta}\del_{\beta}\varphi_S.
\endaligned
$$
Here we emphasize that $\ih$ is {\it not} a dummy index. Due to the Klein--Gordon self-interaction terms $E^{\ih}_{\jh\kh}\phi^{\kh}\phi^{\jh}$,  $F_{\jh\kh}^{\ih\alpha}\phi^{\jh}\del_{\alpha}\phi^{\kh}$. We have introduced the non-resonant condition in Definition \ref{def2-09-03-2022-M}.
\end{subequations}
Then we consider the Cauchy problem associate to \eqref{eq2-27-02-2022-M} with the following initial data.
\begin{equation}\label{eq3-27-02-2022-M}
\aligned
&\phi^1(1,x) = \phi^1_0(x),\quad &&\del_t\phi^1(1,x) = \phi_1(x),
\\
&\phi^{\ih}(1,x) = \phi^{\ih}_0(x),\quad &&\del_t\phi^{\ih}(1,x) = \phi^{\ih}_1(x).
\endaligned
\end{equation}
Based on the above formulation, we state our main result in the following form.
\begin{theorem}[Main result, PDE version]
Consider the Cauchy problem associate to \eqref{eq2-27-02-2022-M} together with the initial data \eqref{eq3-27-02-2022-M}.
Then there exists a positive constant $\vep_0$, such that if for $|I|\leq N$, $N\geq 4$, 
\begin{equation}\label{eq1-22-04-2022-M}
\aligned
&\|\la r\ra^{\eta + 1/2 + |I|}\del^I\del_x\phi^1_0\|_{L^2(\RR^2)} 
+ \|\la r\ra^{\eta+1/2+|I|}\del^I \phi^1_1\|_{L^2(\RR^2)}\leq \vep,
\\
&\|\la r\ra^{\eta + |I|}\del^I \phi^{\ih}_0\|_{L^2(\RR^2)} + \|\la r\ra^{\eta + |I|}\del^I\del_x \phi^{\ih}_0\|_{L^2(\RR^2)} 
+\|\la r\ra^{\eta + |I|}\del^I\phi^{\ih}_1\|_{L^2(\RR^2)} \leq \vep
\endaligned
\end{equation}
with $\eta\in(1/2,1)$ and $\vep\leq \vep_0$, the associate local solution extends to time infinity. 
%
\end{theorem}

\begin{remark}
We remark that in \eqref{eq1-27-02-2022_M} there is no quadratic terms $\del_{\alpha}v\del_{\beta}v$. This feature permits us to relax the demand on the regularity of the initial data, say $N\geq 4$. However, our mechanism accepts this kind of nonlinearity (with possibly more demand on the initial regularity). 
\end{remark}
\begin{remark}
In the present version, we only need a relatively weak weight $\la r\ra^{\eta}$ on the initial energies of Klein-Gordon components (cf. \cite{Georgiev-1992}).
The weight on wave component can be relaxed. In fact if the following version:
\begin{equation}
\|\la r\ra^{\eta  + |I|}\del^I\del_x\phi_0^1\|_{L^2(\RR^2)} +  \|\la r\ra^{\eta + |I|}\phi^1_1\|_{L^2(\RR^2)}\leq \vep	
\end{equation}
is acceptable, provided that $N\geq 9$. The proof will be a bit different.
\end{remark}

\section{Examples}\label{sec1-13-05-2022-M}
The main system \eqref{eq2-27-02-2022-M} covers the most general case. In this section we show several concrete examples which demonstrate that all these nonlinear terms do appear in general case, although in they may vanish in some particular cases.

\subsection{$2-$surface of revolution in $\RR^3$}
We consider a $2-$surface of revolution in $\RR^3$:
\begin{equation}\label{eq3-10-05-2022-M}
\Sigma: (x^1,x^2)\mapsto (A(x^2)\cos x^1,A(x^2)\sin x^1, B(x^2))
\end{equation}
with $A,B$ sufficiently regular functions defined on an interval $(-\vep,\vep)$. Here for simplicity we suppose that $|A'|^2+|B'|^2\equiv 1$, i.e., $x^2$ is the arc-length parameter of the curve $(A(x^2), B(x^2))$. We denote by
$$
\del_1 = \del_{x^1} = (-A(x^2)\sin x^1, A(x^2)\cos x^1, 0),\quad \del_2 = \del_{x^2} = (A'(x^2)\cos x^1, A'(x^2)\sin x^1, B'(x^2)).
$$
We firstly show that, if $A(0) = 1, A'(0) = 0$, then the following curve
$$
\gamma: (-\infty,+\infty)\mapsto \Sigma \subset \RR^3,\quad x^1\mapsto (\cos x^1,\sin x^1, 0)
$$
is a geodesic of $\Sigma$. In fact we remark that in $\RR^3$, 
$$
\dot\gamma = (-\sin x^1, \cos x^1, 0),\quad \ddot \gamma  = (-\cos x^1, -\sin x^1, 0).
$$
Remark that $|\dot\gamma| = 1$ and $\ddot\gamma\perp \Sigma$.  This leads to
\begin{equation}\label{eq2-10-05-2022-M}
\nabla_{\dot\gamma}\dot\gamma = \ddot\gamma^{\perp} = 0,
\end{equation}
where $\nabla$ is the connection of $\Sigma$ associate to the induced metric from $\RR^3$, and $\ddot\gamma^{\perp}$ is the orthogonal projection of $\ddot \gamma$ on $T\Sigma$. Thus $\gamma$ is a geodesic of $\Sigma$. On the other hand, the curve
$$
\eta_{x^1_0}:(-\vep,\vep)\mapsto \Sigma,\quad x^2\mapsto (A(x^2)\cos x^1_0,A(x^2)\sin x^1_0, B(x^2))
$$
is also a geodesic. This can be checked similarly. We remark that
$$
\dot\eta_{x^1_0} = (A'(x^2)\cos x^1_0,A'(x^2)\sin x^1_0, B'(x^2)),\quad
\ddot\eta_{x^1_0} = (A''(x^2)\cos x^1_0, A''(x^2)\sin x^1_0, B''(x^2)).
$$
Then $\dot\eta_{x^1_0}\perp \ddot \eta_{x^1_0}$  due to the fact $|A'|^2+|B'|^2\equiv 1$. On the other hand, 
$$
\del_1|_{\eta_{x^1_0}} = \del_{x^1}|_{\eta_{x^1_0}} = (-A(x^2)\sin x^1_0, A(x^2)\cos x^1_0,0)
$$ 
which is also perpendicular to $\ddot \eta_{x^1_0}$. Thus $\ddot\eta_{x^1_0}(x^2)\perp \Sigma$ which leads to the fact that $\eta_{x^1_0}$ is also a geodesic of $\Sigma$. Remark that $\dot\eta_{x_0^1}(0) = \del_2$. Thus 
$$
\exp_{x^1}(x^2\del_2) = \eta_{x^1}(x^2) = (A(x^2)\cos x^1,A(x^2)\sin x^1, B(x^2)).
$$
This shows that the parameterization \eqref{eq3-10-05-2022-M} is a chart of geodesic normal coordinates around $\gamma$. We then calculate the Christoffel symbols and Riemann curvature. Remark that the metric components are written in the following form:
$$
\big(g_{ij}\big) = \left(
\begin{array}{cc}
|A(x^2)|^2&0
\\
0&1
\end{array}
\right),
\quad
(g^{ij}) = \left(
\begin{array}{cc}
	|A(x^2)|^{-2}&0
	\\
	0&1
\end{array}
\right).
$$
Then the Christoffel symbols are
\begin{equation}\label{eq1-11-05-2022-M}
\begin{aligned}
&\Gamma_{11}^1 = \Gamma_{22}^1 = 0,\quad &&\Gamma_{12}^1 = \Gamma_{21}^1 = A'(x^2)/A(x^2),
\\
&\Gamma_{22}^2 = \Gamma_{12}^2 = \Gamma_{21}^2 = 0,\quad &&\Gamma_{11}^2 = -A'(x^2)A(x^2),
\end{aligned}
\end{equation}
and the nonzero components are the Riemann curvature are
\begin{equation}\label{eq2-11-05-2022-M}
\begin{aligned}
&R_{122}^1 = -R_{212}^1 = -A''(x^2)/A(x^2),\quad
R_{121}^2 = - R_{211}^2 = A''(x^2)A(x^2),
\\
&R_{1221}|_{\gamma} = R_{2112}|_{\gamma} = -R_{1212}|_{\gamma} = -R_{2121}|_{\gamma} = - A''(0).
\end{aligned}
\end{equation}
Obviously, the Christoffel symbols and the Riemann curvature tensor are independent of $x^1$. Thus $\Sigma$ is locally symmetric along $\gamma$.   On the other hand, we recall \eqref{eq1-09-03-2022-M} and obtain
$$
\kappa_{22} = R_{2112}|_{\gamma_-} = -A''(0).
$$
That is, following the terminology of Definition \ref{def2-09-03-2022-M}, $\gamma$ is space-like linear stable, provided that $A''(0)>0$ and time-like linear stable provided that $A''(0)<0$. Furthermore, there is only one transversal direction, thus there is no risk of mass resonance. Then our result applies on the above example. Furthermore, recalling \eqref{eq1-27-02-2022_M},
$$
\aligned
K^{\ih\alpha} =&\, 2R_{211}^{2}|_{\gamma}\m^{\alpha\beta}\del_{\beta}\varphi_S = -A''(0)\m^{\alpha\beta}\del_{\beta}\varphi_S,
\\
E_{\jh\kh}^{\ih} =&\,\rho C^{22}\del_2\del_2\Gamma_{11}^2|_{\gamma} = - \rho C^{22}A'''(0),\quad
F_{\jh\kh}^{\ih\alpha}=2R_{212}^2\m^{\alpha\beta}\del_{\beta}\varphi_S|_{\gamma}=0.
\endaligned
$$
That is, when $A''(0), A'''(0)\neq 0$, the coefficients $K^{\ih\alpha}$ and $E_{\jh\kh}^{\ih}$ are in general nonzero. 

\subsection{$\mathbb{T}^2$ embedded in $\RR^3$} In the above example, if we take $A(x^2) = 2+\cos (x^2+\pi)$ and $B(x^2) = \sin (x^2+\pi)$, the above surface reduces to the $2-$torus:
$$
\mathbb{T}^2: (x^1,x^2)\mapsto \big((2+\cos(x^2+\pi))\cos x^1, (2+\cos(x^2+\pi))\sin x^1, \sin (x^2+\pi)\big)
$$
and $\gamma$ is the inner equator of the torus. The (nonzero components of) Christoffel symbols and Riemann curvature are calculated via \eqref{eq1-11-05-2022-M} and \eqref{eq2-11-05-2022-M}:
$$
\begin{aligned}
&\Gamma_{12}^1 = \Gamma_{21}^1 = \frac{\sin x^2}{2-\cos x^2},\quad
\Gamma_{11}^2 = -\sin x^2(2-\cos x^2),
\\
&R_{122}^1 = -R_{212}^1 = -\frac{\cos x^2}{2-\cos x^2},\quad
R_{121}^2 = - R_{211}^2 = \cos x^2(2-\cos x^2),
\\
&R_{1221} = R_{2112} = -R_{1212} = -R_{2121} = - \cos x^2(2-\cos x^2).
\end{aligned}
$$
In this case, $\kappa_{22} = -A''(0) = -1$. That is,  $\gamma$ is space-like linear stable. However, in this case $A'''(0) = 0$ which leads to $E_{\jh\kh}^{\ih} = 0$. That is, there is no quadratic terms acting on Klein-Gordon component. The system then reduces to a case similar to the isotropic case. A similar discussion shows that the outer equator is also a geodesic with the property of local symmetry. Contrary to the inner equator, it is {\it time-like linear stable}. 
\subsection{An example in $\RR^4$}
Our purpose in this subsection is to find an example such that the coefficients $F_{\jh\kh}^{\ih\alpha}$ do not vanish. We first remark that when $n=2$, there is only one transversal direction thus the coefficient $F_{\jh\kh}^{\ih\alpha} = 2R_{212}^2\m^{\alpha\beta}\del_{\beta}\varphi_S = 0$. We thus turn to the $3-$dimensional manifolds. We construct the following metric in the $3-$manifold $\Sigma := \RR\times (-\vep,\vep)\times(-\vep,\vep)$:
$$
\big(g_{ij}\big) = 
\left(
\begin{array}{ccc}
A^2&0&B
\\
0&1&0
\\
B&0&1	
\end{array}
\right),
\quad
\big(g^{ij}\big) = 
\frac{1}{A^2-B^2}\left(
\begin{array}{ccc}
1&0&-B
\\
0&1&0
\\
-B&0&A^2	
\end{array}
\right)
$$
where $A,B$ are sufficiently regular function independent of $x^1$. We suppose that $A(0)=1$, $B(0) = 0$ and $\del_iA(0) = \del_iB(0) = 0$. Let
$$
\gamma: (-\infty,+\infty)\mapsto \Sigma,\quad x^1\mapsto (x^1,0,0).
$$
Remark that along $\gamma$, $g_{ij} = g^{ij} = \delta_{ij}$ and $\del_ig_{jk} = \del_ig^{jk} = 0$. Then $\Gamma_{ij}^k|_{\gamma} = 0$. This leads to the fact that $\gamma$ is a geodesic. Then we denote by 
$$
\aligned
\eta_{x^1_0} : (-\vep,\vep)\mapsto \Sigma,\quad x^2\mapsto (x^1_0,x^2,0),
\\
\zeta_{x^1_0}:  (-\vep,\vep)\mapsto \Sigma,\quad x^2\mapsto (x^1_0,0,x^3).
\endaligned
$$
We remark that
$$
\aligned
\Gamma_{22}^1 =&\frac{1}{2}g^{1\delta}(2\del_2g_{2\delta}-\del_{\delta}g_{22}) = 0,\quad
\Gamma_{33}^1 =\frac{1}{2}g^{1\delta}(2\del_3g_{3\delta} - \del_{\delta}g_{33}) = \frac{1}{2}g^{11}\del_3g_{31} = \frac{\del_3B}{2(A^2-B^2)},
\\ 
\Gamma_{22}^2 =& \frac{1}{2}g^{2\delta}(2\del_2g_{2\delta}-\del_{\delta}g_{22}) = 0,\quad
\Gamma_{33}^2 =\frac{1}{2}g^{2\delta}(2\del_3g_{3\delta} - \del_{\delta}g_{33}) = 0,
\\
\Gamma_{22}^3 =& \frac{1}{2}g^{3\delta}(2\del_2g_{2\delta} + \del_{\delta}g_{22}) = 0,\quad  
\Gamma_{33}^3 = \frac{1}{2}g^{3\delta}(2\del_3g_{3\delta}-\del_{\delta}g_{33}) = g^{31}\del_3g_{31} = \frac{-B\del_3B}{A^2-B^2}.
\endaligned
$$
Here if we chose $\del_3B = 0$, we have $\Gamma_{22}^i = \Gamma_{33}^i = 0$. In this case, both $\eta_{x^1_0}$ and $\zeta_{x^1_0}$ are geodesics of $\Sigma$. Thus the natural parameterization of $\Sigma$ forms a chart of geodesic normal coordinates around $\gamma$. Remark that in this case the Christoffel symbols and the Riemann curvature are independent of $x^1$. Thus $\Sigma$ is locally symmetric along $\gamma$.

On the other hand, recalling \eqref{eq2-21-11-2021}, in the geodesic normal coordinates one has
$$
R_{212}^3|_{\gamma} = \del_2\Gamma_{12}^3|_{\gamma},\quad 
\kappa_{22} = R_{211}^2|_{\gamma} = \del_2\Gamma_{11}^2|_{\gamma},\quad 
\kappa_{33} = R_{311}^3|_{\gamma} = \del_3\Gamma_{11}^3|_{\gamma}.
$$
Direct calculation shows that
$$
\Gamma_{12}^3 = \frac{1}{2}g^{31}\del_2g_{11} + \frac{1}{2}g^{33}\del_2g_{13},\quad
\Gamma_{11}^2 = -\frac{1}{2}\del_2g_{11},\quad \Gamma_{11}^3 = -\frac{1}{2}g^{33}\del_3g_{11}.
$$
Then
$$
R_{212}^3|_{\gamma} = \frac{1}{2}\del_2\del_2B(0),\quad
R_{211}^2|_{\gamma} = -\del_2\del_2A(0),\quad
R_{311}^3|_{\gamma} = -\del_3\del_3A(0).
$$
We can fix for example $B = |x^2|^2$ and $A = \pm\frac{1}{2}(|x^2|^2+\lambda|x^3|^2)+1$. Then
$$
R_{212}^3|_{\gamma} = 1,\quad R_{211}^2|_{\gamma} = \mp 1,  \quad R_{311}^3|_{\gamma} = \mp\lambda.
$$
That is, $\gamma$ is space-like / time-like linear stable and $\Sigma$ is locally symmetric along it. When $\lambda \notin\{ 2, 1/2\}$, $\gamma$ is non-resonant. In this case, 
$F_{22}^{3\alpha} = 2R_{212}^3|_{\gamma}\m^{\alpha\beta}\del_{\beta}\varphi_S = \m^{\alpha\beta}\del_{\beta}\varphi_S$ do not vanish in general.
\part{Euclidean-hyperboloidal foliation in $\RR^{1+2}$.}
\section{Construction of the Euclidean-hyperboloidal foliation}
\label{sec1-11-04-2022-M}
\subsection{Objective}
In this part we reformulate the Euclidean-hyperboloidal foliation in $\RR^{1+2}$. In Sections \ref{sec1-11-04-2022-M} -- \ref{sec1-22-10-2021} we recall the basic results. In Section \ref{sec1-10-04-2022-M} we give an improvement on the Sobolev decay estimates. The last two sections are devoted two new ingredients: the conformal estimate on Euclidean-hyperboloidal slices and the adapted normal form transform. For the parallel tools and result we will not give detailed proof but only translate the statement into the present system of notation.

\subsection{Basic notation}
We use the expression $A \lesssim B$ for $A\leq CB$ with $C$ a universal constant, and $A\cong B$ for $A\lesssim B$ and $B\lesssim A$.  We are working in $\RR^{1+2} = \{(t,x) = (x^0,x^1,x^2)\}$ equipped with the standard Minkowski metric
$$
\m_{\alpha\beta}dx^{\alpha}dx^{\beta} = dt^2  - d(x^1)^2 - d(x^2)^2.
$$
The D'Alembert operator is defined as $\Box = \m^{\alpha\beta}\del_{\alpha}\del_{\beta} = \del_t\del_t - \del_1\del_1 - \del_2\del_2$. 

Let $\chi:\RR\rightarrow \RR$ be a smooth cut-off function. It satisfies the following properties:
\begin{subequations}
\begin{equation}\label{eq1-03-10-2021}
\chi(x) = 
\left\{
\aligned
&1,\quad &x\geq 1,
\\
&0,\quad &x\leq 0.
\endaligned
\right. 
\end{equation} 
Furthermore, on the interval $(0,1)$ we demand that
\begin{equation}\label{eq6-19-10-2021}
\chi'(x)>0, \quad |\chi'(x)| + |\chi''(x)|\lesssim \chi(x)^{1/2}.
\end{equation}
\end{subequations}
In \cite{M-2018} an explicit example of such $\chi$ is given. 

\subsection{Construction of the foliation}
We briefly recall the construction of the Euclidean-hyperboloidal slice. The following results are parallel to those established in \cite{M-2018} for $\RR^{1+1}$ case. Recall the cut-off function 
\begin{equation}\label{eq2-03-10-2021}
\xi(s,r): = 1 - \chi(r - \rhoH(s))\quad \text{where}\quad \rhoH(s) := (s^2-1)/2.
\end{equation}
Then $\xi(s,\cdot): \RR\rightarrow \RR$ is a smooth cut-off function. 
\begin{equation}
\xi(s,r) = \left\{
\aligned
& 1, \quad && r\leq \rhoH(s),
\\
& 0,\quad && r\geq \rhoH(s)+1.
\endaligned
\right.
\end{equation}
This function is strictly decreasing for $r\in (\rhoH(s),\rhoH(s)+1)$.

We then introduce the time function by the following ODE for $s\geq 2$:
\begin{equation}\label{eq1-04-10-2021}
\del_r T(s,r) = \frac{\xi(s,r)r}{(s^2+r^2)^{1/2}},\quad T(s,0) = s.
\end{equation}
Direct calculation shows that ($r:=|x|$)
\begin{equation}\label{eq3-03-10-2021}
T(s,x) = \left\{
\aligned
&(r^2+s^2)^{1/2},&&0\leq r\leq \rhoH(s),
\\
& s + \int_0^r\frac{\xi(s,\rho)\rho}{\sqrt{s^2+\rho ^2}}d\rho, &&\rhoH(s) < r < \rhoH(s)+1,
\\
&T(s) = s + \int_0^{\rhoH(s)+1}\frac{\xi(s,\rho)\rho}{\sqrt{s^2+\rho^2}}d\rho, &&\rhoH(s)+1\leq r <\infty.
\endaligned
\right.
\end{equation}
Then we define the Euclidean-hyperboloidal slice by
$$
\Fcal_s := \{(t,x)| t = T(s, r)\}
$$
which will take the role of the space-like hyperboloids $\Hcal_s$ in the hyperboloidal foliation framework.
\begin{figure}[ht]
\includegraphics[height=7cm]{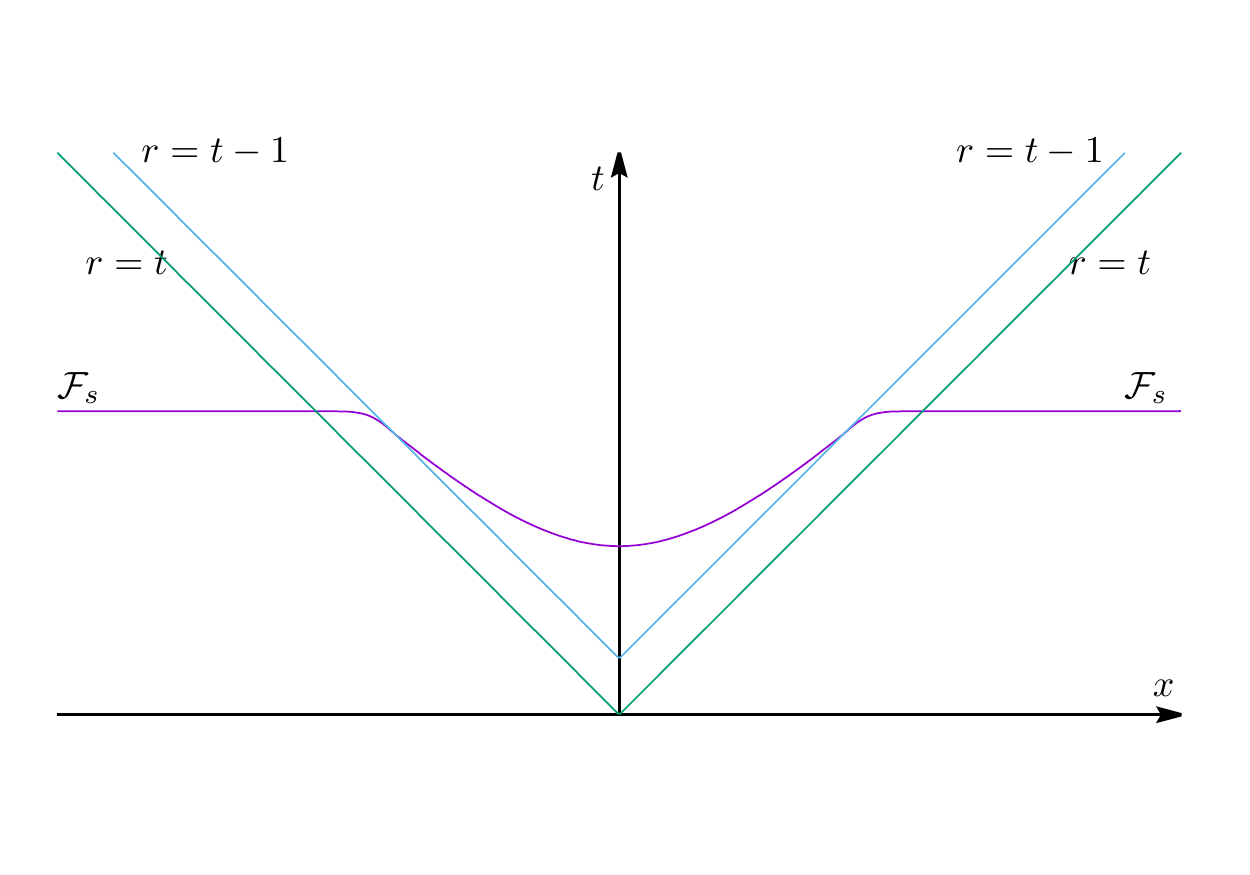}
\centering
\caption{}
\label{fig:picture}
\end{figure}
The figure \ref{fig:picture} shows approximately the form of such slices (in $\RR^{1+1}$). For the convenience of discussion, we introduce the following notation:
$$
\Fcal_{[s_0,s_1]} = \{(t,x)|T(s_0,x)\leq t\leq T(s_1,x)\}, \quad \Fcal_{[s_0,+\infty)} = \{(t,x)|T(s_0,x)\leq t\}
$$
the region limited by one or two such surfaces. We write
$$
\Fcal_{[2,+\infty)} = \bigcup_{s\geq 2}\Fcal_s
$$
which is an one-parameter foliation of the region $\Fcal_{[2,+\infty)}\subset\RR^{1+2}$,  called the {\sl Euclidean-hyperboloidal foliation}.

\subsection{Geometry of the Euclidean-hyperboloidal foliation}
We recall the results  established in \cite{M-2018} on the geometry of $\Fcal_s$.
\begin{proposition}\label{prop1-07-10-2021}
	$\Fcal_s$ are $C^{\infty}$ surfaces. For $r\leq\rhoH(s)$,
	$$
	T(s,x) = (s^2+r^2)^{1/2}.
	$$
	and $\Fcal_s$ coincides with $\Hcal_s$.
	For $r\geq \rhoH(s)+1$, $T(s,x)$ is constant with respect to $r$ and
	\begin{equation}\label{eq9-07-10-2021}
	\frac{s^2+1}{2}\leq T(s,x)\leq \frac{\sqrt{s^4+6s^2+1}}{2}.
	\end{equation}
\end{proposition}

We list out some geometric facts about $\Fcal_s$. The normal vector (with respect to Euclidian metric) of $\Fcal_s$ is
\begin{equation}\label{eq10-07-10-2021}
\vec{n} =
\left\{
\aligned
&(t^2+r^2)^{-1/2}(t,-x) = \Big(\frac{s^2+r^2}{s^2+2r^2}\Big)^{1/2}\Big(1,\frac{-x}{(s^2+r^2)^{1/2}}\Big),\quad &&\text{in }\Hcal^*_s,
\\
&\Big(\frac{s^2+r^2}{s^2+(1+\xi(s,r))r^2}\Big)^{1/2}\Big(1,\frac{-\xi(s,r)x}{(s^2+r^2)^{1/2}}\Big),\quad &&\text{in } \Tcal_s,
\\
&(1,0),\quad  &&\text{in } \Pcal_s.
\endaligned
\right.
\end{equation}
The volume element of $\Fcal_s$ (with respect to the Euclidean metric) is
\begin{equation}\label{eq11-07-10-2021}
d\sigma = \sqrt{1+|\del_r T|^2} dx =
\left\{
\aligned
& \frac{\sqrt{x^2+t^2}}{t} =\frac{\sqrt{s^2+2x^2}}{\sqrt{s^2+x^2}},\quad && \text{in } \Hcal^*_s,
\\
& \Big(\frac{s^2+(1+\xi^2(s,r))r^2}{s^2+r^2}\Big)^{1/2}, \quad &&\text{in } \Tcal_s,
\\
& 1,\quad  &&\text{in }\Pcal_s.
\endaligned
\right.
\end{equation}

For the convenience of discussion, the slice $\Fcal_s$ is divided into three pieces:
$$
\aligned
&\text{The hyperbolic part}\quad &&\Hcal^*_s := \big\{(T(s,r),x)\big|0\leq r\leq \rhoH(s)\big\},
\\
&\text{The transition part} \quad &&\Tcal_s: = \big\{(T(s,r),x)\big|\rhoH(s)\leq r\leq \rhoH(s)+1\big\},
\\
&\text{The flat part} \quad &&\Pcal_s:=\big\{(T(s,r),x)\big|r\geq \rhoH(s)+1\big\}.
\endaligned
$$
Remark that the part $\Hcal^*_s$ is a part of the hyperboloid with radius $s$, and $\Pcal_s$ is part of the plane $\{t=T(s) = T(s,\rhoH(s)+1)\}$. The part $\Tcal_s$ joints the above two pars together in a smooth manner. We also introduce the following partition of spacetime:
$$
\aligned
&\text{Hyperbolic region}\quad&&\Hcal^*_{[s_0,\infty)} : = \big\{(t,x)\big|T(s_0,x)\leq t, r\leq \rhoH(s)\big\},
\\
&\text{Transition region}\quad&&\Tcal_{[s_0,\infty)} : = \big\{(t,x)\big|T(s_0,x)\leq t , \rhoH(s)\leq r\leq \rhoH(s)+1\big\},
\\
&\text{Flat region}\quad&&\Pcal_{[s_0,\infty)} : = \big\{ T(s_0,x)\leq t\big| r\geq \rhoH(s)+1 \big\}.
\endaligned
$$
We also denote by
$$
\aligned
&\Hcal^*_{[s_0,s_1]} : = \big\{(t,x)\big|T(s_0,x)\leq t\leq T(s_1,x), r\leq \rhoH(s)\big\},
\\
&\Tcal_{[s_0,s_1]} : = \big\{(t,x)\big|T(s_0,x)\leq t\leq T(s_1,x), \rhoH(s)\leq r\leq \rhoH(s)+1\big\},
\\
&\Pcal_{[s_0,s_1]} : = \big\{ (t,x)\big| T(s_0,x)\leq t\leq T(s_1,x), r\geq \rhoH(s)+1\big\}.
\endaligned
$$

For the convenience of discussion, we also denote by
$$
\TPcal_s := \Tcal_s\cup \Pcal_s,\quad 
\TPcal_{[s_0,\infty)}:= \Tcal_{[s_0,\infty)}\cup\Pcal_{[s_0,\infty)}, \quad 
\TPcal_{[s_0,s_1]} := \Tcal_{[s_0,s_1]}\cup\Pcal_{[s_0,s_1]}
$$
which are called ``exterior part'' and ``exterior regions''. The frontier between exterior region and hyperbolic region
$$
\aligned
&\Kcal_s := \{(t,x)| r = \rhoH(s), t = \rhoH(s)+1\},
\\
&\Kcal_{[s_0,\infty)} :=\bigcup_{s_0\leq s}\Kcal_s = \{(t,x)|r=t-1,\rhoH(s_0)+1\leq t \},
\\
&\Kcal_{[s_0,s_1]} := \bigcup_{s_0\leq s\leq s_1}\Kcal_s = \{(t,x)|r=t-1,\rhoH(s_0)+1\leq t\leq \rhoH(s_1)+1 \},
\endaligned
$$

In the above regions, the following bounds hold:
\begin{lemma}\label{lem1-07-10-2021}
	Let $(t,x)\in \Fcal_{[s_0,s_1]}$ and $s_0\geq 2$, then
	\begin{equation}\label{eq6-07-10-2021}
	r\,\left\{
	\aligned
	\leq& t-1, \quad&&(t,x)\in \Hcal^*_{[s_0,s_1]},
	\\
	\in& [t-1,t - c(s)],\quad &&(t,x)\in\Tcal_{[s_0,s_1]},
	\\
	\geq& t - c(s),\quad &&(t,x)\in \Pcal_{[s_0,s_1]}.
	\endaligned
	\right.
	\end{equation}
	where $1>c(s)>0$ is determined by the function $\chi$.
\end{lemma}
This result only depends on the time function $T$. We recall \cite{M-2018} (Lemma 2.2) for detailed proof. For the convenience of discussion, we introduce
$$
\TPcal_s^{\near}:= \TPcal_s\cap \{t-1\leq r\leq 3t\},\quad  \TPcal_s^{\far} := \TPcal_s\cap \{r\geq 2t\},
$$
which are referred to as ``near-light-cone region'' and ``away-from-light-cone region''.

Remark that the Euclidean-hyperboloidal foliation also defines a parameterization $\Fcal_{[s_0,\infty)}$ by $(s,x)$. We now recall some key features of this parameters established in \cite{M-2018}. First, recalling the relation
$$
t = T(s,r), \quad x^a=x^a,
$$
the Jacobian between parameterizations $(s,x)$ and $(t,x)$ is 
\begin{equation}\label{eq3-07-10-2021}
J = \det\bigg(\frac{\del(t,x)}{\del(s,x)}\bigg) = \del_sT.
\end{equation}
Then we recall the following estimates on $\del_sT$.
\begin{lemma}\label{lem2-07-10-2021}
Let $T$ be defined as in \eqref{eq1-04-10-2021}. Then in $\Fcal_{[s_0,\infty)}$ with $s_0\geq 2$,
\begin{equation}\label{eq7-07-10-2021}
0<(1-\xi(s,r))s + \frac{\xi(s,r)s}{\sqrt{s^2+x^2}}\leq
\del_sT(s,r)
\leq \frac{\xi(s,r)s}{\sqrt{s^2+r^2}} + 2(1-\xi(s,r))s.
\end{equation}
Especially, 
\begin{equation}\label{eq8-07-10-2021}
\aligned
\del_sT(s,r) =& \frac{s}{(s^2+r^2)^{1/2}} = s/t,\quad &&\text{in}\quad \Hcal^*_{[s_0,\infty)},
\\
s\leq \del_sT(s,r)\leq& s(1 + 2(s^2-1)^{-1}),\quad &&\text{in}\quad \Pcal_{[s_0,\infty)}.
\endaligned
\end{equation}
\end{lemma}
This was established in \cite{M-2018} Lemma 2.3.

\subsection{Frames and vector fields}
Parallel to the case of $\RR^{1+1}$ in \cite{M-2018} and $\RR^{3+1}$ in \cite{LM-2022}, we introduce in $\Fcal_{[s_0,+\infty)}$ the following vector fields:
$$
\del_0 = \del_t,\quad \del_a = \del_{x^a}.
$$
We introduce the following vector filed:
$$
L_a := x^a\del_t+t\del_a
$$
which is called the {\sl Lorentzian boosts}.
We also denote by
$$
\delu_a := (x^a/t)\del_t + \del_a.
$$

In $\Fcal_{[s_0,\infty)}$ we introduce the following semi-hyperboloidal frame ({\sl SHF} for short) :
$$
\delu_0 := \del_t,\quad \delu_a = \delu_a.
$$
The transition matrices between SHF and the canonical frame $\{\del_t,\del_a\}$ are:
$$
\Phiu_{\alpha}^{\beta} = \left(
\begin{array}{ccc}
1 &0 &0
\\
x^1/t &1 &0
\\
x^2/t &0 &1
\end{array}
\right),
\quad
\Psiu_{\alpha}^{\beta} = \left(
\begin{array}{ccc}
1 & 0 &0
\\
-x^1/t &1 &0
\\
-x^2/t &0 &1
\end{array}
\right)
$$
with
$$
\delu_\alpha = \Phiu_{\alpha}^{\beta}\del_\beta,\quad \del_{\alpha} = \Psiu_{\alpha}^{\beta}\delu_{\beta}.
$$

We also introduce the following {\sl semi-null} frame ({\sl SNF} for short)\footnote{The word ``null'' stands for the fact that $\delt_a$ are tangent to the null cones $\{r = t+C\}$} in the region $\Fcal_{[s_0,\infty)}\cap \{r > t/2\}$, defined as following:
$$
\delt_0:= \del_t,\quad \delt_a = (x^a/r)\del_t + \del_a.
$$
The transition matrices between SNF and the canonical frame are:
$$
\Phit_{\alpha}^{\beta} = \left(
\begin{array}{ccc}
1 &0 &0
\\
x^1/r &1 &0
\\
x^2/r &0 &1
\end{array}
\right),
\quad
\Psit_{\alpha}^{\beta} = \left(
\begin{array}{ccc}
1 & 0 & 0
\\
-x^1/r &1 &0
\\
-x^2/r &0 &1
\end{array}
\right)
$$
with
$$
\delt_\alpha = \Phit_{\alpha}^{\beta}\del_\beta,\quad \del_{\alpha} = \Psit_{\alpha}^{\beta}\delt_{\beta}.
$$

In $\Fcal_{[2,+\infty)}$, we define the following frame (called the {\sl tangent frame}, or {\sl TF} for short):
$$
\delb_0 := \del_t ,\quad \delb_a := \delb_a = \frac{\xi(s,r)x^a}{(s^2+r^2)^{1/2}}\del_t + \del_a.
$$
The transition matrices between TF and the canonical frame are
$$
\Phib_{\alpha}^{\beta} = \left(
\begin{array}{ccc}
1 &0 &0
\\
\frac{\xi(s,r)x^1}{(s^2+r^2)^{1/2}} &1 &0
\\
\frac{\xi(s,r)x^2}{(s^2+r^2)^{1/2}} &0 &1
\end{array}
\right),
\quad
\Psib_{\alpha}^{\beta} = \left(
\begin{array}{ccc}
1 &0 &0
\\
\frac{-\xi(s,r)x^1}{\sqrt{s^2+x^2}} &1 &0
\\
\frac{-\xi(s,r)x^2}{\sqrt{s^2+x^2}} &0 &1
\end{array}
\right),
$$
with
$$
\delb_{\alpha} = \Phib_{\alpha}^{\beta}\del_{\beta}u, \quad \del_{\alpha} = \Psib_{\alpha}^{\beta}\delb_{\beta}.
$$

Let $T$ be a two tensor defined in $\Fcal_{[2,+\infty)}$. Then it can be written in different frame as following:
$$
T = T^{\alpha\beta}\del_{\alpha}\otimes\del_{\beta} = \Tu^{\alpha\beta}\delu_{\alpha}\otimes\delu_{\beta} = \Tb^{\alpha\beta}\delb_{\alpha}\otimes\delb_{\beta} = \Tt^{\alpha\beta}\delt_{\alpha}\otimes\delt_{\beta}.
$$
For a three-tensor $Q$, the same relation holds:
$$
Q = Q^{\alpha\beta\gamma}\del_{\alpha}\otimes\del_{\beta}\otimes\del_{\gamma} =
\Qu^{\alpha\beta\gamma}\delu_{\alpha}\otimes\delu_{\beta}\otimes\delu_{\gamma} =
\Qb^{\alpha\beta\gamma}\delb_{\alpha}\otimes\delb_{\beta}\otimes\delb_{\gamma} = \Qt^{\alpha\beta\gamma}\delt_{\alpha}\otimes\delt_{\beta}\otimes\delt_{\gamma}.
$$
We remark the following relation:
$$
\aligned
&\Tu^{\alpha\beta} = \Psiu_{\alpha'}^{\alpha}\Psiu_{\beta'}^{\beta}T^{\alpha'\beta'},\quad 
\quad
&&\Tb^{\alpha\beta} = \Psib_{\alpha'}^{\alpha}\Psib_{\beta'}^{\beta}T^{\alpha'\beta'},
\quad
&&\Tt^{\alpha\beta} = \Psit_{\alpha'}^{\alpha}\Psit_{\beta'}^{\beta}T^{\alpha'\beta'},
\\
&\Qu^{\alpha\beta\gamma} = 
\Psiu_{\alpha'}^{\alpha}\Psiu_{\beta'}^{\beta}\Psiu_{\gamma'}^{\gamma}Q^{\alpha'\beta'\gamma'},
\quad 
&&\Qb^{\alpha\beta\gamma} = \Psib_{\alpha'}^{\alpha}\Psib_{\beta'}^{\beta}\Psib_{\gamma'}^{\gamma}Q^{\alpha'\beta'\gamma'},
\quad
&&\Qt^{\alpha\beta\gamma} = \Psit_{\alpha'}^{\alpha}\Psit_{\beta'}^{\beta}\Psit_{\gamma'}^{\gamma}Q^{\alpha'\beta'\gamma'}.
\endaligned
$$

\section{Energy estimate with Euclidean-hyperboloidal foliation}
This section is devoted to the discussion on the energy estimate within Euclidean-hyperboloidal foliation. We firstly introduced / recall some frequently used norms and functional spaces in the first subsection. The energy estimate is established in the second subsection. In the last subsection we we give a detailed analysis on the structure of the energy on Euclidean-hyperboloidal slices. 
\subsection{Functional spaces}
We are working in $\Fcal_{[s_0,s_1]}$. Let $\Fcal_{(s_0,s_1)} = \{(t,x)| T(s_0,r)<t<T(s_1,r)\}$. We denote by $C_c^{\infty}(D)$ the smooth functions defined in $\RR^d$ with supports contained in $D$. Then let
$\Scal_{[s_0,s_1]} := \{u\in C_c^{\infty}(\Fcal_{(s_0-1,s_1+1)})\}$. Let $u\in \Scal_{(s_0,s_1)}$ and for $s\in(s_0,s_1)$, we denote by 
$$
u_s : = u(T(s,r),x)
$$
the restriction of $u$ on $\Fcal_s$. Then for $1\leq p<\infty$, 
$$
\|u\|_{L^p(\Fcal_s)}^p := \int_{\RR^2} \big|u(T(s,r),x)\big|^pdx,
\quad
\|u\|_{L^{\infty}(\Fcal_s)} := \sup_{x\in\RR^2}|u_s(x)|.
$$
Furthermore, 
$$
\|u\|_{L_p^\infty([s_0,s_1])} := \sup_{s\in[s_0,s_1]}\big\{\|u\|_{L^p(\Fcal_s)}\big\},
\quad
\|u\|_{L_p^q([s_0,s_1])} := \Big(\int_{[s_0,s_1]}\|w_s\|_{L^p(\RR)}^q\ ds\Big)^{1/q},\quad 1\leq q<\infty.
$$
We denote by $L^q_p([s_0,s_1])$ the completion of $\mathcal{S}_{[s_0,s_1]}$ with respect to the norm $\|\cdot \|_{L_p^q([s_0,s_1])}$. In the following discussion, almost all functions under discussion are in $L^{\infty}_2([s_0,s_1])$.

\subsection{Weight function}
For our purpose we need to introduce a special weight function. Let $\aleph: \RR\rightarrow \RR$ such that
$$
\aleph(r) = 
\left\{
\aligned
&0, \quad&& r\leq 0,
\\
&r+1,\quad &&r\geq 1
\endaligned
\right.
$$
and $2\geq \aleph^{\prime}\geq 0$. 
Then we introduce the weight function
$$
\omega(t,r) := 1 + \aleph(2+r-t)
=\left\{
\aligned
&1,\quad && r-t\leq -2,
\\
&2+r-t,&& r-t\geq -1.
\endaligned
\right.
$$
which is non-trivial only in the flat region. Furthermore,
\begin{equation}\label{eq2-05-10-2021}
\del_a\omega = (x^a/r)\aleph^{\prime}(2+r-t),
\quad
\del_t\omega = -\aleph^{\prime}(2+r-t).
\end{equation}
Here remark that $\aleph^{\prime}$ is bounded. Thus
$$
|\del_a\omega| + |\del_t\omega|\leq C(\aleph).
$$

\subsection{Energy estimate}\label{subsec1-24-11-2021}
We apply the multiplier $2\omega^{2\eta}\del_t u$ and obtain the following identity:
\begin{equation}\label{eq1-06-10-2021}
2\omega^{2\eta}\del_tu\, \big(\Box u + c^2u\big) = \text{div} V_{\eta,c}[u]
 + 2\eta \omega^{2\eta-1}\aleph^{\prime}(2+r-t) \Big(\sum_a|\delt_a u|^2  + c^2u^2\Big)
\end{equation}
with
$$
V_{\eta,c}[u] := 
\Big(\omega^{2\eta}(|\del_t u|^2 + \sum_a|\del u|^2 + c^2u^2),\, -2\omega^{2\eta}\del_tu\del_au\Big).
$$
Then we integrate \eqref{eq1-06-10-2021} in $\Fcal_{[s_0,s_1]}$ and apply the Stokes' formula (with respect to Euclidean metric). Here we need to remark that
\begin{equation}\label{eq3-26-01-2022}
\vec{n} d\sigma = \Big(1,-\frac{\xi(s,r)x^a}{(s^2+r^2)^{1/2}}\Big).
\end{equation}
Then we obtain
\begin{equation}\label{eq1-07-10-2021}
\aligned
\Ebf_{\eta,c}(s_1,u) - \Ebf_{\eta,c}(s_0,u) 
+& 2\eta\int_{\Fcal_{[s_0,s_1]}}\!\!\!\!\!\!\!\!\omega^{2\eta-1}\aleph^{\prime}(2+r-t)\Big(\sum_a|\delt_au|^2 +c^2u^2 \Big)dxdt 
\\
=& 2\int_{\Fcal_{[s_0,s_1]}}\!\!\!\!\!\!\!\!\omega^{2\eta}\del_t u\big(\Box u + c^2u\big)dxdt.
\endaligned
\end{equation}
Here 
\begin{equation}\label{eq2-07-10-2021}
\aligned
\Ebf_{\eta,c}(s,u) =& \int_{\Fcal_s} \omega^{2\eta}\Big(|\del_t u|^2 + \sum_a|\del_au|^2 + \frac{\xi(s,r)x^a}{(s^2+r^2)^{1/2}}\del_au\del_tu + c^2u^2\Big)dx
\\
=&\int_{\Fcal_s} \omega^{2\eta}\Big(|\zeta \del_t u|^2 + \sum_a|\delb_au|^2 + c^2u^2\Big) dx
\\
=& \int_{\Fcal_s} \omega^{2\eta}\Big(\sum_a |\zeta\del_a u|^2 + |\delb_{\perp}u|^2 + \frac{\xi^2(s,r)r^2}{s^2+r^2}|r^{-1}\Omega u|^2 + c^2u^2\Big)dx
\endaligned
\end{equation}
where
\begin{equation}\label{eq5-23-01-2022}
\zeta^2 = 1 - \frac{\xi^2(s,r)r^2}{s^2+r^2} ,\quad \delb_{\perp} := \del_t + \frac{\xi(s,r)x^a}{(s^2+r^2)^{1/2}}\del_a
\end{equation}
and $\Omega = x^1\del_2 - x^2\del_1$ the rotation vector.
Here we remark that the derivative $\delb_{\perp}$ is orthogonal to $\Fcal_s$ with respect the Minkowski metric.       
Then recall \eqref{eq3-07-10-2021} and the relation $dtdx = Jdsdx$, \eqref{eq1-07-10-2021} is written as
$$
\aligned
\Ebf_{\eta,c}(s_1,u)  +& 2\eta\int_{s_0}^{s_1}\int_{\Fcal_s} \omega^{2\eta-1}\aleph^{\prime}(2+r-t)\Big(\sum_a|\delt_a u|^2 + c^2u^2\Big)\,Jdxds
\\
& =\Ebf_{\eta,c}(s_0,u) + 2\int_{s_0}^{s_1}\int_{\Fcal_s} \omega^{2\eta}\del_tu\big(\Box u + c^2 u^2\big)\,Jdxds
\endaligned
$$

For the convenience of discussion, we denote by 
$$
\ebf_c[u] := |\del_t u|^2 + \sum_a|\del_au|^2 + \frac{\xi(s,r)x^a}{(s^2+r^2)^{1/2}}\del_au\del_tu + c^2u^2
$$
the (unweighted) energy density. We use the abbreviation $\ebf[u] = \ebf_0[u]$.  Then we decompose the energy in three pieces:
\begin{equation}
	\aligned
	\Ebf_{\eta,c}(s,u)
	=\int_{\Hcal^*_s} + \int_{\Tcal_s}  + \int_{\Pcal_s}\omega^{2\eta}\ebf_c[u]dx
	=: \Ebf^{\Hcal}_{\eta,c}(s,u) + \Ebf^{\Tcal}_{\eta,c}(s,u) + \Ebf^{\Pcal}_{\eta,c}(s,u).
	\endaligned
\end{equation} 
In the following discussion we often use the notation $\Ebf^{\TPcal}_{\eta,c}(s,u) := \Ebf^{\Tcal}_{\eta,c}(s,u) + \Ebf^{\Pcal}_{\eta,c}(s,u)$. 

Then we can also integrate \eqref{eq1-06-10-2021} in $\TPcal_{[s_0,s_1]}$ and obtain an exterior energy estimate, as following. We denote by $\del\Kcal_{[s_0,s_1]} = \{r = t-1| (s_0^2+1)/2\leq  t\leq (s_1^2+1)/2\}$ which is the frontier between $\Hcal^*_{[s_0,s_1]}$ and $\TPcal_{[s_0,s_1]}$. Its normal vector and volume element (with respect to Euclidean metric) is
\begin{equation}
\vec{n} = 2^{-1/2}(1,-x^a/r),\quad d\sigma = 2^{1/2}dx.
\end{equation}
Then we integrate  \eqref{eq1-06-10-2021} in $\TPcal_{[s_0,s_1]}$,
\begin{equation}\label{eq5-30-10-2021}
\aligned
 &\Ebf_{\eta,c}^{\TPcal}(s_1,u) - \Ebf_{\eta,c}^{\TPcal}(s_0,u) 
+ \int_{\del\Kcal_{[s_0,s_1]}}\!\!\!\!\!\!\!\!\!\!\!\!V_{\eta,c}[u]\cdot \vec{n}d\sigma 
\\
&+ 2\eta\int_{\TPcal_{[s_0,s_1]}}\!\!\!\!\!\!\!\!\!\!\!\!
\omega^{2\eta-1}
\aleph^\prime(2+r-t)
|\delt_a u|^2\,Jdx ds
= \int_{\TPcal_{[s_0,s_1]}}\!\!\!\!\!\!\!\!\!\!\!\!
\omega^{2\eta}\del_tu \big(\Box u + c^2u\big)\,Jdxds.
\endaligned
\end{equation}
We remark that
$$
\int_{\del\Kcal_{[s_0,s_1]}}\!\!\!\!\!\!\!\!\!\!\!\!V_{\eta,c}\cdot \vec{n}d\sigma 
= \int_{s_0}^{s_1}\int_{\{|x| = \rhoH(s)\}}\!\!\!\!\!\!\!\! c^2u^2 + \sum_a|\delt_au|^2\,J d\phi \geq 0,
$$
where $d\phi$ is the volume element of the $1$-sphere $\{t = \rhoH(s)+1, |x| = \rhoH(s)\}$. 
Then we conclude by the following result.
\begin{proposition}\label{prop2-24-11-2021}
Let $u$ be a $C^2$ solution to the following wave / Klein--Gordon equation
\begin{equation}
\Box u + c^2u = f
\end{equation}
and vanishes sufficiently fast at spatial infinity. Then when $\eta>0$,
\begin{equation}\label{eq7-22-03-2022-M}
\aligned
&\Ebf_{\eta,c}(s,u) 
+2\eta\int_{s_0}^{s_1}\int_{\Fcal_s}\omega^{2\eta-1}\aleph^{\prime}(2+r-t)\Big(\sum_a|\delt_au|^2+c^2u^2\Big)\,Jdxds 
\\
&\qquad \leq\, \Ebf_{\eta,c}(s_0,u) + 2\int_{s_0}^{s_1}\int_{\Fcal_s}|\omega^{2\eta}f\del_t u|\,Jdxds .
\endaligned
\end{equation}
Furthermore,
\begin{equation}\label{eq5-14-03-2022-M}
\aligned
&\Ebf_{\eta,c}^{\TPcal}(s,u) + 2\eta\int_{s_0}^{s_1}\int_{\TPcal_s}\omega^{2\eta-1}\Big(\sum_a|\delt_au|^2 + c^2u^2\Big)\,Jdxds 
\\
&\qquad \leq \Ebf_{\eta,c}^{\TPcal}(s_0,u) + 2\int_{s_0}^{s_1}\int_{\TPcal_s} |\omega^{2\eta}f\del_t u|\,Jdxds.
\endaligned
\end{equation}
\end{proposition}
\begin{remark}
This is the flat version of Propositions 3.6 and 3.7 of \cite{LM-2022}. However, different from the case in $\RR^{1+3}$, here the space-time integration controlled on the left-hand side will play essential roles in the following discussion. For the Klein-Gordon equation ($c>0$), this is essentially the ghost weight energy introduced in \cite{Dong-2021,Dong-2020-2}.
\end{remark}
\begin{remark}
In Subsection \ref{subsec1-22-04-2022-M} when we construct the initial data from $\{t=2\}$ to $\Fcal_2$, we need to control the energy $\Ebf_{\eta,c}(2,u)$ by the flat energy on $\{t=2\}$. This is done by integrating \eqref{eq1-06-10-2021} in the region $\Fcal^{\flat}_2 := \{(t,x)|2\leq t\leq T(2,|x|)\}$. We denote by 
$$
\Ebf^{\flat}_{\eta,c}(t,u) : = \int_{\RR^2} \omega^{2\eta}(|\del_t u|^2 + \sum_a|\del u|^2 + c^2u^2)(t,\cdot)\, dx
$$
the weighted flat energy. Then
\begin{equation}\label{eq2-22-04-2022-M}
\Ebf_{\eta,c}(2,u) = \Ebf^{\flat}_{\eta,c}(2,u) + 2\int_{\Fcal^{\flat}_2}\omega^{2\eta}\del_tu\, \big(\Box u + c^2u\big)dxdt.
\end{equation}
\end{remark}

\subsection{Analysis on the energy}
For further discussion, we need a detailed analysis on the weight function $\zeta$ appearing naturally in the expression of $\Ebf_{\eta,c}(s,u)$. We firstly write (see also in \cite{M-2018} and \cite{LM-2022})
\begin{equation}
\zeta(s,x) = \Big(1-\frac{\xi(s,r)r^2}{s^2+r^2}\Big)^{1/2} 
=\left\{
\aligned
&\frac{s}{(s^2+r^2)^{1/2}} = s/t,\quad && \text{in } \Hcal^*_s,
\\
&\Big(1-\frac{\xi(s,r)r^2}{s^2+r^2}\Big)^{1/2} ,\quad && \text{in } \Tcal_s,
\\
&1,\quad &&\text{in } \Pcal_s.
\endaligned
\right. 
\end{equation}

Then

On $\Hcal_s^*$, remark that $\xi(s,r)\equiv 1, \zeta = (s/t)$ and $\omega\simeq 1$. Then
$$
\aligned
\Ebf^{\Hcal}_{\eta,c}(s,u) \simeq \Ebf_c^{\Hcal}(s,u):= & \int_{\Hcal^*_s}\Big(\big|(s/t)\del_tu\big|^2 + \sum_a|\delu_a u|^2 + c^2u^2\Big)\, dx
\\
=& \int_{\Hcal^*_s}\Big(\big|\delu_{\perp}u\big|^2 + \sum_a|(s/t)\del_a u|^2 + |t^{-1}\Omega u|^2 + c^2u^2\Big)\, dx
\endaligned
$$
which is the standard energy on hyperboloids. The we turn to the exterior region.  On $\Pcal_s$, remark that $\xi(s,r) \equiv 0, \zeta\equiv 1$. Then 
$$
\Ebf^{\Pcal}_{\eta,c}(s,u) = \int_{\Pcal_s} \omega^{2\eta} \Big(\sum_{\alpha}|\del_{\alpha} u|^2 + c^2u^2\Big)dx.
$$
It is clear that $\|\omega^{\eta}\delt_\alpha u\|_{L^2(\Pcal_s)}$ is bounded by $\Ebf^{\Pcal}_{\eta,c}(s,u)$. 

Now it is the turn of $\Ebf_{\eta,c}^{\Tcal}(s,u)$. We have the following result which is parallel to Lemma 2.5 of \cite{M-2018}.
\begin{lemma}\label{lem1-26-10-2021}
Let $u$ be a sufficiently regular function defined in $\Tcal_{[s_0,s_1]}$. Then the following quantities:
\begin{equation}\label{eq6-26-10-2021}
|\delb_a u|,\quad  \zeta|\del_{\alpha}u|,\quad \big(s^{-1}\xi + (1-\xi)^{1/2}\big)|\del_{\alpha}u|,\quad \la r-t\ra t^{-1}|\del_{\alpha}u|,\quad |\delt_a u|
\end{equation}
are bounded by $(\ebf[u])^{1/2}$ in $\Tcal_{[s_0,s_1]}$. Here $\la \cdot\ra$ represents the Japaneses bracket, say,
$\la \rho\ra := (1+\rho^2)^{1/2}$.
\end{lemma}
\begin{proof}
The first two are trivial. For the third one, we apply \eqref{eq12-07-10-2021}. The forth term is by \eqref{eq6-23-01-2022}.

For the last one, remark that
\begin{equation}\label{eq10-11-03-2022-M}
\aligned
\delt_a u =& \frac{x^a}{r}\del_t u + \del_au 
= \Big(\frac{\xi(s,r)x^a}{(s^2+r^2)^{1/2}}\del_tu + \del_a u\Big) 
+ \frac{x^a}{r}\Big(1-\frac{\xi(s,r)r}{(s^2+r^2)^{1/2}}\Big)\del_t u
\\
=& \delb_au + \frac{x^a}{r}\Big(1+\frac{\xi(s,r)r}{(s^2+r^2)^{1/2}}\Big)^{-1}\zeta^2 \del_t u.
\endaligned
\end{equation}
The first term on the right-hand side is bounded by the energy density. For the second term, we only need to remark that on $\Tcal_s$, 
$$
\Big|\frac{x^a}{r}\Big(1+\frac{\xi(s,r)r}{(s^2+r^2)^{1/2}}\Big)^{-1}\Big|\lesssim 1.
$$
\end{proof}

Now we conclude the above results by the following Proposition.
\begin{proposition}\label{prop1-26-10-2021}
	Let $u$ be a sufficiently regular function defined in $\Fcal_{[s_0,s_1]}$. Then for $s_0\leq s\leq s_1$, the following quantities:
	\begin{equation}\label{eq5-07-10-2021}
	\aligned
	&\|\omega^{\eta}\delb_au\|_{L^2(\Fcal_s)},\quad\|\omega^{\eta}\zeta\del_{\alpha}u\|_{L^2(\Fcal_s)},
	\\
	&\|\omega^{\eta}(s^{-1}\xi + (1-\xi)^{1/2})\del_\alpha u\|_{L^2(\Fcal_s)},\quad
	\|\la r+t\ra^{-1}\la r-t\ra\del_{\alpha} u\|_{L^2(\Fcal_s)},
	\quad
	\|\omega^{\eta}\delt_au\|_{L^2(\TPcal_s)}
	\endaligned
	\end{equation}
	are bounded by $\Ebf_{\eta,c}(s,u)^{1/2}$.
\end{proposition}
%
%


\section{Bounds on high-order derivatives}\label{sec1-22-10-2021}
\subsection{Basic notations}
In practical we need to bound high-order derivatives of the solution with respect to $\{\del_{\alpha}, L_a,\Omega\}$. We recall the notation and results introduced in \cite{M-2020-strong} and \cite{LM-2022}. Let $Z$ be a $p$ order derivative composed by $\{\del_{\alpha}, L_a,\Omega\}$ which contains $k$ Lorentzian boosts and rotations. Then we denote by $\ord(Z) = p$ and $\rank(Z) = k$. 
Let $\mathcal{I}_{p,k}$ be the set composed by high-order derivatives composed by $\{\del_{\alpha},L_a,\Omega\}$ with order $\leq p$ and rank $\leq k$.
$$
\aligned
&|u|_{p,k} :=  \max_{Z\in\mathcal{I}_{p,k}}|Z u|,\quad 
&&|\del u|_{p,k} :=\max_{\alpha=0,1,2,3}|\del_{\alpha} u|_{p,k},\quad
&&|\del^m u|_{p,k}:= \max_{|I|= m}|\del^I u|_{p,k},\quad
\\
&|u|_p := \max_{0\leq k\leq p}|u|_{p,k},\quad
&&|\del u|_p :=  \max_{0\leq k\leq p}|\del u|_{p,k},\quad
&&|\del^m u|_{p,k}:= \max_{0\leq k\leq p}|\del^m u|_{p,k}.
\endaligned
$$
We also recall the notation
$$
\ebf_c^{p,k}[u]:= \sum_{Z\in\mathcal{I}_{p,k}}\ebf_c[Z u], \quad
\Ebf_{\eta,c}^{p,k}(s,u) := \sum_{Z\in \mathcal{I}_{p,k}}\Ebf_{\eta,c}(s,Z u) 
$$
for high-order energy densities and energies. Similarly, one also has
$$
\Ebf_{\eta,c}^{\Hcal,p,k}(s,u),\quad \Ebf_{\eta,c}^{\Tcal,p,k}(s,u),\quad \Ebf_{\eta,c}^{\Pcal,p,k}(s,u),\quad
\Ebf_{\eta,c}^{\TPcal,p,k}(s,u)
$$
for high-order energies restricted in the corresponding regions. 

The main task of this section is to bound the high-order derivatives of divers of quantities associated to the solution by the above energy densities with appropriate weights. This is done separately in different regions: the away-from-light-cone region $\TPcal^{\far}_{[s_0,\infty)}$, the near-light-cone region $\TPcal^{\near}_{[s_0,\infty)}$, and the hyperbolic region $\Hcal^*_{[s_0,\infty)}$. 

\subsection{Controlling high-order derivatives in $\TPcal^{\far}_{[s_0,\infty)}$}
This is the simplest region. We will establish the following results.
\begin{proposition}
Let $u$ be a sufficiently regular function defined in $\TPcal^{\far}_{[s_0,s_1]}$. Then:
\begin{equation}
|\del u|_{p,k}\leq C(p)\big(\ebf^{p,k}[u]\big)^{1/2}
\end{equation}	
where $C(p)$ is a constant determined by $p$.
\end{proposition}
\begin{proof}[Sketch of Proof]
This is a direct result of Proposition \ref{prop1-23-10-2021}. Remark that for all $Z$ satisfying $\ord(Z)\leq p$ and $\rank(Z)\leq k$, $Z\del_{\alpha}$ is a operator with order $p+1$ and rank $k$. Then
$$
Z\del_{\alpha} u 
= \sum_{1\leq |I|\leq p+1-k\atop |J|+l\leq k}\!\!\!\!\Gamma(Z)_{IJl}\del^IL^J\Omega^l u 
= \sum_{1\leq |I'|\leq p-k\atop |J|+l\leq k}\!\!\!\!\Gamma(Z)_{\alpha IJl}\del_{\alpha}\del^{I'}L^J\Omega^l u
$$
where $\Gamma(Z)_{\alpha IJl}$ are constants. Then regarding the expression of $\ebf^{p,k}[u]$, the desired result is established.
\end{proof}

\subsection{Controlling high-order derivatives in $\TPcal^{\near}_{[s_0,\infty)}$}
When near the light cone we need to do semi-null decomposition. So we need to concentrate on quantities associated to $\delt_a$. So we introduced the following notation in $\TPcal_{[s_0,\infty)}$:
$$
\aligned
&|\delt u|_{p,k} := \max_{a=1,2}|\delt_a u|_{p,k},\quad
&&|\del\delt u|_{p,k}:=\max_{\alpha = 0,1,2\atop b=1,2}\big\{|\del_{\alpha}\delt_bu|,|\delt_b\del_{\alpha}u|\big\},\quad
&&|\delt\delt u|_{p,k}:=\max_{a,b=1,2}|\delt_a\delt_bu|,\quad
\\
&|\delt u|_p := \max_{0\leq k\leq p}|\delu u|_{p,k},\quad
&&|\del\delt u|_{p,k}:=\max_{0\leq k\leq p}|\del\delt u|_{p,k},\quad
&&|\del\delt u|_{p,k}:=\max_{0\leq k\leq p}|\delt\delt u|_{p,k}.
\endaligned
$$
The following result is the two-dimensional version of Section 4 of \cite{M-2018} in $\RR^{1+1}$ and Section 6 of \cite{LM-2022} in $\RR^{3+1}$. The proof is nearly the same to the Proof of Proposition 6.13 of \cite{LM-2022}. For the convenience of the reader we give a sketch in Appendix \ref{sec2-22-10-2021}.

\begin{proposition}\label{prop2-23-10-2021}
	Let $u$ be a sufficiently regular function defined in $\TPcal^{\near}_{[s_0,s1]}$. Then 
	\begin{subequations}
	\begin{equation}\label{eq2-26-10-2021}
	\zeta| \del u|_{p,k} + |\delt u|_{p,k}\leq C(p) \big(\ebf^{p,k}[u]\big)^{1/2},
	\end{equation}
	\begin{equation}\label{eq1-27-10-2021}
	c|u|_{p,k}\leq C(p)\big(\ebf_c^{p,k}[u]\big)^{1/2},\quad c|\del u|_{p,k}\leq C(p)\big(\ebf_c^{p+1,k}[u]\big)^{1/2},
	\end{equation}
	where $C(p)$ is a constant determined by $p$.
	\end{subequations}
\end{proposition}

\subsection{Controlling high-order derivatives in $\Hcal^*_{[s_0,\infty)}$}
We introduce the following notation.
$$
\aligned
&|\delu u|_{p,k} := \max_{a=1,2}|\delu_a u|_{p,k},\quad
&&|\del\delu u|_{p,k}:=\max_{\alpha = 0,1,2\atop b=1,2}\big\{|\del_{\alpha}\delu_bu|,|\delu_b\del_{\alpha}u|\big\},\quad
&&|\delu\delu u|_{p,k}:=\max_{a,b=1,2}|\delu_a\delu_bu|,\quad
\\
&|\delu u|_p := \max_{0\leq k\leq p}|\delu u|_{p,k},\quad
&&|\del\delu u|_{p,k}:=\max_{0\leq k\leq p}|\del\delu u|_{p,k},\quad
&&|\delu\delu u|_{p,k}:=\max_{0\leq k\leq p}|\delu\delu u|_{p,k}.
\endaligned
$$
Then we state the following estimates.
\begin{proposition}\label{prop1-20-10-2021}
	Let $u$ be a sufficiently regular function defined in $\Hcal^*_{[s_0,s_1]}$. Then
\begin{subequations}
\begin{equation}\label{eq5-21-01-2022}
|\delu u|_{p,k} + |(s/t)\del u|_{p,k}\leq C(p) \big(\ebf^{p,k}[u]\big)^{1/2},
\end{equation}
\begin{equation}
t|\del\delu u|_{p,k} + (s/t^2)|\delu\delu u|_{p,k} \leq C(p) \big(\ebf^{p+1,k+1}[u]\big)^{1/2},
\end{equation}
\end{subequations}
\begin{subequations}
\begin{equation}
c|u|_{p,k} \leq C(p) \big(\ebf_c^{p,k}[u]\big)^{1/2},\quad c|\del u|_{p,k}\leq C(p) \big(\ebf_c^{p+1,k}[u]\big)^{1/2} ,
\end{equation}
\begin{equation}
ct|\delu u|_{p,k}\leq C(p) \big(\ebf_c^{p+1,k+1}[u]\big)^{1/2},\quad ct^2|\delu\delu u|_{p,k}\leq C(p) \big(\ebf_c^{p+2,k+2}[u]\big)^{1/2}.
\end{equation}
\end{subequations}
Here $C(p)$ is a constant determined by $p$.
\end{proposition}
These are also established in \cite{LM1}. The present version is stated in \cite{M-2020-strong} ((2.21) -- (2.24)) and briefly proved in Appendix B therein. 

%

\section{Global Sobolev inequalities and decay estimates}
\label{sec1-10-04-2022-M}
This section is devoted to the global Sobolev decay on $\Fcal_s$. Due to the fact that $\Hcal^*_s$ and $\TPcal_s$ have very different nature, we need to distinguish between the hyperbolic domain and the transition-Euclidean domain. However, in this case we need to restrict a function originally defined on $\Fcal_s$ into $\Hcal^*_s$ and $\TPcal_s$. These restrictions can not be approximated by compactly supported smooth functions defined in $\Hcal^*_s$ or $\TPcal_s$. In \cite{LM-2022} we have established the Sobolev inequalities on the positive cone $\RR^3_+ = \{(x^1,x^2,x^3)| x^a\geq 0\}$. In the present 2-D case, we preform the same strategy and obtain the following results.

\subsection{Global Sobolev inequalities}
We recall the following two results.
\begin{proposition}\label{prop1-21-10-2021}
	Let $u$ be a $C^2$ function defined in $\Hcal^*_{[s_0,s_1]}$, sufficiently regular. Then 
	\begin{equation}\label{eq4-17-10-2021}
	|u(t,x)|\lesssim t^{-1}\sum_{|J|\leq 2}\|L^J u\|_{L^2(\Hcal_s^*)}.
	\end{equation}
\end{proposition}

\begin{proposition}\label{prop2-21-10-2021}
	Let $u$ be a $C^2$ function defined in $\TPcal_{[s_0,s_1]}$ and vanishes sufficiently fast at the spatial infinity. Then for $\eta\geq 0$, and $(t,x)\in\TPcal_s$,
	\begin{equation}\label{eq5-17-10-2021}
	\omega^\eta(t,r)\la r\ra^{1/2} |u(t,x)|
	\lesssim (1+\eta)\sum_{i+j\leq 2}\|\omega^{\eta}\delb_r^i\Omega^j u\|_{L^2(\TPcal_s)},
	\end{equation}	
	here $\delb_r : = (x^a/r)\delb_a = \frac{\xi(s,r)r}{(s^2+r^2)^{1/2}}\del_t + (x^a/r)\del_a$. For $(t,x)\in\Pcal_s$,
\begin{equation}\label{eq5-22-01-2022}
\omega^\eta(t,r)\la r\ra^{1/2} |u(t,x)|
\lesssim (1+\eta)\sum_{i+j\leq 2}\|\omega^{\eta}\del_r^i\Omega^j u\|_{L^2(\Pcal_s)}.
\end{equation}

\end{proposition}
These are parallel cases of Lemma 3.2, Lemma 3.3 and Proposition 3.4 in \cite{M-2018} in $\RR^{1+1}$. Their proof are similar to the case in $\RR^{1+3}$, which are detailed in \cite{LM-2022} Section 4. For the convenience of reader,  we give a sketch in Appendix \ref{sec1-23-10-2021}.

\subsection{Sobolev decay estimates I}

In the following application, we need to bound the right-hand side of \eqref{eq4-17-10-2021} and \eqref{eq5-17-10-2021} by energies defined in the corresponding region. 
\begin{proposition}\label{prop1-21-01-2022}
Let $u$ be a function defined in $\Hcal^*_{[s_0,s_1]}$, sufficiently regular. Then
\begin{equation}\label{eq2-21-01-2022}
s|\del_{\alpha}u|_{p,k} + t|\delu_a u|_{p,k}\leq C(p) \Ebf^{\Hcal,p+2,k+2}(s,u)^{1/2},
\end{equation}
\begin{equation}\label{eq4-21-01-2022}
t|cu|_{p,k}\leq C(p) \Ebf_{c}^{\Hcal, p+2,k+2}(s,u)^{1/2}.
\end{equation}
\end{proposition}
The bounds in $\Hcal^*_{[s_0,s_1]}$ are essentially established in \cite{LM1}. We only give a sketch. 
\begin{proof}
	Remark that \eqref{eq4-21-01-2022} is by \eqref{eq4-17-10-2021}. For \eqref{eq2-21-01-2022}, let $|J|\leq 2$ and  $Z$ satisfying $\ord(Z)\leq p$ and $\rank(Z)\leq k$. Then in $\Hcal^*_{[s_0,s_1]}$ the following estimate holds:
	$$
	\big|L^J\big((s/t)Z\del_{\alpha}u\big)\big| \lesssim (s/t)\sum_{|J'|\leq 2}|L^{J'}Z\del_{\alpha}u|\leq C(p) \big(\ebf^{p+2,k+2}[u]\big)^{1/2}
	$$
	where for the first inequality, the bound $L^J(s/t)\leq C(|J|)(s/t)$ in $\Hcal^*_{[s_0,\infty)}$ is applied. This bounded is proved in \cite{LM1} and can be checked directly by the induction on $|J|$. For the second inequality we have applied \eqref{eq5-21-01-2022}. Then direct application of \eqref{eq4-17-10-2021} leads to \eqref{eq2-21-01-2022}. 
\end{proof}

\begin{proposition}\label{prop1-23-01-2022}
Let $u$ be a sufficiently regular function defined in $\TPcal_{[s_0,s_1]}$ and vanishes sufficiently fast at the spatial infinity. Then for $\eta\geq 0$,
\begin{equation}\label{eq1-21-01-2022}
 \omega^{1/2+\eta}|\del_{\alpha} u|+ \omega^{\eta}\la r\ra^{1/2}|\delt_a u|
\lesssim (1+\eta)\Ebf_{\eta}^{2,\TPcal}(s,u)^{1/2},
\end{equation}	
\begin{equation}\label{eq2-23-01-2022}
\omega^{\eta}\la r\ra^{1/2}|\del_{\alpha} u|\lesssim (1+\eta)\Ebf_{\eta}^{2,\Pcal}(s,u)^{1/2},\qquad \text{in } \Pcal_{[s_0,s_1]}, 
\end{equation}
\begin{equation}\label{eq3-21-01-2022}
\omega^{\eta}\la r\ra^{1/2}|cu|\lesssim \Ebf_{\eta,c}^{2,\TPcal}(s,u)^{1/2}.
\end{equation}
\end{proposition}

Before the proof of this proposition, we need some technical preparations. 
\begin{lemma}\label{lem1-23-01-2022}
Let $u$ be a $C^2$ function defined in $\TPcal_{[s_0,s_1]}$. Then
\begin{equation}\label{eq1-22-01-2022}
|\delb_r\delb_ru| \lesssim \big(\ebf^{1,0}[u]\big)^{1/2},
\end{equation}
\begin{equation}\label{eq1-23-01-2022}
|\delb_r\delb_r\delb_r u|\lesssim \sum_{|I|\leq 2,a}|\delb_a\del^I u| + \zeta|\del u|_{1,0}
\end{equation}
in $\TPcal_{[s_0,s_1]}$.
\end{lemma}
\begin{proof}
We remark that
$$
\aligned
\delb_r\delb_r u = \delb_r \Big(\frac{\xi(s,r)r}{(s^2+r^2)^2}\del_t u + \del_a u\Big) 
= \delb_r\Big(\frac{\xi(s,r)r}{(s^2+r^2)^{1/2}}\Big)\del_tu + \frac{\xi(s,r)r}{(s^2+r^2)^{1/2}}\delb_r\del_tu + \delb_r\del_au.
\endaligned
$$
The (squares of) last two terms are bounded by $|\delb_a\del_{\alpha} u|^2$ thus bounded by $\ebf^{1,0}[u]$ in $\TPcal_s$.  For the first term, we write it into the following from:
$$
\delb_r\Big(\frac{\xi(s,r)r}{(s^2+r^2)^{1/2}}\Big)\del_t u
= \frac{r\del_r\xi(s,r)}{(s^2+r^2)^{1/2}}\del_t u + \frac{\xi(s,r)}{(s^2+r^2)^{1/2}}\del_t u 
- \frac{\xi(s,r)r^2}{(s^2+r^2)^{3/2}}\del_t u.
$$
Again, the last two are bounded by $\ebf[u]$ because of \eqref{eq12-07-10-2021} and the fact that
$$
\Big|\frac{\xi(s,r)}{(s^2+r^2)^{1/2}}\Big| + \Big|\frac{\xi(s,r)r^2}{(s^2+r^2)^{3/2}}\Big|\lesssim 1/s.
$$
For the first term, recall that $\del_r\xi(s,r) = -\chi'(r-\rhoH(s))$. Recall \eqref{eq6-19-10-2021}, 
$$
|\del_r\xi(s,r)|\lesssim \big(1-\xi(s,r)\big)^{1/2}.
$$
Then by \eqref{eq12-07-10-2021},
$$
\Big| \frac{r\del_r\xi(s,r)}{(s^2+r^2)^{1/2}}\Big|\lesssim (1-\xi(s,r))^{1/2}\lesssim \zeta.
$$
This leads to the bound in $\delb_r\delb_r u$. In the same manner, we can show the bound on $|\delb_r\delb_r\delb_r u|$ in $\Tcal_{[s_0,s_1]}$. This is also by direct calculation and the bound $|\chi''(\rho)|\lesssim \chi(\rho)^{1/2}$ supplied \eqref{eq6-19-10-2021}. We omit the detail. 
\end{proof}

\begin{proof}[Proof of Proposition \ref{prop1-23-01-2022}]
We get started with \eqref{eq2-23-01-2022}. This is by direct application of \eqref{eq5-22-01-2022} and the fact that
$$
|\del_r^i\Omega^j\del_{\alpha}u|\lesssim \ebf^{2}[u],\qquad \text{in } \Pcal_{[s_0,s_1]}.
$$

Then we turn to \eqref{eq3-21-01-2022}. Remark that
\begin{equation}\label{eq2-22-01-2022}
|\delb_r^i\Omega^j u|\lesssim |\Omega^ju|_{i,0},\quad i\leq 2.
\end{equation}
This is because, when $i=0$ it is trivial. When $i=1$, we only need to remark that that $|\delb_ru|\lesssim |\del u|$. When $i=2$, we recall \eqref{eq1-22-01-2022} and the fact that $|\zeta|\leq 1$. Then we apply \eqref{eq5-17-10-2021}.

Then we turn to \eqref{eq1-21-01-2022}. We firstly remark that, by \eqref{eq2-26-10-2021},
$$
|\Omega^j\delt_au|^2\lesssim \ebf^j[u].
$$
For $\delb_r\Omega^j\delt_a u, j\leq 1$, we recall that
$$
\aligned
\delb_r\Omega^j \delt_au 
=& \,
\Omega^j\delb_r\big((x^a/t)\del_t + \del_a\big)u 
= \Omega^j\big((x^a/r)\delb_r\del_t u + \delb_r(x^a/r)\del_tu + \delb_r\del_au\big)
\\
=&\, \sum_{j_1+j_2=j}\Omega^{j_1}(x^a/r)\,\delb_r\Omega^{j_2}\del_tu
+ \sum_{j_1+j_2 = j}\Omega^{j_1}\delb_r(x^a/r)\,\Omega^{j_2}\del_t u + \delb_r\Omega^j\del_a u.
\endaligned
$$
Then the three terms on the right-hand side are bounded by $(\ebf^{2,1}[u])^{1/2}$. For the first term the coefficients $\Omega^{j_1}(x^a/r)$ is uniformly  bounded in $\TPcal_{[s_0,s_1]}$ due to the homogeneity. The second term vanishes because $\delb_r(x^a/r)\equiv 0$. 

For the term $\delb_r^2\delt_au$, one only needs to remark that
$$
\delb_r\delb_r\delt_au = \delb_r\delb_r\big((x^a/t)\del_tu\big) + \delb_r\delb_r\del_au, 
$$
where thanks to \eqref{eq1-22-01-2022} and Lemma \ref{lem1-26-10-2021}, the second term on the right-hand side is bounded by $(\ebf^{2,0}[u])^{1/2}$. For the first term, 
$$
\delb_r\delb_r\big((x^a/t)\del_tu\big) = (x^a/r)\delb_r\delb_r\del_tu + 2\delb_r(x^a/r)\delb_r\del_tu + \delb_r\delb_r(x^a/r)\,\del_tu.
$$
By the relation $\delb_r(x^a/r) \equiv 0$, the last two terms vanish. Thus by \eqref{eq1-22-01-2022} and Lemma \ref{lem1-26-10-2021}, these terms are bounded by $(\ebf^{2,0}[u])^{1/2}$. Then we conclude that 
$$
|\delt_a u|\lesssim \omega^{\eta}\la r\ra^{-1/2}\Ebf_{\eta}^{2,\TPcal}(s,u)^{1/2}.
$$
On the other hand, we recall \eqref{eq3-23-01-2022} and obtain, for $i+j\leq 2$,
\begin{equation}\label{eq4-22-01-2022}
\aligned
\Big|\delb_r^i\Omega^j\Big(\Big(\frac{r-t+3}{r}\Big)^{1/2}\del_{\alpha}u\Big)\Big|
\lesssim \Big(\frac{\la r-t\ra}{r}\Big)^{1/2}|\del u|_2 \lesssim \ebf^{2}[u].
\endaligned
\end{equation}
Then by \eqref{eq5-17-10-2021}, we obtain, remark that in $\TPcal_{[s_0,\infty)}$ one has $|r-t|+1\leq r-t+3$,
\begin{equation}
\omega^{1/2 + \eta}|\del_{\alpha} u|\leq C\Ebf_{\eta}^{2,\TPcal}(s,u)^{1/2}.
\end{equation}
This concludes \eqref{eq1-21-01-2022}.

\end{proof}
We also write the high-order version of Proposition \ref{prop1-23-01-2022}.
\begin{proposition}
Let $u$ be a sufficiently regular function defined in $\TPcal_{[s_0,s_1]}$ and vanishes sufficiently fast at the spatial infinity. Then for $\eta\geq 0$,
\begin{equation}\label{eq2-15-03-2022-M}
\omega^{1/2+\eta}|\del u|_{p,k} + \omega^{\eta}\la r\ra^{1/2}|\delt u|_{p,k}
\lesssim C(p) (1+\eta)\Ebf_{\eta}^{p+2,k+2,\TPcal}(s,u)^{1/2},
\end{equation}	
\begin{equation}\label{eq3-15-03-2022-M}
\omega^{\eta}\la r\ra^{1/2}|\del u|_{p,k}\lesssim C(p) (1+\eta)\Ebf_{\eta}^{p+2,k+2,\Pcal}(s,u)^{1/2},\qquad \text{in } \Pcal_{[s_0,s_1]}, 
\end{equation}
\begin{equation}\label{eq4-15-03-2022-M}
\omega^{\eta}\la r\ra^{1/2}|cu|_{p,k}\lesssim C(p) \Ebf_{\eta,c}^{p+2,k+2,\TPcal}(s,u)^{1/2}.
\end{equation}	
where $C(p)$ is a constant determined by $p$.
\end{proposition}
\begin{proof}
For the bounds on $|\del u|_{p,k}$, we only need to recall \eqref{prop1-23-10-2021} together with the corresponding bounds in Proposition \ref{prop1-23-01-2022}. For the bounds on $|\delt_a u|_{p,k}$, we recall \eqref{eq3-23-10-2021} in $\TPcal^{\near}_{[s_0,s_1]}$. When in the region $\TPcal^{\far}_{[s_0,s_1]}$, one remark that $\delt_a u = (x^a/r)\del_t u + \del_au$ with $(x^a/r)$ homogeneous of degree zero. Thus it enjoys the same decay of $|\del u|_{p,k}$ in $\TPcal^{\far}_{[s_0,s_1]}$. However in this region $\la r-t\ra\simeq \la r\ra$. We thus obtain the desired result.
\end{proof}

\subsection{Sobolev decay estimates II}
In $\TPcal_s$, the bound \eqref{eq1-21-01-2022} is not sufficient. In this section we will establish a sharper decay estimate --  the cost is one order regularity. 
\begin{proposition}\label{prop1-29-10-2021}
Let $u$ be a regular function defined in $\TPcal_{[s_0,s_1]}$ and vanishes sufficiently fast at spatial infinity. Then for $\eta\in(0,1)$, 
\begin{equation}\label{eq1-29-10-2021}
\la r\ra^{1/2}\omega^{\eta}|\del_{\alpha}u| + \la r\ra^{1/2+\eta}|\delt_au|\lesssim (1-\eta)^{-1} \Ecal_{\eta}^{\TPcal,3}(s,u)^{1/2}.
	\end{equation}
\end{proposition}
\begin{lemma}
Let $u$ be a $C^3$ function defined in $\TPcal_{[s_0,s_1]}$. Then
\begin{equation}\label{eq5-20-10-2021}
\omega^{\eta}\la r\ra^{1/2}|\delb_r u(t,x)|\lesssim \Ebf_{\eta}^{\TPcal,2}(s,u)^{1/2}.
\end{equation}
\end{lemma}
\begin{proof}
This is a direct result of Lemma \ref{lem1-23-01-2022}. We remark that
$$
\|\omega^{\eta}\delb_r^i\Omega^j\delb_ru\|_{L^2(\TPcal_s)} = \|\omega^{\eta}\delb_r^{i+1}\Omega^ju\|_{L^2(\TPcal_s)} \lesssim \Ebf_{\eta}^{2,\TPcal}(s,u)^{1/2}.
$$
Then we apply \eqref{eq5-17-10-2021}.
\end{proof}

\begin{proof}[Proof of Proposition \ref{prop1-29-10-2021}]
In $\TPcal_s^{far}$, the estimate \eqref{eq2-23-01-2022} covers \eqref{eq1-29-10-2021} because in this case $r\lesssim \omega$. We thus concentrate on $\TPcal_s^{\near}$. We apply \eqref{eq5-20-10-2021} on $L_a u$ and obtain
	$$
	|\delb_rL_au(t,x)|\lesssim (1+r)^{-1/2}(2+r-t)^{-\eta}\Ebf^{\TPcal,2}_{\eta}(s,L_au)^{1/2}.
	$$
	For $(t,x)\in\TPcal_s^{\near}$, let $(t_0,x_0)\in \TPcal_s^{\far}\cap\TPcal_s^{\near}$ with $x/|x| = x_0/|x_0|$. Then we integrate $\delb_r L_au$ along the radial direction on $\TPcal_s$ and obtain
	\begin{equation}\label{eq6-29-10-2021}
	L_au(t,x) = L_au(t_0,x_0) - \int_{|x|}^{|x_0|}\delb_rL_au(T(\rho,s),\rho)d\rho.
	\end{equation}
	Remark that \eqref{eq2-23-01-2022} leads us to
	$$
	|L_au(t_0,x_0)|\lesssim t_0|\del u|\lesssim |x_0|^{1/2-\eta}\Ebf_{\eta}^{\TPcal,2}(s,u).
	$$ 
	Then one has
	$$
	\aligned
	|L_au(t,x)|\lesssim |x_0|^{1/2-\eta}\Ebf_{\eta}^{\TPcal,2}(s,u) 
	+ \Ebf^{\TPcal,2}_{\eta}(s,L_au)^{1/2}\int_{|x|}^{|x_0|}(1+\rho)^{-1/2}(2+\rho - T(\rho,s))^{-\eta}d\rho.
	\endaligned
	$$
	When $(t,x)\in\Pcal_s\cap \TPcal^{\near}_s$, one has $(t_0,x_0)\in \Pcal_s$. In this case $|x|/3\leq T(s) = T(\rhoH(s),s) = t = t_0 \leq |x_0|/2$, and $r_0\leq 3t_0 = 3t$ Thus
	$$
	\aligned
	|L_au(t,x)|\lesssim& \la r\ra^{1/2-\eta}\Ebf_{\eta}^{\TPcal,2}(s,u)^{1/2}
	+ \la r\ra^{-1/2}\Ebf_{\eta}^{\TPcal,3}(s,u)^{1/2}\int_{|x|}^{|x_0|}\!\!\!\!(2+\rho-t)^{-\eta}d\rho
	\\
	\lesssim&\la r\ra^{1/2-\eta}\Ebf_{\eta}^{\TPcal,2}(s,u)^{1/2} 
	+ (1-\eta)^{-1}\la r\ra^{1/2-\eta}\Ebf_{\eta}^{\TPcal,3}(s,u)^{1/2}.
	\endaligned
	$$
	Now remark that $\delt_au = t^{-1}L_au + (x^a/r)(1-r/t)\del_t u$ and the bound 
	$$
	|\del u|\lesssim \la r\ra^{-1/2}\la r-t\ra^{-\eta}\Ebf_{\eta}^{\TPcal,2}(s,u)^{1/2},
	$$
	we obtain
	\begin{equation}\label{eq5-29-10-2021}
	|\delt_au(t,x)|\lesssim (1-\eta)^{-1}\la r\ra^{-1/2-\eta}\Ebf_{\eta}^{\TPcal,3}(s,u)^{1/2}
	\end{equation}
	which is the bound for $\delt_au$ in $\Pcal_s$. For the bound in $\Tcal_s$, we take \eqref{eq6-29-10-2021} with $(t,x)\in \Tcal_s$ and $(t_0,x_0)\in\Pcal_s\cap \Tcal_s$, $x/|x| = x_0/|x_0|$. Then remark that in this case $1\geq|x_0|-|x|\geq 0$. Thus
	$$
	\aligned
	&|L_au(t,x)|
	\\
	\lesssim& |L_au(t_0,x_0)| + \sup_{\Tcal_s}\{|\delb_rL_au|\}
	\lesssim (1-\eta)^{-1}\la r_0\ra^{1/2-\eta}\Ebf_{\eta}^{\TPcal,3}(s,u)^{1/2} 
	+ \la r_0-1\ra^{-1/2} \Ebf_{\eta}^{\TPcal,3}(s,u)^{1/2} 
	\\
	\lesssim& (1-\eta)^{-1}\la r\ra^{1/2-\eta}\Ebf_{\eta}^{\TPcal,3}(s,u)^{1/2}.
	\endaligned
	$$
	Here we recall that $s_0\geq 2$ leads to $r_0\geq 3/2$. The last inequality is because of Lemma \ref{lem1-07-10-2021} which tell us that for $(t_0,x_0)\in \Tcal_s\cap\Pcal_s$, $t_0- 1\leq |x_0|\leq t_0 = T(s)\simeq s^2$, and for $(t,x)\in \Tcal_s$, $|x|\geq \rhoH(s)\simeq s^2$. Then also by the relation $\delt_au = t^{-1}L_au + (x^a/r)(1-r/t)\del_t$, we extend \eqref{eq5-29-10-2021} in $\TPcal_s$. Thus the bound on $|\delt_au|$ in \eqref{eq1-29-10-2021} is established.
	 
	 For the bound on $\del_{\alpha}u$, we remark that on $\Pcal_s$, \eqref{eq2-23-01-2022} already guarantees the desired bound because $\zeta\equiv 1$. For $(t,x)\in \Tcal_s$, we only need to integrate (thanks to \eqref{eq5-20-10-2021})
	 $$
	 |\delb_r\del_{\alpha}u|\lesssim \la r\ra^{-1/2}\la r-t\ra^{-\eta}\Ebf_{\eta}^{\TPcal,3}(s,u)^{1/2}
	 $$
	 on $\Tcal_s$ along the radial direction from $\Tcal_s\cap \Pcal_s$. 
\end{proof}
We now concludes these two subsections with the following decay estimate on high-order derivatives:
\begin{proposition}
Let $u,v$ be sufficiently regular functions defined in $\TPcal_{[s_0,s_1]}$ and vanish sufficiently fast at the spatial infinity. Then for $\eta\geq 0$,	
\begin{equation}\label{eq1-30-10-2021}
\|\la r\ra^{1/2}\omega^{\eta}|\del u|_{p,k}\|_{L^\infty(\TPcal_s)} 
+\|\la r\ra^{1/2+\eta}|\delt u|_{p,k}\|_{L^{\infty}(\TPcal_s)} \leq C(p)\Ebf_{\eta}^{\TPcal,p+3}(s,u)^{1/2},
\end{equation}
\begin{equation}\label{eq2-30-10-2021}
\|\la r\ra^{1/2}\omega^{\eta}|\del v|_{p,k}\|_{L^\infty(\TPcal_s)} 
+\|\la r\ra^{1/2+\eta}|\delt v|_{p,k}\|_{L^{\infty}(\TPcal_s)} 
\leq C(p)\Ebf_{\eta,c}^{\TPcal,p+3}(s,v)^{1/2},
\end{equation}
\begin{equation}\label{eq3-30-10-2021}
c\|\la r\ra^{1/2}\omega^{\eta}|v|_{p,k}\|_{L^{\infty}(\TPcal_s)} \leq C(p)\Ebf_{\eta,c}^{p+2}(s,v)^{1/2}.
\end{equation}
Here $C(p)$ is a constant determined by $p$.
\end{proposition}
\begin{proof}
For \eqref{eq1-30-10-2021} and \eqref{eq2-30-10-2021}, we recall \eqref{prop1-23-10-2021} for the bound on $|Z\del u|$ and \eqref{eq3-23-10-2021} for $|Z\delt_au|$. Then we apply \eqref{eq1-29-10-2021} on $Zu$ with $\ord(Z)\leq p, \rank(Z)\leq k$. For \eqref{eq3-30-10-2021}, we apply \eqref{eq4-21-01-2022} on $Zu$ with $\ord(Z)\leq p,\rank(Z)\leq k$.
\end{proof}

\section{Decay estimates on wave  and Klein--Gordon equations}
\label{sec2-11-04-2022-M}
In this section we recall some decay estimates on wave and Klein--Gordon equations. These results depend on the linear structure of the equations.
\subsection{Estimates on Hessian of wave equations}
The following result is essentially established in \cite{LM1}. Here is a version formulated in \cite{M-2020-strong} (Proposition 2.4 therein).
\begin{proposition}
	For any function $u$ defined in $\Hcal^*_{[s_0,\infty)}$, 
	\begin{equation}\label{eq1-28-10-2021}
		(s/t)^2|\del\del u|_{p,k}\lesssim C(p)\big(|\Box u|_{p,k} +  t^{-1}|\del u|_{p+1,k+1}\big).
	\end{equation}	
\end{proposition}

\begin{proof}[Sketch of proof]
	The first two bounds are direct results of 
	\begin{equation}\label{eq6-27-10-2021}
		\Box u = \big(1-(r/t)^2\big)\del_t\del_tu - t^{-1}\big(L_a\del_a - (x^a/t)L_a\del_t\big)u
	\end{equation}
	Together with the relations
	\begin{subequations}
		\begin{equation}\label{eq1a-29-10-2021}
			\del_t\del_a = \del_a\del_t = t^{-1}L_a\del_t - (x^a/t)\del_t\del_t,
		\end{equation}
		\begin{equation}\label{eq1b-29-10-2021}
			\del_a\del_b = t^{-1}L_a\del_b - (x^a/t^2)L_b\del_t + (x^ax^b/t^2)\del_t\del_t.
		\end{equation}
	\end{subequations}
	These identities are based on the expression $\del_a = t^{-1}L_a - (x^a/t)\del_t$.

\end{proof}

\subsection{Decay estimates on wave equation based on integration along time-like hyperbolas}
We recall Proposition 4.1 of \cite{M-2020-strong}. Here we present a slightly modified version with which we do not need the assumption on the support of initial data. We firstly recall the following notation. Let $\mathcal{J} := \del_t + \frac{2tx^a}{t^2+r^2}\del_a$ which is a smooth vector field defined in $\Fcal_{[s_0,\infty)}$. Let $(t,x)\in \Hcal^*_{[s_0,\infty)}$. We denote by $\gamma_{t,x}$ the integral curve passing $(t,x)$ with $\gamma_{t,x}(t) = (t,x)$. In fact we can calculate explicitly:
$$
\aligned
\gamma_{t,x}: [t_0,\infty) &\mapsto \Hcal^*_{[s_0,\infty)}
\\
\tau &\mapsto \Big(\tau, \frac{x^a}{|x|}\sqrt{\tau^2+\frac{1}{4}C(t,x)^2} - \frac{1}{2}C(t,x)\Big)
\endaligned
$$	
with $C(t,x) := (t^2-|x^2|)/|x|$. Then we state the following estimate:
\begin{proposition}\label{prop1-31-10-2021}
	Let $u$ be a sufficiently regular function defined in $\Hcal^*_{[s_0,s_1]}$. Then for $(t,x)\in \Hcal^*_{[s_0,s_1]}$ and $s = \sqrt{t^2-|x|^2}$, 
	\begin{equation}\label{eq1-31-10-2021}
		|t^{1/2}\del_t u|\lesssim 
		\sup_{\Hcal^*_{s_0}\cup \del\Kcal_{[s_0,s]}}\{t^{1/2}|\del_t u|\} 
		+ \int_{s_0}^tW_{t,x}[u](\tau)e^{-\int_{\tau}^tP_{t,x}(\eta)d\eta}\,d\tau,
	\end{equation}
	where
	$$	
	\aligned
	W_{t,x}[u](\tau) :=& \frac{t^2}{t^2+r^2}t^{1/2}\big(\Box u + \sum_a\delu_a\delu_a u \big)\Big|_{\gamma_{t,x}(\tau)},
	\\
	P_{t,x}(\tau):=& \frac{3t^2}{2(t^2+r^2)}(s/t)^2t^{-1}\geq \frac{3}{4}(s/t)^2t^{-1}\Big|_{\gamma_{t,x}(\tau)}.
	\endaligned
	$$
\end{proposition}
\begin{proof}[Sketch of Proof]
	We only need to make the following additional observation. Firstly, remark that $(\tau,(x^a/|x|)\tau)$ are integral curves of $\mathcal{J}$. Thus for any $(t_0,x_0)\in\Hcal^*_{[s_0,\infty)}$, $\gamma_{t_0,x_0}\subset H^*_{[s_0,\infty)}$. So one has $t^2-r^2>0$ along $\gamma_{t_0,x_0}$. Then
	$$
	\mathcal{J}(r/t) = -r^3t^{-2}(t^2+r^2)^{-1}<0,\quad \mathcal{J}s^2 = \frac{2t^3}{t^2+r^2}>0.
	$$
	Thus $\forall (t,x)\in \Hcal^*_{[s_0,\infty)}$ and $0<t-\tau$ sufficiently small, $\gamma_{t,x}(\tau)\in\Hcal^*_{[s_0,s]}$ (with $s^2 = t^2-r^2$). Let 
	$$
	t_0: = t_0(t,x) = \inf\{\lambda| \forall \lambda\leq \tau\leq t, \gamma_{t,x}(\tau)\in \Hcal^*_{[s_0,s]}\}
	$$
	It is clear that $t_0\geq s_0$. On the other hand, we remark that $\gamma_{t,x}(t_0)\in \Hcal^*_{s_0}\cup \del\Kcal_{[s_0,s_1]}$ (because $s$ strictly increases along $\gamma_{t,x}$, $\gamma_{t,x}(t_0)\notin \Hcal^*_s)$). 
	
	On the other hand, we recall the decomposition of wave operator (which is a special case of (4.2) in \cite{M-2020-strong}, with $\alpha=1/2, \beta = 0$)
	$$
	\mathcal{J}(t^{1/2}\del_tu) + P(t,r)t^{1/2}\del_tu = S^w[u] + \Delta^w[u]
	$$
	where
	$$
	S^w[u]:= t^{1/2}\frac{t^2\Box u}{t^2+r^2},\quad \Delta^w[u] := t^{1/2}\frac{t^2\delu_a\delu_a u}{t^2+r^2}.
	$$
	Integrating the above identity along $\gamma_{t,x}$ for $\tau\in [t_0,t]$ and remark that $|t-r|$ remains uniformly bounded on $\Hcal^*_{s_0}\cup \del\Kcal_{[s_0,s]}$, the desired result is established.
\end{proof}

\subsection{Decay estimates on Klein--Gordon equations in transition-exterior region}
Klein--Gordon components enjoy extra decay in the form of $t^{-k}\la r-t\ra^k$ near the light cone. This phenomena was detected in \cite{Ho1,Kl1993}. In \cite{LM1} it was reestablished within the hyperboloidal foliation. The following one is a slight generalization of version established in \cite{LM-2022}.
\begin{proposition}\label{eq4-30-10-2021}
	Let $v$ be a sufficiently regular solution to the equation 
	$$
	\Box v + c^2v  = f
	$$
	in $\TPcal^{\near}_{[s_0,s_1]}$. Then
	$$
	c^2|v|_{p,k} \leq 
	\begin{cases}
		C(p) \la r\ra^{-3/2}\la r-t\ra^{1-\eta}\Ebf_{\eta}^{\TPcal,p+4,k+4}(s,u)^{1/2} + |f|_{p,k},
		\\
		C(p) \la r\ra^{-1}\la r-t\ra^{1/2-\eta}\Ebf_{\eta}^{\TPcal,p+3,k+3}(s,u)^{1/2} + |f|_{p,k}.
	\end{cases}
	$$
\end{proposition}
\begin{proof}[Sketch of proof]
	Thanks to \eqref{eq6-27-10-2021}, one has
	\begin{equation}\label{eq5-21-03-2022-M}
		c^2|v|_{p,k}\leq C(p)\big(r^{-1}\la r-t\ra|\del v|_{p+1,k} + r^{-1}|\del v|_{p+1,k+1} + |f|_{p,k}\big). 
	\end{equation}
	Then substituting \eqref{eq1-21-01-2022}, \eqref{eq1-29-10-2021} into the above expression, we obtain the desired results.
\end{proof}

\section{Conformal energy estimate on Euclidean-hyperboloidal hypersurface}
\label{sec-2-10-04-2022-M}
\subsection{Basic identities}
We recall that for wave equation, there is the well-known conformal energy estimate in the flat foliation (cf. \cite{Al-book1}). In \cite{MH-2017} a parallel estimate was established in the hyperboloidal foliation context (see also \cite{Wong-2017}). Here we make a generalization in the framework of Euclidean-hyperboloidal foliation. 

To get started, we recall the following differential identity (cf. \cite{Al-book1}, Section 6.7). Let $u$ be a $C^2$ function defined in $\Fcal_{[s_0,s_1]}$, then in $\RR^{1+n}$ case,
\begin{equation}\label{eq2-26-01-2022}
	\aligned
	&\big(K_2 + (n-1)t\big)u\,\Box u 
	\\
	=& \frac{1}{2}\del_t\Big((t^2+r^2)\big(|\del_tu|^2 + \sum_a|\del_a u|^2\big) + 4tx^a\del_tu\del_au + 2(n-1)tu\del_t u - (n-1)u^2\Big)
	\\
	& + \del_a \Big(tx^a\big(\sum_b|\del_b u|^2 - |\del_tu|^2\big) - (r^2+t^2)\del_tu\del_au - 2tx^b\del_bu\del_au
	- (n-1)tu\del_a u\Big).
	\endaligned
\end{equation}
Here $K_2 := (t^2+r^2)\del_t + 2tx^a\del_a$. Then we integrate this identity in $\Fcal_{[s_0,s_1]}$. Recalling \eqref{eq3-26-01-2022}, one obtains
\begin{equation}
	\aligned
	\Econ(s_1,u) - \Econ(s_0,u) =& \int_{\Fcal_{[s_0,s_1]}}(K_2 + (n-1)t)u\,\Box u dxdt
	\\
	=& \int_{s_0}^{s_1}\int_{\Fcal_s}(K_2 + (n-1)t)u\,\Box u\, J dx ds
	\endaligned
\end{equation}
with $\Econ(s,u) := \int_{\Fcal_s} \econ[u]\,dx$ and the energy density
\begin{equation}\label{eq4-22-04-2022-M}
\aligned
\econ[u] 
=& \frac{1}{2} \Big(t^2+r^2 + \frac{2\xi(s,r)tr^2}{(s^2+r^2)^{1/2}}\Big)|\del_tu|^2 
+ \frac{1}{2}\Big(t^2+r^2 - \frac{2\xi(s,r)tr^2}{(s^2+r^2)^{1/2}}\Big)\sum_a|\del_au|^2
\\
& + \Big(2tx^a + \frac{\xi(s,r)(t^2+r^2)x^a}{(s^2+r^2)^{1/2}}\Big)\del_tu\del_au + \frac{2t\xi(s,r)|r\del_r u|^2}{(s^2+r^2)^{1/2}}
\\
&+(n-1)tu\Big(\del_t + \frac{\xi(s,r)x^a}{(s^2+r^2)^{1/2}}\del_a\Big)u\, dx - \frac{n-1}{2} u^2.
\endaligned
\end{equation}
In this subsection we first establish the following result.
\begin{lemma}
Suppose that $u$ is a sufficiently regular function defined in $\Fcal_{[s_0,s_1]}$. Then
\begin{equation}\label{eq1-29-01-2022}
\aligned
\econ[u]=&\frac{\zeta^2}{2A} |K_2u + t(n-1)u|^2 + \frac{1}{2}A\sum_{a>b}|r^{-1}\Omega_{ab}u|^2
\\
&+ \frac{(t^2-r^2)^2}{2A} |\delb_r u|^2 
+(n-1)t\Big(\delb_r t - \frac{2\zeta^2 rt}{A}\Big)u\delb_ru -\frac{(n-1)^2\zeta^2 t^2u^2}{2A} - \frac{n-1}{2}u^2.
\endaligned
\end{equation}
Here $\delb_r = (x^a/r)\delb_a$.
\end{lemma}
\begin{proof}
For our purpose, in  the following calculation we rely on the parameterization $(s,x)$ instead of $(t,x)$. The variable $t$ is regarded as a function of $(s,x)$ with $t = T(s,x)$. In this context, the natural frame associate to $(s,x)$ is $\{\delb_s, \delb_a\}$. 
For the simplicity of expression, we remark that
$$
\delb_r t = (x^a/r) \del_a T = \frac{\xi(s,r)r}{(s^2+r^2)^{1/2}},\quad \zeta^2 = 1-|\delb_rt|^2,
$$ 
and we denote by 
$$
A := r^2+t^2 - 2rt\delb_r t.
$$
Then 
$$
K_2u = (A\del_t + 2rt\delb_r) u,\quad
\del_r u = \delb_ru - \delb_r t\, \del_tu,\quad
|\del_r u|^2 = |\delb_r u|^2 - 2\delb_r t\, \del_tu\delb_ru + (\delb_rt)^2|\del_t u|^2.
$$
Recalling that $\sum_a|\del_a u|^2 = |\del_ru|^2 + \sum_{a<b}|r^{-1}\Omega_{ab}u|^2$, we have the following calculation on the energy density of $\Econ(s,u)$:
\begin{equation}\label{eq1-02-02-2022-M}
\aligned
\econ[u] 
=& \frac{1}{2}A\zeta^2|\del_t u|^2 +\frac{1}{2}\big(t^2+r^2 + 2tr\delb_rt\big)|\delb_ru|^2
+\frac{1}{2}A\sum_{a>b}|r^{-1}\Omega_{ab}u|^2
\\
&+2rt\zeta^2\del_tu\delb_r u
+(n-1)tu\big(\del_t  + \delb_r t\del_r\big)u - \frac{n-1}{2}u^2.
\endaligned
\end{equation}
Remark that
$$
\aligned
&\frac{1}{2A}\zeta^2|K_2u|^2 = \frac{\zeta^2}{2A}\big(A\del_t u + 2tr\delb_ru\big)^2
=\frac{A}{2}\zeta^2|\del_tu|^2 + 2\zeta^2rt\del_tu\delb_ru + \frac{2\zeta^2 r^2t^2}{A}|\delb_ru|^2.
\endaligned
$$
Then the energy density is written as
\begin{equation}\label{eq3-02-02-2022}
	\aligned
	\econ[u] =& \frac{\zeta^2}{2A} |K_2u|^2 
	+ \frac{A}{2}\sum_{a>b}|r^{-1}\Omega_{ab}|^2
	+ \frac{(t^2-r^2)^2}{2A} |\delb_r u|^2 
	+(n-1)tu\big(\del_t  +\delb_r t\del_r\big)u - \frac{n-1}{2}u^2.
	\endaligned
\end{equation}
	Here for the coefficient of $|\delb_r u|^2$, we have made the following calculation:
	$$
	\aligned
	&\frac{1}{2}\Big(t^2+r^2+\frac{2\xi(s,r)tr^2}{(s^2+r^2)^{1/2}}\Big) - \frac{2\zeta^2r^2t^2}{t^2+r^2-\frac{2\xi(s,r)tr^2}{(s^2+r^2)^{1/2}}}
	\\
	=&\Big(t^2+r^2-\frac{2\xi(s,r)tr^2}{(s^2+r^2)^{1/2}}\Big)^{-1}
	\Big(\frac{1}{2}(t^2+r^2)^2 - \frac{2tr^2 \xi(s,r)^2r^2}{s^2+r^2} - 2r^2t^2\Big(1-\frac{\xi(s,r)^2r^2}{s^2+r^2}\Big)\Big)
	\\
	=&\frac{1}{2}\Big(t^2+r^2-\frac{2\xi(s,r)tr^2}{(s^2+r^2)^{1/2}}\Big)^{-1}(t^2-r^2)^2.
	\endaligned
	$$
Then we concentrate on the fourth term of \eqref{eq3-02-02-2022}. Firstly we remark that
$$
\aligned
(n-1)tu\big(\del_t  + \delb_r t\,\del_r\big)u 
=& (n-1)tu\big(1 - |\delb_rt|^2\big)\del_t u + (n-1)t\delb_rt\,u\delb_r u 
\\
=& (n-1)\zeta^2tu\del_tu + (n-1)t\delb_rt\,u \delb_r u.
\endaligned
$$
On the other hand, 
$$
\aligned
\frac{\zeta^2}{2A}\big(2(n-1)tu K_2u + (n-1)^2t^2u^2\big)
=&\frac{(n-1)\zeta^2 tu}{A}\big(A\del_tu + 2tr\delb_ru\big)
+ \frac{(n-1)^2\zeta^2}{2A}(tu)^2
\\
=&(n-1)\zeta^2tu\del_t u + \frac{2(n-1)tr\zeta^2 tu}{A}\delb_ru
+ \frac{(n-1)^2\zeta^2}{2A}(tu)^2.
\endaligned
$$
We thus obtain
$$
\aligned
(n-1)tu\big(\del_t  + \delb_r t\,\del_r\big)u 
=&\frac{\zeta^2}{2A}\Big(2(n-1)tu K_2u + (n-1)^2t^2u^2\Big)
\\
&+(n-1)t\Big(\delb_r t - \frac{2\zeta^2 rt}{A}\Big)u\delb_ru
-\frac{(n-1)^2\zeta^2 t^2u^2}{2A}.
\endaligned
$$
Substitute the above identity into \eqref{eq3-02-02-2022}, we obtain the desired result.

\end{proof}

\subsection{Positivity of the conformal energy}
Now we need to show that the energy is positive. For this purpose, we rely on polar coordinates.  Let $v = r^{(n-1)/2}u$. Then one has
$$
\aligned
\delb_r u =&\, r^{-(n-1)/2}\delb_r v - \frac{n-1}{2}r^{-(n+1)/2}v.
\\
|\delb_r u|^2 =&\, r^{1-n}|\delb_r v|^2 - (n-1)r^{-n}v\delb_rv + \frac{(n-1)^2}{4}r^{-n-1}v^2,
\\
u\delb_r u =&\, u\big(r^{-(n-1)/2}\delb_r v - \frac{n-1}{2}r^{-(n+1)/2}v\big)
=r^{1-n}v\delb_rv - \frac{n-1}{2}r^{-n}v^2.
\endaligned
$$
Substitute the above relations into \eqref{eq1-29-01-2022}. The last four terms in \eqref{eq1-29-01-2022} can be written as
\begin{equation}\label{eq1-30-01-2022}
\aligned
L_n[u] :=&\, \frac{(t^2-r^2)^2}{2A} |\delb_r u|^2 
+(n-1)t\Big(\delb_r t - \frac{2\zeta^2 rt}{A}\Big)u\delb_ru -\frac{(n-1)^2\zeta^2 t^2u^2}{2A} - \frac{n-1}{2}u^2
\\
=&\,\frac{(t^2-r^2)^2r^{1-n}}{2A} |\delb_r v|^2 - (n-1)\frac{t^2+r^2}{2r}r^{1-n}v\delb_rv
\\
&\,+ r^{1-n}\frac{(n-1)^2}{8r^2}A v^2 - \frac{(n-1)r^{1-n}}{2}v^2.
\endaligned
\end{equation}
In the above identity, for the coefficient of $v\delb_r v$, we make the following calculation: 
	$$
	\aligned
	&- \frac{(n-1)(t^2-r^2)^2}{2A}r^{-n}
	+ (n-1)t\delb_rt\, r^{1-n}
	- \frac{2(n-1)\zeta^2rt^2}{A}r^{1-n}
	\\
	=&(n-1)t\delb_rt\, r^{1-n}
	- (n-1)r^{-n}\frac{4\zeta^2t^2r^2 +(t^2-r^2)^2}{2A}
	\\
	=&(n-1)t\delb_rt\, r^{1-n}
	-(n-1)r^{-n}\frac{4t^2r^2 - 4|\delb_rt|^2r^2t^2 +(t^2-r^2)^2}{2A}
	\\
	=&(n-1)t\delb_rt\, r^{1-n}
	-(n-1)r^{-n}\frac{(t^2+r^2)^2 - 4|\delb_rt|^2r^2t^2}{2A}
	\\
	=&(n-1)t\delb_rt\, r^{1-n}- \frac{(n-1)}{2r^n}\big(t^2+r^2 + 2rt\delb_r\big)
	=-\frac{n-1}{2r^n}(t^2+r^2).
	\endaligned
	$$
For the coefficient of $v^2$, we make the following calculation:
$$
\aligned
&r^{-n}\frac{(n-1)^2}{4r}\frac{(t^2-r^2)^2}{2A}
-r^{-n}\frac{n-1}{2}(n-1)\frac{\xi(s,r)tr}{(s^2+r^2)^{1/2}}
\\
&+r^{-n}\frac{n-1}{2}\frac{2(n-1)\zeta^2rt^2}{A} 
-r^{-n}\frac{(n-1)^2\zeta^2 t^2r}{2A} - \frac{(n-1)r^{1-n}}{2}
\\
=&r^{-n}\frac{(n-1)^2}{4r}\frac{(t^2-r^2)^2}{2A} -r^{-n}\frac{(n-1)^2}{2}\frac{\xi(s,r)tr}{(s^2+r^2)^{1/2}}
+r^{-n}\frac{(n-1)^2\zeta^2 t^2r}{2A} - \frac{(n-1)r^{1-n}}{2}
\\
=&r^{-n+1}\frac{(n-1)^2}{8r^2 A}\big((t^2-r^2)^2 + 4\zeta^2t^2r^2\big)
- r^{-n}\frac{(n-1)^2}{2}\frac{\xi(s,r)tr}{(s^2+r^2)^{1/2}} - \frac{(n-1)r^{1-n}}{2}
\\
=&r^{1-n}\frac{(n-1)^2}{8r^2A}
\big((t^2+r^2)^2 - 4|\delb_r t|^2t^2r^2\big)
-r^{-n}\frac{(n-1)^2}{2}\frac{\xi(s,r)tr}{(s^2+r^2)^{1/2}} - \frac{(n-1)r^{1-n}}{2}
\\
=&r^{1-n}\frac{(n-1)^2}{8r^2}\big(t^2+r^2 +2rt\delb_r \big) - r^{1-n}\frac{(n-1)^2}{2}rt\del_rt - \frac{(n-1)r^{1-n}}{2}
\\
=&r^{1-n}\frac{(n-1)^2A}{8r^2} - \frac{(n-1)r^{1-n}}{2}.
\endaligned
$$
We remark the following differential identity {\sl along $\Fcal_s$}:
$$
\aligned
\frac{r^2-t^2}{r}v\delb_rv = \frac{1}{2}\delb_r\Big(\frac{r^2-t^2}{r}v^2\Big) - \frac{1}{2}v^2\Big(1+t^2/r^2 - 2(t/r)\delb_rt\Big) 
=\frac{1}{2}\delb_r\Big(\frac{r^2-t^2}{r}v^2\Big) - \frac{A}{2r^2}v^2.
\endaligned 
$$
Integrate the above identity within polar coordinates with volume element $dx = r^{n-1}drd\sigma$ in the region $\{r\geq \vep\}\cap \Fcal_s$, one obtains
\begin{equation}\label{eq1-31-01-2022}
-\int_{\vep}^{\infty}\frac{r^2-t^2}{r}\int_{\mathbb{S}^{n-1}}v\delb_rv\,d\sigma dr 
= -\frac{1}{2}\int_{\vep}^{\infty}\delb_r\Big(\frac{r^2-t^2}{r}\int_{\mathbb{S}^{n-1}}v^2d\sigma\Big)\,dr
+\frac{1}{2}\int_{\{r\geq \vep\}}r^{-2}A\int_{\mathbb{S}^{n-1}}v^2\,d\sigma dr
\end{equation}
Remark that $A>0$. Then
$$
-\frac{r^2-t^2}{r}v\delb_rv 
+ \frac{1}{2}\Big(\frac{A^{1/2}v}{\sqrt{2}r}
+ \frac{\sqrt{2}(r^2-t^2)\delb_r v}{A^{1/2}}\Big)^2
=
\frac{1}{2}\Big(\frac{Av^2}{2r^2}+ \frac{2(r^2-t^2)^2|\delb_r v|^2}{A}\Big).
$$
This leads to
$$
\frac{(r^2-t^2)^2|\delb_r v|^2}{2A} 
=
- \frac{r^2-t^2}{2r}v\delb_rv
-\frac{A v^2}{8r^2}
+ \frac{1}{4}\Big(\frac{A^{1/2}v}{\sqrt{2}r} + \frac{\sqrt{2}(r^2-t^2)\delb_r v}{A^{1/2}}\Big)^2.
$$
Thus by \eqref{eq1-31-01-2022}, we obtain
\begin{equation}\label{eq2-31-01-2022}
\aligned
&\frac{1}{2}\int_{\vep}^{\infty}\frac{(r^2-t^2)^2}{A}
\int_{\mathbb{S}^{n-1}}|\del_rv|^2d\sigma\,dr
\\
=&\int_{\vep}^{\infty}\int_{\mathbb{S}^{n-1}}- \frac{r^2-t^2}{2r}v\delb_rv
-\frac{A v^2}{8r^2}
+ \frac{1}{4}\Big(\frac{A^{1/2}v}{\sqrt{2}r} + \frac{\sqrt{2}(r^2-t^2)\delb_r v}{A^{1/2}}\Big)^2\,d\sigma dr
\\
=&  - \frac{1}{4}\int_{\vep}^{\infty}\delb_r\Big((r-t^2/r)\int_{\mathbb{S}^{n-1}}v^2d\sigma\Big)\,dr 
+ \frac{1}{8} \int_{\vep}^{\infty}r^{-2}A\int_{\mathbb{S}^{n-1}}v^2d\sigma\,dr
\\
&+\frac{1}{4}\int_{\vep}^{\infty}\int_{\mathbb{S}^{n-1}}\Big(\frac{A^{1/2}v}{\sqrt{2}r}
+ \frac{\sqrt{2}(r^2-t^2)\delb_r v}{A^{1/2}}\Big)^2\,d\sigma dr.
\endaligned
\end{equation}	

On the other hand, we remark that
$$
	\aligned
	-\frac{(n-1)(t^2+r^2)}{2r}v\delb_rv =& -\frac{n-1}{4}(r+(t^2/r))\delb_r(v^2) 
	\\
	=& -\frac{n-1}{4}\delb_r\big((r+t^2/r)v^2\big) + \frac{n-1}{4}(1-(t/r)^2+2(t/r)\delb_rt)v^2.
	\endaligned
$$
By integration on $\Fcal_s\cap\{r\geq \vep\}$, one obtains:
\begin{equation}\label{eq1-02-02-2022}
\aligned
&-\int_{\vep}^{\infty}\frac{(n-1)(t^2+r^2)}{2r}\int_{\mathbb{S}^{n-1}}v\delb_rv d\sigma\,dr 
\\
&=
-\frac{n-1}{4}\int_{\vep}^{\infty}\delb_r\Big((r+t^2/r)\int_{\mathbb{S}^{n-1}}v^2d\sigma\Big)\,dr 
+ \frac{n-1}{4}\int_{\vep}^{\infty}r^{-2}(r^2-t^2+2rt\delb_r t)\int_{\mathbb{S}^{n-1}}v^2d\sigma\,dr.
\endaligned
\end{equation}
Thus in the case $n=2$, by summing \eqref{eq1-02-02-2022} and \eqref{eq2-31-01-2022} (remark that the two singular terms at $r=0$ cancels each other),
$$
\aligned
\int_{\{r\geq \vep\}\cap \Fcal_s} L_2[u]\,dx
=&\frac{1}{4}\int_{\vep}^{\infty}\int_{\mathbb{S}^{n-1}}\Big(\frac{A^{1/2}v}{\sqrt{2}r} + \frac{\sqrt{2}(r^2-t^2)\delb_r v}{A^{1/2}}\Big)^2\,d\sigma dr
-\frac{1}{2}\int_{\vep}^{\infty}\delb_r\Big(r\int_{\mathbb{S}^2}v^2d\sigma\Big)\,dr.
\endaligned
$$
Then by letting $\vep\rightarrow 0^+$, we obtain, recalling that $v = r^{1/2}u$,
\begin{equation}\label{eq2-02-02-2022}
\aligned
\int_{\{r\geq \vep\}\cap \Fcal_s} L_2[u]\,dx 
=&\, \frac{1}{4}\int_{\vep}^{\infty}\int_{\mathbb{S}^{n-1}}\Big(\frac{\sqrt{2r}(r-t\delb_rt)}{A^{1/2}} u - \frac{\sqrt{2r}(t^2-r^2)}{A^{1/2}}\delb_r u\Big)^2\,d\sigma dr
\\
=&\, \frac{1}{2}\int_{\vep}^{\infty}\int_{\mathbb{S}^1} \frac{r}{A}\big((t^2-r^2)\delb_r u - (r-t\delb_r t)u\big)^2\,d\sigma dr
\\
=&\, \frac{1}{2}\int_{\Fcal_s} A^{-1}\big((t^2-r^2)\delb_r u - (r-t\delb_r t)u\big)^2\,dx\geq 0.
\endaligned
\end{equation}
And thus we arrive at the following expression:
\begin{equation}\label{eq4-02-02-2022}
\aligned
\Econ(s,u) =& \frac{1}{2}\int_{\Fcal_s}A^{-1}\zeta^2(K_2u + tu)^2\,dx 	+ \frac{1}{2}\int_{\Fcal_s}A|r^{-1}\Omega u|^2\,dx 
\\
&+ \frac{1}{2}\int_{\Fcal_s} A^{-1}\big((t^2-r^2)\delb_r u - (r-t\delb_r t)u\big)^2\,dx
\endaligned
\end{equation}\
with
$$
A(s,r) = t^2+r^2 - 2rt\delb_r t = t^2+r^2-\frac{2\xi(s,r)tr^2}{(s^2+r^2)^{1/2}}.
$$
\subsection{Conformal energy estimate}
Based on \eqref{eq4-02-02-2022}, we are ready to establish the following conformal energy estimate on $\Fcal_s$:
\begin{proposition}\label{prop1-22-03-2022-M}
Let $u$ be a $C^2$ function defined in $\Fcal_{[s_0,s_1]}$ and decreases sufficiently fast at spatial infinity. Then
	\begin{equation}\label{eq1-06-02-2022}
		\Econ(s_1,u)^{1/2}\leq \Econ(s_0,u)^{1/2} + \frac{\sqrt{2}}{2} \int_{s_0}^s\|J\zeta^{-1}A^{1/2}\Box u\|_{L^2(\Fcal_\tau)} \,ds.
	\end{equation}
\end{proposition}
\begin{proof}
Once \eqref{eq4-02-02-2022} is established, the proof of \eqref{eq1-06-02-2022} becomes a standard energy estimate. We differentiate \eqref{eq2-26-01-2022} with respect to $s$ and obtain:
$$
\aligned
\frac{d}{ds} \Econ(s,u) =& \int_{\Fcal_s}\frac{\zeta}{A^{1/2}}(K_2 + t)u\, \frac{A^{1/2}}{\zeta}\Box u J dx
\leq \|\frac{\zeta}{A^{1/2}}(K_2 + t)u\|_{L^2(\Fcal_s)}\|\frac{JA^{1/2}}{\zeta}\Box u\|_{L^2(\Fcal_s)}
\\
\leq& \sqrt{2}\Econ(s,u)^{1/2}\|\frac{JA^{1/2}}{\zeta}\Box u\|_{L^2(\Fcal_s)}.
\endaligned
$$
On the other hand, $\frac{d}{ds}\Econ(s,u) = 2\Econ(s,u)^{1/2}\frac{d}{ds}\Econ(s,u)^{1/2}$. Then 
$$
\frac{d}{ds}\Econ(s,u)^{1/2}\leq \frac{\sqrt{2}}{2}\|J\zeta^{-1}A^{1/2}\Box u\|_{L^2(\Fcal_s)}.
$$
Integrate the above inequality from on $[s_0,s_1]$, the desired bound is obtained.
\end{proof}
\subsection{Analysis on the conformal energy}
The expression \eqref{eq4-02-02-2022} is {\sl not} satisfactory because it does not control the gradient of $u$. In the present subsection we will establish the following estimates. 
\begin{proposition}\label{prop1-11-03-2022-M}
Let $u$ be a function defined in $\Fcal_{[s_0,s_1]}$, sufficiently regular and decreasing sufficiently fast at spatial infinity. Suppose that $s_0\geq 2$, then the following quantities 
\begin{equation}
\aligned
&\|\zeta u\|_{L^2(\Fcal_s)},\qquad &&\|\zeta t^{-1}|r^2-t^2|\del_\alpha u\|_{L^2(\Fcal_s)},
\\
&\|\zeta^{-1}t^{-1}|r^2-t^2|\delt_au\|_{L^2(\Fcal_s)} ,\qquad &&\|\zeta^{-1}t^{-1}|r^2-t^2|\delb_au\|_{L^2(\Fcal_s)}
\endaligned
\end{equation}
are controlled by
\begin{equation}\label{eq1-11-03-2022-M}
\Fbf(s,u;s_0)^{1/2} := \|\zeta u\|_{L^2(\Fcal_{s_0})} + \Ebf(s,u)^{1/2} + \int_{s_0}^s\tau^{-1}\Ebf(\tau,u)^{1/2}d\tau
\end{equation}
modulo a universal constant.
\end{proposition}
In \cite{M3} we have established this estimate on hyperboloids. Here is its generalization on Euclidean-hyperboloidal slices. The proof is done in several steps. We get started with the following Lemma.
\begin{lemma}
The conformal energy can be written into the following form when $u$ decreases sufficiently fast at spatial infinity: 
\begin{equation}\label{eq2-02-03-2022-M}
\aligned
\Econ(s,u) =& \frac{1}{2}\int_{\Fcal_s}\big(\zeta^2t\del_tu + (r+t\delb_rt)\delb_ru + u\big)^2\, dx
\\
&+ \frac{1}{2}\int_{\Fcal_s}\big(\zeta(r-t\delb_rt) \del_tu + \zeta t\delb_ru\big)^2\,dx
+ \frac{1}{2}\int_{\Fcal_s} A|r^{-1}\Omega u|^2\,dx.
\endaligned
\end{equation}
\end{lemma} 
\begin{proof}
Recalling \eqref{eq1-02-02-2022-M}, we remark the following identity:
$$
\aligned
&\econ[u]
\\
=&\, \frac{1}{2}\zeta^2\big(t^2 - (t\delb_rt)^2\big)|\del_t u|^2 + \frac{1}{2}\big(r^2+t^2|\delb_rt|^2 + 2rt\delb_rt\big)|\delb_r u|^2 
\\
&\, + \frac{1}{2}\zeta^2\big(r^2 + t^2|\delb_rt|^2 - 2rt\delb_rt\big)|\del_tu|^2 
+ \frac{1}{2}\big(t^2 - (t\delb_rt)^2\big)|\delb_r u|^2 
 + \zeta^2t(r+t\delb_rt)\del_tu\delb_ru 
\\
&\,+ \zeta^2 t(r-t\delb_r t)\delb_ru\del_tu
 + u\big(\zeta^2t\del_tu + (r + t\delb_rt)\delb_ru\big) - ru\delb_ru - \frac{1}{2}u^2 + \frac{1}{2}A|r^{-1}\Omega u|^2
\\
=&\, \frac{1}{2}\big(\zeta^2t\del_tu + (r+t\delb_rt)\delb_ru\big)^2 
+ \frac{1}{2}\big(\zeta(r-t\delb_rt) \del_tu + \zeta t\delb_ru\big)^2
\\
&\, + u\big(\zeta^2t\del_tu + (r + t\delb_rt)\delb_ru\big) - \delb_a(x^au^2) + \frac{1}{2}u^2 + \frac{1}{2}A|r^{-1}\Omega u|^2
\\
=&\, \frac{1}{2}\big(\zeta^2t\del_tu + (r+t\delb_rt)\delb_ru + u\big)^2
+ \frac{1}{2}\big(\zeta(r-t\delb_rt) \del_tu + \zeta t\delb_ru\big)^2
+ \frac{1}{2}A|r^{-1}\Omega u|^2 - \delb_a(x^au^2). 
\endaligned
$$
The last term is in divergence form and when integrating on $\Fcal_s$, the desired result is established.
\end{proof}
Then we establish the following relation.
\begin{lemma}
Let $u$ be a function defined in $\Fcal_{[s_0,s_1]}$, sufficiently regular and decreasing sufficiently fast at spatial infinity. Then
\begin{equation}\label{eq6-10-03-2022-M}
	\aligned
	\frac{d}{2ds}\int_{\Fcal_s}\big(1-(r/t)\delb_rt\big)u^2\,dx
	=& 
	s^{-1}\int_{\Fcal_s}\big((s/t)\zeta^{-1}\delb_rt\big)\,Ju\,\big(\zeta(r-t\delb_rt) \del_tu + \zeta t\delb_ru\big)\, dx
	\\
	&- s^{-1}\int_{\Fcal_s}(s/t)Ju\,\big(\zeta^2t\del_tu + (r+t\delb_rt)\delb_ru + u\big) dx.
	\endaligned
\end{equation}	
\end{lemma}
\begin{proof}
We firstly remark that $\delb_t = J^{-1} \delb_s $
$$
\aligned
&u\,\big(\zeta^2t\del_tu + (r+t\delb_rt)\delb_ru + u\big)J = \frac{1}{2}\zeta^2t\, J\del_t(u^2) 
+ J u\,t\delb_rt\delb_ru + Ju(r\delb_ru + u)
\\
=&\, \frac{1}{2}\big(t - r\delb_rt\big)\delb_s\big(u^2\big)
+ Ju\,\big(r\delb_rt - t(\delb_rt)^2\big)\delb_tu + Ju\,t\delb_rt\delb_ru 
+ \frac{1}{2}\delb_st\delb_a(x^au^2)
\\
=&\, \frac{1}{2}\big(t - r\delb_rt\big)\delb_s\big(u^2\big) 
+ J\zeta^{-1}\delb_rt\, u\,\big(\zeta(r-t\delb_rt) \del_tu + \zeta t\delb_ru\big)
+ \frac{1}{2}\delb_st\delb_a(x^au^2).
\endaligned
$$
Then
$$
\aligned
&(s/t)u\,\big(\zeta^2t\del_tu + (r+t\delb_rt)\delb_ru + u\big)J 
\\
=&\frac{s}{2}\delb_s\big(u^2\big) - \frac{s}{2}(r/t)\delb_rt\delb_s\big(u^2\big) 
+ \big((s/t)J\zeta^{-1}\delb_rt\big)u\,\big(\zeta(r-t\delb_rt) \del_tu + \zeta t\delb_ru\big) 
+ \frac{1}{2}(s/t)\delb_st\delb_a\big(x^au^2\big)
\\
=&\frac{s}{2}\delb_s\big(u^2\big)  - \frac{s}{2}\delb_s\big((r/t)\delb_rt\,u^2\big)
+ \frac{s}{2}\delb_s\big((r/t)\delb_rt\big)u^2 
\\
& + \big((s/t)J\zeta^{-1}\delb_rt\big)u\,\big(\zeta(r-t\delb_rt) \del_tu + \zeta t\delb_ru\big) 
 + \frac{1}{2}\delb_a\big((s/t)\delb_st\,x^au^2\big) - \frac{1}{2}ru^2\delb_r\big((s/t)\delb_st\big) 
\\
=& (s/2)\delb_s\Big(\big(1-(r/t)\delb_rt\big)u^2\Big)
+ \big((s/t)J\zeta^{-1}\delb_rt\big)u\,\big(\zeta(r-t\delb_rt) \del_tu + \zeta t\delb_ru\big)  
+ \frac{1}{2}\delb_a\big((s/t)\delb_st\,x^au^2\big) 
\\
& + (s/2)(r/t)\delb_s\delb_rt\, u^2 - (s/2)rt^{-2}\delb_st\delb_rt\, u^2
- \frac{s}{2}(r/t)u^2 \delb_r\delb_st + (s/2)rt^{-2}\delb_st\delb_rt\,u^2.
\endaligned
$$
Here the last four terms cancel each other. We conclude by
\begin{equation}\label{eq1-05-03-2022-M}
\aligned
&(s/t)Ju\,\big(\zeta^2t\del_tu + (r+t\delb_rt)\delb_ru + u\big) 
\\
=& 
(s/2)\delb_s\Big(\big(1-(r/t)\delb_rt\big)u^2\Big)
+ \big(J\zeta^{-1}\delb_rt\big)\,(s/t)u\,\big(\zeta(r-t\delb_rt) \del_tu + \zeta t\delb_ru\big)  
+ \frac{1}{2}\delb_a\big((s/t)\delb_st\,x^au^2\big).
\endaligned
\end{equation}
Integrating the above identity on $\Fcal_s$ respect to the variable $x$, we obtain
$$
\aligned
\int_{\Fcal_s}(s/t)Ju\,\big(\zeta^2t\del_tu + (r+t\delb_rt)\delb_ru + u\big)\, dx
=&\,\frac{s}{2}\int_{\Fcal_s}\delb_s\Big(\big(1-(r/t)\delb_rt\big)u^2\Big)dx
\\
&\, +\int_{\Fcal_s}\big((s/t)J\zeta^{-1}\delb_rt\big)u\,\big(\zeta(r-t\delb_rt) \del_tu + \zeta t\delb_ru\big)  dx.
\endaligned
$$
This leads to the desired identity.
\end{proof}
Remark that 
\begin{equation}\label{eq3-10-02-2022-M}
1-(r/t)\delb_rt = 
\left\{
\aligned
& (s/t)^2, && r\leq \rhoH_s,
\\
& (1-r/t) + c,&& \rhoH_s\leq r\leq \rhoH_s+1,
\\
& 1,&& r\geq \rhoH_s+1.
\endaligned
\right.
\end{equation}
Here $c = c(s)$ is determined in Lemma \ref{lem1-07-10-2021}. We will show that
\begin{lemma}\label{lem1-10-01-2022-M}
In the region $\Fcal_{[s_0,\infty)}$ with $s_0\geq 2$, the following relations hold:
\begin{equation}\label{eq5-10-03-2022-M}
(s/t)\zeta^{-1}J\delb_rt\lesssim (1-(r/t)\delb_rt)^{1/2},
\end{equation}
\begin{equation}\label{eq4-10-03-2022-M}
(s/t)J\lesssim (1-(r/t)\delb_rt)^{1/2},
\end{equation}
\begin{equation}\label{eq2-11-03-2022-M}
 1-(r/t)\delb_r \simeq \zeta^2.
\end{equation}
\end{lemma}
\begin{proof}
In $\Hcal^*_{[s_0,\infty)}$ and $\Pcal_{[s_0,\infty)}$, the above relations holds trivially. We only need to check them in $\Tcal_{[s_0,\infty)}$.

For the first two bounds, we only need to remark that 
\begin{equation}\label{eq2-10-03-2022-M}
\zeta^2\lesssim 1 - (r/t)\delb_rt. 
\end{equation}
Remark that by Lemma \ref{lem1-07-10-2021}, in $\TPcal_{[s_0,\infty)}$ one has $t-1\leq r\leq t$. Thus
$$
1 - r/t\geq 0.
$$
On the other hand, $(r/t)\geq 1/2$ when $s_0\geq 2$. Thus by \eqref{eq5-23-01-2022} and \eqref{eq3-10-02-2022-M}
$$
\zeta^2 = 1 - \frac{\xi^2(s,r)r^2}{s^2+r^2} \lesssim 1 - \frac{\xi(s,r)r}{(s^2+r^2)^{1/2}}\lesssim 1-(r/t)\delb_rt.
$$
Thus \eqref{eq2-10-03-2022-M} is concluded. Then recalling \eqref{eq9-07-10-2021}, we remark that in $\TPcal_{[s_0,\infty)}$, $s^2\simeq t$. By \eqref{eq13-07-10-2021}, 
$$
(s/t)J\lesssim (s/t)s\zeta^2\lesssim \zeta^2\lesssim 1 - (r/t)\delb_rt \lesssim (1-(r/t)\delb_rt)^{1/2}
$$
which concludes \eqref{eq4-10-03-2022-M}. For \eqref{eq5-10-03-2022-M}, we remark that
$$
(s/t)\zeta^{-1}J\delb_rt\lesssim \zeta\lesssim (1 - (r/t)\delb_r t)^{1/2}.
$$
For the last bound \eqref{eq2-11-03-2022-M}, we firstly remark that in $\Tcal_{[s_0,\infty)}$, 
$$
(r/t) \Big(1-\frac{\xi(s,r)r}{(s^2+r^2)^{1/2}}\Big)\lesssim \zeta^2.
$$
On the other hand, recalling Lemma \ref{lem2-23-01-2022},
$$
0\leq 1-r/t = ((t-r)/r) (t/r)\lesssim \zeta^2.
$$
Combing this together with \eqref{eq2-10-03-2022-M}, we conclude \eqref{eq2-11-03-2022-M}.
\end{proof}
\begin{proof}[Proof of Proposition \ref{prop1-11-03-2022-M}]
Based on the above Lemma \ref{lem1-10-01-2022-M}, we observe that \eqref{eq6-10-03-2022-M} leads to
$$
\aligned
&\big\|(1-(r/t)\delb_rt)^{1/2}u\big\|_{L^2(\Fcal_s)}
\frac{d}{ds}\big\|(1-(r/t)\delb_rt)^{1/2}u\big\|_{L^2(\Fcal_s)} 
\\
\lesssim&
s^{-1} \big\|(1-(r/t)\delb_rt)^{1/2}u\big\|_{L^2(\Fcal_s)}\|\zeta(r-t\delb_rt)\del_tu + \zeta t\delb_r u\|_{L^2(\Fcal_s)}
\\
&+ s^{-1}\big\|(1-(r/t)\delb_rt)^{1/2}u\big\|_{L^2(\Fcal_s)}\|\zeta^2t\del_tu + (r+t\delb_r)\delb_ru + u\|_{L^2(\Fcal_s)}
\endaligned
$$
thus by \eqref{eq2-02-03-2022-M},
\begin{equation}\label{eq7-10-03-2022-M}
\frac{d}{ds}\big\|(1-(r/t)\delb_rt)^{1/2}u\big\|_{L^2(\Fcal_s)} \lesssim s^{-1}\Econ(s,u)^{1/2}.
\end{equation}
Integrate the above inequality on $[s_0,s]$, we obtain
$$
\aligned
\big\|(1-(r/t)\delb_rt)^{1/2}u\big\|_{L^2(\Fcal_s)}
\leq& \,
\big\|(1-(r/t)\delb_rt)^{1/2}u\big\|_{L^2(\Fcal_{s_0})}	+ C\int_{s_0}^s\tau^{-1}\Econ(\tau,u)^{1/2}d\tau.
\\
\lesssim&\, \|\zeta u\|_{L^2(\Fcal_{s_0})} + \int_{s_0}^s\tau^{-1}\Econ(\tau,u)^{1/2}d\tau\lesssim \Fbf(s,u;s_0)^{1/2}
\endaligned
$$
where \eqref{eq2-11-03-2022-M} is applied. This leads to, thanks to \eqref{eq2-11-03-2022-M},
\begin{equation}\label{eq3-11-03-2022-M}
\|\zeta u\|_{L^2(\Fcal_s)}\lesssim \Fbf(s,u;s_0)^{1/2}.
\end{equation}

Then we recall \eqref{eq2-02-03-2022-M}, the $L^2$ norm of the following quantities are bounded by $\Fbf(s,u;s_0)^{1/2}$:
$$
\aligned
\zeta^2t\del_tu + (r+t\delb_rt)\delb_ru + u, \quad A^{1/2}r^{-1}\Omega u,\quad 
\zeta(r-t\delb_rt) \del_tu + \zeta t\delb_ru,\quad\zeta u.
\endaligned
$$
Remark that
$$
\aligned
&\zeta\big(\zeta^2t\del_tu + (r+t\delb_rt)\delb_ru + u\big) - \zeta u 
- \frac{r+t\delb_r}{t}\big(\zeta(r-t\delb_rt) \del_tu + \zeta t\delb_ru\big)
\\
=& \Big(\zeta^3t - \frac{\zeta}{t}\big(r^2-(t\delb_rt)^2\big)\Big)\del_tu = \zeta t^{-1}(t^2-r^2)\del_t u,
\endaligned
$$
and
$$
\aligned
&-\frac{r-t\delb_rt}{\zeta^2t}\Big(\zeta\big(\zeta^2t\del_tu + (r+t\delb_rt)\delb_ru + u\big) - \zeta u\Big) 
+ \big(\zeta(r-t\delb_rt) \del_tu + \zeta t\delb_ru\big)
\\
=&\, -\zeta(r-t\delb_rt)\del_tu - \zeta^{-1}t^{-1}\big(r^2 - (t\delb_rt)^2\big)\delb_ru + \zeta(r-t\delb_rt)\del_tu + \zeta t\delb_ru
= \zeta^{-1} t^{-1}(t^2-r^2)\delb_ru.
\endaligned
$$
Here we remark that 
$$
\frac{r-t\delb_rt}{\zeta^2t} = \frac{r-t}{\zeta^2t} + \Big(1 + \frac{\xi(s,r)r}{(s^2+r^2)^{1/2}}\Big)^{-1},
$$
thus, thanks to \eqref{eq3-23-01-2022} and \eqref{eq5-23-01-2022}, 
\begin{equation}\label{eq7-11-03-2022-M}
|r-t\delb_rt|\lesssim \zeta^2t.
\end{equation}
These leads to the bound
\begin{equation}\label{eq4-11-03-2022-M}
\|\zeta t^{-1}|r^2-t^2|\del_tu\|_{L^2(\Fcal_s)} + \|\zeta^{-1} t^{-1}|r^2-t^2|\delb_ru\|_{L^2(\Fcal_s)} \lesssim \Fbf(s,u;s_0)^{1/2}.
\end{equation}

Recalling the following relation:
$$
\del_r u = \delb_ru - \frac{\xi(s,r)r}{(s^2+r^2)^{1/2}}\del_tu.
$$
This leads to 
\begin{equation}\label{eq5-11-03-2022-M}
\|\zeta t^{-1}|r^2-t^2|\del_r u\|_{L^2(\Fcal_s)}\lesssim \Fbf(s,u;s_0)^{1/2}.
\end{equation}

In order to bound $\del_a u$ and $\delb_au$, we need to remark the following relation:
$$
\sum_{a}|\del_a u|^2  = |r^{-1}\Omega u|^2 + |\del_r u|^2.
$$
On the other hand, remark that in $\Hcal^*_{[s_0,\infty)}$,
$$
\zeta t^{-1}|r^2-t^2| = s^3t^{-2}  \lesssim s = \sqrt{t^2-r^2} = A^{1/2}  
$$
and in $\Pcal_{[s_0,\infty)}$,
$$
\zeta t^{-1}|r^2-t^2| = t^{-1}|r^2-t^2| \lesssim \sqrt{r^2+t^2} = A^{1/2}.
$$
In $\Tcal_{[s_0\infty)}$, we remark that $t-1\leq r\leq t$. Then
$$
\zeta t^{-1}(t^2-r^2)\lesssim A^{1/2} \Leftarrow (t^2-r^2)^2\lesssim t^4 + r^2t^2 - 2rt^3\delb_rt
\Leftarrow r^4\lesssim r^2t^2 + 2t^2r^2-2rt^3\delb_rt
$$
where we have remarked that $0<\zeta\leq 1$ and $1\leq \delb_rt\leq 1$. Then we conclude that
$$
\zeta t^{-1}|(r^2-t^2)\del_au|\lesssim \zeta t^{-1}|(r^2-t^2)\del_ru| + A^{-1/2}|r^{-1}\Omega u|
$$
and this leads to 
\begin{equation}\label{eq6-11-03-2022-M}
\|\zeta t^{-1}|r^2-t^2|\del_a u\|_{L^2(\Fcal_s)}\lesssim \Fbf(s,u;s_0)^{1/2}.	
\end{equation}
In the same manner, we remark 
\begin{equation}\label{eq8-11-03-2022-M}
\sum_{a}|\delb_a u|^2  = |r^{-1}\Omega u|^2 + |\delb_r u|^2.
\end{equation}
We remark that in $\Tcal_{[s_0,\infty)}$,
$$
0\leq \frac{|t^2-r^2|}{\zeta t} = \zeta^{-1}\frac{t-r}{t}(r+t)\lesssim \frac{|t-r|^{1/2}}{t^{1/2}}(r+t),
$$
and
$$
A= t^2+r^2-2rt + 2rt(1-\delb_rt) = (t-r)^2 + 2rt(1-\delb_rt) \geq 2rt(1-\delb_rt).
$$
Remark that
$$
\frac{|t-r|}{t}\lesssim \zeta^2 = 1-(\delb_rt)^2\lesssim 1-\delb_rt
$$
and $r\simeq t$ in $\Tcal_{[s_0\infty)}$, we obtain
\begin{equation}\label{eq9-11-03-2022-M}
\frac{|t^2-r^2|^2}{\zeta^2t^2}\lesssim A.
\end{equation}
This bound holds trivially in $\Hcal^*_s$ and $\Pcal_s$. Now combining \eqref{eq8-11-03-2022-M} with \eqref{eq2-02-03-2022-M} and \eqref{eq9-11-03-2022-M}, we obtain
\begin{equation}
\|\zeta^{-1}t^{-1}|r^2-t^2|\delb_au\|_{L^2(\Fcal_s)}\lesssim \Fbf(s,u;s_0)^{1/2}.
\end{equation}
Finally we turn to $\delt_au$. We only need to recall \eqref{eq10-11-03-2022-M} and obtain
\begin{equation}
\|\zeta^{-1}t^{-1}|r^2-t^2|\delt_au\|_{L^2(\Fcal_s)}\lesssim \Fbf(s,u;s_0)^{1/2}.
\end{equation}
\end{proof}

\section{Normal form transform}
\label{sec1-03-04-2022-M}
\subsection{Objective and technical preparations}
We remark that in the auxiliary system \eqref{eq1-24-11-2021}, the Klein--Gordon equations have pure quadratic Klein--Gordon terms coupled on the right-hand side, i.e., 
$$
E_{\jh\kh}(t,x) v^{\jh}v^{\kh} +  F_{\jh\kh}^{\ih\alpha}(t,x)v^{\jh}\del_{\alpha}v^{\kh}.
$$
As explained in Introduction, these terms are critical in $\RR^{1+2}$ case and will be handled by normal from transform. This section is devoted to a variety of this technique. 

We only need to carry out the normal form transform in the region $\Fcal_{[s_0,\infty)}\cap\{r\leq 3t\}$ (or ``far from the spatial infinity''). The rest of this subsection is devoted to some technique preparations in this region.

We firstly recall the cut-off function $\chi$ defined in \eqref{eq1-03-10-2021}. We define the following conical cut-off function
$$
\chih(t,x) := 1- \chi(|x|/t-2).
$$
Remark that $\chih(t,x)\equiv 0$ when $|x|/t\geq 3$ and $\chih(t,x)\equiv 1$ when $|x|/t\leq 2$. Furthermore, $|\del_{\alpha}\chi(t,x)| = 0$ when $|x|/t\leq 2$ and $|x|/t\geq 3$, 
\begin{equation}\label{eq4-07-02-2022}
|\del\chih|_{p,k}\leq C(p) \la r+t\ra^{-1-p+k}\mathbbm{1}_{\{2\leq r/t\leq 3\}}.
\end{equation}
This can be easily checked by homogeneity, we omit the detail. 

A second result is the bound on the function $(t^2-r^2)/t^2$. We prove that
\begin{lemma}
In $\Fcal_{[s_0,\infty)}\cap\{r\leq 3t\}$,
\begin{equation}
|(t^2-r^2)/t^2|_{p,k}\leq
\begin{cases} 
&C(p) \la t + r\ra^{-p+k},\quad p>k,
\\
&C(p) |t^2-r^2|t^{-2},\quad p=k.
\end{cases}
\end{equation}
\end{lemma}
\begin{proof}
In $\Hcal^*_{s_0,\infty}$ this is exactly the bound on $(s/t)^2$ established in \cite{LM1}. Here we need to generalize it to a larger region. The argument is in fact the same. Firstly, by \eqref{prop1-23-10-2021}, we only need to bound $\big|\del^IL^J \big((t^2-r^2)/t^2\big)\big|$. Remark that
$$
L_a \big((t^2-r^2)/t^2\big) = -2(x^a/t)\big((t^2-r^2)/t^2\big),
$$
where $-2(x^a/t)$ is interior-homogeneous. Then by induction we conclude that 
$$
L^J \big((t^2-r^2)/t^2\big) = \Lambda^J\big((t^2-r^2)/t^2\big),
$$
where $\Lambda^J$ is homogeneous of degree zero. Thus the case $p=k$ is deduced. When $p>k$, we regard
$$
\del^IL^J\big((t^2-r^2)/t^2\big) = \del^I\big(\Lambda^J\big((t^2-r^2)/t^2\big)\big).
$$ 
Then by homogeneity, we conclude by the desired result.
\end{proof}

\subsection{Basic differential identities}
Recalling the semi-hyperboloidal frame $\{\delu_0=\del_t,\,\delu_a = (x^a/t)\del_t+\del_a\}$ is well defined in the region $\{r\leq 3t\}$. We firstly remark that 
\begin{equation}\label{eq1-20-01-2022}
	\Box u =  \frac{t^2-r^2}{t^2}\del_t\del_t u + t^{-1}A_\m[u].
\end{equation}
where 
$$
A_\m[u] = \Big(2(x^a/t)\del_tL_a - \sum_a\delu_aL_a - (x^a/t)\delu_a + (2+(r/t)^2)\del_t\Big)u.
$$
Here the subscript ``$\m$'' represents the Minkowski metric. $\m^{00} = 1$, $\m^{aa} = -1$ and the rest components are zero. It is clear that in $\{r\leq 3t\}$,
\begin{equation}\label{eq2-20-01-2022}
	|A_\m[u]|_{p,k}\lesssim |\del u|_{p+1,k+1}.
\end{equation}
On the other hand, in semi-hyperboloidal frame,
$$
\big(\minu^{\alpha\beta}\big)_{\alpha\beta} 
=\left(
\begin{array}{ccc}
	1-(r/t)^2 &x^1/t &x^2/t
	\\
	x^1/t &-1 &0
	\\
	x^2/t &0 &-1
\end{array} 
\right).
$$
The ``good components'' of the quadratic form $\m(\del u,\del v) := \m^{\alpha\beta}\del_{\alpha}u\del_{\beta}v$ is denoted by
$$
\minu(\del u,\del v) = \sum_{(\alpha,\beta) \neq (0,0)}\minu^{\alpha\beta}\del_{\alpha}u\del_{\beta}v
= (x^a/t^2)(\del_tu L_av + \del_tvL_au) - t^{-2}\sum_aL_auL_av.
$$
Remark that in $\{r\leq 3t\}$,
\begin{equation}\label{eq3-20-01-2022}
	\aligned
	\big|\minu(\del u,\del v)\big|_{p,k}
	\lesssim&\, t^{-1}\sum_{p_1+p_2=p\atop k_1+k_2=k}\big(|\del u|_{p_1,k_1}|v|_{p_2+1,k_2+1} 
	+ |\del u|_{p_1,k_1}|u|_{p_2+1,k_2+1}\big) 
	\\
	&+ t^{-2}\!\!\!\sum_{p_1+p_2=p\atop k_1+k_2=k}|u|_{p_1+1,k_1+1}|v|_{p_2+1,k_2+1}
	\\
	\lesssim&\, t^{-1}\!\!\!\!\sum_{p_1+p_2=p}\big(|\del u|_{p_1}|v|_{p_2+1} + |\del v|_{p_1}|u|_{p_2+1}\big).
	\endaligned
\end{equation}

From now on, we consider the Klein--Gordon system 
$$
\Box v_i +c_i^2v_i = F_i
$$
with the following {\sl non-resonant} condition:
\begin{equation}
c_i+c_j\neq c_k,\quad i,j,k = 1,2,\cdots,n.
\end{equation}
The above condition leads to 
\begin{equation}\label{eq10-07-02-2022}
(c_i^2 - c_j^2 - c_k^2)^2 - 4c_j^2c_j^2\neq 0.
\end{equation}

Then we establish three basic identity for normal form transform. We get started with  the following observation.
\begin{equation}\label{eq9-04-02-2022}
(\Box + c_i^2)\vt^0_{jk} = (c_i^2-c_j^2-c_k^2)v_jv_k + 2 \frac{t^2-r^2}{t^2}\del_tv_j\del_tv_k + \tilde{\Rbb}^0_{ijk}[v]
\end{equation}
with
\begin{equation}\label{eq5-07-02-2022}
\aligned
\vt^0_{jk} :=&\, v_jv_k, 
\\
\tilde{\Rbb}^0_{jk}[v] :=&\, 2\minu(\del v_j,\del v_k) + v_jF_k + v_kF_j.
\endaligned
\end{equation}
On the other hand, thanks to \eqref{eq1-20-01-2022},
	$$
	\aligned
	\Box (\del_tv_j\del_tv_k) =& \del_tv_j(\Box\del_t v_k) + \del_tv_k\Box(\del_tv_j) + 2\frac{t^2-r^2}{t^2}\del_t\del_tv_j\del_t\del_tv_k + 2\minu(\del\del_tv_j,\del\del_tv_k)
	\\
	=& \del_tv_j (\del_tF_k - c_k^2\del_tv_k) + \del_tv_k(\del_tF_j-c_j^2\del_tv_j) 
	+ 2\del_t\del_tv_j\big(\Box v_k - t^{-1} A_\m[v_k]\big)
	\\
	& +  2\minu(\del\del_tv_j,\del\del_tv_k)
	\\
	=&-(c_j^2+c_k^2)\del_tv_j\del_tv_k - 2c_k^2v_k\del_t\del_tv_j 
	\\
	&+ 2\del_t\del_tv_j(F_k- t^{-1}A_\m[v_k])  +  2\minu(\del\del_tv_j,\del\del_tv_k) + \del_tv_j\del_tF_k + \del_tv_k\del_tF_j.
	\endaligned
	$$
Apply again \eqref{eq1-20-01-2022} on $\del_t\del_t v_j$, we obtain
	$$
	\aligned
	&\frac{t^2-r^2}{t^2}\big(\Box + c_i^2\big)  (\del_tv_j\del_tv_k) 
	\\
	=& \frac{t^2-r^2}{t^2}(c_i^2-c_j^2-c_k^2)\del_tv_j\del_tv_k + 2c_j^2c_k^2v_jv_k 
	-2c_k^2v_kF_j + 2t^{-1}c_k^2v_kA_\m[v_j]
	\\
	&+ \frac{t^2-r^2}{t^2}\big(2\minu(\del\del_t v_j,\del\del_tv_k) + \del_tv_j\del_tF_k + \del_tv_k\del_tF_j + 2\del_t\del_tv_jF_k 
- 2t^{-1}\del_t\del_tv_jA_\m[v_k] \big).
	\endaligned
	$$
\begin{equation}\label{eq6-04-02-2022}
	\big(\Box + c_i^2\big)\vt^2_{jk} 
	=  2c_j^2c_k^2v_jv_k + (c_i^2-c_j^2-c_k^2)\frac{t^2-r^2}{t^2}\del_tv_j\del_tv_k + \tilde{\Rbb}^2_{jk}[v]
\end{equation}
with 
\begin{equation}\label{eq8-07-02-2022}
\aligned
\vt^2_{jk} :=&\, (t^2-r^2)t^{-2}\del_tv_j\del_tv_k,
\\
\tilde{\Rbb}^2_{jk}[v] :=&\, v_jv_k\Box\big((t^2-r^2)t^{-2}\big) + 2\m\big(\del\big((t^2-r^2)t^{-2}\big),\del(v_jv_k)\big)
 -2c_k^2v_kF_j + 2t^{-1}c_k^2v_kA_\m[v_j]
\\
&+ \frac{t^2-r^2}{t^2} \big(2\minu(\del\del_t v_j,\del\del_tv_k) + \del_tv_j\del_tF_k + \del_tv_k\del_tF_j + 2\del_t\del_tv_jF_k - 2t^{-1}\del_t\del_tv_jA_\m[v_k]\big) .
\endaligned
\end{equation}
Now we define , recalling \eqref{eq10-07-02-2022},
\begin{equation}\label{eq9-02-07-2022}
\aligned
v^0_{ijk} =&\, \big( (c_i^2-c_j^2-c_k^2)^2 - 4c_j^2c_k^2\big)^{-1}
\big((c_i^2-c_j^2-c_k^2)\vt^0_{jk} - 2\vt^2_{jk}\big),
\\
v^2_{ijk} =&\, \big( (c_i^2-c_j^2-c_k^2)^2 - 4c_j^2c_k^2\big)^{-1}
\big((c_i^2-c_j^2-c_k^2)\vt^2_{jk} - 2c_j^2c_k^2\vt^0_{jk} \big),
\endaligned
\end{equation}
and obtain, thanks to \eqref{eq9-04-02-2022} and \eqref{eq6-04-02-2022},
\begin{equation}\label{eq1-05-02-2022}
	(\Box + c_i^2)v^0_{ijk} = v_jv_k + \Rbb^0_{ijk}[v] ,
	\quad
	(\Box + c_i^2)v^2_{ijk} = \frac{t^2-r^2}{t^2}\del_tv_j\del_tv_k + \Rbb^2_{ijk}[v] ,
\end{equation}
with
\begin{equation}\label{eq10-02-07-2022}
\aligned
\Rbb^0_{ijk}[v] &=  \big( (c_i^2-c_j^2-c_k^2)^2 - 4c_j^2c_k^2\big)^{-1}
\big((c_i^2-c_j^2-c_k^2)\tilde{\Rbb}^0_{jk} - 2\tilde{\Rbb}^2_{jk}\big),
\\
\Rbb^2_{ijk}[v] &=  \big( (c_i^2-c_j^2-c_k^2)^2 - 4c_j^2c_k^2\big)^{-1}
\big((c_i^2-c_j^2-c_k^2)\tilde{\Rbb}^2_{jk} - 2c_j^2c_k^2\tilde{\Rbb}^0_{jk} \big).
\endaligned
\end{equation}

Finally, we remark that
$$
\aligned
\Box (v_j\del_t v_k) + c_i^2(v_j\del_t v_k) 
=&\, (c_i^2-c_j^2-c_k^2)v_j\del_tv_k - 2c_k^2v_k\del_tv_j
\\
&\,+F_j\del_tv_k + v_j\del_tF_k + 2\del_tv_jF_k + 2\minu(\del v_j,\del\del_t v_k) - 2t^{-1}\del_tv_j A_\m[v_k].
\endaligned
$$
Let $\vt^1_{jk} := v_j\del_tv_k$, the above identity is written as
\begin{subequations}\label{eq2-05-02-2022}
\begin{equation}\label{eq2a-05-02-2022}
(\Box + c_i^2)\vt^1_{jk} = (c_i^2-c_j^2-c_k^2)v_j\del_tv_k - 2c_k^2v_k\del_tv_j + \tilde{\Rbb}^1_{jk}[v]
\end{equation}
with
\begin{equation}\label{eq2b-05-02-2022}
\aligned
&\tilde{\Rbb}^1_{jk}[v] := F_j\del_tv_k + v_j\del_tF_k + 2\del_tv_jF_k + 2\minu(\del v_j,\del\del_t v_k) - 2t^{-1}\del_tv_j A_\m[v_k].
\endaligned
\end{equation}\label{eq2c-05-02-2022}
Exchanging $j,k$ in \eqref{eq2a-05-02-2022}, one has
\begin{equation}
(\Box + c_i^2)\vt^1_{kj} = (c_i^2-c_j^2-c_k^2)v_k\del_tv_j - 2c_j^2v_j\del_tv_k + \tilde{\Rbb}^1_{kj}[v].
\end{equation}
\end{subequations}
Then from the above two identities, one obtains, 
\begin{equation}\label{eq3-05-02-2022}
(\Box + c_i^2)v^1_{ijk} = v_j\del_tv_k + \Rbb^1_{ijk}[v].
\end{equation}
where, thanks to \eqref{eq10-07-02-2022}
\begin{equation}\label{eq11-07-02-2022}
\aligned
v^1_{ijk} =& -\big((c_i^2-c_j^2-c_k^2)^2-4c_j^2c_k^2\big)^{-1}
\big((c_i^2-c_j^2-c_k^2)\vt^1_{jk} - 2c_k^2\vt^1_{kj}\big),
\\
\Rbb^1_{ijk}[v] =& -\big((c_i^2-c_j^2-c_k^2)^2-4c_j^2c_k^2\big)^{-1}
\big((c_i^2-c_j^2-c_k^2)\tilde{\Rbb}^1_{jk}[v] - 2c_k^2\tilde{\Rbb}^1_{kj}[v]\big).
\endaligned
\end{equation}
\subsection{Normal form transform}\label{subsec1-04-04-2022-M}
Recalling \eqref{eq1-27-02-2022_M}, we now consider the following nonlinear Klein--Gordon system
\begin{equation}\label{eq8-04-02-2022}
	(\Box + c_i^2)v_i = A_i^{jk}v_jv_k + B_i^{jk}v_j\del_tv_k + f_i,
\end{equation}
where $c_i,A_i^{jk}$ and $B_i^{jk}$ are constants. $A_i^{jk} = A_i^{kj}$, $c_i>0$. The functions $f_i$ are defined in $\Fcal_{[s_0,s_1]}$. $v = (v_1v_2,\cdots, v_n)$ is a solution of \eqref{eq8-04-02-2022} defined in $\Fcal_{[s_0,s_1]}$, sufficiently regular and decreases sufficiently fast at spatial infinity. Then we consider the following auxiliary variable
$$
w_i =: v_i -\chih(t,x)\big(A_i^{jk}v^0_{ijk} + B_i^{jk}v^1_{ijk}\big). 
$$
By \eqref{eq1-05-02-2022} and \eqref{eq3-05-02-2022}, we obtain
\begin{equation}\label{eq4-05-02-2022}
(\Box + c_i^2)w_i = (1-\chih)\big( A_i^{jk}v_jv_k + B_i^{jk}v_j\del_tv_k\big)
+f_i + \Rbb_i[v],
\end{equation}
where
$$
\aligned
\Rbb_i[v] :=& (1-\chih)\big( A_i^{jk}\Rbb^0_{jk}[v] + B_i^{jk}\Rbb^1_{jk}[v]\big)
\\
&+\big( A_i^{jk}v^0_{ijk} + B_i^{jk}v^1_{ijk} \big)\Box\chih
+ 2\m\big(\del \chih,\del\big( A_i^{jk}v^0_{ijk} + B_i^{jk}v^1_{ijk}\big)\big).
\endaligned
$$
The main feature is that, on the right-hand side of \eqref{eq4-05-02-2022} the quadratic terms vanishes in $\{r\leq 2t\}$. 

It remains to show that the remaining terms on the right-hand side are ``super-critical'' in the sens that their $L^2$ norms are integrable with respect to the time function $s$. This will be the main work of the next subsection.

\subsection{Bounds on remaining terms}
Remark that we are working in the region $\{r\leq 3t\}$. For the simplicity we make the following convention.
$$
\aligned
&|v|_{p,k} := \max_{i=1,2,\cdots,n}|v_i|_{p,k},
\quad
&&\ebf^{p,k}[v] := \sum_{i=1}^n\ebf^{p,k}[v],
\\
&|F|_{p,k} :=  \max_{i=1,2,\cdots,n}|F_i|_{p,k},
\quad 
&&|f|_{p,k} :=  \max_{i=1,2,\cdots,n}|f_i|_{p,k}.
\endaligned
$$

We firstly remark in our case of \eqref{eq8-04-02-2022},
\begin{equation}\label{eq6-07-02-2022}
|F|_{p,k} \leq C(p) \sum_{p_1+p_2=p}\big(|v|_{p_1}|v|_{p_2} + |\del v|_{p_1}|v|_{p_2} + |\del v|_{p_1}|\del v|_{p_2}\big) + |f|.
\end{equation}

Then we recall \eqref{eq5-07-02-2022} and obtain
\begin{equation}\label{eq7-07-02-2022}
\aligned
|\vt^0_{jk}|_{p,k}\leq& C(p) \sum_{p_1+p_2=p}|v|_{p_1}|v|_{p_2},
\\
|\tilde{\Rbb}^0_{jk}[v]|_{p,k}\leq& C(p)\sum_{p_1+p_2=p}\big(t^{-1}|\del v|_{p_1}|v|_{p_2+1} + |v|_{p_1}|F|_{p_2}\big).
\endaligned
\end{equation}

Then we turn to $\vt^2$ and $\tilde{\Rbb}^2$. By \eqref{eq8-07-02-2022}, we obtain
\begin{equation}\label{eq3-07-02-2022}	
\aligned
|\tilde{\Rbb}^2_{jk}|_{p,k}\leq& C(p)t^{-1}\sum_{p_1+p_2=p}(|\del v|_{p_1+1}+|v|_{p_1})(|\del v|_{p_2+1}+|v|_{p_2})
\\
&+C(p) \big(|\del v|_{p_1}|\del F|_{p_2} +|\del v|_{p_1+1}|F|_{p_2} + |v|_{p_1}|F|_{p_2}\big),
\endaligned
\end{equation}
\begin{equation}
|\vt^2_{jk}|_{p,k}\leq C(p)\zeta^2\sum_{p_1+p_2=p}|\del v|_{p_1}|\del v|_{p_2}.
\end{equation}
Here \eqref{eq6-23-01-2022} is applied on $\frac{t^2-r^2}{t^2}$ for the region $\TPcal_[s_,s_1]$. In $\Hcal^*_{[s_0,s_1]}$ one has $\zeta^2 = (s/t)^2 = \frac{t^2-r^2}{t^2}$.

Finally we turn to $\vt^1$ and $\tilde{\Rbb}^1$. By \eqref{eq2b-05-02-2022}, 
\begin{equation}
\aligned
&|\tilde{\Rbb}^1_{jk}[v]|_{p,k}\leq C(p)t^{-1}|v|_{p_1+1}|\del v|_{p_2+1} + |\del v|_{p_1}|F|_{p_2} + |v|_{p_1}|\del F|_{p_2},
\\
&|\vt^1_{jk}|_{p,k}\leq C(p)t^{-1}|v|_{p_1}|\del v|_{p_2}
\endaligned
\end{equation}

Form \eqref{eq9-02-07-2022}, \eqref{eq10-02-07-2022}, \eqref{eq11-07-02-2022} together with the non-resonant condition \eqref{eq10-07-02-2022}, we obtain
\begin{equation}\label{eq2-07-02-2022}
\aligned
|v^0_{ijk}|_{p,k}\leq& C(p)\sum_{p_1+p_2=p}(|v|_{p_1}|v|_{p_2} + \zeta^2|\del v|_{p_1}|\del v|_{p_2}),
\quad
|v^1_{ijk}|_{p,k} \leq C(p)\sum_{p_1+p_2=p}|v|_{p_1}|\del v|_{p_2}.
\endaligned
\end{equation}
\begin{equation}\label{eq12-07-02-2022}
\aligned
|\Rbb^1_{ijk}|_{p,k}+|\Rbb^0_{ijk}[v]|_{p,k}\leq& C(p)t^{-1}\sum_{p_1+p_2=p}(|\del v|_{p_1+1} + |v|_{p_2+1})(|\del v|_{p_1+1} + |v|_{p_2+1})
\\
&+ \sum_{p_1+p_2=p}(|\del v|_{p_1} + |v|_{p_1})|\del F|_{p_2} + (|\del v|_{p_1+1} + |v|_{p_1})|F|_{p_2}.
\endaligned
\end{equation}

Now we are ready to estimate $\Rbb_i[v]$ on the right-hand side of \eqref{eq4-05-02-2022}. Remark that
$$
|\Box \chih|_{p,k} + |\del \chih|_{p,k}\leq C(p)t^{-1}\mathbbm{1}_{\{2t\leq r\leq 3t\}}.
$$
Furthermore,
%
\begin{equation}\label{eq1-08-02-2022}
\aligned
|\Rbb_i[v]|_{p,k}\lesssim & C(p)t^{-1}\sum_{p_1+p_2=p}(|\del v|_{p_1+1} + |v|_{p_2+1})(|\del v|_{p_1+1} + |v|_{p_2+1})
\\
&+ \sum_{p_1+p_2=p}(|\del v|_{p_1} + |v|_{p_1})|\del F|_{p_2} + (|\del v|_{p_1+1} + |v|_{p_1})|F|_{p_2}.
\endaligned
\end{equation}
On the other hand, we remark that
$$
\aligned
\big|(1-\chih)\big( A_i^{jk}v^0_{ijk}[v] + B_i^{jk}v^1_{ijk}[v]\big)\big|_{p,k}
\leq C(p) \mathbbm{1}_{\{r\geq 2t\}}\sum_{p_1+p_2=p}(|v|_{p_1}|v|_{p_2} + |\del v|_{p_1}|v|_{p_2}).
\endaligned
$$
We thus obtain:
\begin{equation}\label{eq2-08-02-2022}
\aligned
|(\Box + c_i^2)w_i|_{p,k}
\leq& |f|_{p,k}  + \big(C(p)t^{-1} + \mathbbm{1}_{\{r\geq 2t\}}\big)\sum_{p_1+p_2=p}(|\del v|_{p_1+1} + |v|_{p_1+1})(|\del v|_{p_2+1} + |v|_{p_2+1})
\\
&+ \sum_{p_1+p_2=p}(|\del v|_{p_1} + |v|_{p_1})|\del F|_{p_2} + (|\del v|_{p_1+1} + |v|_{p_1})|F|_{p_2}.
\endaligned
\end{equation}

In the following discussion, we will prove that, under suitable bootstrap assumptions, the right-hand side of \eqref{eq2-08-02-2022} enjoys integrable $L^2$ bound. Thus one expects a uniform energy bound on $w_i$. Finally, we need to deduce the bound on $v$ form the bound on $w$. This relies on the following estimate:
\begin{equation}
|w_i-v_i|_{p,k}\leq C(p) \sum_{p_1+p_2=p}\big(|v|_{p_1}|v|_{p_2} + |\del v|_{p_1}|v|_{p_2} + \zeta^2|\del v|_{p_1}|\del v|_{p_2}\big).
\end{equation} 
This is a direct result of \eqref{eq2-07-02-2022}. As a conclusion, we establish the following energy estimate.
\begin{proposition}\label{prop1-04-04-2022-M}
Suppose that $(v_i)_{i=1,2,\cdots,N}$ be a sufficiently regular solution to \eqref{eq8-04-02-2022} in $\Fcal_{[s_0,s_1]}$ with $s_0\geq 2$. 
Suppose furthermore that
\begin{equation}
\zeta^{-1}|v|_{[p/2]} + |\del v|_{[p/2]}\leq \vep_s
\end{equation}
with $\vep_s$ sufficiently small (determined by $p$). Then
\begin{equation}\label{eq6-04-04-2022-M}
\Ebf_{\eta,c}^p(s,v) + \int_{s_0}^s\la r-t\ra^{2\eta-1}(|\delt v|_p + |v|_p)J\,dxd\tau
\lesssim \sum_{k\leq p}\int_{s_0}^s\int_{\Fcal_s}|\del v|_p|T|_{p,k} \,Jdxd\tau. 
\end{equation}
Here $T$ represents the right-hand side of \eqref{eq2-08-02-2022}.
\end{proposition}
\part{Global stability result}
\section{Construction of auxiliary system and initial data}
\label{sec3-11-04-2022-M}
\subsection{Construction of auxiliary system}

Following the strategy  of \cite{Duan-Ma-2020}, we introduce an auxiliary  system, which make use of the divergence structure of the first term on the right-hand side of \eqref{eq1-26-02-2022-M}. Instead of considering $\phi^1$, we consider the ``primitive'' of $\phi^1$. 
We denote by $u = w_0 + A^{\alpha\jh}\del_{\alpha}w_{\jh}$ and consider the following {\sl auxiliary system}:
\begin{equation}\label{eq1-24-11-2021}
\aligned
\Box w_0 =& T_W[u,v] + S_W[u,v],
\\
\Box w_{\jh} =& v_{\jh}^2,
\\
\big(\Box + c_{\ih}^2\big) v_{\ih} =&K_{\ih}^{\alpha}v_{\ih}\del_{\alpha}u
 + E_{\ih}^{\jh\kh} v_{\jh}v_{\kh} +  F_{\ih}^{\alpha\jh\kh}v_{\jh}\del_{\alpha}v_{\kh}
 + T_{KG}^{\ih}[u,v] + S_{KG}^{\ih}[u,v]
\endaligned
\end{equation}
with $2\leq \ih,\jh,\kh,\cdots\leq n$ and the coefficients defined in Subsection \ref{subsec1-15-04-2022-M}. We impose the following initial data
\begin{equation}\label{eq2-24-11-2021}
\aligned
&w_0(2,x) = \phi^1_0(x),\quad \del_t w_0(2,x) = \phi^1_1(x) - \sum_{\ih}A^{0\ih}(\phi^{\ih}_0(x))^2,
\quad
w_{\ih}(2,x) = \del_t w_{\ih}(2,x) = 0,
\\
&v_{\ih}(2,x) = \phi^{\ih}_0(x),\quad \del_t v_{\ih}(2,x) = \phi^{\ih}_1(x),\quad 2\leq \ih\leq n.
\endaligned
\end{equation}
Then we state the following result.
\begin{proposition}\label{prop1-24-11-2021}
Let $(w_0,w_{\ih},v_{\ih})$ be a $C^3$ solution to \eqref{eq1-24-11-2021} with initial data \eqref{eq2-24-11-2021}. Then $(\phi^1,\phi^{\ih})$ is a $C^2$ solution to  \eqref{eq2-27-02-2022-M} with initial data \eqref{eq3-27-02-2022-M} with
	$$
	\phi^1 = w_0 + A^{\alpha\ih}\del_{\alpha}w_{\ih},\quad \phi^{\ih} = v_{\ih},\quad 2\leq \ih\leq n.
	$$
\end{proposition}
\begin{proof}
	Direct calculation shows that $(u = w_0 + A^{\alpha\ih}\del_{\alpha}w_{\ih},v_{\ih})$ and $(\phi^1,\phi^{\ih})$ solve the same Cauchy problem. Uniqueness result guarantees the desired result.
\end{proof}

\subsection{Key structures of the auxiliary system}  
In $\RR^{1+2}$, a wave-Klein-Gordon system with general quadratic and/or cubic nonlinearties blows up in finite time with small regular initial data. Thus the structure of the nonlinear terms contained in \eqref{eq1-24-11-2021} becomes crucial. 

We first remark that at the quadratic level, $w_0$ is decoupled from the system. $w_{\ih}$ are the essential wave components and they are coupled in the Klein-Gordon equations only by their Hessian components. This is of course due to the divergence structure on the right-hand side of \eqref{eq1-26-02-2022-M}, based on which the auxiliary is constructed. As explained in \cite{Duan-Ma-2020}, the Hessian components enjoy better $L^2$ and pointwise decay (see \eqref{eq1-25-03-2022-M} and \eqref{eq2-04-04-2022-M}) than the gradient. This feature helps us to obtain sufficient energy estimates. 

Then we regard the cubic terms. In $\RR^{1+2}$, in general cubic terms may leads to finite blow-up (see \cite{Zhou-Han-2011}). However in the present case, each cubic term contains at least one Klein-Gordon factor, which is sufficient for global existence. For more general cases about the cubic wave and Klein-Gordon systems, the reader is referred to \cite{Katayama-Matsumura-Sunagawa-2015} (for wave) and \cite{Cheng-2021} (for wave-Klein-Gordon). 

%

\subsection{Construction of initial data}\label{subsec1-22-04-2022-M}
For our purpose, we need the following initial energy estimates.
\begin{proposition}
Suppose that \eqref{eq1-22-04-2022-M} holds with $N\geq 4$ and $\vep$ sufficiently small (determined by $N$ and the system). Then the following initial energy bounds hold  
\begin{equation}\label{eq3-22-04-2022-M}
\Ebf^N_{\eta}(2,w_0)  
+ \Econ^{N-1}(2,w_0) + \|\zeta |w_0|_{N-1}\|_{L^2(\Fcal_2)}^2 
+
\sum_{\ih} (\Ebf^N_{\eta}(2,w_{\ih}) + \Ebf^N_{\eta,c}(2,v_{\ih}))\leq (C_0\vep)^2,
\end{equation}
where $C_0$ is determined by $N$ and the system. 
\end{proposition}
\begin{proof}[Sketch of proof]
We recall \eqref{eq1-22-04-2022-M} and \eqref{eq2-27-02-2022-M}. The local solution exists and extends to $\Fcal_2$ provided that $\vep$ sufficiently small. Then we recall \eqref{eq2-22-04-2022-M}. Following the local theory, the source term on the right-hand side is controlled by $C^*\vep$ where $C^*$ is a constant determined by the system and $N$. Then \eqref{eq3-22-04-2022-M} reduces to the estimates on $\Ebf^{\flat,N}_{\eta}(s,\phi^1)$ and $\Ebf^{\flat,N}_{\eta,c}(s,\phi^{\ih})$. These bounds are guaranteed by \eqref{eq1-22-04-2022-M}. Here we remark that for the bounds on $\del_t\del_tZ\phi^i$, we need to use the equation and express it as a linear combination of $\Delta Z\phi^i$ together with the nonlinear terms on the right-hand side. 

The conformal energy $\Econ(2,Zw_0)$ will be bounded by $\Ebf_{1/2+\eta,c}(2,Zw_0)$. Thanks to \eqref{eq4-22-04-2022-M}, we remark that for $\ord(Z)\leq N-1$, 
\begin{equation}\label{eq5-22-04-2022-M}
\Econ(2,Zw_0)^{1/2}\lesssim \|\la r\ra\del Z w_0\|_{L^2(\Fcal_2)} + \|\zeta Zw_0\|_{L^2(\Fcal_2)}.
\end{equation}
The first term is bounded by $\Ebf_{1/2+\eta}(2,Zw_0)^{1/2}$ due to the fact that $1/2+\eta>1$, For the second term, we make the following observation. Let $u$ be a $C^1$ function defined in $\RR^2$, decreasing sufficiently fast at infinity. Let $x = r\theta$,
$$
\int_{\mathbb{S}^1}u^2(x)d\theta 
= -2\int_{\mathbb{S}^1}\int_0^{\infty}\mathbbm{1}_{\{\rho\geq r\}}u(\rho\theta)\del_r u(\rho\theta)\,d\rho d\theta.
$$
Then
$$
\aligned
\int_{\RR^2}u^2dx =&\, \int_0^{\infty}r\int_{\mathbb{S}^1}u^2d\sigma\,dr
= -2\int_0^{\infty}\int_{\mathbb{S}^1}\int_0^{\infty}\mathbbm{1}_{\{\rho\geq r\}}u(\rho\theta)\del_r u(\rho\theta)\,r d\rho d\theta dr
\\
=&\, -2\int_0^{\infty}\int_{\mathbb{S}^1}u(\rho\theta)\del_ru(\rho\theta)\Big(\int_0^\rho rdr\Big)d\theta d\rho
= -\int_0^{\infty}\int_{\mathbb{S}^1}\rho^2u(\rho\theta)\del_ru(\rho\theta)d\theta d\rho 
\\
=&\, - \int_{\RR^2} u(x)\del_r u(x) rdx.
\endaligned
$$
This leads to
$$
\|u\|_{L^2(\RR^2)}^2\lesssim \|u\|_{L^2(\RR^2)}\|\la r\ra\del u\|_{L^2(\RR^2)}\quad \Rightarrow 
\quad \|u\|_{L^2(\RR^2)}\lesssim \|\la r\ra\del u\|_{L^2(\RR^2)}. 
$$
Then we take $u(x) = Zw_0(2,x)$. Then $\del_a u(x) = \delb_a Zw_0(2,x)$, which is bounded by $\Ebf_{1/2+\eta}^{N-1}(2,w_0)$ because $1/2+\eta>1$. Then the second term on the right-hand side of \eqref{eq5-22-04-2022-M} is also bounded $C^*\vep$ with $C^*$ a constant determined by $N$. This concludes the desired result.
\end{proof}
\section{Bootstrap argument and linear estimates}\label{sec4-11-04-2022-M}
In this section we start the bootstrap argument. We fix $0<\delta\ll (\eta-1/2)< 1/2$ and $N\geq 4 $. The symbol ``$\lesssim$'' meas smaller or equal to modulo a constant determined by $s_0,\delta,\eta$ and  $N$. We denote by $|w|_{p,k} = \max_{\ih}|w_{\ih}|_{p,k}$, $|v|_{p,k} := \max_{\ih}|v_{\ih}|_{p,k}$. Furthermore,
$$
\Ebf_{\eta}^{p,k}(s,w):=\sum_{\ih}\Ebf_{\eta}^{p,k}(s,w_{\ih}),\quad 
\Ebf_{\eta,c}^{p,k}(s,v):=\sum_{\ih}\Ebf_{\eta,c}^{p,k}(s,v_{\ih}).
$$
\subsection{Bootstrap assumption and Sobolev decay}
Suppose that on a time interval $[s_0,s_1]$ the following energy bounds hold:
\begin{subequations}\label{eq7-29-10-2021}
\begin{equation}\label{eq7d-29-10-2021}
\Ebf_{\eta}^N(s,\del w) + \int_{s_0}^s\int_{\TPcal_{\tau}}\la r-t\ra^{2\eta-1} |\delt\del w|_N^2\,J dxd\tau 
\leq (C_1\vep)^2 s^{2\delta},
\end{equation}
\begin{equation}\label{eq7a-29-10-2021}
\Ebf_{\eta}^N(s,w)  + \int_{s_0}^s\int_{\TPcal_{\tau}}\la r-t\ra^{2\eta-1}|\delt w|_N^2\,J dxd\tau 
\leq (C_1\vep)^2 s^{2\delta},
\end{equation}
\end{subequations}

\begin{equation}\label{eq4-17-03-2022-M}
\Ebf_{\eta}^N(s,w_0) + \int_{s_0}^s\int_{\TPcal_{\tau}}\la r-t\ra^{2\eta-1}|\delt w_0|_N^2\,J dxd\tau 
\leq (C_1\vep)^2s^{2\delta},
\end{equation}

\begin{subequations}\label{eq8-14-03-2022-M}
\begin{equation}\label{eq7c-29-10-2021}
\Ebf_{\eta,c}^N(s,v) 
+ \int_{s_0}^s\int_{\TPcal_\tau}\la r-t\ra^{2\eta-1}\big(|v|_N^2 + |\delt v|_N^2 \big)\,Jdxd\tau 
 \leq(C_1\vep)^2 s^{2\delta},
\end{equation}
\begin{equation}\label{eq7e-29-10-2021}
\Ebf_{\eta,c}^{N-1}(s,v) 
+ \int_{s_0}^s\int_{\TPcal_\tau}\la r-t\ra^{2\eta-1}\big(|v|_{N-1}^2 + |\delt v|_{N-1}^2 \big)\,Jdxd\tau 
\leq (C_1\vep)^2.
\end{equation}
\end{subequations}
Recalling the relation $u = w_0 + A^{\alpha\jh}\del_{\alpha}w_{\jh}$, we obtain

\begin{equation}\label{eq3-17-03-2022-M}
\Ebf_{\eta}^N(s,u)^{1/2} + \int_{s_0}^s\int_{\TPcal_{\tau}}\la r-t\ra^{2\eta-1}|\delt u|_N^2\,J dxd\tau 
\leq C_1\vep s^{\delta},
\end{equation}

Our purpose is establishing, on the same interval, the following {\sl improved energy bounds}:
\begin{subequations}\label{eq3-14-03-2022-M}
\begin{equation}\label{eq3a-14-03-2022-M}
\Ebf_{\eta}^N(s,w) + \int_{s_0}^s\int_{\TPcal_{\tau}}\la r-t\ra^{2\eta-1}|\delt w|_N^2\,J dxd\tau \leq \frac{1}{2}(C_1\vep)^2 s^{2\delta},
\end{equation}
\begin{equation}\label{eq3b-14-03-2022-M}
\Ebf_{\eta}^N(s,\del w) + \int_{s_0}^s\int_{\TPcal_{\tau}}\la r-t\ra^{2\eta-1}|\delt\del w|_N^2 \,J dxd\tau 
\leq \frac{1}{2}(C_1\vep)^2 s^{2\delta},
\end{equation}
\begin{equation}\label{eq3c-14-03-2022-M}
\Ebf_{\eta}^N(s,w_0) + \int_{s_0}^s\int_{\TPcal_{\tau}}\la r-t\ra^{2\eta-1}|\delt w_0|^2\,J dxd\tau \leq \frac{1}{2}(C_1\vep)^2 s^{2\delta},
\end{equation}
\begin{equation}\label{eq4a-14-03-2022-M}
\Ebf_{\eta,c}^N(s,v) + \int_{s_0}^s\int_{\TPcal_\tau}\la r-t\ra^{2\eta-1}\big(|v|_N^2 + |\delt v|_N^2\big)\,Jdxd\tau \leq \frac{1}{2}(C_1\vep)^2 s^{2\delta},
\end{equation}
\begin{equation}\label{eq4b-14-03-2022-M}
\Ebf_{\eta,c}^{N-1}(s,v) + \int_{s_0}^s\int_{\TPcal_\tau}\la r-t\ra^{2\eta-1}\big(|v|_{N-1}^2 + |\delt v|_{N-1}^2\big)\,Jdxd\tau \leq \frac{1}{2}(C_1\vep)^2.
\end{equation}
\end{subequations}
These bounds will be established via the energy estimates Proposition \ref{prop2-24-11-2021}. More precisely, we need to bound 
\begin{equation}\label{eq7-14-03-2022-M}
\int_{s_0}^s\int_{\Fcal_{\tau}}\la r-t\ra^{2\eta} |ZT\,\del_t Z\phi|\,Jdxd\tau,
\end{equation}
where $T$ represents the terms on the right-hand side of \eqref{eq1-24-11-2021} and $\phi = w_0,w,v$, $\ord(Z)\leq N$.  The rest of this article is essentially devoted to the estimates on the above quantity.

\subsection{Sobolev decay}
\paragraph{Sobolev bounds in $\TPcal_{[s_0,s_1]}$.} We combine \eqref{eq7-29-10-2021}, \eqref{eq8-14-03-2022-M} with \eqref{eq2-15-03-2022-M} and \eqref{eq1-30-10-2021}. Remark that in $\TPcal_s$, $\omega \simeq \la r-t\ra$, we thus obtain
\begin{subequations}\label{eq4-31-10-2021}
\begin{equation}\label{eq4a-31-10-2021}
\frac{\la r\ra^{1/2}}{\la r-t\ra^{1/2}}|\delt \bar{w}|_p + |\del \bar{w}|_p
\lesssim
\left\{
\aligned
& C_1\vep \la r-t\ra^{-1/2-\eta}s^{\delta}, \quad && p=N-2,
\\
& C_1\vep \la r\ra^{-1/2}\la r-t\ra^{-\eta}s^{\delta},&&p=N-3,
\endaligned
\right.
\end{equation}
with $\bar{w} = w_0, w_{\ih}$ or $\del w_{\ih}$. 
\begin{equation}\label{eq4e-31-10-2021}
\frac{\la r\ra^{1/2}}{\la r-t\ra^{1/2}}(|\delt v|_p + |v|_p) + |\del v|_p\lesssim 
\left\{
\aligned 
&C_1\vep \la r-t\ra^{-1/2-\eta}s^{\delta},\quad &&p=N-2,
\\
&C_1\vep \la r-t\ra^{-1/2-\eta},\quad &&p=N-3.
\endaligned
\right.
\end{equation}
\begin{equation}\label{eq4f-31-10-2021}
\frac{\la r\ra^\eta}{\la r-t\ra^\eta}|\delt v|_p + |\del v|_p\lesssim 
\left\{
\aligned 
&C_1\vep \la r\ra^{-1/2}\la r-t\ra^{-\eta}s^{\delta},\quad &&p=N-3,
\\
&C_1\vep \la r\ra^{-1/2}\la r-t\ra^{-\eta},\quad &&p=N-4.
\endaligned
\right.
\end{equation}
Also by the relation $u = w_0 + A^{\alpha\jh}\del_{\alpha}w_{\jh}$, we obtain
\begin{equation}\label{eq5-17-03-2022-M}
\frac{\la r\ra^{1/2}}{\la r-t\ra^{1/2}}|\delt u|_p + |\del u|_p
\lesssim 
\left\{
\aligned
&C_1\vep \la r-t\ra^{-1/2-\eta}s^{\delta},\quad && p=N-2,
\\
&C_1\vep \la r\ra^{-1/2}\la r-t\ra^{-\eta}s^{\delta},&& p=N-3.
\endaligned
\right.
\end{equation}
\end{subequations}

\paragraph{Bounds in $\Hcal^*_{[s_0,s_1]}$.} We combine \eqref{eq7-29-10-2021}, \eqref{eq8-14-03-2022-M} with Proposition \ref{prop1-21-01-2022} and denote by $\bar{w} = w_0,w_{\ih}$ or $\del w$,
\begin{subequations}\label{eq8-15-03-2022-M}
\begin{equation}\label{eq3-22-03-2022-M}
(s/t)|\del\bar{w}|_{N-2} + |\delu \bar{w}|_{N-2} \lesssim C_1\vep (s/t)s^{-1+\delta},
\end{equation}
\begin{equation}\label{eq3-31-10-2021}
(s/t)|\del v|_{p-2} + |\delu v|_{p-2} + |v|_{p-2} + |\del v|_{p-3}\lesssim
\left\{
\aligned
& C_1\vep (s/t)s^{-1+\delta},\quad &&p=N,
\\
& C_1\vep (s/t)s^{-1}, && p=N-1,
\endaligned
\right.
\end{equation}
\begin{equation}\label{eq7-17-03-2022-M}
(s/t)|\del u|_{N-2} + |\delu u|_{N-2}\lesssim C_1\vep (s/t)s^{-1+\delta}.
\end{equation}
\end{subequations}

We can obtain a rough bound on $|u|_{N-2}$ by integrating $|\del_r Zu|$ from spatial infinity, where $\ord (Z)=p\leq N-2$:
$$
|Zu(t_0,x_0)| \leq \int_{r_0}^{\infty} |\del_rZu(t_0,\rho(x_0/|x_0|))|d\rho.
$$
Firstly, in $\TPcal_{[s_0,s_1]}$, by \eqref{eq5-17-03-2022-M},
\begin{equation}\label{eq1-20-03-2022-M}
|u|_p\lesssim 
\left\{
\aligned
&C_1\vep s^{\delta},\quad && p=N-2,
\\
&C_1\vep, \quad && p =N-3.
\endaligned
\right.
\end{equation}
When $(t_0,x_0)\in \Hcal^*_{[s_0,s_1]}$, one has
\begin{equation}\label{eq2-20-03-2022-M}
|u|_{N-2}\lesssim C_1\vep s^{\delta}.
\end{equation}
\section{Estimates on Klein--Gordon components in $\TPcal_{[s_0,s_1]}$}  
\subsection{Objective of this section}
This section is devoted to the following estimates.
\begin{proposition}\label{prop1-17-03-2022-M}
Suppose that \eqref{eq7-29-10-2021}, \eqref{eq4-17-03-2022-M} and \eqref{eq8-14-03-2022-M} hold on $[s_0,s_1]$ with $C_1\vep$ sufficiently small (determined by the system and $N$). Then in $\TPcal_{[s_0,s_1]}$,
\begin{equation}\label{eq1-15-03-2022-M}
\frac{\la r\ra|v|_{p-4}}{\la r-t\ra} + \frac{\la r\ra^{1/2}|v|_{p-3}}{\la r-t\ra^{1/2}}\lesssim 
\left\{
\aligned
&C_1\vep \la r\ra^{-1/2}\la r-t\ra^{-\eta} s^{\delta}, \quad && p=N,
\\
&C_1\vep \la r\ra^{-1/2}\la r-t\ra^{-\eta} , \quad && p=N-1.
\endaligned
\right.
\end{equation}
\begin{equation}\label{eq5-15-03-2022-M}
\|\la r\ra\la r-t\ra^{-1}\zeta\omega^{\eta}|v|_{N-1}\|_{L^2(\TPcal_s)} \lesssim C_1\vep s^{2\delta}.
\end{equation}
\begin{equation}\label{eq2-02-04-2022-M}
\| \la r\ra\la r-t\ra^{-1}\zeta\omega^{\eta}|v|_p\|_{L^2(\TPcal_s)}\lesssim C_1\vep,\quad p\leq N-2,
\end{equation}
\begin{equation}\label{eq4-04-04-2022-M}
\|\la r\ra^{-3/2}\la r-t\ra^{3/2}\zeta\omega^{\eta}|v|_{N-3}\|_{L^2(\TPcal_s)}\lesssim C_1\vep s^{3\delta}.
\end{equation}
\end{proposition}
Let us give a description on the strategy. We mainly rely on Proposition \ref{eq4-30-10-2021}. For this purpose we need to bound the right-hand side of the Klein--Gordon equations. The pointwise bounds will be firstly established and then they will be applied in the proof of $L^2$ bounds. At the end of this Section we give a direct result of \eqref{eq1-15-03-2022-M} and \eqref{eq5-15-03-2022-M}, which is the following uniform energy bound:
\begin{proposition}\label{prop1-21-03-2022-M}
Suppose that \eqref{eq1-15-03-2022-M} and \eqref{eq5-15-03-2022-M} hold on $[s_0,s_1]$. Then
\begin{equation}\label{eq6-21-03-2022-M}
\Ebf_{\eta}^{\TPcal,N}(s,w) + \int_{s_0}^s\int_{\TPcal_\tau}\omega^{2\eta-1}|\delt w|_N^2\,Jdxd\tau
\leq (C_1\vep)^2.
\end{equation}
\end{proposition}

\subsection{The pointwise estimates}
We remark that in $\TPcal^{\far}_{[s_0,s_1]}$, \eqref{eq1-15-03-2022-M} is covered by \eqref{eq4e-31-10-2021}. We only need to prove them in $\TPcal^{\near}_{[s_0,s_1]}$. In order to apply Proposition \ref{eq4-30-10-2021}, we need the following estimates.
\begin{lemma}
Suppose that \eqref{eq7-29-10-2021}, \eqref{eq8-14-03-2022-M} hold on $[s_0,s_1]$. Then
\begin{equation}\label{eq1-28-11-2021}
\big|(\Box+c_{\ih}^2\big)v_{\ih}|_{p,k} \lesssim C_1\vep |v|_{p,k},\quad p\leq N-3
\end{equation}
\end{lemma}
\begin{proof}
We only need to remark that on the right-hand side of \eqref{eq1-27-02-2022_M}, each term contains at least one Klein--Gordon factor. Then we apply the Sobolev bounds \eqref{eq4-31-10-2021} and \eqref{eq1-20-03-2022-M}.
\end{proof}
\begin{proof}[Proof of pointwise bound \eqref{eq1-15-03-2022-M}]
Based on Proposition \ref{eq4-30-10-2021}, one has, for $p\leq N-3$
$$
\aligned
c_{\ih}^2|v_{\ih}|_{p,k}\leq C(p)C_1\vep |v|_{p,k}+
\begin{cases}
C(p) \la r\ra^{-3/2}\la r-t\ra^{1-\eta}\Ebf_{\eta}^{\TPcal,p+4,k+4}(s,u)^{1/2}  ,
\\
C(p) \la r\ra^{-1}\la r-t\ra^{1/2-\eta}\Ebf_{\eta}^{\TPcal,p+3,k+3}(s,u)^{1/2}.
\end{cases}
\endaligned
$$
Taking $C_1\vep$ sufficiently small and recalling \eqref{eq7-29-10-2021}, \eqref{eq8-14-03-2022-M} , the above inequalities lead to \eqref{eq1-15-03-2022-M}.
\end{proof}

\subsection{$L^2$ estimates on high-order terms}
Once \eqref{eq1-15-03-2022-M} is established, we are able to obtain sufficient $L^2$ bounds on the higher-order terms on the right-hand side of the Klein--Gordon equations. More precisely, we will establish the following result.
\begin{lemma}\label{lem1-20-03-2022-M}
Suppose that \eqref{eq7-29-10-2021}, \eqref{eq8-14-03-2022-M} hold on $[s_0,s_1]$. Then for $p\leq N$,
\begin{equation}\label{eq4-20-03-2022-M}
\aligned
|T_W[u,v]|_p + |T_{KG}[u,v]|_p
\lesssim 
&\, (C_1\vep)^2 s^{2\delta} \la r\ra^{-3/2}\la r-t\ra^{1/2-2\eta}(|\del u|_p + |\del v|_p) 
\\
& + (C_1\vep)^2s^{2\delta} \la r\ra^{-1}\la r-t\ra^{-2\eta}(|\delt u|_p + |\delt v|_p + |v|_p),
\endaligned
\end{equation}
and for $p\leq N-1$,
\begin{equation}\label{eq4-21-03-2022-M}
\aligned
|T_W[u,v]|_p + |T_{KG}[u,v]|_p
\lesssim&\, 
(C_1\vep)^2 s^{2\delta} \la r\ra^{-2}\la r-t\ra^{1-2\eta}(|\del u|_p + |\del v|_p )
\\
& + (C_1\vep)^2s^{2\delta} \la r\ra^{-3/2}\la r-t\ra^{1/2-2\eta}(|\delt u|_p + |\delt v|_p + |v|_p),
\endaligned
\end{equation}
In the same manner, for $p\leq N$,
\begin{equation}\label{eq5-20-03-2022-M}
\aligned
|S_W[u,v]|_p + |S_{KG}[u,v]|_p \lesssim &\, (C_1\vep)^2 s^{3\delta} \la r\ra^{-3/2}\la r-t\ra^{1/2-2\eta}
(|\del u|_p + |\del v|_p) 
\\
& + (C_1\vep)^2s^{3\delta} \la r\ra^{-1}\la r-t\ra^{-2\eta}(|\delt u|_p + |\delt v|_p + |v|_p),
\endaligned
\end{equation}
and for $p\leq N-1$,
\begin{equation}\label{eq3-21-03-2022-M}
\aligned
|S_W[u,v]|_p + |S_{KG}[u,v]|_p \lesssim &\, (C_1\vep)^2 s^{3\delta} \la r\ra^{-2}\la r-t\ra^{1-2\eta}(|\del u|_p + |\del v|_p )
\\
& + (C_1\vep)^2s^{3\delta} \la r\ra^{-3/2}\la r-t\ra^{1/2-2\eta}(|\delt u|_p + |\delt v|_p + |v|_p).
\endaligned
\end{equation}
\end{lemma}
\begin{proof}
We first consider the bounds with order $N$. Remark that when $N\geq 4$,
\begin{equation}\label{eq7-15-03-2022-M}
\aligned
|ABC|_p\lesssim& |A|_p(|B|_{N-2}|C|_{N-3} + |B|_{N-3}|C|_{N-2}) 
 + |B|_p(|A|_{N-2}|C|_{N-3} + |A|_{N-3}|C|_{N-2})
\\
& + |C|_p(|A|_{N-2}|B|_{N-3} + |A|_{N-3}|B|_{N-2}).
\endaligned
\end{equation} 
We substitute the Sobolev decay bounds \eqref{eq4-31-10-2021}, \eqref{eq8-15-03-2022-M} and \eqref{eq1-15-03-2022-M} into the corresponding expression (on $|\cdot|_{N-2}$ and $|\cdot|_{N-3}$). 

We then regard the null terms in $T_W[u,v]$ and $T_{KG}[u,v]$. 
$$
|\m(du,dv_{\jh})v_{\kh}|\lesssim |v\delt u\, \del v| + |v\del u\,\delt v|,
$$
where
$$
\aligned
|v\delt u\del v|_p\lesssim& |\delt u|_p\big(|\del v|_{N-2}|v|_{N-3} + |\del v|_{N-3}|v|_{N-2}\big)
+ |\del v|_p\big(|\delt u|_{N-2}|v|_{N-3} + |\delt u|_{N-3}|v|_{N-2}\big)
\\
&+ |v|_p\big(|\delt u|_{N-2}|\del v|_{N-3} + |\delt u|_{N-3}|\del v|_{N-2}\big)
\\
\lesssim& (C_1\vep)^2s^{2\delta} \la r\ra^{-1}\la r-t\ra^{-2\eta}  (|\delt u|_p + |v|_p) 
+ (C_1\vep)^2s^{2\delta} \la r\ra^{-3/2}\la r-t\ra^{1/2-2\eta} |\del v|_p.
\endaligned
$$
and similarly,
$$
\aligned
|v\del u\,\delt v|_p
\lesssim&\, (C_1\vep)^2s^{2\delta}\la r\ra^{-3/2}\la r-t\ra^{1/2-2\eta}|\del u|_p 
+ (C_1\vep)^2s^{2\delta} \la r\ra^{-1}\la r-t\ra^{-2\eta}(|\delt v|_p + |v|_p).
\endaligned
$$
So we conclude that
\begin{equation}\label{eq2-16-03-2022-M}
\aligned
|\m(du,dv_{\jh})v_{\kh}|_p 
\lesssim&\, (C_1\vep)^2 s^{2\delta} \la r\ra^{-3/2}\la r-t\ra^{1/2-2\eta}(|\del u|_p + |\del v|_p) 
\\
& + (C_1\vep)^2s^{2\delta} \la r\ra^{-1}\la r-t\ra^{-2\eta}(|\delt u|_p + |\delt v|_p + |v|_p).
\endaligned
\end{equation}
Similarly, one can show that all null terms enjoy the same bound. 

For the rest of the terms, the proofs are direct calculations and we omit the details. Then we arrive at \eqref{eq4-20-03-2022-M}.

The higher-order terms $S_{W}, S_{KG}$ are bounded similarly. For the bound on $|u|_{N-2}$ we need \eqref{eq1-20-03-2022-M}.  We omit the details.

For the lower-order case, wen remark that for $p\leq N-1$, $[p/2]\leq N-3$ provided that $N\geq 4$. Then
\begin{equation}
\aligned
|ABC|_p\lesssim& |A|_p(|B|_{N-3}|C|_{N-3} + |B|_{N-3}|C|_{N-3}) 
 + |B|_p(|A|_{N-3}|C|_{N-3} + |A|_{N-3}|C|_{N-3})
\\
& + |C|_p(|A|_{N-3}|B|_{N-3} + |A|_{N-3}|B|_{N-3}). 
\endaligned
\end{equation}
Then repeating the above discussion, we obtain the lower-order bounds. In fact we remark that, comparing the pointwise bounds enjoyed by 
$$
|\del u|_{N-2},\quad |\del v|_{N-2},\quad |\delt u|_{N-2},\quad |\delt v|_{N-2},\quad |v|_{N-2}
$$
with the bounds enjoyed by 
$$
|\del u|_{N-3},\quad |\del v|_{N-3},\quad |\delt u|_{N-3},\quad |\delt v|_{N-3},\quad |v|_{N-3},
$$
we find that $|T|_{N-3}\lesssim \la r\ra^{-1/2}\la r-t\ra^{1/2}|T|_{N-2}$ where $T$ represents any of the above terms. This fact directly leads to \eqref{eq4-21-03-2022-M} and \eqref{eq3-21-03-2022-M}.
\end{proof}

From Lemma \ref{lem1-20-03-2022-M}, we obtain the following results.
\begin{corollary}\label{lem1-04-04-2022-M}
Suppose that the estimates in Lemma \ref{lem1-20-03-2022-M} hold. Then
\begin{equation}\label{eq2-21-03-2022-M}
\int_{s_0}^s\int_{\TPcal_\tau}\omega^{2\eta}|\del_t Z v_{\jh}|\,|Z(T_{KG}^{\ih} + S_{KG}^{\ih})|J\,dxd\tau\lesssim 
\left\{
\aligned
& (C_1\vep)^4s^{\delta},\quad && \ord (Z) = N,
\\
& (C_1\vep)^4,\quad&& \ord(Z)\leq N-1.
\endaligned
\right.
\end{equation}
\begin{equation}\label{eq6-22-03-2022-M}
\int_{s_0}^{s}\int_{\TPcal_\tau}\omega^{2\eta}|\del_t Z w_0|\,|Z(T_W + S_W)|J\,dxd\tau\lesssim 
\left\{
\aligned
& (C_1\vep)^4s^{\delta},\quad && \ord (Z) = N,
\\
& (C_1\vep)^4,\quad&& \ord(Z)\leq N-1.
\endaligned
\right.
\end{equation}
\end{corollary}
\begin{proof}
We only write the prove of \eqref{eq2-21-03-2022-M}.

We first treat the top-order case. Suppose that $\ord(Z) = p = N$. Then by \eqref{eq4-20-03-2022-M}, \eqref{eq5-20-03-2022-M} and \eqref{eq13-07-10-2021},
\begin{equation}\label{eq1-21-03-2022-M}
\aligned
&\int_{s_0}^s\int_{\TPcal_\tau}\omega^{2\eta}|\del_t Z v_{\ih}|\,| ZT_{KG}^{\ih}[u,v]|J\,dxd\tau
\lesssim \int_{s_0}^s \int_{\TPcal_\tau}|\omega^{\eta}\zeta\del_t Zv_{\ih}|\,J\zeta^{-1}\la r-t\ra^{\eta}|T_{KG}[u,v]|_p\,dxd\tau
\\
\lesssim&\,(C_1\vep)^2 \int_{s_0}^s\tau^{1+2\delta}
\Ebf_{\eta,c}^{\TPcal,p}(\tau,v)^{1/2}\|\la r\ra^{-3/2}\la r-t\ra^{1/2-\eta}\zeta(|\del u|_p+|\del v|_p)\|_{L^2(\TPcal_\tau)}\,d\tau
\\
&+(C_1\vep)^2\int_{s_0}^s\tau^{2\delta}\Ebf_{\eta,c}^{\TPcal,p}(\tau,v)^{1/2}
\|J\zeta^{-1}\la r\ra^{-1}\la r-t\ra^{-\eta}(|\delt u|_p + |\delt v|_p + |v|_p)\|_{L^2(\TPcal_\tau)}\,d\tau.
\endaligned
\end{equation}
The first term on the right-hand side is bounded as follows.
$$
\aligned
&\int_{s_0}^s\tau^{1+2\delta}\Ebf_{\eta,c}^{\TPcal,p}(\tau,v)^{1/2}
\|\la r\ra^{-3/2}\la r-t\ra^{1/2-\eta}\zeta(|\del u|_p+|\del v|_p)\|_{L^2(\TPcal_\tau)}\,d\tau
\\
\lesssim&\, 
\int_{s_0}^s\tau^{-2 + 2\delta}\Ebf_{\eta,c}^{\TPcal,p}(\tau,v)^{1/2}\|\la r-t\ra^{\eta}\zeta(|\del u|_p + |\del v|_p)\|_{L^2(\TPcal_\tau)}\,d\tau
\\
\lesssim&\, \int_{s_0}^s\tau^{-2+2\delta}\Ebf_{\eta,c}^{\TPcal,N}(\tau,v)\,d\tau
\lesssim (C_1\vep)^2\int_{s_0}^s\tau^{-1+4\delta}\,d\tau
\lesssim (C_1\vep)^2,
\endaligned
$$
where we have applied \eqref{eq2-26-10-2021}, \eqref{eq3-17-03-2022-M}, and the fact that $s^2\lesssim t\lesssim r$.
For the second term on right-hand side of \eqref{eq1-21-03-2022-M}, we remark that
$$
\aligned
&\int_{s_0}^s\tau^{2\delta}\Ebf_{\eta,c}^{\TPcal,p}(\tau,v)^{1/2}
\|J\zeta^{-1}\la r\ra^{-1}\la r-t\ra^{-\eta}(|\delt u|_p + |\delt v|_p + |v|_p)\|_{L^2(\TPcal_\tau)}\,d\tau
\\
\lesssim&\,\int_{s_0}^s\tau^{1/2+2\delta}\Ebf_{\eta,c}^{\TPcal,p}(\tau,v)^{1/2}
\|J^{1/2}\la r\ra^{-1}\la r-t\ra^{-\eta-1/2}(|\delt u|_p + |\delt v|_p + |v|_p)\|_{L^2(\TPcal_\tau)}\,d\tau
\\
\lesssim&\, \int_{s_0}^s\tau^{-3/2+2\delta}\Ebf_{\eta,c}^{\TPcal,p}(\tau,v)^{1/2}
\|J^{1/2}\la r-t\ra^{\eta-1/2}(|\delt u|_p + |\delt v|_p + |v|_p)\|_{L^2(\TPcal_\tau)}\,d\tau
\\
\lesssim&\,C_1\vep\int_{s_0}^s\tau^{-3/2+3\delta}
\|J^{1/2}\la r-t\ra^{\eta-1/2}(|\delt u|_p + |\delt v|_p + |v|_p)\|_{L^2(\TPcal_\tau)}\,d\tau
\\
\lesssim&\,C_1\vep \Big(\int_{s_0}^s\tau^{-3+6\delta}d\tau\Big)^{1/2}\
\Big(\int_{s_0}^s\|J^{1/2}\la r-t\ra^{\eta-1/2}(|\delt u|_p + |\delt v|_p + |v|_p)\|_{L^2(\TPcal_\tau)}^2d\tau\Big)^{1/2}
\lesssim (C_1\vep)^2s^{\delta},
\endaligned 
$$
where \eqref{eq3-17-03-2022-M} is applied. The term containing the higher-order terms $S_{KG}$ is treated in the same way.  So we conclude by the case $p=N$ of  \eqref{eq2-21-03-2022-M}.

The lower-order case is relatively simpler. We only need to remark that by \eqref{eq4-21-03-2022-M},
$$
\aligned
|T_{KG}[u,v]|_p
\lesssim&\, 
(C_1\vep)^2 s^{2\delta} \la r\ra^{-2}\la r-t\ra^{1-2\eta}(|\del u|_p + |\del v|_p + |v|_p).
\endaligned
$$
Then
$$
\aligned
&\int_{s_0}^s\int_{\TPcal_\tau}\omega^{2\eta}|\delt Zv_{\ih}||ZT_{KG}^{\ih}[u,v]|J\,dxd\tau
\\
\lesssim&\, (C_1\vep)^2\int_{s_0}^s\tau^{-3 + 2\delta}\Ebf_{\eta,c}^{\TPcal,p}(\tau,v)^{1/2}
\|J\zeta^{-1}\la r-t\ra^{\eta}(|\del u|_p+|\del v|_p+|v|_p)\|_{L^2(\TPcal_\tau)}\,d\tau
\\
\lesssim&\, (C_1\vep)^2\int_{s_0}^s\tau^{-2+2\delta}\Ebf_{\eta,c}^{\TPcal,p}(\tau,v)^{1/2}\Ebf_{\eta}^{\TPcal,p}(\tau,u)^{1/2}\,d\tau
\lesssim (C_1\vep)^2.
\endaligned
$$
The bound on the terms containing $S_{KG}$ is similar, we omit the details.
\end{proof}

\subsection{The $L^2$ estimates}
This subsection is devoted to the proofs of \eqref{eq5-15-03-2022-M} -- \eqref{eq4-04-04-2022-M}. We recall that in $\TPcal^{\far}_{[s_0,s_1]}$, the weight $\la r\ra\la r-t\ra^{-1}$ is trivial. So we only need to bound the $L^2$ norm on $\TPcal^{\near}_s$. to this juncture let us recall \eqref{eq5-21-03-2022-M}:
\begin{equation}\label{eq7-21-03-2022-M}
c^2\la r\ra\la r-t\ra^{-1}|v|_p \lesssim |\del v|_{p+1} + \la r\ra\la r-t\ra^{-1}|f|_p
\end{equation}
where $f$ represents the right-hand  side of the Klein-Gordon equations. We denote by
$$
Q_{KG}[u,v] := K_{\ih}^{\alpha}v_{\ih}\del_{\alpha}u
+ E_{\ih}^{\jh\kh} v_{\jh}v_{\kh} +  F_{\ih}^{\alpha\jh\kh}v_{\jh}\del_{\alpha}v_{\kh}
$$
the quadratic terms on the right-hand side of the Klein-Gordon equations. Then when $p\leq N-1$, one has $[p/2]\leq N-3$. Then by the Sobolev decay and \eqref{eq1-15-03-2022-M},
$$
\aligned
|Q_{KG}[u,v]|_p\lesssim&\, |v|_p(|\del u|_{N-3} + |\del v|_{N-3} + |v|_{N-3}) + |\del v|_p|v|_{N-3} + |\del u|_p|v|_{N-3}
\\
\lesssim&\, C_1\vep|v|_p + C_1\vep s^{\delta}\la r\ra^{-1}\la r-t\ra^{1/2-\eta} (|\del u|_p + |\del v|_p).
\endaligned
$$
Substitute this bound together with \eqref{eq4-21-03-2022-M} and \eqref{eq3-21-03-2022-M} in \eqref{eq7-21-03-2022-M}, we obtain, thanks to \eqref{eq6-23-01-2022},
$$
\aligned
&\|\la r\ra\la r-t\ra^{-1}\zeta\omega^{\eta}|v|_p\|_{L^2(\TPcal^{\near}_s)}
\\
\lesssim&\, \|\zeta\omega^{\eta}|\del v|_{p+1}\|_{L^2(\TPcal^{\near}_s)}
 + C_1\vep\|\la r\ra\la r-t\ra^{-1}\zeta\omega^{\eta}|v|_p\|_{L^2(\TPcal^{\near}_s)}
\\
&+ C_1\vep s^{\delta}\|\la r-t\ra^{-1/2-\eta}\zeta\omega^{\eta}(|\del u|_p + |\del v|_p)\|_{L^2(\TPcal^{\near}_s)}
\\
&+ (C_1\vep)^2 s^{3\delta}\|\la r\ra^{-1/2}\la r-t\ra^{-1/2-2\eta}\zeta\omega^{\eta}(|\del u|_p + |\del v|_p + |v|_p)\|_{L^2(\TPcal^{\near}_s)}
\\
\lesssim&\,  C_1\vep\|\la r\ra\la r-t\ra^{-1}\zeta\omega^{\eta}|v|_p\|_{L^2(\TPcal^{\near}_\tau)}
+  s^{\delta}\Ebf_{\eta,c}^{\TPcal,p+1}(s,v)^{1/2} + s^{-1+3\delta}\Ebf_{\eta}^{\TPcal,p}(s,u)^{1/2},
\endaligned
$$
which proves \eqref{eq5-15-03-2022-M} provided that $C_1\vep$ sufficiently small.

For \eqref{eq2-02-04-2022-M}, we remark that when $p\leq N-2$, the quadratic terms are bounded as follows:
\begin{equation}\label{eq3-04-04-2022-M}
|Q_{KG}[u,v]|_p\lesssim C_1\vep|v|_p + C_1\vep s^{\delta}\la r\ra^{-3/2}\la r-t\ra^{1-\eta}(|\del u|_p + |\del v|_p)
\end{equation}
Then similar to the higher-order case, 
$$
\aligned
&\|\la r\ra\la r-t\ra^{-1}\zeta\omega^{\eta}|v|_p\|_{L^2(\TPcal^{\near}_s)}
\\
\lesssim&\, \|\zeta\omega^{\eta}|\del v|_{p+1}\|_{L^2(\TPcal^{\near}_s)}
 + C_1\vep \|\la r\ra\la r-t\ra^{-1}\zeta\omega^{\eta}|v|_p\|_{L^2(\TPcal^{\near}_s)}
 \\
& + C_1\vep s^{\delta}\|\la r\ra^{-1/2}\la r-t\ra^{-\eta}\zeta\omega^{\eta}(|\del u|_p + |\del v|_p)\|_{L^2(\TPcal^{\near}_s)} 
\\
&+ (C_1\vep)^2 s^{3\delta}\|\la r\ra^{-1/2}\la r-t\ra^{-1/2-2\eta}\zeta\omega^{\eta}(|\del u|_p + |\del v|_p + |v|_p)\|_{L^2(\TPcal^{\near}_s)}
\\
\lesssim&\, 
C_1\vep \|\la r\ra\la r-t\ra^{-1}\zeta\omega^{\eta}|v|_p\|_{L^2(\TPcal^{\near}_s)} 
+ \Ebf_{\eta,c}^{\TPcal,p+1}(s,v)^{1/2} + s^{-1+3\delta}\Ebf_{\eta}^{\TPcal,p}(s,u)^{1/2}.
\endaligned
$$
Then we obtain \eqref{eq2-02-04-2022-M}.

For \eqref{eq4-04-04-2022-M}, we apply twice \eqref{eq7-21-03-2022-M}. We remark that
$$
\aligned
c^2\la r\ra^{3/2}\la r-t\ra^{-3/2}|v|_{N-3}\lesssim&  \la r\ra^{1/2}\la r-t\ra^{-1/2}|\del v|_{N-2} 
+ \la r\ra^{3/2}\la r-t\ra^{-3/2}|f|_{N-3}
\\
\lesssim&\, \la r\ra^{1/2}\la r-t\ra^{-1/2}|v|_{N-1}  + \la r\ra^{3/2}\la r-t\ra^{-3/2}|f|_{N-3}.
\endaligned
$$
Then 
$$
\aligned
&\|\la r\ra^{3/2}\la r-t\ra^{-3/2}\zeta \omega^{\eta}|v|_{N-3}\|_{L^2(\TPcal^{\near}_s)}
\\
&\lesssim \|\la r\ra^{1/2}\la r-t\ra^{-1/2}\zeta\omega^{\eta}|v|_{N-1}\|_{L^2(\TPcal^{\near}_s)} 
+ C_1\vep \|\la r\ra^{3/2}\la r-t\ra^{-3/2}\zeta \omega^{\eta}|v|_{N-3}\|_{L^2(\TPcal^{\near}_s)}
\\
&\quad + C_1\vep s^{\delta}\|\la r-t\ra^{-1/2-\eta}\zeta\omega^{\eta}(|\del u|_p + |\del v|_p)\|_{L^2(\TPcal^{\near}_s)} 
\\
&\quad +(C_1\vep)^2 s^{3\delta}\|\la r-t\ra^{-1-2\eta}\zeta\omega^{\eta}(|\del u|_p + |\del v|_p + |v|_p)\|_{L^2(\TPcal^{\near}_s)}
\\
&\lesssim C_1\vep s^{3\delta} 
+ C_1\vep \|\la r\ra^{3/2}\la r-t\ra^{-3/2}\zeta \omega^{\eta}|v|_{N-3}\|_{L^2(\TPcal^{\near}_s)},
\endaligned
$$
where \eqref{eq5-15-03-2022-M} is applied. Then when $C_1\vep$ sufficiently small, we obtain \eqref{eq4-04-04-2022-M}.
\subsection{Proof of Proposition \ref{prop1-21-03-2022-M}}
Recalling Lemma \ref{lem3-07-10-2021},
$$
\aligned
&\|\omega^{\eta}J\zeta^{-1}|v_{\ih}^2|_N\|_{L^2(\TPcal^{\far}_s)} \lesssim s\|\omega^{\eta}|v_{\ih}^2|_N\|_{L^2(\TPcal^{\far}_s)}
\lesssim C_1\vep s^{1+\delta}\|r^{-1/2-\eta}\, \omega^{\eta}|v|_N\|_{L^2(\TPcal^{\far}_s)}
\\
\lesssim& C_1\vep s^{-2\eta+\delta}\|\omega^{\eta}|v|_N\|_{L^2(\TPcal^{\far}_s)} \lesssim (C_1\vep)^2 s^{-1-2\delta},
\endaligned
$$
where we have remark that $\delta\ll \eta-1/2$. 

The estimate in $\TPcal^{\near}$ is more involved. We firstly remark that, because $N\geq 4$, one has $[N/2]\leq N-2$. Then
$$
\aligned
&\|\omega^{\eta}J\zeta^{-1}|v_{\ih}^2|_N\|_{L^2(\TPcal^{\near}_s)} 
\lesssim s\|\omega^{\eta}\zeta|v_{\ih}^2|_N\|_{L^2(\TPcal^{\near}_s)}
\\
\lesssim& s\| \omega^{\eta}|v_{\ih}|_{N-4}|v_{\ih}|_N\|_{L^2(\TPcal^{\near}_s)} 
+ s\|\omega^{\eta}|v_{\ih}|_{N-2}\zeta|v_{\ih}|_{N-1}\|_{L^2(\TPcal^{\near}_s)} 
\\
\lesssim& C_1\vep s^{1+\delta}\|\la r\ra^{-3/2}\la r-t\ra^{1-\eta}\,\omega^{\eta}|v|_N\|_{L^2(\TPcal^{\near}_s)}
\\
&+ s\|\la r\ra^{-3/2}\la r-t\ra^{1-\eta}|v|_{N-2}\zeta \la r\ra \la r-t\ra^{-1} \omega^{\eta}|v|_{N-1}\|_{L^2(\TPcal^{\near}_s)}
\\
\lesssim& \, C_1\vep s^{-2\eta+\delta}\|\omega^{\eta}|v|_N\|_{L^2(\TPcal^{\near}_s)}
+ C_1\vep s^{-2\eta+\delta}\|\la r\ra\la r-t\ra^{-1}\zeta\omega^{\eta}|v|_{N-1}\|_{L^2(\TPcal^{\near}_s)}
\\
\lesssim&\, (C_1\vep)^2s^{-2\eta+3\delta}\lesssim (C_1\vep)^2 s^{-1-2\delta}.
\endaligned
$$
Here for the last inequality we have applied \eqref{eq5-15-03-2022-M} and the fact that $\delta\ll \eta-1/2$. This bound is integrable with respect to $s$. Then by Proposition \ref{prop2-24-11-2021}, for $\ord(Z)\leq N$,
$$
\aligned
\int_{s_0}^s\int_{\TPcal_s}|\omega^{2\eta}\del_t Zw_{\ih}\,Z(v_{\ih}^2)|J\,dxd\tau
\lesssim&
\int_{s_0}^s\|\omega^{\eta}\zeta|\del w_{\ih}|_N\|_{L^2(\TPcal_\tau)}\|\omega^{\eta}J\zeta^{-1}|v_{\ih}^2|_N\|_{\TPcal_\tau}\,d\tau
\\
\lesssim& (C_1\vep)^2\int_{s_0}^s\Ebf_{\eta}^{\TPcal,N}(\tau,w)^{1/2}\tau^{-1-2\delta}\,d\tau
\lesssim (C_1\vep)^3.
\endaligned
$$
Then by Proposition \ref{prop2-24-11-2021} applied on 
$$
\Box Z w_{\ih} = Z(v_{\ih}^2)
$$
with $\ord(Z) \leq N$, we obtain 
$$
\Ebf_{\eta}^{\TPcal,N}(s,w) + 2\eta\int_{s_0}^s\int_{\TPcal_\tau}\la r-t\ra^{2\eta-1}|\delt w|_NJ\,dxd\tau\lesssim (C_0\vep)^2 + (C_1\vep)^3.
$$
Recall that $C_0<C_1$ and $C_1\vep \leq 1$, we obtain \eqref{eq6-21-03-2022-M}.

\section{Estimates on $w_0$}
\subsection{Objective of this section}
In this section we establish the following result:
\begin{proposition}[Improved energy estimates]
Suppose that the estimates in Lemma \ref{lem1-20-03-2022-M} hold. Then
\begin{equation}\label{eq8-21-03-2022-M}
\Ebf_{\eta}^p(s,w_0) + \int_{s_0}^s\int_{\TPcal_\tau}\la r-t\ra^{2\eta-1}|\delt w_0|_p^2\,Jdxd\tau
\leq (C_0\vep)^2
 + \left\{
\aligned
&C(C_1\vep)^4s^{\delta},\quad p=N,
\\
&C(C_1\vep)^4,\quad p\leq N-1,
\endaligned
\right.
\end{equation}
\begin{equation}\label{eq9-21-03-2022-M}
\Econ^{N-1}(s,w_0)^{1/2}\lesssim (C_1\vep)^3s^{4\delta}.
\end{equation}
\end{proposition}


\subsection{Standard energy estimate}
We differentiate the equation of $w_0$ with respect to a high-order operator $Z$ with $\ord(Z) = p\leq N$: 
\begin{equation}\label{eq10-21-03-2022-M}
\Box Zw_0 = Z(T_W[u,v] + S_{W}[u,v]).
\end{equation}
We recall Proposition \ref{prop2-24-11-2021}. Our main task is to estimate the following quantities:
\begin{equation}\label{eq11-21-03-2022-M}
\int_{s_0}^s\int_{\Fcal_\tau}|\omega^{2\eta}\del_tZ w_0\,Z(T_W[u,v] + S_{W}[u,v])|J\,dxd\tau.
\end{equation}

We first concentrate on the hyperbolic domain.  The Sobolev decay estimates are already sufficient. More precisely, we will establish the following result.
\begin{lemma}\label{lem1-22-03-2022-M}
	Suppose \eqref{eq7-29-10-2021}, \eqref{eq4-17-03-2022-M}, \eqref{eq8-14-03-2022-M} together with  \eqref{eq8-15-03-2022-M} hold in $\Hcal^*_{[s_0,s_1]}$. Then for $p\leq N$,
	\begin{equation}\label{eq1-22-03-2022-M}
		|T_W[u,v]|_p + |S_W[u,v]|_p\lesssim \,(C_1\vep)^2s^{-2+2\delta}(s/t)(\ebf^p[u] +\ebf_c^p[v])^{1/2},
	\end{equation}
	\begin{equation}\label{eq2-22-03-2022-M}
		|T_{KG}[u,v]|_p + |S_{KG}[u,v]|_p \lesssim \, (C_1\vep)^2s^{-2+3\delta}(s/t)(\ebf^p[u] +\ebf_c^p[v])^{1/2},
	\end{equation}
	in $\Hcal^*_{[s_0,s_1]}$.
\end{lemma}
\begin{proof}
We recall \eqref{eq7-15-03-2022-M}. In general, we only need to substitute the Sobolev pointwise bounds into the corresponding expression. However, for the null terms we need their null structure:\footnote{see for example \cite{M-2020-strong}, Appendix B, for a detailed proof.}
	\begin{equation}\label{eq4-22-03-2022-M}
		\aligned
		|\m(du,dv)|_p\lesssim&\, (s/t)^2(|\del u|_p|\del v|_{N-2} + |\del u|_{N-2}|\del v|_p) 
		\\
		+& |\delu u|_p|\del v|_{N-2} + |\del v|_p|\delu u|_{N-2} + |\del u|_p|\delu v|_{N-2} + |\del u|_{N-2}|\delu v|_p.
		\endaligned
	\end{equation}
	Denote by $\phi = u$ or $v_{\ih}$, one has
	\begin{equation}\label{eq5-22-03-2022-M}
		|\m(\del \phi,\del \phi)v_{\ih}|_p\lesssim (C_1\vep)^2(s/t)s^{-2+2\delta}((s/t)|\del \phi|_p + |\delu \phi|_p + |v|_p).
	\end{equation}
	The rest of the terms in $T_W,T_{KG}, S_W$ and $S_{KG}$ are bounded directly by \eqref{eq7-15-03-2022-M}, we omit the details.
\end{proof}

\begin{proof}[Proof of \eqref{eq8-21-03-2022-M}]
 Recall that
$$
\aligned
&\int_{s_0}^s\int_{\Fcal_\tau}|\omega^{2\eta}\del_tZ w_0\,Z(T_W[u,v] + S_{W}[u,v])|J\,dx d\tau
\\
=&\, \int_{s_0}^s\int_{\Hcal^*_\tau} + \int_{s_0}^s\int_{\TPcal_\tau}|\omega^{2\eta}\del_tZ w_0\,Z(T_W[u,v] + S_{W}[u,v])|J\,dx d\tau
=: I^{\Hcal} + I^{\TPcal}.
\endaligned
$$
Based on Lemma \ref{lem1-22-03-2022-M} and thanks to  \eqref{eq4-17-03-2022-M}, \eqref{eq7c-29-10-2021} and \eqref{eq3-17-03-2022-M},
$$
\aligned
I^{\Hcal} \lesssim &\,
\int_{s_0}^s\|(s/t)|\del_t w_0|_p\|_{L^2(\Hcal^*_\tau)}\||T_W[u,v] + S_{W}[u,v]|_p\|_{L^2(\Hcal^*_\tau)}\,d\tau
\\
\lesssim&\, (C_1\vep)^2\int_{s_0}^s\tau^{-2+2\delta}\Ebf^{\Hcal,p}(\tau,w_0)^{1/2}(\Ebf^{\Hcal,p}(\tau,u)^{1/2} + \Ebf_c^{\Hcal,p}(\tau,v)^{1/2})d\tau
\lesssim (C_1\vep)^4.
\endaligned
$$
On the other hand, $I^{\TPcal}$ is bounded by \eqref{eq6-22-03-2022-M}. We thus conclude by \eqref{eq7-22-03-2022-M}.
\end{proof}
\subsection{Conformal energy estimate}
This subsection is devoted to the proof of \eqref{eq9-21-03-2022-M}. Recalling Proposition \ref{prop1-22-03-2022-M}, we need to bound the following quantity with $\ord(Z) = p\leq N-1$:
$$
\int_{s_0}^s\|J\zeta^{-1}A^{1/2}\Box Z w_0\|_{L^2(\Fcal_\tau)} \,d\tau.
$$
which reduces to the sum of the following two terms:
$$
\aligned
I_{\text{con}}^{\Hcal} := \int_{s_0}^s \tau\||\Box w_0|_p\|_{L^2(\Hcal^*_\tau)}d\tau,
\quad
I^{\TPcal}_{\text{con}} := \int_{s_0}^s\big\|J\zeta^{-1} (t^2+r^2-2rt\delb_r t)^{1/2}|\Box w_0|_p\big\|_{L^2(\TPcal_\tau)} d\tau.
\endaligned
$$
Recalling its expression and \eqref{eq1-22-03-2022-M},
$$
\aligned
I^{\Hcal}_{\text{con}} 
\lesssim (C_1\vep)^2\int_{s_0}^s\tau^{-1+3\delta}\Big(\int_{\Hcal^*_\tau}(\tau/t)^2(\ebf^p[u]+ \ebf_c^p[v])dx\Big)^{1/2}\,d\tau
\lesssim (C_1\vep)^3s^{4\delta}.
\endaligned
$$
On the other hand, we recall \eqref{eq4-21-03-2022-M} and \eqref{eq3-21-03-2022-M},
\begin{equation}\label{eq5-23-03-2022-M}
\aligned
I^{\TPcal}_{\text{con}} =& \int_{s_0}^s\big\|J\zeta^{-1} (t^2+r^2-2rt\delb_r t)^{1/2}|\Box w_0|_p\big\|_{L^2(\TPcal_\tau)} d\tau
\\
\lesssim&\, 
(C_1\vep)^2\int_{s_0}^s\tau^{1+3\delta}\|\la r\ra^{-1}\la r-t\ra^{1-2\eta}\zeta(|\del u|_p + |\del v|_p)\|_{L^2(\TPcal_\tau)}d\tau
\\
&\, + (C_1\vep)^2\int_{s_0}^s\tau^{1/2+3\delta}\|J^{1/2}\la r\ra^{-1/2}\la r-t\ra^{1/2-2\eta}
(|\delt u|_p + |\delt v|_p + |v|_p)\|_{L^2(\TPcal_\tau)}d\tau.
\endaligned
\end{equation}
The first term on the right-hand side is bounded directly:
$$
\aligned
\int_{s_0}^s\tau^{1+3\delta}\|\la r\ra^{-1}\la r-t\ra^{1-2\eta}\zeta(|\del u|_p + |\del v|_p)\|_{L^2(\TPcal_\tau)}d\tau
\lesssim& 
\int_{s_0}^s\tau^{-1+3\delta}\big(\Ebf_{\eta}^p(\tau,u)^{1/2} + \Ebf_{\eta,c}^p(\tau,v)^{1/2}\big)d\tau
\\
\lesssim&\, C_1\vep s^{4\delta}.
\endaligned
$$
For the second term in right-hand side of \eqref{eq5-23-03-2022-M}, 
$$
\aligned
&\int_{s_0}^s\tau^{1/2+3\delta}\|J^{1/2}s^{1/2}\la r\ra^{-1/2}\la r-t\ra^{1/2-2\eta}
(|\delt u|_p + |\delt v|_p + |v|_p)\|_{L^2(\TPcal_\tau)}d\tau
\\
\lesssim&\,  \int_{s_0}^s\tau^{-1/2 + 3\delta}\|J^{1/2}\la r-t\ra^{\eta-1/2}(|\delt u|_p + |\delt v|_p + |v|_p)\|_{L^2(\TPcal_\tau)}d\tau
\\
\lesssim&\, \Big(\int_{s_0}^s\tau^{-1+6\delta}d\tau\Big)^{1/2}
\Big(\int_{s_0}^s\la r-t\ra^{2\eta-1}(|\delt u|_p + |\delt v|_p + |v|_p)^2\,Jdxd\tau\Big)^{1/2}
\lesssim C_1\vep s^{4\delta}.
\endaligned
$$
This concludes the desired energy estimate.

\subsection{Pointwise estimate based on improved energy bounds}
For our purpose, we need the following result.
\begin{proposition}\label{prop1-26-03-2022-M}
Suppose that \eqref{eq9-21-03-2022-M}, then in $\Hcal^*_{[s_0,s_1]}$
\begin{equation}\label{eq1-26-03-2022-M}
|\del w_0|_{N-3}\lesssim C_1\vep (s/t)^{-1}s^{-2+4\delta}.
\end{equation}	
Suppose that \eqref{eq8-21-03-2022-M} holds. Then
\begin{equation}\label{eq1-02-04-2022-M}
|\del w_0|_{N-4}\lesssim C_1\vep \la r\ra^{-1/2}\la r-t\ra^{-\eta},\quad \text{in } \TPcal_{[s_0,s_1]},
\end{equation}
\begin{equation}\label{eq2-26-03-2022-M}
|\del w_0|_{N-3}\lesssim
\left\{
\aligned
& C_1\vep s^{-1},\quad && \text{in } \Hcal^*_{[s_0,s_1]},
\\
& C_1\vep\la r-t\ra^{-1/2-\eta}, && \text{in } \TPcal_{[s_0,s_1]}.
\endaligned
\right.
\end{equation} 
\end{proposition}
\begin{proof}
The estimate \eqref{eq1-02-04-2022-M} is direct from \eqref{eq1-30-10-2021} and the uniform energy bound in \eqref{eq8-21-03-2022-M}. The estimate \eqref{eq2-26-03-2022-M} is a direct consequence of the uniform energy bound on $\Ebf_{\eta}^{N-1}(s,w_0)$ combined with \eqref{eq2-21-01-2022} and \eqref{eq2-15-03-2022-M}.

The bound \eqref{prop1-26-03-2022-M} is  based on the global Sobolev inequality on $H^*_s$. We remark that by Proposition \ref{prop1-11-03-2022-M},
\begin{equation}\label{eq1-04-04-2022-M}
\|(s/t)^2s|\del w_0|_{N-1}\|_{L^2(\Hcal^*_s)}\lesssim \|\zeta |u|_{N-1}\|_{L^2(\Fcal_{s_0})} + \Econ^{N-1}(s,u)^{1/2} + \int_{s_0}^s\tau^{-1}\Econ^{N-1}(\tau,u)^{1/2}d\tau
\lesssim C_1\vep s^{4\delta},
\end{equation}
and for $\ord(Z_1) \leq 2$, $\ord(Z_2)\leq N-3$, 
$$
|Z_1((s/t)^2sZ_2\del w_0)|\lesssim (s/t)^2s|\del w_0|_{N-1}\quad \text{in } \Hcal^*_{[s_0,s_1]}.
$$
Then by \eqref{prop1-21-10-2021}, we obtain the desired result.
\end{proof}

\section{Estimates on $w_{\ih}$}

\subsection{Improved energy estimates}
We first establish the following estimates.
\begin{proposition}
Suppose that \eqref{eq7-29-10-2021}, \eqref{eq8-14-03-2022-M} and the Sobolev bounds \eqref{eq4e-31-10-2021}, \eqref{eq4f-31-10-2021} hold. Then
\begin{equation}\label{eq2-25-03-2022-M}
\Ebf_{\eta}^N(s,\del w) + \int_{s_0}^s\int_{\TPcal_\tau}\la r-t\ra^{2\eta-1}|\del\delt w|_N\,J dxd\tau
\leq (C_0\vep)^2 + C(C_1\vep)^3s^{2\delta}.
\end{equation}
\begin{equation}\label{eq2-24-03-2022-M}
\Ebf_{\eta}^N(s,w) 
+ \int_{s_0}^s\int_{\TPcal_\tau}\la r-t\ra^{2\eta-1} |\delt w|_N\,Jdxd\tau
\leq(C_0\vep)^2 + C(C_1\vep)^3s^{2\delta},
\end{equation}
\end{proposition}
\begin{proof}
We will only write the prove of \eqref{eq2-25-03-2022-M} because $|v_{\ih}^2|_p$ satisfies all bounds enjoyed by $|v_{\ih}\del v_{\ih}|_p$. 	
Recalling Proposition \ref{prop2-24-11-2021}, we need to bound
$$
\aligned
\mathcal{J} :=&\, \|J\zeta^{-1}\omega^{\eta}|v_{\ih}\del v_{\ih}|_N\|_{L^2(\Fcal_s)}
\leq \||v_{\ih}\del v_{\ih}|_N\|_{L^2(\Hcal^*_s)} + \|J\zeta^{-1}\omega^{\eta}|v_{\ih}\del v_{\ih}|_N\|_{L^2(\TPcal_s)}
\\
=:&\, \mathcal{J}^{\Hcal} + \mathcal{J}^{\TPcal}.
\endaligned
$$
The quantity $\mathcal{J}^{\Hcal}$ is bounded directly by the bootstrap bounds and Sobolev decay. We remark that (recalling $N\leq 4$)
$$
\aligned
\mathcal{J}^{\Hcal}\lesssim&\, \||\del v_{\ih}|_N|v_{\ih}|_{N-3}\|_{L^2(\Hcal^*_s)} 
+ \||\del v_{\ih}|_{N-1}|v_{\ih}|_{N-2}\|_{L^2(\Hcal^*_s)}
\\
& + \|| v_{\ih}|_N|\del v_{\ih}|_{N-3}\|_{L^2(\Hcal^*_s)}
+ \| v_{\ih}|_{N-1}|\del v_{\ih}|_{N-2}\|_{L^2(\Hcal^*_s)}
\\
\lesssim&\, C_1\vep s^{-1}\|(s/t)|\del v|_N\|_{L^2(\Hcal^*_s)} 
+ C_1\vep s^{-1+\delta}\|(s/t)|\del v|_{N-1}\|_{L^2(\Hcal^*_s)} 
\\
& + C_1\vep s^{-1}\|| v_{\ih}|_N|\|_{L^2(\Hcal^*_s)} 
+ C_1\vep s^{-1+\delta}\|| v_{\ih}|_{N-1}\|_{L^2(\Hcal^*_s)} 
\lesssim (C_1\vep)^2s^{-1+\delta}.
\endaligned
$$ 
Here we remark that the uniform energy bound \eqref{eq7e-29-10-2021} the loss of decay of $|v|_{N-2}$ and $|\del v|_{N-2}$.

For the bound on $\mathcal{J}^{\TPcal}$, we recall that $N\leq 4$, then
$$
\aligned
\mathcal{J}^{\TPcal} \lesssim &\, \|J\zeta^{-1}\omega^{\eta}|\del v|_N|v|_{N-4}\|_{L^2(\TPcal_s)} 
+ \|J\zeta^{-1}\omega^{\eta}|\del v|_{N-1}|v|_{N-3}\|_{L^2(\TPcal_s)} 
\\
&+ \|J\zeta^{-1}\omega^{\eta}|\del v|_{N-2}|v|_{N-2}\|_{L^2(\TPcal_s)} 
+ \|J\zeta^{-1}\omega^{\eta}| v|_{N-1}|\del v|_{N-3}\|_{L^2(\TPcal_s)} 
\\
&+ \|J\zeta^{-1}\omega^{\eta}|v|_N|\del v|_{N-4}\|_{L^2(\Hcal_s)} 
=: \mathcal{J}^{\TPcal}_1+\mathcal{J}^{\TPcal}_2+\mathcal{J}^{\TPcal}_3+\mathcal{J}^{\TPcal}_4+J^{\TPcal}_5.
\endaligned
$$
Each term will be bounded in different manner.
$$
\aligned
\mathcal{J}_1^{\TPcal}\lesssim&\,
C_1\vep s^{1+\delta}\|\la r\ra^{-3/2}\la r-t\ra^{1-\eta}\zeta\omega^{\eta}|\del v|_N\|_{L^2(\TPcal_s)} 
\lesssim\,C_1\vep s^{-2\eta+\delta}\|\zeta\omega^{\eta}|\del v|_N\|_{L^2(\TPcal_s)} 
\\
\lesssim&\, (C_1\vep)^2s^{-1+\delta},
\endaligned
$$
where \eqref{eq1-15-03-2022-M} case $p=N$ is applied.
$$
\aligned
\mathcal{J}_2^{\TPcal}\lesssim&\, s\||v|_{N-3}\zeta\omega^{\eta}|\del v|_{N-1}\|_{L^2(\TPcal_s)} 
\lesssim C_1\vep s^{1+\delta}\|\la r\ra^{-1}\la r-t\ra^{1/2-\eta}\zeta\omega^{\eta}|\del v|_{N-1}\|_{L^2(\TPcal_s)} 
\\
\lesssim&\, C_1\vep s^{-1+\delta}\Ebf^{N-1}_{\eta,c}(\tau,v)^{1/2}d\tau
\lesssim (C_1\vep)^2s^{-1+\delta},
\endaligned
$$
where for the last inequality we apply \eqref{eq7e-29-10-2021}, which offsets the loss of decay $s^{\delta}$.
For the third term, we apply the Sobolev decay on $|v|_{N-2}$ and observe that $|\del v|_{N-2}\lesssim |v|_{N-1}$. Then we obtain
$$
\aligned
\mathcal{J}_3^{\TPcal}\lesssim&\, 
C_1\vep s^{1+\delta} \|\la r\ra^{-3/2}\la r-t\ra^{1-\eta}\zeta \la r\ra\la r-t\ra^{-1} \omega^{\eta}|v|_{N-1}\|_{L^2(\TPcal_s)} 
\lesssim 
(C_1\vep)^2 s^{-2\eta+3\delta}
\\
\lesssim&\, (C_1\vep)^2s^{-1+\delta}.
\endaligned
$$
Here \eqref{eq5-15-03-2022-M} is applied. For the fourth term is bounded in the same manner. We omit the details.
For the last term, we recall that
$$
\aligned
\mathcal{J}_5^{\TPcal}\lesssim&\, \|J\zeta^{-1}\omega^{\eta}|v|_N|\del v|_{N-4}\|_{L^2(\Hcal_s)} 
\lesssim C_1\vep s^{1/2}\||v|_{N-3} J^{1/2}\omega^{\eta}|v|_N\|_{L^2(\TPcal_s)} 
\\
\lesssim&\, C_1\vep s^{1/2+\delta}\|\la r\ra^{-1}\la r-t\ra^{1-\eta}J^{1/2}\la r-t \ra^{\eta-1/2}|v|_N\|_{L^2(\TPcal_s)} 
\\
\lesssim&\, C_1\vep s^{-2\eta+1/2+\delta}\|J^{1/2}\la r-t\ra^{\eta-1/2}|v|_N\|_{L^2(\TPcal_s)} 
\lesssim C_1\vep s^{-1/2-2\delta}\|J^{1/2}\la r-t\ra^{\eta-1/2}|v|_N\|_{L^2(\TPcal_s)}.
\endaligned
$$
We thus conclude that
\begin{equation}\label{eq3-25-03-2022-M}
\mathcal{J}\lesssim (C_1\vep)^2s^{-1+\delta}
+ C_1\vep s^{-1/2-2\delta}\|J^{1/2}\la r-t\ra^{\eta-1/2}|v|_N\|_{L^2(\TPcal_s)}.
\end{equation}
Then we make the following calculation:
$$
\aligned
&\int_{s_0}^s \omega^{\eta}|\del_t Zw_{\ih}\, Z(v_{\ih}\del v_{\ih})|J\,dxd\tau
\lesssim \int_{s_0}^s\Ebf_{\eta,c}^N(\tau,v)^{1/2}(\mathcal{J}^{\Hcal} + \mathcal{J}^{\TPcal})d\tau
\\
\lesssim&\, (C_1\vep)^3\int_{s_0}^s \tau^{-1+2\delta} d\tau 
+ (C_1\vep)^2 \int_{s_0}^s\tau^{-1/2-\delta} \|J^{1/2}\la r-t\ra^{\eta-1/2}|v|_N\|_{L^2(\TPcal_\tau)}\,d\tau
\\
\lesssim&\, (C_1\vep)^3s^{2\delta} 
+ \Big(\int_{s_0}^s\tau^{-1-2\delta}\Big)^{1/2}\Big(\int_{s_0}^s\la r-t\ra^{2\eta-1}|v|_NJ\,dxd\tau\Big)^{1/2}.
\endaligned
$$
This concludes \eqref{eq2-25-03-2022-M}.
\end{proof}

\subsection{Estimates on Hessian of $w_{\ih}$}
This subsection is devoted to the following result:
\begin{proposition}
Suppose that the Sobolev decay estimates \eqref{eq4-31-10-2021}, \eqref{eq8-15-03-2022-M} hold. Then
\begin{equation}\label{eq1-25-03-2022-M}
|\del\del w|_{N-3}\lesssim C_1\vep (s/t)^{-1}s^{-2+2\delta},\quad \text{in } \Hcal^*_{[s_0,s_1]},
\end{equation}
\begin{equation}\label{eq2-04-04-2022-M}
\|s(s/t)^2|\del\del w|_{N-1}\|_{L^2(\Hcal^*_s)}\lesssim C_1\vep s^{2\delta}.
\end{equation}
\end{proposition}
\begin{proof}
This is by applying \eqref{eq1-28-10-2021}. We first remark that in $\Hcal^*_{[s_0,s_1]}$,
$$
|\Box w|_{N-3}\lesssim |v_{\ih}^2|_{N-3}\lesssim (C_1\vep)^2(s/t)^2s^{-2+2\delta}.
$$
This leads to
$$
|\del\del w|_{N-3}\lesssim (s/t)^{-2}t^{-1}|\del w|_{N-2} + (s/t)^{-2}|\Box w|_{N-3}\lesssim C_1\vep (s/t)^{-1}s^{-2+2\delta},
$$
where \eqref{eq3-22-03-2022-M} is applied. On the other hand, 
$$
|\Box w|_{N-1}\lesssim |v_{\ih}^2|_{N-1}\lesssim |v_{\ih}|_{N-1}|v_{\ih}|_{N-3}\lesssim (C_1\vep)^2 (s/t)s^{-1+\delta}|v|_{N-1}.
$$
Then
$$
\|s(s/t)^2|\del\del w|_{N-1}\|_{L^2(\Hcal_s)}
\lesssim \|s|\Box w|_{N-1}\|_{L^2(\Hcal^*_s)} + \|(s/t)|\del w|_N\|_{L^2(\Hcal^*_s)}
\lesssim C_1\vep s^{2\delta}.
$$
\end{proof}

\subsection{Sharp decay estimates on $w_{\ih}$}
This subsection is devoted to 
\begin{equation}\label{eq5-31-10-2021}
|\del\del w|_{N-4}\lesssim C_1\vep t^{-1/2}\la r-t\ra^{-1},\quad \text{in } \Hcal^*_{[s_0,s_1]}.
\end{equation}
This is proved by Proposition \ref{prop1-31-10-2021}. To do so we need the following preparation. 
\begin{lemma}
Suppose that Proposition \ref{prop1-21-03-2022-M} holds. Then
\begin{equation}\label{eq1-31-10-2012}
	\sup_{\Hcal^*_{s_0}\cup\del\Kcal_{[s_0,s_1]}}\{t^{1/2}|\del w|_{N-3}\}\lesssim (C_0+C_1)\vep.
\end{equation}
\end{lemma}
\begin{proof}
We apply \eqref{eq1-29-10-2021} together with the uniform bound \eqref{eq6-21-03-2022-M}, and obtain
\begin{equation}\label{eq1-05-04-2022-M}
|\del w|_{N-3}\lesssim C_1\vep \la r\ra^{-1/2}\la r-t\ra^{-\eta},\quad \text{in } \TPcal_{[s_0,s_1]}.
\end{equation}
This leads to 
$$
\sup_{\del\Kcal_{[s_0,s_1]}}\{t^{1/2}|\del w|_{N-3}\}\lesssim C_1\vep.	
$$
The bound on initial slice $\Hcal^*_{s_0}$ is direct by its compactness and the fact that $t$ is uniformly bounded on $\Hcal^*_{s_0}$. This leads us to the desired result.
\end{proof}

On the other hand, we remark that for $\ord(Z)\leq N-3$, by \eqref{eq3-31-10-2021}, 
\begin{equation}\label{eq2-31-10-2012}
	| \Box Z w_{\ih}| =  |v_{\ih}|_{N-3}^2\lesssim (C_1\vep)^2t^{-2}s^{2\delta},
\end{equation}
\begin{equation}\label{eq3-31-10-2012}
	|\delu_a\delu_a Z w|\lesssim t^{-1}|\delu w|_{N-2} \lesssim C_1\vep t^{-2}s^{\delta}.
\end{equation}
Substitute the above bounds into \eqref{eq1-31-10-2021}, we obtain, for $(t_1,x_1)\in \Hcal^*_{[s_0,s_1]}$
$$
\aligned
|t_1^{1/2}\del_t Zw(t_1,x_1)|
&\lesssim (C_0+C_1)\vep
+ C_1\vep\int_{s_0}^{t_1} (t-r)^{\delta}t^{-3/2 + \delta}\Big|_{\gamma_{t_1,x_1}(\tau)}e^{-\int_{\tau}^tP_{t_1,x_1}(\eta)d\eta}d\tau
\\
&\lesssim (C_0+C_1)\vep 
+ C_1\vep\int_{s_0}^{t_1} \tau^{-3/2 + 2\delta}d\tau\lesssim (C_0+C_1)\vep\lesssim C_1\vep.
\endaligned
$$
Thus in $\Hcal^*_{[s_0,s_1]}$ and for $\ord(Z)\leq N-3$,
\begin{equation}
	|\del_tZ w|\lesssim C_1\vep t^{-1/2}.
\end{equation}
Remark that $|\del_aZw| \leq |\del_t Zw| +|\delu_aZw|\lesssim C_1\vep t^{-1/2}$. Then by Proposition \ref{prop1-23-10-2021},

\begin{equation}\label{eq2-31-10-2021}
	|\del w|_{N-3}\lesssim C_1\vep t^{-1/2}.
\end{equation}

Now recalling \eqref{eq1-28-10-2021}  and 
$$
|\Box w|_{N-2}\lesssim (C_1\vep)^2t^{-2}s^{2\delta},
$$
we arrive at \eqref{eq5-31-10-2021}.


\section{Conclusion of the bootstrap}
\subsection{Objective} In this section we will establish the improved energy estimates \eqref{eq3-14-03-2022-M}. Recalling \eqref{eq8-21-03-2022-M}, \eqref{eq2-25-03-2022-M} and \eqref{eq2-24-03-2022-M}, the refined energy bounds are already established on $w_0, \del w_{\ih}$ and $w_{\ih}$. We still need to bound the Klein--Gordon components. In the present section we will establish the following result.
\begin{proposition}\label{prop1-25-03-2022-M}
Suppose that the bootstrap assumptions \eqref{eq7-29-10-2021} -- \eqref{eq8-14-03-2022-M} hold on $[s_0,s_1]$ with $C_1\vep$ sufficiently small (determined by the system and $s_0, \eta, \delta, N$), then
\begin{equation}\label{eq4-25-03-2022-M}
\Ebf_{\eta,c}^p(s,v) + \int_{s_0}^s\int_{\Fcal_{\tau}}\la r-t\ra^{2\eta-1}\big(|\delt v|_p\ + |v|_p\big)\,Jdxd\tau
\leq (C_0\vep)^2 +
\left\{
\aligned
&C(C_1\vep)^3s^{2\delta},\quad &&p=N,
\\
&C(C_1\vep)^3,\quad &&p\leq N-1.
\endaligned
\right.
\end{equation}
\end{proposition}

Once \eqref{eq4-25-03-2022-M} is established, we fix $C_1>\sqrt{2}C_0$, and make the following choice on  $\vep$:
\begin{equation}
\vep\leq \frac{C_1^2 - 2C_0^2}{CC_1^3},
\end{equation}
where $C$ is a constant determined by the system and $s_0, \delta, \eta$ and $N$. On the other hand, we demand $C_1\vep$ sufficiently small such that Proposition \ref{prop1-04-04-2022-M}, Proposition \ref{prop1-17-03-2022-M} and Lemma \ref{lem2-04-04-2022-M} hold.  Then \eqref{eq8-21-03-2022-M}, \eqref{eq2-25-03-2022-M}, \eqref{eq2-24-03-2022-M} and \eqref{eq4-25-03-2022-M} lead to \eqref{eq3-14-03-2022-M}. Then the bootstrap argument is concluded and the global existence is established. The rest of this section is devoted to the proof of Proposition \ref{prop1-25-03-2022-M}.

\subsection{Improved energy estimate: top order}
In this subsection we establish \eqref{eq4-25-03-2022-M} for $p=N$. To do so we will rely on Proposition \ref{prop2-24-11-2021}. More precisely, we need the following result:
\begin{lemma}\label{lem2-04-04-2022-M}
Suppose that the bootstrap assumptions \eqref{eq7-29-10-2021} -- \eqref{eq8-14-03-2022-M} hold on $[s_0,s_1]$ with $C_1\vep\leq 1$ sufficiently small. Then
\begin{equation}\label{eq6-03-04-2022-M}
\int_{s_0}^s\int_{\Fcal_\tau}|\del_t Z v_{\ih}\,Z(\Box v_{\ih} + c_{\ih}^2v_{\ih})|J\,dxd\tau \lesssim (C_1\vep)^3s^{2\delta}.
\end{equation}
\end{lemma}
\begin{proof}
We denote by
$$
I^{\Hcal} := \||\Box v_{\ih} + c_{\ih}^2v_{\ih}|_N\|_{L^2(\Hcal^*_s)}
,\quad
I^{\TPcal} := \|J\zeta^{-1}\omega^{\eta}|\Box v_{\ih} + c_{\ih}^2v_{\ih}|_N\|_{L^2(\TPcal_s)}.
$$
For the bound in $\Hcal^*{[s_0,s_1]}$, recalling \eqref{eq1-24-11-2021}, we need to bound the following terms:
\begin{equation}\label{eq6-25-03-2022-M}
\aligned
&\||v\del w_0|_N\|_{L^2(\Hcal^*_s)},\quad \||v\del\del w|_N\|_{L^2(\Hcal^*_s)},\quad
\||v\del v|_N\|_{L^2(\Hcal^*_s)},\quad
\||v_{\ih}^2|_N\|_{L^2(\Hcal^*_s)},
\\
&\||T_{KG}[u,v] + S_{KG}[u,v]|_N\|_{L^2(\Hcal^*_s)}.
\endaligned
\end{equation}
We will show that the above terms are bounded by $(C_1\vep)^2 s^{-1+\delta}$ (modulo a constant determined by the system and $s_0,\delta,\eta$ and $N$). For the pure Klein-Gordon quadratic terms, we recall that $N\leq 4$, then
$$
\aligned
\||v\del v|_N\|_{L^2(\Hcal^*_s)}
&\lesssim \||\del v|_{N-3}|v|_N\|_{L^2(\Hcal^*_s)} + \||\del v|_{N-2}|v|_{N-2}\|_{L^2(\Hcal^*_s)}
 + \||v|_{N-3}|\del v|_N\|_{L^2(\Hcal^*_s)}
\\
&\lesssim C_1\vep s^{-1}\big(\||v|_N\|_{L^2(\Hcal^*_s)} + (s/t)|\del v|_N\|_{L^2(\Hcal^*_s)}\big) 
 + C_1\vep s^{-1+\delta}\||v|_{N-1}\|_{L^2(\Hcal^*_s)}
\\
&\lesssim (C_1\vep)^2s^{-1+\delta}.
\endaligned
$$
Here the uniform energy bound on $|v|_{N-1}, |\del v|_{N-1}$ offsets the loss of decay from $|v|_{N-2}\lesssim C_1\vep s^{\delta}t^{-1}$ and $|\del v|_{N-2}\lesssim C_1\vep s^{-1+\delta}$. The term $|v_{\ih}^2|_N$ is bounded in the same manner. The term $|v\del w_0|$ is bounded similarly:
$$
\aligned
\||v\del w_0|_N\|_{L^2(\Hcal^*_s)}\lesssim&\, \||v|_{N-3}|\del w_0|_N\|_{L^2(\Hcal^*_s)}
+ \||v|_{N-2}|\del w_0|_{N-2}\|_{L^2(\Hcal^*_s)}
+ \||v|_{N}|\del w_0|_{N-3}\|_{L^2(\Hcal^*_s)} 
\\
\lesssim&\, (C_1\vep)^2s^{-1+\delta}.
\endaligned
$$
Here we have applied \eqref{eq3-31-10-2021} (case $p=N-1$) and \eqref{eq7a-29-10-2021} on the first term;  \eqref{eq3-31-10-2021} and \eqref{eq8-21-03-2022-M} (case $p=N-1$) on the second term; \eqref{eq7c-29-10-2021} and \eqref{eq2-26-03-2022-M} on the third term; \eqref{eq7e-29-10-2021} and \eqref{eq3-22-03-2022-M} on the last term. 
The term $|v\del\del w|_N$ is bounded as following. Recall that $N\leq 4$,
$$
\aligned
&\||v\del\del w|_N\|_{L^2(\Hcal^*_s)}
\\
\lesssim&\, \||\del\del w|_N|v|_{N-3}\|_{L^2(\Hcal^*_s)} 
+ \||\del\del w|_{N-3}|v|_{N-1}\|_{L^2(\Hcal^*_s)}
+ \||\del\del w|_{N-4}|v|_N\|_{L^2(\Hcal^*_s)}
\\
\lesssim&\, C_1\vep s^{-1}\|(s/t)|\del\del w|_N\|_{L^2(\Hcal^*_s)}  
+ C_1\vep s^{-1+\delta}\||v|_{N-1}\|_{L^2(\Hcal^*_s)}
+ C_1\vep \|t^{-1/2}\la r-t\ra^{-1}|v|_N\|_{L^2(\Hcal^*_s)}
\\
\lesssim&\, (C_1\vep)^2 s^{-1+\delta}.
\endaligned
$$
where on the last term we have applied \eqref{eq5-31-10-2021} and on the rest of the terms the bootstrap energy bounds and the Sobolev bounds are sufficient. 
Finally, we return to the high-order terms $T_{KG}$ and $S_{KG}$. This relies on \eqref{eq2-22-03-2022-M}.
Based on this bound and \eqref{eq2-22-03-2022-M}, one has
$$
\aligned
\||T_{KG}[u,v] + S_{KG}[u,v]|_N\|_{L^2(\Hcal^*_s)}
\lesssim&\, (C_1\vep)^3s^{-2+3\delta}\big(\Ebf^{\Hcal,N}(s,u)^{1/2} + \Ebf_c^{\Hcal,N}(s,v)^{1/2}\big) 
\\
\lesssim&\,(C_1\vep)^4s^{-2+4\delta}.
\endaligned
$$
Concluding the above bounds, we thus have
\begin{equation}\label{eq4-03-04-2022-M}
I^{\Hcal}\lesssim (C_1\vep)^2s^{-1+\delta}.
\end{equation}

Then we turn to the estimate on $I^{\TPcal}$. We need to bound the quantities in the following list:
\begin{equation}\label{eq5-26-03-2022-M}
\aligned
&\|\la r-t\ra^{\eta}J\zeta^{-1}|v\del w_0|_N\|_{L^2(\TPcal_s)},\quad 
\|\la r-t\ra^{\eta}J\zeta^{-1}|v\del\del w|_N\|_{L^2(\TPcal_s)},
\\
&\|\la r-t\ra^{\eta}J\zeta^{-1}|v\del v|_N\|_{L^2(\TPcal_s)},\quad
\|\la r-t\ra^{\eta}J\zeta^{-1}|v_{\ih}^2|_N\|_{L^2(\TPcal_s)},
\\
&\|\la r-t\ra^{\eta}J\zeta^{-1}|v(T_{KG}[u,v] + S_{KG}[u,v])|_N\|_{L^2(\TPcal_s)}.
\endaligned
\end{equation}
We begin with the pure Klein-Gordon quadratics. 
$$
\aligned
&\|\la r-t\ra^{\eta}J\zeta^{-1}|v\del v|_N\|_{L^2(\TPcal_s)}
\\
\lesssim&\, s\|\la r-t\ra^{\eta}\zeta|v|_{N-4}|\del v|_N\|_{L^2(\TPcal_s)}
+ s\|\zeta\la r-t\ra^{\eta}|v|_{N-3}|\del v|_{N-1}\|_{L^2(\TPcal_s)}
\\
&+ s^{1/2}\|\la r-t\ra^{\eta}J^{1/2} |v|_{N-2}|\del v|_{N-2}\|_{L^2(\TPcal_s)}
+ s\|\la r-t\ra^{\eta}\zeta |v|_{N-1}|\del v|_{N-3}\|_{L^2(\TPcal_s)}
\\
&+ s^{1/2}\|\la r-t\ra^{\eta}J^{1/2}|v|_N|\del v|_{N-4}\|_{L^2(\TPcal_s)}
\\
\lesssim&\, C_1\vep s^{1+\delta}\|\la r-t\ra^{-3/2}\la r-t\ra^{1-\eta}\zeta \la r-t\ra^{\eta}|\del v|_N\|_{L^2(\TPcal_s)}
\\
&+ C_1\vep s^{1+\delta} 
\|\la r\ra^{-1}\la r-t\ra^{1/2-\eta}\zeta\la r-t\ra^{\eta}|\del v|_{N-1}\|_{L^2(\TPcal_s)}
\\
&+ C_1\vep s^{1+\delta} \|\la r\ra^{-1/2}\la r-t\ra^{-\eta}\zeta\la r-t\ra^{\eta}|v|_{N-1}\|_{L^2(\TPcal_s)}
\\
&+ C_1\vep s^{1+\delta} \|\la r\ra^{-1/2}\la r-t\ra^{-\eta}\zeta\la r-t\ra^{\eta}|v|_{N-1}\|_{L^2(\TPcal_s)}
\\
&+  C_1\vep s^{1/2+\delta} \|\la r\ra^{-1}\la r-t\ra^{1/2-\eta} J^{1/2}\la r-t\ra^{\eta}|v|_N\|_{L^2(\TPcal_s)}
\\
\lesssim&\, C_1\vep s^{-2\eta+\delta}\Ebf_{\eta,c}^N(s,v)^{1/2} 
+ C_1\vep s^{-1+\delta}\Ebf_{\eta,c}^{N-1}(s,v)^{1/2} 
\\
&\,+ C_1\vep s^{1+\delta}\|\la r\ra^{-3/2}\la r-t\ra^{1-\eta}\zeta\la r\ra\la r-t\ra^{-1+\eta}|v|_{N-1}\|_{L^2(\TPcal_s)}
\\
&\,+ C_1\vep s^{1/2+\delta}\|\la r\ra^{-1}\la r-t\ra^{1/2-\eta} J^{1/2}\la r-t\ra^{\eta-1/2}|v|_N\|_{L^2(\TPcal_s)}
\endaligned
$$
This leads to
$$
\aligned
\|\la r-t\ra^{\eta}J\zeta^{-1}|v\del v|_N\|_{L^2(\TPcal_s)}
\lesssim&\, (C_1\vep)^2 s^{-1 + \delta}
 + C_1\vep s^{-1/2}\|J^{1/2}\la r-t\ra^{\eta-1/2}|v|_N\|_{L^2(\TPcal_s)},
\endaligned
$$
where \eqref{eq5-15-03-2022-M} and the fact that $\delta\ll \eta-1/2<1/2$ are applied. The term $|v_{\ih}^2|_N$ is bounded in the same manner because $|v|$ always enjoys better estimates than $|\del v|$. 
For the term $|v\del w_0|$, we remark that 
$$
\aligned
&\|\la r-t\ra^{\eta}J\zeta^{-1}|v\del w_0|_N\|_{L^2(\TPcal_s)}
\\
\lesssim & s^{1/2}\||\del w_0|_{N-4}\la r-t\ra^{\eta}J^{1/2}|v|_N\|_{L^2(\TPcal_s)} 
+ s\||\del w_0|_{N-3}\zeta\la r-t\ra^{\eta}|v|_{N-1}\|_{L^2(\TPcal_s)}
\\
& + s\||v|_{N-2}\zeta\la r-t\ra^{\eta}|\del w_0|_{N-2}\|_{L^2(\TPcal_s)}
+ s\||v|_{N-3}\zeta\la r-t\ra^{\eta}|\del w_0|_{N-1}\|_{L^2(\TPcal_s)}
\\
&+ s\||v|_{N-4}\zeta\la r-t\ra^{\eta}|\del w_0|_N\|_{L^2(\TPcal_s)}
\\
\lesssim& C_1\vep s^{1/2}\|\la r\ra^{-1/2}J^{1/2}\la r-t\ra^{\eta-1/2}|v|_N\|_{L^2(\TPcal_s)}
\\
&+ C_1\vep s^{1+\delta}\|\la r\ra^{-1/2}\la r-t\ra^{-\eta}\zeta\la r-t\ra^{\eta}|v|_{N-1}\|_{L^2(\TPcal_s)}
\\
&+ C_1\vep s^{1+\delta}\|\la r-t\ra^{-1/2-\eta}\zeta\la r-t\ra^{\eta}|v|_{N-2}\|_{L^2(\TPcal_s)}
\\
&+ C_1\vep s^{1+\delta}\|\la r\ra^{-1}\la r-t\ra^{1/2-\eta}\zeta\la r-t\ra^{\eta}|\del w_0|_{N-1}\|_{L^2(\TPcal_s)}
\\
&+ C_1\vep s^{1+\delta}\|\la r\ra^{-3/2}\la r-t\ra^{1-\eta}\zeta\la r-t\ra^{\eta}|\del w_0|_N\|_{L^2(\TPcal_s)}
\\
\lesssim&\, C_1\vep s^{-1/2}\|J^{1/2}\la r-t\ra^{\eta-1/2}|v|_N\|_{L^2(\TPcal_s)}
\\
&\,+ C_1\vep s^{1+\delta} 
\|\la r\ra^{-3/2}\la r-t\ra^{1-\eta}\zeta\la r\ra\la r-t\ra^{-1+\eta}|v|_{N-1}\|_{L^2(\TPcal_s)}
\\
&\, + C_1\vep s^{1+\delta} 
\|\la r\ra^{-1}\la r-t\ra^{1/2-\eta}\zeta\la r\ra\la r-t\ra^{-1+\eta}|v|_{N-2}\|_{L^2(\TPcal_s)}
\\
&\, +C_1\vep s^{-1+\delta}\Ebf_{\eta}^{\TPcal,N-1}(s,w_0)^{1/2} 
+ C_1\vep s^{-2\eta+\delta} \Ebf_{\eta}^{\TPcal,N}(s,w_0)^{1/2}.
\endaligned
$$
Then we apply \eqref{eq5-15-03-2022-M} and \eqref{eq2-02-04-2022-M} together with the bootstrap energy bounds, and obtain
\begin{equation}\label{eq1-03-04-2022-M}
\aligned
\|\la r-t\ra^{\eta}J\zeta^{-1}|v\del w_0|_N\|_{L^2(\TPcal_s)}
\lesssim&  C_1\vep s^{-1/2}\|J^{1/2}\la r-t\ra^{\eta-1/2}|v|_N\|_{L^2(\TPcal_s)}
 +(C_1\vep)^2 s^{-1+\delta} .
\endaligned
\end{equation}
Then we turn to the term $|v\del\del w|$. Thanks to \eqref{eq1-05-04-2022-M},
$$
\aligned
|v\del\del w|_N\lesssim& 
|\del\del w|_{N-4}|v|_N + |\del\del w|_{N-3}|v|_{N-1} 
\\
&+ |\del\del w|_{N-2}|v|_{N-2} + |\del\del w|_{N-1}|v|_{N-3} 
+ |\del\del w|_N|v|_{N-4}
\\
\lesssim& |\del w|_{N-3}|v|_N + C_1\vep s^{\delta}\la r\ra^{-1/2}\la r-t\ra^{-\eta}|v|_{N-1} 
\\
&+ C_1\vep s^{\delta}\la r-t\ra^{-1/2-\eta}|v|_{N-2} 
+ C_1\vep s^{\delta}\la r\ra^{-1}\la r-t\ra^{1/2-\eta}|\del w|_N
\\
& +C_1\vep s^{\delta}\la r\ra^{-3/2}\la r-t\ra^{1-\eta}|\del\del w|_N 
\\
\lesssim& C_1\vep \la r\ra^{-1/2}\la r-t\ra^{-\eta}|v|_N + C_1\vep s^{\delta}\la r\ra^{-1/2}\la r-t\ra^{-\eta}|v|_{N-1} 
\\
&+ C_1\vep s^{\delta}\la r-t\ra^{-1/2-\eta}|v|_{N-2} 
+ C_1\vep s^{\delta}\la r\ra^{-1}\la r-t\ra^{1/2-\eta}|\del w|_N
\\
& +C_1\vep s^{\delta}\la r\ra^{-3/2}\la r-t\ra^{1-\eta}|\del\del w|_N.
\endaligned
$$
Then
$$
\aligned
\|\la r-t\ra^{\eta}J\zeta^{-1}|v\del\del w|_N\|_{L^2(\TPcal_s)}
&\lesssim C_1\vep s^{1/2}\|\la r\ra^{-1/2}\la r-t\ra^{1/2-\eta}J^{1/2}\la r-t\ra^{\eta-1/2}|v|_N\|_{L^2(\TPcal_s)}
\\
&\quad + C_1\vep s^{1+\delta}
\|\la r\ra^{-3/2}\la r-t\ra^{1-\eta}\zeta\la r\ra\la r-t\ra^{\eta-1}|v|_{N-1}\|_{L^2(\TPcal_s)}
\\
&\quad + C_1\vep s^{1+\delta}
\|\la r\ra^{-1}\la r-t\ra^{1/2-\eta}\zeta\la r\ra\la r-t\ra^{\eta-1}|v|_{N-2}\|_{L^2(\TPcal_s)}
\\
&\quad + C_1\vep s^{1+\delta}
\|\la r\ra^{-1}\la r-t\ra^{1/2-\eta}\zeta\la r-t\ra^{\eta}|\del w|_N\|_{L^2(\TPcal_s)}
\\
&\quad + C_1\vep s^{1+\delta}
\|\la r\ra^{-3/2}\la r-t\ra^{1-\eta}\zeta\la r-t\ra^{\eta}|\del\del w|_N\|_{L^2(\TPcal_s)}.
\endaligned
$$
We apply \eqref{eq5-15-03-2022-M}, \eqref{eq2-02-04-2022-M}, \eqref{eq6-21-03-2022-M} and \eqref{eq7a-29-10-2021}.
This leads to 
\begin{equation}\label{eq5-03-04-2022-M}
\aligned
\|\la r-t\ra^{\eta}J\zeta^{-1}|v\del\del w|_N\|_{L^2(\TPcal_s)}
\lesssim& C_1\vep s^{-1/2}\|J^{1/2}\la r-t\ra^{\eta-1/2}|v|_N\|_{L^2(\TPcal_s)}
\\
& + (C_1\vep)^2s^{-1+\delta}.
\endaligned
\end{equation}

For the higher-order terms ($T_{KG}, S_{KG}$), we apply \eqref{eq2-21-03-2022-M}. 
Then we conclude that
\begin{equation}\label{eq3-03-04-2022-M}
\aligned
I^{\TPcal}\lesssim& (C_1\vep)^2s^{-1+\delta} + 
C_1\vep s^{-1/2}\|J^{1/2}\la r-t\ra^{\eta-1/2}(|\delt u|_N + |\delt v|_N + |v|_N)\|_{L^2(\TPcal_s)}.
\endaligned
\end{equation}
Now we make the following calculation:
$$
\aligned
&\int_{s_0}^s\int_{\TPcal_\tau}|\del_t Z v_{\ih}\,Z(\Box v_{\ih} + c_{\ih}^2v_{\ih})|J\,dxd\tau
\leq\int_{s_0}^s\Ebf_{\eta,c}(\tau,v)^{1/2}I^{\TPcal}(\tau)d\tau
\\
\lesssim&\, (C_1\vep)^3s^{2\delta} 
+ C_1\vep \int_{s_0}^s\tau^{-1/2+\delta}\|J^{1/2}\la r-t\ra^{\eta-1/2}(|\delt u|_N + |\delt v|_N + |v|_N)\|_{L^2(\TPcal_\tau)}d\tau
\\
\lesssim&\, (C_1\vep)^3s^{2\delta} + C_1\vep \Big(\int_{s_0}^s\tau^{-1+2\delta}d\tau\Big)^{1/2}
\Big(\int_{s_0}^s\int_{\TPcal_s}\la r-t\ra^{2\eta-1}(|\delt u|_N + |\delt v|_N + |v|_N)J\,dxd\tau\Big)^{1/2}
\\
\lesssim&\, (C_1\vep)^3s^{2\delta}.
\endaligned
$$
$$
\aligned
\int_{s_0}^s\int_{\Hcal^*_\tau}|\del_t Z v_{\ih}\,Z(\Box v_{\ih} + c_{\ih}^2v_{\ih})|J\,dxd\tau
\lesssim&\, \int_{s_0}^s\|(s/t)\del v\|_{L^2(\Hcal^*_{\tau})}I^{\Hcal}d\tau 
= (C_1\vep)^3\int_{s_0}^s\tau^{1+2\delta}d\tau
\\
\lesssim& (C_1\vep)^3s^{1+2\delta}.
\endaligned
$$
These estimates concludes \eqref{eq6-03-04-2022-M}
\end{proof}
Then from Proposition \ref{prop2-24-11-2021}, we conclude \eqref{eq4-25-03-2022-M} with $p=N$.

\subsection{Improved energy estimates: lower order}
The improved energy bound \eqref{eq4b-14-03-2022-M} is a uniform bound, which is much more difficult. Our strategy is to apply the normal form transform introduced in Section \ref{sec1-03-04-2022-M}. 
Before the proof of \eqref{eq4-25-03-2022-M} case $p\leq N-1$, we need the following estimates.
\begin{lemma}
Let $Z$ be a high-order operator with $\ord(Z)\leq N-1$. Then
\begin{equation}\label{eq5-04-04-2022-M}
\int_{s_0}^s\int_{\Fcal_s}\omega^{2\eta}\big|\del_t Zv_{\ih}\,Z(v_{\ih}\del_{\alpha}(w_0 + A^{\beta\jh}\del_{\beta}w_{\jh}))\big|J\,dxd\tau
\lesssim (C_1\vep)^3.
\end{equation}
\end{lemma}
\begin{proof}
We firstly regard the hyperbolic domain. Remark that
$$
\aligned
\|(s/t)|v_{\ih}\del_{\beta}w_0|_{N-1}\|_{L^2(\Hcal^*_s)}
&\lesssim \||v|_{N-1} (s/t)|\del w_0|_{N-3}\|_{L^2(\Hcal^*_s)} + \||v|_{N-3}(s/t)|\del w_0|_{N-1}\|_{L^2(\Hcal^*_s)}
\\
&\lesssim C_1\vep s^{-2+4\delta}\|(s/t)|v|_{N-1}\|_{L^2(\Hcal^*_s)}
 + C_1\vep s^{-2+\delta}\|(s/t)^2s|\del w_0|_{N-1}\|_{L^2(\Hcal^*_s)}
\\
&\lesssim (C_1\vep)^2s^{-1-2\delta},
\endaligned
$$
where we have applied \eqref{eq1-26-03-2022-M}, \eqref{eq1-04-04-2022-M} and the fact that $\delta\ll 1$. On the other hand,
$$
\aligned
\|(s/t)|v_{\ih}\del_{\alpha}\del_{\beta}w_{\jh}|_{N-1}\|_{L^2(\Hcal^*_s)}
&\lesssim \|(s/t)|v|_{N-1}|\del\del w|_{N-3}\|_{L^2(\Hcal^*_s)} 
+ \|(s/t)|v|_{N-3}|\del\del w|_{N-1}\|_{L^2(\Hcal^*_s)} 
\\
&\lesssim C_1\vep s^{-2+2\delta}\||v|_{N-1}\|_{L^2(\Hcal^*_s)} 
+ C_1\vep s^{-2+\delta}\|(s/t)^2s|\del\del w|_{N-1}\|_{L^2(\Hcal^*_s)}
\\
&\lesssim (C_1\vep)^2s^{-2+3\delta} \lesssim (C_1\vep)^2s^{-1-2\delta}.
\endaligned
$$
where \eqref{eq1-25-03-2022-M} and \eqref{eq2-04-04-2022-M} are applied.

Then we work in the transition-flat region.
$$
\aligned
&\|\omega^{\eta}J|v_{\ih}\del_{\beta}w_0|_{N-1}\|_{L^2(\TPcal_s)}
\\
&\lesssim \|J\omega^{\eta}|v|_{N-1}|\del w_0|_{N-3}\|_{L^2(\TPcal_s)}
+ \|J\omega^{\eta}|v|_{N-3}|\del w_0|_{N-2}\|_{L^2(\TPcal_s)}
+ \|J\omega^{\eta}|v|_{N-4}|\del w_0|_{N-1}\|_{L^2(\TPcal_s)}
\\
&\lesssim C_1\vep s^{1+\delta}
\|\la r\ra^{-3/2}\la r-t\ra^{1-\eta}\zeta\la r\ra\la r-t\ra^{-1}\omega^{\eta}|v|_{N-1}\|_{L^2(\TPcal_s)}
\\
&\quad + C_1\vep s^{1+\delta}\|\la r\ra^{-3/2}\la r-t\ra^{1-\eta}\zeta \la r\ra^{3/2}\la r-t\ra^{-3/2}\omega^{\eta}|v|_{N-3}\|_{L^2(\TPcal_s)}
\\
&\quad + C_1\vep s^{1+\delta}
\|\la r\ra^{-3/2}\la r-t\ra^{1-\eta}\zeta\omega^{\eta}|\del w_0|_{N-1}\|_{L^2(\TPcal_s)}
\\
&\lesssim\, (C_1\vep)^2 s^{-2\eta+4\delta} \lesssim (C_1\vep)^2s^{-1-2\delta},
\endaligned
$$
where \eqref{eq5-15-03-2022-M} and \eqref{eq4-04-04-2022-M} are applied.

The term $\|\omega^{\eta}J|v_{\ih}\del_{\alpha}\del_{\beta} w_{\jh}|_{N-1}\|_{\TPcal_s}$ enjoys the same bound because $|\del\del w_{\ih}|$ satisfies all the Soblev pointwise bounds and bootstrap energy bounds enjoyed by $|\del w_0|$. 

As a conclusion, we have
$$
\|J|v_{\ih}\del_{\alpha}(w_0 + A^{\beta\jh}\del_{\beta}w_{\jh})|_{N-1}\|_{L^2(\Fcal_s)}\lesssim (C_1\vep)^2s^{-1-2\delta}.
$$
Then direct calculation leads us to \eqref{eq5-04-04-2022-M}.
\end{proof}

According to the notation in Subsection \ref{subsec1-04-04-2022-M}, the above estimate \eqref{eq5-04-04-2022-M} together with \eqref{eq2-21-03-2022-M} (case $p = N-1$) leads to
\begin{equation}\label{eq7-04-04-2022-M}
\int_{s_0}^s\int_{\Fcal_s}\omega^{2\eta}|\del v|_{N-1}|f|_{N-1}J\,dxd\tau\lesssim (C_1\vep)^3.
\end{equation}
For the second term on the right-hand side of \eqref{eq2-08-02-2022}, we remark that due to the decreasing factor $t^{-1}$, it enjoys integrable $L^2$ bounds. We only write the most critical term:
$$
\aligned
&\sum_{p_1+p_2=N-1}\!\!\!\!\|t^{-1}\omega^{2\eta}J\zeta^{-1}|\del v|_{p_1+1}|\del v|_{p_2+1}\|_{L^2(\TPcal_s)}
\\
&\lesssim C_1\vep s \|t^{-1}\omega^{\eta}\zeta|\del v|_N|\del v|_{N-3}\|_{L^2(\TPcal_s)}
+ C_1\vep s^{1/2} \|t^{-1}\omega^{\eta}J^{1/2}\zeta |\del v|_{N-1}|\del v|_{N-2}\|_{L^2(\TPcal_s)}
\\
&\lesssim C_1\vep s^{1+\delta} \|t^{-1}\la r\ra^{-1/2}\la r-t\ra^{-\eta}\zeta \omega^{\eta}|\del v|_N\|_{L^2(\TPcal_s)}
\\
&\quad + C_1\vep s^{1+\delta}\|t^{-1}\la r-t\ra^{-1/2-\eta} \zeta \omega^{\eta}|v|_{N-1}\|_{L^2(\TPcal_s)}
\\
&\lesssim C_1\vep s^{-2+\delta} \Ebf_{\eta,c}^{\TPcal,N-1}(s,v)^{1/2}
+ C_1\vep s^{1+\delta}\|t^{-1}\la r\ra^{-1}\la r-t\ra^{1/2-\eta}\zeta \la r\ra\la r-t\ra^{-1+\eta}|v|_{N-1}\|_{L^2(\TPcal_s)}
\\
&\lesssim (C_1\vep)^2 s^{-2+2\delta}\lesssim (C_1\vep)^2 s^{-1-2\delta},
\endaligned
$$ 
where \eqref{eq5-15-03-2022-M} is applied. Thus
$$
\aligned
&\sum_{p_1+p_2=N-1}\int_{s_0}^s\int_{\Fcal_s} 
\omega^{2\eta}|\del v|_{N-1}(t^{-1}|\del v|_{p_1+1}|\del v|_{p_2+1})J\,dxd\tau
\\
&\lesssim \sum_{p_1+p_2=N-1}\int_{s_0}^s\|\omega^{\eta}\zeta|\del v|_{N-1}\|_{L^2(\TPcal_s)}
\|t^{-1}J\zeta^{-1}|\del v|_{p_1+1}|\del v|_{p_2+1}\|_{L^2(\TPcal_s)}d\tau
\\
&\lesssim (C_1\vep)^2 \int_{s_0}^s \tau^{-1-2\delta}\Ebf_{\eta,c}^{\TPcal,N-1}(s,v)^{1/2}\,d\tau
\lesssim (C_1\vep)^3.
\endaligned
$$
The third term only concerns the estimate in the far-light-cone region, where $\la r-t\ra\simeq \la r\ra$. We remark that
$$
\aligned
&\|\omega^{\eta}J\zeta^{-1}(|\del v|_N + |v|_N)(|\del v|_{N-2} + |v|_{N-2})\|_{L^2(\TPcal^{\far}_s)}
\\
&\lesssim C_1\vep s^{1+\delta}\|\la r\ra^{-1/2-\eta}\zeta\omega^{\eta}(|\del v|_N+|v|_N)\|_{L^2(\TPcal^{\far}_s)}
\lesssim (C_1\vep)^2s^{-2\eta + 2\delta}\lesssim (C_1\vep)^2s^{-1-2\delta}.
\endaligned
$$ 
The last term are high-order ones. Remark that among these terms each one contains at least two Klein-Gordon factor. This structure leads to an integrable $L^2$ bound ( {\tt To be checked}). So we conclude by \eqref{eq4-25-03-2022-M} case $p\leq N-1$.

Finally, thanks to Proposition \ref{prop2-24-11-2021}, we obtain \eqref{eq4-25-03-2022-M} case $p\leq N-1$.
\appendix

\section{Properties of the weight functions}\label{app1-05-10-2021}
In this subsection we establish the estimates on $\zeta$ defined in \eqref{eq5-23-01-2022}. 
\begin{lemma}
In $\TPcal_{[s_0,\infty)}$, the following estimates hold:
\begin{equation}\label{eq12-07-10-2021}
	\sqrt{1-\xi(s,r)}\leq \zeta(s,x),\quad \frac{1}{s}\lesssim \frac{s}{\sqrt{s^2+r^2}}\leq \zeta(s,x).
\end{equation}
\begin{equation}\label{eq6-23-01-2022}
r^{-1} \leq \frac{|r-t|+1}{r}\lesssim \zeta^2\leq \zeta\leq 1.
\end{equation}
\end{lemma}
\begin{proof}
We only need to verify these bounds in $\Tcal_{[s_0,\infty)}$ because in $\Pcal_{[s_0,\infty)}$, $\zeta\equiv 1$.  \eqref{eq12-07-10-2021} is due to the fact that
$$
\zeta^2(s,r) = \frac{s^2}{s^2+r^2} + \frac{(1-\xi^2(s,r))r^2}{s^2+r^2}
$$ 
and \eqref{eq9-07-10-2021}, i.e., on $\Tcal_s$ one has $s^2\cong t\cong r$.  The bound \eqref{eq6-23-01-2022} is similar because in $\Tcal_s$, $|r-t|\leq 1$ due to \eqref{eq6-07-10-2021}.
	
\end{proof}

\begin{lemma}\label{lem2-23-01-2022}
	In $\TPcal_{[s_0,\infty)}$, the following estimate holds:
	\begin{equation}\label{eq3-23-01-2022}
		\Big|\delb_r^i\Big(\big(\frac{r-t+3}{r}\big)^{1/2}\Big)\Big| \lesssim \big(\frac{r-t+3}{r}\big)^{1/2}\lesssim \zeta,\quad 0\leq i\leq 2.
	\end{equation}
\end{lemma}

\begin{proof}[Proof of Lemma \ref{lem2-23-01-2022}]
	When $i=0$ this is \eqref{eq6-23-01-2022}.
	
	We firstly make the following calculation with parameter $(s,x^a)$:
	$$
	\aligned
	\delb_r\Big(\frac{r-t+3}{r}\Big) 
	=&\frac{t-3}{r^2} - \frac{\xi(s,r)}{\sqrt{s^2+r^2}}
	= 1-\frac{\xi(s,r)}{\sqrt{s^2+r^2}} + \frac{t-r+3}{r^2} - 6r^{-2}
	\\
	=&\frac{\zeta^2}{1+\frac{\xi(s,r)^2}{s^2+r^2}} +  \frac{t-r+3}{r^2} - 6r^{-2}.
	\\
	\delb_r\delb_r\Big(\frac{r-t+3}{r}\Big) =& \delb_r\Big(\frac{t-3}{r^2}\Big) - \delb_r\Big(\frac{\xi(s,r)}{\sqrt{s^2+r^2}}\Big)
	\\
	=&-\frac{2(t-3)}{r^3} + \frac{\xi(s,r)}{r\sqrt{s^2+r^2}} + \frac{\xi(s,r)r}{(s^2+r^2)^{3/2}} -  \frac{\del_r\xi(s,r)}{\sqrt{s^2+r^2}} - r^{-2}.
	\endaligned
	$$	
	Then, let $f(s,r) = \frac{r-t(s,r)+3}{r}$, one has
	$$
	|\delb_rf|\lesssim \zeta^2 + r^{-2},\quad |\delb_r\delb_r f|\lesssim r^{-2} + r^{-1}|\del_r\xi(s,r)|,\qquad \text{in }\Tcal_{[s_0,\infty)}. 
	$$
	Remark that by \eqref{eq6-23-01-2022},
	$$
	|\delb_r f^{1/2}| = |f^{-1/2}\delb_rf|\lesssim \Big(\frac{r-t+3}{r}\Big)^{-1/2}\big(\zeta^2 + r^{-2}\big)
	\lesssim \zeta.
	$$
	 On the other hand
	$$
	|\delb_r\delb_rf^{1/2}|\lesssim |f^{-3/2}||\delb_r f|^2 + |f^{-1/2}||\delb_r\delb_r f|
	\lesssim \zeta + \zeta^{-1} (r^{-2} + r^{-1}|\delb_r\xi(s,r)|)\lesssim \zeta.
	$$
\end{proof}

There is an important result of \eqref{eq12-07-10-2021} which is frequently used (see also \cite{LM-2022}).
\begin{lemma}\label{lem3-07-10-2021}
	In $\TPcal_{[s_0,\infty)}$ with $s_0\geq 2$, 
	\begin{equation}\label{eq13-07-10-2021}
		J\zeta^{-1}\lesssim s\zeta
	\end{equation}
\end{lemma}
\begin{proof}
	From \eqref{eq12-07-10-2021} one has
	$$
	s\zeta^2\geq \frac{1}{2}(1-\xi^2(s,r))s,\quad \xi(s,r)\geq \frac{Cs^2}{s^2+r^2}\geq \frac{C\xi(s,r)s^2}{s^2+r^2}.
	$$
	Then recall Lemma \ref{lem2-07-10-2021}, one has
	$$
	Cs\zeta^2\geq J
	$$
	where $C$ is a universal constant. This leads to the desired result.
\end{proof}

\section{Proofs in Section \ref{sec1-22-10-2021}}\label{sec2-22-10-2021}

\subsection{Basic ordering property}
The following decomposition is established in \cite{LM-2022}.
\begin{proposition}\label{prop1-23-10-2021}
Let $Z$ be a high-order operator composed by $\{\del_{\alpha},L_a,\Omega\}$. Then in any region of $\RR^{1+2}$, 
\begin{equation}
Z u = \sum_{|I| = p-k\atop |J|+l=k}\Gamma(Z)_{IJl}\del^IL^J\Omega^l u
\end{equation}
with $\Gamma(Z)_{IJl}$ constants determined by $Z$ and $I,J,l$.
\end{proposition}

For the proof we need the following commutations relations.
\begin{equation}\label{eq1-22-10-2021}
[L^J,\del^I]  = \sum_{|I'|=|I|\atop|J'|<|J|}\Gamma^{JI}_{I'J'}\del^{I'}L^{J'},
\end{equation}
\begin{equation}\label{eq2-22-10-2021}
[\Omega^l,\del^I] = \sum_{|I'|=|I|\atop l'<l}\Delta^{KI}_{I'l'}\del^{I'}\Omega^{l'},\quad
[\Omega^l,L^J]u = \sum_{|J'|=|J|\atop l'<l}\Theta^{KJ}_{J'l'}L^{J'}\Omega^{l'}u
\end{equation}
where $\Gamma^{JI}_{I'J'}, \Delta^{KI}_{I'l'}, \Theta^{KJ}_{J'l'}$ are constants. These can be checked by induction on $|I|,|J|$ and $|K|$.

\begin{proof}[Sketch of proof of Proposition \ref{prop1-23-10-2021}]
Let $Z$ be a high-order operator composed by vectors in $\{\del_{\alpha},L_a\}$, $\ord(Z) = p,\rank(Z) = k$. Then
\begin{equation}\label{eq1-23-10-2021}
Z = \sum_{|I| = p-k\atop |J| = k}\Gamma(Z)_{IJ}\del^IL^J.
\end{equation}
This is established in \cite{M-2020-strong} ((B.2) in Appendix B) in $\Hcal^*_{[s_0,\infty)}$. It can be checked directly that, the proof therein remains valid in $\Fcal_{[s_0,\infty)}$. Then Let $Z$ be a high-order operator composed by vectors in $\{\del_{\alpha},L_a,\Omega\}$. By \eqref{eq1-23-10-2021}, $Z$ can be written as a finite linear combination of the following terms with constant coefficients:
$$
\del^{I_1}L^{J_1}\Omega^{l_1}\del^{I_2}L^{J_2}\Omega^{l_2}\cdots \del^{I_m}L^{J_m}\Omega^{l_m}
$$
with $\sum_{j=1}^m|I_j| = |I|,\sum_{j=1}|J_j| \leq |J|$ and $\sum_{j=1}^ml_j = l$. Then we commute $\Omega^{l_j}$ with $\del^{I_{j+1}}L^{J_{j+1}}$. Thanks to \eqref{eq2-22-10-2021}, we obtain the desired result.
\end{proof}

The operators $\del^IL^J\Omega^l$ are called {\bf ordered} operators and they will play a special role in the following discussion.  

\subsection{Basic calculus in $\TPcal^{\near}_{[s_0\infty)}$}
In this Subsection we firstly recall some basic rules of calculation near the light cone. These are established in \cite{LM-2022}.
\begin{definition} 
	A smooth function $f$ defined in the domain $\big\{ r \geq t/2, \, t > 0 \big\}$ is called {\bf exterior-homogeneous\footnote{When there is no ambiguity, we will simply write `homogeneous'.}
		of degree $k$}
	if it satisfies ($\mathbb{S}^1 \subset \RR^2$ denoting the $1$-sphere)
	$$
	\aligned
	f(\lambda t, \lambda x) & =   \lambda^{k} f(t, x),
	& \lambda>0; 
	\qquad
	| \del^I f(t, \omega)|& \leq  C \, (I),   & \omega\in \mathbb{S}^1, \quad 0 < t\leq 2.
	\endaligned
	$$
	Similarly, a smooth function $g$ in $\big\{r\leq 3t, t\geq 1\big\}$ is called interior-homogeneous of degree $k$, if
	$$	
	\aligned
	f(\lambda t, \lambda x) & =   \lambda^{k} f(t, x),
	& \lambda>0; 
	\qquad
	| \del^I f(1, x)|& \leq  C \, (I),   & |x|\leq 3.
	\endaligned
	$$
\end{definition}
These coefficients enjoy the following properties.
\begin{lemma}
	\label{lem 1 homo-ext}
	Let $f$ and $g$ be exterior/interior-homogeneous functions of degree $m$ and $n$, respectively.  
	\begin{itemize} 	
		\item When $m=n$, then $\alpha f + \beta g$ is exterior/interior-homogeneous of degree $m$ (for any reals $\alpha,\beta$).
		
		\item The product $fg$ is exterior/interior-homogeneous of degree $(m+n)$.
		
		\item $Z f$ is exterior/interior-homogeneous of degree $m+k-p$ with $\ord(Z)= p$ and $\rank(Z) = k$ and, moreover, 
		it holds $|Z f| \lesssim C(Z) r^{m+k-p}$ in $\TPcal_{[s_0, + \infty)}$.		
	\end{itemize} 
\end{lemma}
Then we regard a special weight function. 
\begin{lemma}\label{lem1-31-05-2021}
	In the region $\TPcal^{\near}_{[s_0,+\infty)}$, the following estimate holds 
	(with implied constants determined by the multi-indices $I,J$): 
	\begin{equation}\label{eq1-31-05-2021}
	\big|\del^IL^J (r-t)\big| \leq
	\begin{cases}
	 C(J)|r-t|,\quad  & |I| = 0,
	\\
	C(I,J)r^{-|I|+1}\leq|r-t|+1,\quad  & |I| \geq 1.
	\end{cases}
	\end{equation}
\end{lemma}
This is the Lemma 5.1 of \cite{M-2018}. The proof is by induction on $|J|$ and $|I|$. We omit the detail.

\subsection{Proof of Proposition \ref{prop2-23-10-2021}}
Remark that in $\TPcal^{\near}_{[s_0,\infty)}$, $t^{-1}$ is interior-homogeneous and $(x^a/r)$ is exterior-homogeneous. Furthermore, $t-1\leq r\leq 3t$. We will firstly establish the following bound.
\begin{lemma}
In $\TPcal^{\near}_{[s_0,\infty)}$
\begin{equation}\label{eq3-23-10-2021}
|\delt u|_{p,k}
\lesssim C(p)  \sum_{\ord(Z)\leq p \atop \rank(Z)\leq k}\Big(\sum_a|\delt_a  Zu| + \frac{|r-t|+1}{r}\sum_{\alpha}|\del_\alpha Z u|\Big).
\end{equation} 
\end{lemma}
\begin{proof}[Sketch of proof]
We need to bound $Z\delt_au$. Let $Z$ be a high-order operator with $\ord(Z) = p$ and $\rank(Z) = k$. Thanks to Proposition \ref{prop1-23-10-2021}, we only need to consider $\del^IL^J\Omega^l \delt_a u$
Then we remark the following identity:
\begin{equation}\label{eq2-23-10-2021}
\delt_a u = t^{-1}L_au - (x^a/r)(r-t)t^{-1}\del_t u.
\end{equation}
Remark that in $\TPcal^{\near}_{[s_0,\infty)}$, $|\del^IL^J\Omega^K(t^{-1})|\leq C(I,J) t^{-1}$ (this is due to the homogeneity of $t^{-1}$ in the region $\{t/2\leq r\leq 3t\}$), and $|\del^IL^J\Omega^l(x^a/r)|\leq C(I,J,K)$ (because $x^a/r$ is homogeneous of degree zero). Thanks to \eqref{eq1-31-05-2021}, $|\del^IL^J\Omega^l(r-t)|\leq C(I,J)(|r-t|+1)$. Then  (remark that $\Omega t = \Omega r = 0$), 
\begin{equation}\label{eq1-26-10-2021}
\aligned
\del^IL^J\Omega^l\delt_au =& \sum_{I_1+I_2=I\atop J_1+J_2=J}
\!\!\!\!\del^{I_1}L^{J_1}\big(t^{-1}\big)\del^{I_2}L^{J_2}\Omega^lL_a u 
\\
&+ \sum_{I_1+I_2=I,J_1+J_2=J\atop l_1+l_2=l}
\!\!\!\!\del^{I_1}L^{J_1}\Omega^{l_1}\big(t^{-1}(x^a/t)(r-t)\big)\del^{I_2}L^{J_2}\Omega^{l_2}\del_t u 
\endaligned
\end{equation}
For the first term on the right-hand side, remark that $\big|\del^{I_1}L^{J_1}\big(t^{-1}\big)\big|\leq C(p)t^{-1}$. When $|I_2|\geq 1$, this term is bounded by $t^{-1}|\del_{\alpha}Z'u|$ with $\ord(Z')\leq p$ and $\rank(Z')\leq k$, thus bounded by the right-hand side of \eqref{eq3-23-10-2021}. When $|I_2|=0$ and $|J_2|\geq 1$, 
$$
\aligned
\big|\del^IL^{J_1}(t^{-1})L^{J_2}\Omega^lu\big| \leq& C(p) \big|t^{-1}L_aL^{J_2'}\Omega^lu\big|
\leq C(p) |\delu_aL^{J_2'}\Omega^{l_2}u| 
\\
\leq& C(p)|\delt_aL^{J_2'}\Omega^l u| + C(p)t^{-1}|r-t||\del_tL^{J_2'}\Omega^lu|.
\endaligned
$$
In the region $\TPcal^{\near}$ one has $t\simeq r$. Thus the above terms are bounded by the right-hand side of \eqref{eq3-23-10-2021}.

For the second term on the right-hand side of \eqref{eq1-26-10-2021}, we remark that in $\TPcal^{\near}_{[s_0,\infty)}$
$$
\big|\del^{I_1}L^{J_1}\Omega^{l_1}\big(t^{-1}(x^a/t)(r-t)\big)\big|\leq C(p)r^{-1}(|r-t|+1)
$$
where \eqref{eq1-31-05-2021} and the homogeneity of $(x^a/t^2)$ in $\TPcal^{\near}$ are applied. The term $\del^{I_2}L^{J_2}\Omega^{l_2}\del_t u $ is bounded via Proposition \ref{prop1-23-10-2021} by $|\del_\alpha Z'u|$ with
$\ord{Z}\leq p, \rank(Z)\leq k$.
\end{proof}

\paragraph{Proof of \eqref{eq2-26-10-2021}}
The bound on $\zeta|\del u|_{p,k}$ is a direct result of Proposition \ref{prop1-23-10-2021} together with the expression of $\ebf[u]$. For the bound on $|\delt u|_{p,k}$, we apply \eqref{eq3-23-10-2021} together with Lemma \ref{lem1-26-10-2021}.

\section{Proofs of Sobolev inequalities}\label{sec1-23-10-2021}
\subsection{Sobolev inequality in positive cone}

For the convenience of discussion, we recall the notion for $\rho>0$,
$$
C_{x_0,\rho} := \{x_0^a\leq x^a\leq \rho\},\quad C_{x_0,-\rho} :=\{x_0^a-\rho\leq x^a\leq x^a\} 
$$
and $\RR^2_+ = C_{0,+\infty} = \{0\leq x^a\}$. Then we recall the following result parallel to the $3-D$ case in \cite{LM-2022}.
\begin{proposition}\label{prop1-16-10-2021}
	Let $u$ be a $C^2$ function defined in $C_{0,\rho} = \{0\leq x^a\leq \rho\}$. Then
	\begin{equation}\label{eq1-16-10-2021}
	|u(0,0)|\leq C(\rho)\sum_{|I|\leq 2}\|\del^Iu\|_{L^2(C_{0,\rho})}. 
	\end{equation}
\end{proposition}
The proof relies on the following two lemmas.
\begin{lemma}[Sobolev-Poincar\'e inequality in positive cone]
	Let $v$ be a $C^1$ function defined in $\RR^2_+$ and $v(x)\equiv 0$ for $|x|\geq \rho$. Then $\forall\, p \geq  2$, 
	\begin{equation}\label{eq3-16-10-2021}
	\|v\|_{L^p(C_{0,\rho})}\leq (p/2)\rho^{p/2}\,\|\del v\|_{L^2(C_{0,\rho})}.
	\end{equation}
\end{lemma}
\begin{proof}
	This is a slight modification of the classical Sobolev's inequality on compactly supported functions. We introduce
	$$
	w^1(x^2) = \sup_{x^1\geq 0} \big\{|v(x^1,x^2)|^{p/2}\big\},\quad 
	w^2(x^1) = \sup_{x^2\geq 0} \big\{|v(x^1,x^2)|^{p/2}\big\}.
	$$
	$w^a$ are defined on $[0,\infty)$ and $w^a(\tau)\equiv 0$ when $\tau\geq \rho$. On the other hand, remark that $\del_a|v|^{p/2} = (p/2)|v|^{p/2-2}v\del_a v$. Then
	$$
	\aligned
	|v(x^1,x^2)|^{p/2} \leq &\int_0^\infty \big|\del_1|v|^{p/2}(\rho,x^2)\big|d\rho 
	\leq (p/2)\int_0^{\infty}|v(\rho,x^2)|^{p/2-1}|\del_1v(\rho,x_2)|d\rho,
	\\
	|v(x^1,x^2)|^{p/2} \leq &\int_0^\infty \big|\del_2|v|^{p/2}(x^1,\rho)\big|d\rho 
	\leq (p/2)\int_0^{\infty}|v(x^1,\rho)|^{p/2-1}|\del_2v(x^1,\rho)|d\rho.
	\endaligned
	$$
	These lead to
	$$
	\aligned
	\int_0^\infty w^1(x^2)dx^2 \leq& (p/2) \|\del_1 v\|_{L^2(\RR^2_+)}\Big(\int_{\RR^2_+}|v|^{p-2}dx\Big)^{1/2},
	\\
	\int_0^\infty w^2(x^1)dx^1 \leq& (p/2) \|\del_1 v\|_{L^2(\RR^2_+)}\Big(\int_{\RR^2_+}|v|^{p-2}dx\Big)^{1/2}.
	\endaligned
	$$
	Multiply the above two inequalities, we have
	\begin{equation}\label{eq2-16-10-2021}
	\int_{\RR^2_+}w^1(x^2)w^2(x^1)dx^1dx^2\leq (p/2)^2\|\del v\|_{L^2(\RR^2_+)}^2\int_{\RR^2_+}|v|^{p-2}dx
	\end{equation}
	For the right-hand side, we remark that $v$ is supported in $[0,\rho]\times[0,\rho]$. Thus by H\"order's inequality
	$$
	\int_{\RR^2_+}|v|^{p-2}dx\leq \rho^{4/p}\|v\|_{L^p(\RR^2_+)}^{p-2}.
	$$
	For the left-hand-side we remark that
	$$
	\|v\|_{L^p(\RR^2_+)}^p = \int_{\RR^2_+}|v(x^1,x^2)|^pdx\leq \int_{\RR^2_+}w^1(x^2)w^2(x^1)dx^1dx^2.
	$$
	Then by \eqref{eq2-16-10-2021} we obtain \eqref{eq3-16-10-2021}.	
\end{proof}
\begin{lemma}
	Let $v$ be a $C^1$ function defined in $\RR^2_+$. Then
	\begin{equation}\label{eq4-16-10-2021}
	|v(0)|\leq C(q,\rho)\|v\|_{H^1(C_{0,\rho})},\quad \forall\, q>2.
	\end{equation}
\end{lemma}
\begin{proof}
	This is also parallel to the classical Morrey's inequality. For $\rho> 0$, we remark that $\forall x\in C_{0,\rho}$,
	$$
	|v(x) - v(0)|\leq\int_0^1\Big|\frac{d}{dt}v(tx)\Big|dt\leq \rho\int_0^1\sum_a|\del_a v(tx)|dt.
	$$
	Then 
	$$
	\aligned
	\Big|\rho^{-2}\int_{C_{0,\rho}}v(x)dx - v(0)\Big|\leq& \rho^{-2}\int_{C_{0,\rho}}|v(x) - v(0)|dx
	\leq \rho^{-1}\int_0^1\int_{C_{0,\rho}}\sum_a|\del_a v(tx)|dx\,dt
	\\
	=&\rho^{-1}\sum_a\int_0^1 t^{-2}dt\int_{C_{0,t\rho}}|\del_a v(x)|dx.
	\endaligned
	$$
	On the other hand
	$$
	\int_{C_{0,t\rho}}|\del_a v(x)|dx\leq \|\del_a v\|_{L^q(C_{0,t\rho})}\|1\|_{L^{q'}(C_{0,t\rho})} 
	\leq (t\rho)^{2-2/q}\|\del_a v\|_{L^q(C_{0,\rho})}.
	$$
	Then we arrive at (remark that $q>2$)
	\begin{equation}\label{eq1-17-10-2021}
	\aligned
	\Big|\rho^{-2}\int_{C_{0,\rho}}v(x)dx - v(x_0)\Big|
	\leq \sum_a\rho^{1-2/q}\|\del_a v\|_{L^q(C_{0,\rho})}\int_0^1t^{-2/q}dt
	\\
	\leq C(q)\rho^{1-2/q}\sum_a\|\del_a v\|_{L^q(C_{0,\rho})}.
	\endaligned
	\end{equation}
	
	Let $x_0 = (\rho,\rho)$. Then with similar argument one has
	\begin{equation}\label{eq2-17-10-2021}
	\Big|\rho^{-2}\int_{C_{0,\rho}}v(x)dx - v(0)\Big|\leq C(q)\rho^{1-2/q}\sum_a\|\del_a v\|_{L^q(C_{0,\rho})}.
	\end{equation}
	Here we need to remark that $\forall x\in C_{0,\rho}$, 
	$$
	|v(x) - v(x_0)|\leq \int_0^1\Big|\frac{d}{dt}v(t(x_0-x) + x)\Big|dt 
	\leq \rho\int_0^1\sum_a|\del_av(t(x_0-x)+x)|dt.
	$$
	Then 
	$$
	\aligned
	\Big|\rho^{-2}\int_{C_{0,\rho}}v(x)dx - v(x_0)\Big|\leq& \rho^{-2}\int_{C_{0,\rho}}|v(x) - v(x_0)|dx
	\\
	\leq& \rho^{-1}\int_0^1\int_{C_{0,\rho}}\sum_a|\del_a v(t(x_0-x) + x)|dx\,dt
	\\
	=&\rho^{-1}\sum_a\int_0^1 (1-t)^{-2}dt\int_{C_{x_0,(t-1)\rho}}\!\!\!\!\!\!|\del_a v(x)|dx.
	\endaligned
	$$
	Furthermore,
	$$
	\aligned
	\int_{C_{x_0,(1-t)\rho}}\!\!\!\!\!\!|\del_a v(x)|dx
	\leq \|\del_a v\|_{L^q(C_{x_0,(t-1)\rho})}\|1\|_{L^{q'}(C_{x_0,(t-1)\rho})}
	\leq (1-t)^{2-2/q}\rho^{2-2/q}\|\del_av\|_{L^q(C_{0,\rho})}.
	\endaligned
	$$
	Then, remark that $q>2$, 
	$$
	\aligned
	\Big|\rho^{-2}\int_{C_{0,\rho}}v(x)dx - v(x_0)\Big|
	\leq& \rho^{1-q/2}\sum_a\|\del_a v\|_{L^q(C_{0,\rho})}\int_0^1(1-t)^{-2/q}dt 
	\\
	\leq& C(q)\rho^{1-q/2}\sum_a\|\del_a v\|_{L^q(C_{0,\rho})}
	\endaligned
	$$
	which leads to \eqref{eq2-17-10-2021}. Combining \eqref{eq2-17-10-2021} with \eqref{eq1-17-10-2021}, one has
	\begin{equation}\label{eq3-17-10-2021}
	|v(x_0) - v(0)|\leq C(q) \rho^{1-2/q}\sum_a\|\del_a v\|_{L^q(C_{0,\rho})}.
	\end{equation}
	
	Now let $v$ be a $C^1$ function defined in $\RR^2_+$. Then $w(x) := \chi(1-\rho^{-1}|x|)v(x)$ is again a $C^1$ function defined on $\RR^2_+$ with $w(\rho,\rho) = 0$. We apply \eqref{eq3-17-10-2021} on $w$ with $x_0 = (\rho,\rho)$. Remark that
	$$
	\del_a w(x) \leq C(\rho^{-1}|v(x)| + |\del_a v(x)|),
	$$
	then we conclude by \eqref{eq4-16-10-2021}.
\end{proof}

\begin{proof}[Proof of Proposition \ref{prop1-16-10-2021}]
	We consider the function 
	$$
	v(x) := \chi(1-\rho^{-1}|x|)u(x).
	$$
	It is clear that $v(0) = u(0)$ and $\text{supp}(v)\subset \{|x|\leq \rho\}\cap \RR^2_{+}$. Thanks to \eqref{eq3-16-10-2021}, 
	$$
	\aligned
	\sum_{a,b}\|\del_a v\|_{L^3(C_{0,\rho})}\leq& C(\rho)\sum_b\|\del_b\del_a v\|_{L^2(C_{0,\rho})}
	\leq C(\rho)\sum_{|I|\leq 2}\|\del^I u\|_{L^2(C_{0,\rho})},
	\\
	\|v\|_{L^3(\RR^2_+)}\leq& C(\rho)\sum_a\|\del_a v\|_{L^2(C_{0,\rho})}
	\leq C(\rho)\sum_{|I|\leq 1}\|\del^I u\|_{L^2(C_{0,\rho})}.
	\endaligned
	$$
	Then by \eqref{eq4-16-10-2021} applied on $v$ with $q=3$, we obtain the desired result.

\end{proof}

\subsection{Proof of Proposition \ref{prop1-21-10-2021}}
\begin{proof}
	Compared with the version introduced in \cite{Ho1}, there two difference. The first is that here $u$ need not to vanish near $\{r=t+1\}$. The second is that on the right-hand side we only need the norms on $L^J u$. Let
	$$
	v_s(x) := u\big((s^2+r^2)^{1/2, x}\big)
	$$
	be the restriction of $u$ on $\Hcal^*_s$. $v_s$ is defined for $|x|\leq \rhoH(s) = (s^2-1)/2$. Remark that $\del_a v_s(x) = \delu_a u\big((s^2+r^2)^{1/2},x\big)$. Without loss of generality, we fix $x_0 = 2^{-1/2}(-r_0,-r_0)$ with $r_0 = |x_0|$. 
	
	When $r_0\geq (1/\sqrt{3}) s$, one has $r_0\geq (1/2)\sqrt{s^2+r_0^2} = t_0/2$. Then let
	$$
	w_s(y) := v_s(x_0 + r_0y).
	$$
	Remark that when $0\leq y^a\leq 1/2\sqrt{2}$, 
	$$
	r_0\leq |x_0 + r_0y|\leq r_0/2.
	$$
	Then $w_s$ is well defined in $C_{0,1/2\sqrt{2}}$.
	
	Remark that 
	\begin{subequations}\label{eq1-19-10-2021}
		\begin{equation}\label{eq1a-19-10-2021}
		\del_a w_s(y) = r_0\del_av_s(x_0+r_0y) = r_0\delu_au(t,x)
		\end{equation}
		with $x = x_0 + r_0y$ and $t = \sqrt{s^2+|x|^2}$, and 
		\begin{equation}\label{eq1b-19-10-2021}
		\del_a\del_bw_s(y) = r_0^2\delu_b\delu_a u(t,x).
		\end{equation}
	\end{subequations}
	Remark that when $y\in C_{0,1/2}$, 
	$$
	t^2 = s^2+|x|^2\leq s^2+r_0^2\leq 4r_0^2.
	$$
	Then for $|I|\leq 2$, \eqref{eq1-19-10-2021} lead to 
	\begin{equation}\label{eq2-19-10-2021}
	|\del^I w_s(y)|\lesssim \sum_{|J|\leq 2}|L^J u(t,x)|.
	\end{equation}
	By \eqref{eq1-16-10-2021}, one has
	$$
	\aligned
	|u(\sqrt{s^2+r_0^2},x_0)|^2 =& |w_s(0)|^2\lesssim C\sum_{|I|\leq 2}\int_{C_{0,1/2\sqrt{2}}}\!\!\!\!|\del^Iw_s(y)|^2 d y
	\\
	\leq& C\sum_{|J|\leq 2}\int_{C_{0,1/2\sqrt{2}}}\!\!\!\!|L^Ju\big(\sqrt{s^2+|x|^2}, x\big)|^2dy 
	\\
	=& Cr_0^{-2}\sum_{|J|\leq2}\!\!\!\!\int_{C_{x_0,r_0/2\sqrt{2}}}|L^Ju(\sqrt{s^2+|x|^2},x)|^2dx,
	\endaligned
	$$
	which leads to (thanks to the fact that $t_0 = \sqrt{s^2+r_0^2} \leq 2r_0$)
	$$
	\big|u\big(t_0,x_0\big)\big|\lesssim t_0\sum_{|J|\leq 2}\|L^J u\|_{L^2(\Hcal_s^*)}.
	$$
	
	When $r_0\leq (1/\sqrt{3}) s$, one has $t_0 = \sqrt{s^2+r_0^2}\leq (2/\sqrt{3})s$. 
	$$
	w_s(y) := v_s(x_0 + sy) = u\big(\sqrt{s^2+|x|^2},x\big),\quad x=x_0+sy. 
	$$
	Remark that for $0\leq y^a\leq 1/2\sqrt{2}$, one has (for $s\geq 2$)
	$$
	|x|\leq s/2\leq \rhoH(s) = (s^2-1)/2 .
	$$
	Thus $w_s$ is a $C^2$ function defined on $C_{0,1/2\sqrt{2}}$. Remark that
	$$
	\del_aw_s(y) = s\del_av_s(x_0+sy) = s\delu_au(t,x),
	\quad
	\del_b\del_aw_s(y) = s^2\delu_b\delu_au(t,x).
	$$
	where $x = x_0+sy, t = \sqrt{s^2-|x|^2}$. Also remark that $t_0^2 = s^2+r_0^2\leq (4/3)s^2$, with a similar argument  we obtain 
	$$
	|u(t_0,x_0)| = |w_s(0)|\lesssim s^{-1}\sum_{|J|\leq 2}\|L^J u\|_{L^2(\Hcal^*_s)}
	\leq t_0^{-1}\sum_{|J|\leq 2}\|L^J u\|_{L^2(\Hcal^*_s)}.
	$$
	This concludes the desired result.
\end{proof}

\subsection{Proof of Proposition \ref{prop2-21-10-2021}}
\begin{proof}
	We firstly introduce 
	$$
	v_s(x) := u\big(T(s,|x|),x\big)
	$$
	the restriction of $u$ on $\TPcal_s$. Then $v_s$ is defined for $|x|\geq \rhoH(s)$. 
	
	We consider a point $x_0$ with $|x_0| = r_0\geq \rhoH(s)$. Without loss of generality, we fix $x = 2^{-1/2}(r_0,r_0)$. We introduce
	$$
	w_s(\rho,\theta) = v_s(x) = u(t,x)
	$$
	with $x = \big((r_0 + \rho)\cos(\theta+\pi/4), (r_0+\rho)\sin(\theta+\pi/4)\big)$ and $t = T(s,|x|)$. Remark that when $(\rho,\theta)\in [0,\pi/4]^2$, $|x| \geq \rhoH(s)$. Thus $w_s$ is a $C^2$ function defined in $[0,\pi/4]^2$. Furthermore,
	$$
	\aligned
	\del_\rho w_s(\rho,\theta) =& \del_rv_s(x) =  \delb_ru(t,x),
	\\
	\del_\theta w_s(\rho,\theta) =&-(r_0+\rho)\sin(\theta+\pi/4)\del_1v_s(x) + (r_0+\rho)\cos(\theta+\pi/4)\del_2v_s(x)
	\\
	=& x^1\delb_2u(t,x) - x^2\delb_1u(t,x) = \Omega u(t,x).
	\endaligned 
	$$
	In the same manner, 
	$$
	\aligned
	\del_{\rho}\del_{\rho} w_s(\rho,\theta) =& \delb_r\delb_r u(t,x),\quad
	\del_{\theta}\del_{\rho} w_s(\rho,\theta) = \del_{\rho}\del_{\theta}w_s(\rho,\theta) = \delb_r\Omega u(t,x),
	\\
	\del_{\theta}\del_{\theta}w_s(\rho,\theta) =& \Omega^2u(t,x).
	\endaligned
	$$
	Then for $|I|\leq 2$,
	$$
	|\del^I w_s(r,\theta)|\leq C\sum_{k+|J|\leq 2} |\delb_r^k\Omega^J u(t,x)|. 
	$$
	Then by \eqref{eq1-16-10-2021},
	$$
	\aligned
	|u(t_0,x_0)|^2 =& |w_s(0,0)|^2 \leq C\sum_{|I|\leq 2}\int_{C_{0,\pi/4}}\!\!\!\!|\del^I w_s|^2d\rho d\theta
	\leq C\sum_{k+|J|\leq 2}\int_{D}r^{-1}|\delb_r^k\del^J u(t,x)|^2 rdrd\theta,
	\endaligned
	$$
	where $r = r_0+\rho$ and 
	$$
	D = \big\{\big(r\cos(\theta+\pi/4),r\sin(\theta + \pi/4)\big)|r_0 \leq r\leq r_0+\pi/4, 0\leq \theta\leq \pi/4\big\}.
	$$
	Here remark that in $D$, $r_0\leq r\leq r_0+\pi/4$ which leads to 
	$r^{-1}\leq r_0^{-1}$. Then
	\begin{equation}\label{eq4-19-10-2021}
	|u(t_0,x_0)|^2\leq Cr_0^{-1}\sum_{k+|J|\leq 2}\|\delb_r^k\del^J u\|_{L^2(\TPcal_s)}^2.
	\end{equation}
	This is the non-weighted version of \eqref{eq5-17-10-2021}.

	For the weighted version, recall that $\Omega \omega^\eta = 0$ and
	$$
	\delb_r\omega^\eta = \eta\omega^{\eta-1}\aleph^{\prime}(2+r-t)\Big(1 - \frac{\xi(s,r)r}{(s^2+r^2)^{1/2}}\Big)
	$$
	thus
	\begin{equation}\label{eq3-19-10-2021}
	\big|\delb_r\omega^{\eta}\big|\leq \eta\omega^{\eta-1}\zeta\leq \eta\omega^{\eta}.
	\end{equation}
	Then apply \eqref{eq4-19-10-2021} on $\omega^\eta u(t,x)$ and remark that
	\begin{equation}\label{eq5-19-10-2021}
	\big|\delb_r^k\Omega^J \omega^{\eta}u\big| = \sum_{k_1+k_2=k}\big|\delb_r^{k_1}\omega^{\eta} \delb_r^{k_2}\Omega^J u\big|
	\leq (|\eta|+1)\sum_{k'\leq k}\omega^{\eta} |\delb_r^{k'}\Omega^J u|.
	\end{equation}
	This leads to \eqref{eq5-17-10-2021}. The proof of \eqref{eq5-22-01-2022} is quite similar. We only need to consider the restriction of $u$ on $\Pcal_s$ and remark that in this case $\delb_r = \del_r$. 
\end{proof}

\bibliographystyle{elsarticle-num}
\bibliography{WKG-09-01-2022}

\begin{thebibliography}{10}
\expandafter\ifx\csname url\endcsname\relax
  \def\url#1{\texttt{#1}}\fi
\expandafter\ifx\csname urlprefix\endcsname\relax\def\urlprefix{URL }\fi
\expandafter\ifx\csname href\endcsname\relax
  \def\href#1#2{#2} \def\path#1{#1}\fi

\bibitem{Shatah-1998}
J.~Shatah, M.~Struwe, Geometric wave equations, Courant Institute of
  Mathematical Sciences, 2000.

\bibitem{Kri07}
J.~Krieger, Global regularity and singularity development for wave maps,
  Surveys in differential geometry 12~(1) (2007) 167--202.

\bibitem{Ab-2019}
L.~Abbrescia, Y.~Chen, Global stability of some totally geodesic wave maps, J.
  Differ. Equations 284~(2) (2021) 219--252.
\newblock \href {https://doi.org/10.1016/j.jde.2021.02.056}
  {\path{doi:10.1016/j.jde.2021.02.056}}.

\bibitem{Gray-2003}
A.~Gray, Tubes, Birkhäuser Verlag, 2003.

\bibitem{Duan-Ma-2020}
S.Duan, Y.~Ma, Global solutions of wave-klein-gordon system in two spatial
  dimensions with strong couplings in divergence form, SIAM J. MATH. ANAL .
  54~(3).

\bibitem{Dong-Wyatt-2021}
S.~Dong, Z.~Wyatt, Stability of some two dimensional wave maps,
  arXiv:2103.05318v1.

\bibitem{Chrsitodoulou-Shadi-1993-1}
D.~Christodoulou, A.~S. Tahvildar-Zadeh, On the regularity of spherically
  symmetric wave maps, Communications on Pure and Applied Mathematics 46 (1993)
  1041 -- 1091.
\newblock \href {https://doi.org/10.1002/cpa.3160460705}
  {\path{doi:10.1002/cpa.3160460705}}.

\bibitem{Chrsitodoulou-Shadi-1993-2}
D.~Christodoulou, A.~S. Tahvildar-Zadeh, On the asymptotic behavior of
  spherically symmetric wave maps, Duke Mathematical Journal 71 (07 1993).
\newblock \href {https://doi.org/10.1215/S0012-7094-93-07103-7}
  {\path{doi:10.1215/S0012-7094-93-07103-7}}.

\bibitem{Muller-Struwe-1996}
S.~M\"uller, M.~Struwe, Global existence of wave maps in 1 + 2 dimensions with
  finite energy data, Topological methods in nonlinear analysis 7 (06 1996).
\newblock \href {https://doi.org/10.12775/TMNA.1996.011}
  {\path{doi:10.12775/TMNA.1996.011}}.

\bibitem{Freire-1997}
A.~Freire, S.~Müller, M.~Struwe, Weak convergence of wave maps from
  (1+2)-dimensional minkowski space to riemannian manifolds, Invent. Math. 130
  (1997) 589--617.
\newblock \href {https://doi.org/10.1007/s002220050195}
  {\path{doi:10.1007/s002220050195}}.

\bibitem{Tao-2000}
T.~Tao, Global regularity of wave maps {II}. small energy in two dimensions,
  Comm. Math. Phys. 224 (11 2000).
\newblock \href {https://doi.org/10.1007/PL00005588}
  {\path{doi:10.1007/PL00005588}}.

\bibitem{Klainerman-Selberg-2001}
S.~Klainerman, S.~Selberg, Remark on the optimal regularity for equations of
  wave maps type, Communications in Partial Differential Equations 22 (02
  2001).
\newblock \href {https://doi.org/10.1080/03605309708821288}
  {\path{doi:10.1080/03605309708821288}}.

\bibitem{Tataru-2001}
D.~Tataru, On global existence and scattering for the wave maps equation,
  American Journal of Mathematics 123 (02 2001).
\newblock \href {https://doi.org/10.1353/ajm.2001.0005}
  {\path{doi:10.1353/ajm.2001.0005}}.

\bibitem{Kri-2004}
J.~Krieger, Global regularity of wave maps from ${\bf r}^{2+1}$ to ${\bf h}^2$.
  small energy, Communications in Mathematical Physics 250 (10 2004).
\newblock \href {https://doi.org/10.1007/s00220-004-1088-5}
  {\path{doi:10.1007/s00220-004-1088-5}}.

\bibitem{Tataru-2005}
D.~Tataru, Rough solutions for the wave maps equation, American Journal of
  Mathematics - AMER J MATH 127 (04 2005).
\newblock \href {https://doi.org/10.1353/ajm.2005.0014}
  {\path{doi:10.1353/ajm.2005.0014}}.

\bibitem{Sterbenz-Tataru-2009}
J.~Sterbenz, D.~Tataru, Regularity of wave-maps in dimension 2 + 1, Comm. Math.
  Phys. 298 (07 2009).
\newblock \href {https://doi.org/10.1007/s00220-010-1062-3}
  {\path{doi:10.1007/s00220-010-1062-3}}.

\bibitem{Kir-2009}
J.~Krieger, W.~Schlag, Concentration compactness for critical wave maps (08
  2009).
\newblock \href {https://doi.org/10.4171/106} {\path{doi:10.4171/106}}.

\bibitem{Sterbenz-Tataru-2010}
J.~Sterbenz, D.~Tataru, Energy dispersed large data wave maps in 2 + 1
  dimensions, Comm. Math. Phys. 298 (06 2009).
\newblock \href {https://doi.org/10.1007/s00220-010-1061-4}
  {\path{doi:10.1007/s00220-010-1061-4}}.

\bibitem{Wang-Yu-2012}
J.~Wang, P.~Yu, Long time solutions for wave maps with large data, Journal of
  Hyperbolic Differential Equations 10 (07 2012).
\newblock \href {https://doi.org/10.1142/S0219891613500136}
  {\path{doi:10.1142/S0219891613500136}}.

\bibitem{Lawrie-Oh-2016}
A.~Lawrie, S.~Oh, A refined threshold theorem for (1 + 2)-dimensional wave maps
  into surfaces, Communications in Mathematical Physics 342 (2016) 989--999.

\bibitem{Gavrus-Jiao-Tataru-2021}
C.~Gavrus, C.~Jao, D.~Tataru, Wave maps on (1+2)-dimensional curved spacetimes,
  Analysis \& PDE 14 (2021) 985--1084.
\newblock \href {https://doi.org/10.2140/apde.2021.14.985}
  {\path{doi:10.2140/apde.2021.14.985}}.

\bibitem{Grigoryan2010}
V.~Grigoryan, Stability of geodesic wave maps in dimensions $d \geq 3$, Int.
  Math. Res. Not. 23 (2010) 4544 -- 4584.

\bibitem{Sideris1989}
T.-C. Sideris, Global existence of harmonic maps in {M}inkowski space, Commun.
  Pure Appl. Math. 42~(1) (1989) 1 -- 13.

\bibitem{Sunagawa-2003}
H.~Sunagawa, On global small amplitude solutions to systems of cubic nonlinear
  {K}lein–{G}ordon equations with different mass terms in one space
  dimension, Journal of Differential Equations 192 (2003) 308--325.
\newblock \href {https://doi.org/10.1016/S0022-0396(03)00125-6}
  {\path{doi:10.1016/S0022-0396(03)00125-6}}.

\bibitem{Tsutsumi-2003-2}
Y.~Tsutsumi, Stability of constant equilibrium for the {M}axwell–{H}iggs
  equations, Funkcial. Ekvac. 46 (2003) 41--62.

\bibitem{Katayama-Ozawa-Sunagawa-2012}
S.~Katayama, T.~Ozawa, H.~Sunagawa, A note on the null condition for quadratic
  nonlinear {K}lein-{G}ordon systems in two space dimensions, Communications on
  Pure and Applied Mathematics 65 (09 2012).
\newblock \href {https://doi.org/10.1002/cpa.21392}
  {\path{doi:10.1002/cpa.21392}}.

\bibitem{Dfx}
J.-M. Delort, {D. Fang and R. Xue}, Global existence of small solutions for
  quadratic quasilinear {K}lein-{G}ordon systems in two space dimensions, J.
  Funct. Anal. 211~(2) (2004) 288--323.
\newblock \href {https://doi.org/10.1016/j.jfa.2004.01.008}
  {\path{doi:10.1016/j.jfa.2004.01.008}}.

\bibitem{KS-2011}
Y.~Kawahara, H.~Sunagawa, Global small amplitude solutions for two-dimensional
  nonlinear {K}lein-{G}ordon systems in the presence of mass resonance, J.
  Differ. Equations 251~(9) (2011) 2549--2567.
\newblock \href {https://doi.org/10.1016/j.jde.2011.04.001}
  {\path{doi:10.1016/j.jde.2011.04.001}}.

\bibitem{K-W-Y-2018}
{{S}. Klainerman, {Q}. Wang}, S-W.Yang, Global solution for massive
  {M}axwell-{K}lein-{G}ordon equations, Commun. Pur. Appl. Math. 73~(1) (2018).
\newblock \href {https://doi.org/10.1002/cpa.21864}
  {\path{doi:10.1002/cpa.21864}}.

\bibitem{Stingo-Huneau-2021}
C.~Huneau, A.~Stingo, Global well-posedness for a system of quasilinear wave
  equations on a product space, arXiv:2110.13982v1 [math.AP], preprint (2021).

\bibitem{M-2018}
Y.~Ma, Global solutions of nonlinear wave-{K}lein-{G}ordon system in one space
  dimension, Nonlinear Anal-Theor. 191~(4) (2020) 111641.
\newblock \href {https://doi.org/10.1016/j.na.2019.111641}
  {\path{doi:10.1016/j.na.2019.111641}}.

\bibitem{LM-2022}
P.~LeFloch, Y.~Ma, Nonlinear stability of self-gravitating massive fields,
  Prirint arXiv:1712.10045v2 (2022).

\bibitem{Stingo-2015}
A.~Stingo, Global existence and asymptotics for quasi-linear one-dimensional
  {K}lein-{G}ordon equations with mildly decaying cauchy data, Bulletin de la
  Soci\'eet\'e math\'ematique de France 146~(1) (2015).
\newblock \href {https://doi.org/10.24033/bsmf.2755}
  {\path{doi:10.24033/bsmf.2755}}.

\bibitem{Ionescu-Pausader-2017}
A.~Ionescu, B.~Pausader, On the global regularity for a {W}ave-{K}lein-{G}ordon
  coupled system, Acta Mathematica Sinica 35~(6) (2017).

\bibitem{Stingo-2018}
A.~Stingo, Global existence of small amplitude solutions for a model quadratic
  quasi-linear coupled wave-{K}lein-{G}ordon system in two space dimension,
  with mildly decaying cauchy data, arXiv:1810.10235v1.

\bibitem{Stingo-2019}
M.~Ifrim, A.~Stingo, Almost global well-posedness for quasilinear strongly
  coupled wave-{K}lein-{G}ordon systems in two space dimensions,
  arXiv:1910.12673.

\bibitem{Dong-2020}
S.~Dong, Global solution to the wave and {K}lein-{G}ordon system under null
  condition in dimension two, Journal of functional analysis 281~(11) (2021).

\bibitem{Dong-2021}
S.~Dong, Global solution to the {K}lein-{G}ordon-{Z}akharov equations with
  uniform energy bounds, SIAM Journal on Mathematical Analysis 54~(1) (2022).
\newblock \href {https://doi.org/10.1137/21M1395235}
  {\path{doi:10.1137/21M1395235}}.

\bibitem{Dong-Ma-2021}
S.~Dong, Y.~MA, Global existence and scattering of the
  {K}lein-{G}ordon-{Z}akharov system in two space dimensions, preprint
  arXiv:2111.00244v1 (2021).

\bibitem{Ionescu-Pausader-2022}
A.~Ionescu, B.~Pausader, The {E}instein-{K}lein-{G}ordon Coupled System: Global
  Stability of the {M}inkowski Solution, Vol. 213, AMS, 2022.

\bibitem{Dong-Ma-Yuan-2022-1}
S.~Dong, Y.~Ma, X.~Yuan, Asymptotic behavior of 2d wave-{K}lein-{G}ordon
  coupled system under null condition, preprint arXiv:2202.08139.

\bibitem{Dong-Wyatt-2021-2}
S.~Dong, Z.~Wyatt, Hidden structure and sharp asymptotics for the
  {D}irac–{K}lein-{G}ordon system in two space dimensions, preprint
  arXiv:2105.13780v1 [math.AP] (05 2021).

\bibitem{Dong-Wyatt-2020}
S.Dong, Z.~Wyatt, Two dimensional wave–{K}lein-{G}ordon equations with
  semilinear nonlinearities (2020).

\bibitem{Kubota-Yokoyama-2001}
K.~Kubota, K.~Yokoyama, Global existence of classical solutions to systems of
  nonlinear wave equations with different speeds of propagation, Japanese J.
  Math. 27 (2001) 113--202.

\bibitem{Ka}
S.Katayama, Global existence for coupled systems of nonlinear wave and
  {K}lein-{G}ordon equations in three space dimensions, Math. Z. 270~(1) (2012)
  487--513.
\newblock \href {https://doi.org/10.1007/s00209-010-0808-0}
  {\path{doi:10.1007/s00209-010-0808-0}}.

\bibitem{Shatah85}
J.~Shatah, Normal forms and quadratic nonlinear {K}lein-{G}ordon equations,
  Comm. Pure Appl. Math. 38 (1985) 685--696.
\newblock \href {https://doi.org/10.1002/cpa.3160380516}
  {\path{doi:10.1002/cpa.3160380516}}.

\bibitem{Ozawa-1996}
{T. Ozawa, K. Tsutaya}, Y.~Tsutsumi, Global existence and asymptotic behavior
  of solutions for the {K}lein-{G}ordon equations with quadratic nonlinearity
  in two space dimensions, Math.Z. (1996) 341--362.

\bibitem{Yagi-1994}
K.~Yagi, Normal forms and nonlinear klein-gordon equations in one space
  dimensio, Master thesis, Waseda University (1994).

\bibitem{Moriyama-1997}
K.~Moriyama, Normal forms and global existence of solutions to a class of cubic
  nonlinear {K}lein-{G}ordon equations in one space dimension, Differential and
  Integral Equations~(3) (1997).

\bibitem{Katayama-1999}
S.~Katayama, A note on global existence of solutions to nonlinear
  {K}lein-{G}ordon equations in one space dimension, J. Math. Kyoto Univ.~(2)
  (1999).

\bibitem{Georgiev-1992}
V.~Georgiev, Decay estimates for the {K}lein-{G}ordon equation, Comm. Partial
  Differential Equations  1111--1139.

\bibitem{Vil70}
J.~Vilms, Totally geodesic maps, J. Differ. Geom. 4~(1) (1970) 73--79.

\bibitem{Wu-1962}
H.-H. Wu, On the de {R}ham decomposition theorem, Illinois Journal of
  Mathematics 8 (1964) 291--311.

\bibitem{Zhou-Han-2011}
Y.~Zhou, W.~Han, Sharpness on the lower bound of the lifespan of solutions to
  nonlinear wave equations, Chin. Ann. Math.~(4) (2011) 521--526.

\bibitem{Dong-2020-2}
S.~Dong, Asymptotic behavior of the solution to the {K}lein-{G}ordon-{Z}akharov
  model in dimension two, Commun. Math. Phys. 384~(3) (2021).
\newblock \href {https://doi.org/10.1007/s00220-021-04003-3}
  {\path{doi:10.1007/s00220-021-04003-3}}.

\bibitem{M-2020-strong}
Y.~Ma, Global solutions of nonlinear wave-{K}lein-{G}ordon system in two
  spatial dimensions: {A} prototype of strong coupling case, Journal of
  Differential Equations 287~(4) (2021) 236--294.
\newblock \href {https://doi.org/10.1016/j.jde.2021.03.047}
  {\path{doi:10.1016/j.jde.2021.03.047}}.

\bibitem{LM1}
P.~LeFloch, Y.~Ma, The hyperboloidal foliation method, World Scientific, 2015.

\bibitem{Ho1}
L.~H\"ormander, Lectures on nonlinear hyperbolic differential equations,
  Springer Verlag, 1997.

\bibitem{Kl1993}
S.~Klainerman, Remark on the asymptotic behavior of the {K}lein-{G}ordon
  equation in $\mathbb{R}^{n+1}$, Commun. Pure Appl. Math. 46~(2) (1993)
  137--144.

\bibitem{Al-book1}
S.~Alinhac, Hyperbolic Partial Differential Equations an elementary
  introduction, Springer Verlag, 2009.

\bibitem{MH-2017}
H.-J. Huang, Y.~MA, A conformal-type energy inequality on hyperboloids and its
  application to quasi-linear wave equation in $\mathbb{R}^{3+1}$,
  arXiv:1711.00498v1 [math.AP].

\bibitem{Wong-2017}
W.Wong, Small data global existence and decay for two dimensional wave maps,
  arXiv:1712.07684.

\bibitem{M3}
Y.~Ma, {G}lobal solutions of non-linear wave-{K}lein-{G}ordon system in two
  space dimensions: semi-linear terms interactions, arXiv:1712.05315v1.

\bibitem{Katayama-Matsumura-Sunagawa-2015}
{S. Katayama, A. Matsumura}, H.~Sunagawa, Energy decay for systems of
  semilinear wave equations with dissipative structure in two space dimensions,
  Nonlinear Differ. Equ. Appl. (2015) 601--628.

\bibitem{Cheng-2021}
M.-G. Cheng, Global existence for systems of nonlinear wave and
  {K}lein-{G}ordon equations in two space dimensions under a kind of the weak
  null condition, preprint arXiv:2111.13339v1 (2021).

\end{thebibliography}

\end{document}